\documentclass[11pt,reqno]{amsart}
\numberwithin{equation}{section}
\usepackage{amsmath,amsthm}
\usepackage{mathptmx}  
\usepackage{multicol}
\usepackage{mathtools}
\usepackage{bbm}
\DeclareFontFamily{U}{BOONDOX-calo}{\skewchar\font=45 }
\DeclareFontShape{U}{BOONDOX-calo}{m}{n}{
  <-> s*[1.05] BOONDOX-r-calo}{}
\DeclareFontShape{U}{BOONDOX-calo}{b}{n}{
  <-> s*[1.05] BOONDOX-b-calo}{}
\DeclareMathAlphabet{\mathcalboondox}{U}{BOONDOX-calo}{m}{n}
\SetMathAlphabet{\mathcalboondox}{bold}{U}{BOONDOX-calo}{b}{n}
\DeclareMathAlphabet{\mathbcalboondox}{U}{BOONDOX-calo}{b}{n}
\usepackage{multicol}
\usepackage{pstricks}
\usepackage{float}
\usepackage{mathrsfs}
\usepackage[charter]{mathdesign}
\usepackage{tikz}
\newcommand{\mcb}[1]{{\mathcalboondox #1}}
\usepackage{diagbox}
\usepackage{booktabs}
\usepackage{enumitem}
\usepackage[sans]{dsfont} 
\usepackage{bbm}

\usepackage{changebar}
\usepackage[colorlinks]{hyperref}
\usepackage{a4wide}
\usepackage{physics}  

\usepackage{tikz}
\usetikzlibrary{arrows,shapes,automata,petri,positioning,calc}
\tikzset{
    place/.style={
        circle,
        thick,
        draw=black,
        fill=gray!50,
        minimum size=20mm,
    },
        state/.style={
        circle,
        thick,
        draw=blue!75,
        fill=blue!20,
        minimum size=20mm,
    },
}

\tikzset{
    cross/.pic = {
    \draw[rotate = 45] (-0.2,0) -- (0.2,0);
    \draw[rotate = 45] (0,-0.2) -- (0, 0.2);
    }
}

\usepackage{ulem}
\usepackage{setspace}

\newtheorem{thm}{Theorem}[section]
\newtheorem{lem}[thm]{Lemma}
\newtheorem{cor}[thm]{Corollary}

\newtheorem{prop}[thm]{Proposition}

\newtheorem{definition}[thm]{Definition}

\newtheorem{exmp}{Example}
\newtheorem{rem}[thm]{Remark}
\newtheorem{hipos}[thm]{Hypotheses}
\newtheorem{hipo}[thm]{Hypothesis}

\DeclareMathOperator\supp{supp}
\DeclareMathOperator\Leb{Leb}

\newcommand\eps{\epsilon}

\usepackage{comment}

\usepackage{soul} 
\newcommand{\gab}[1]{{\textcolor{red}{#1}}}

\newcommand{\rmd}{\mathrm{d}}   

\newcommand{\nnorm}[1]{{\left\vert\kern-0.25ex\left\vert\kern-0.25ex\left\vert #1 
		\right\vert\kern-0.25ex\right\vert\kern-0.25ex\right\vert}} 

\usepackage{comment} 
\excludecomment{toexclude} 

\title[Non-polynomial fractional diffusions ]{Derivation of non-polynomial  fractional diffusions from the generalized exclusion with a slow barrier}

\vspace{2.cm}	
\author{Pedro Cardoso,  Patr\'icia   Gon\c calves, Gabriel Nahum}
\newcommand{\Addresses}{{

                \footnotesize
		Pedro Cardoso, \textsc{\noindent Institute for Applied Mathematics \\
			University of Bonn \\
			Endenicher Allee, no. 60, 53115 Bonn, Germany}\par\nopagebreak
		\textit{E-mail address}: \texttt{pgondimc@uni-bonn.de}
		\medskip
		
		Patr\'icia   Gon\c calves, \textsc{\noindent Center for Mathematical Analysis,  Geometry and Dynamical Systems \\
Instituto Superior T\'ecnico, Universidade de Lisboa\\
Av. Rovisco Pais, no. 1, 1049-001 Lisboa, Portugal}\par\nopagebreak
		\textit{E-mail address}: \texttt{pgoncalves@tecnico.ulisboa.pt}
		\medskip			
		
		Gabriel Nahum, \textsc{\noindent DRACULA\\
			Inria Lyon \\
                56 Bd. Niels Bohr, 69199 Villeurbanne, France}\par\nopagebreak
		\textit{E-mail address}: \texttt{gabriel.silva-nahum@inria.fr}

	}}
	\usepackage{multicol}
	

\begin{document}


\begin{abstract}
	In this article we derive in the hydrodynamic limit a generalized fractional porous medium equation, in the sense that the regional fractional Laplacian is applied to a  function of the density given in terms of  a power series, instead of a polynomial.  The hydrodynamic limit is obtained considering a microscopic dynamics of random particles with long range interactions, but the jump rate highly depends on the occupancy near the sites where the interactions take place. This system is also studied in the presence of a "slow barrier" that hinders the flow of mass between two half-spaces.
\end{abstract}

\maketitle

\section{Introduction}

The focus of this article is the derivation of the hydrodynamic limit, namely a scaling limit procedure describing the  space-time evolution of the conserved quantity(ies) of an interacting particle system (IPS) by means of a partial differential equation (PDE). The  question we address,  and which is recurrently asked to probabilists working in the field of IPSs by mathematicians from the fields of fluid dynamics or analysts working in PDEs, can be stated as follows: what sort of equations can be rigorously obtained as the hydrodynamic limit of such physical systems?  In this respect, we  exhibit a collection of models whose hydrodynamic limit can be described by polynomial or even non-polynomial fractional PDEs, the later case being a novelty in the literature. Before explaining in detail our results we describe below our microscopic context. 

\subsection{Generalized exclusion process and hydrodynamic limit}

An interacting particle system  consists of a collection of continuous-time random walks whose evolution typically obeys one or more constraints. A classical example is the (generalized) exclusion rule, which dictates that a site cannot be occupied by more than a fixed number of particles (in this work, denoted by $N_{\text{e}}$). These particles move in some lattice, the \textit{microscopic space} (that is $\mathbb Z^d$ in article), which is the discretization of some continuous domain, the \textit{macroscopic space} (that is $\mathbb R^d$). At every $\{\hat x,\hat y\} \subset \mathbb Z^d$ a Poisson clock is attached, and all those clocks are independent. After the ring of the corresponding clock, a particle can jump from $\hat x$ to $\hat y$, as long as the exclusion rule is obeyed. This dynamics preserves a quantity, the \textit{number  of particles}, to which we associate a random measure, the \textit{empirical measure}.

One challenging goal in the field of IPS is to describe the space-time evolution of the conserved quantity of the system by means of a of a partial differential equation; this is known in the literature as the  \textit{hydrodynamic limit}. This is done at the level of a weak law of large numbers, in the sense that for any fixed $t \geq 0$, some sequence of random empirical measures converges to a deterministic measure, whose density is a solution of a PDE, namely the \textit{hydrodynamic equation}.

In what follows, we confront our microscopic context with another ones that are already known in the community of IPSs, in order to clarify the novelties of our work.

\subsection{Comparison with previous literature}

In this article we are focused on exclusion processes with long jumps. More exactly, a particle jumps from $\hat{x}$ to $\hat{y}$ with transition probability $p_{\gamma}(\hat{y} - \hat{x}  )$, where $p_{\gamma}(\cdot)$ is precisely defined in \eqref{transition prob}, whose moments are ruled by a parameter $\gamma \in (0, 2)$.  Up to the best of our knowledge, the hydrodynamic limit of such a model was first studied in \cite{jara2009hydrodynamic} (with $N_{\text{e}}=1$), where the symmetry of the transition probability {produces a hydrodynamic} equation given in terms of (a multiple of) the fractional Laplacian operator. 

Next, we list  a few variations of the long-range exclusion process.
Inspired by the results in \cite{GLT09, bonorino}, one possibility is to generalize the rates of \cite{GLT09},  as it was done in  \cite{renato} and \cite{CG}, where it was possible to obtain the \textit{fractional} porous medium equation, given in terms of the \textit{fractional} Laplacian.  
 A natural question that  arises is under which conditions the results of \cite{renato} and \cite{CG} can be extended in order to obtain other types of nonlinear equations and also to the multidimensional context. In this work these two generalizations are done, at the cost of introducing more general jump rates.  Indeed, in both the hydrodynamic equations derived in  \cite{CGJ2} and \cite{renato},  the fractional operator was being applied,  respectively, to the linear function $F(\rho)=\rho$ or to the quadratic function $F(\rho)=\rho^2$. Nevertheless, in this work $F(\rho)$ can be given not only by a more general polynomial as in \cite{CG}, but also by non-polynomial functions, such as $F(\rho)=\exp(\rho)$. 
 
The main motivation for the construction of our microscopic models is the definition of the interpolating model in \cite{gabriel}.  There, the authors derived the diffusive equation $\partial_t\rho=\partial_u^2\rho^{m}$, with $m\in(0,2)$. Their starting point was to observe that, for a symmetric one-dimensional  simple exclusion process (where only jumps of size $1$ are allowed), the expected value of the constraints of the model, with respect to the invariant measure, equals the diffusion coefficient $D(\rho)=m\rho^{m-1}$. In this way, from the Generalized Binomial Theorem, by expanding $\rho^{m-1}$ into a power series, one is able to identify the constraints associated with $D(\rho)$ as a (very large) linear combination of porous media models. Our work extends the range of the coefficients of the aforementioned linear combination, in a long-range and multidimensional setting. 

Another extension with respect to all the  aforementioned cases is to ask whether the classical exclusion rule (the presence of at most \textit{one} particle per site, i.e. $N_{e}=1$) is a fundamental assumption in our results. Surprisingly, the answer is negative, since our results also apply to the \textit{generalized} exclusion process for which $N_{e}>1$, for a specific choice of rates (of product form) which was also studied in \cite{beatriz} and references therein. What is crucial is the existence of an uniform upper bound for the value of the  occupation variables. In this case, the invariant measures are known and   instead of being the product of Bernoulli distributions they are product of {binomial ones}. 

Finally,  one possible variation of the model is the introduction of slows bonds in the discrete space $\mathbb Z^d$. For instance, in many models in the literature the jump rates at those bonds have a  multiplicative factor of  $\alpha n^{-\beta}$, where $\alpha >0$, $ n \in \mathbb{N}_+$ and $\beta \geq 0$, see \cite{tertuaihp,tertumariana} and \cite{CGJ2}, for a few examples. For $\beta$ positive, this leads to a hindering of jumps in those bonds as we take $n \rightarrow \infty$. The reader  could wonder if the choice of the aforementioned rate $\alpha n^{-\beta}$ for the slow bonds  could be replaced by any other function of $n$, and how this would affect the results. Keeping this in mind, in this work we replace  $\alpha n^{-\beta}$ by a more general function $\alpha_n$ and our  main theorem  describes  how the hydrodynamic equation depends of the choice of the function $\alpha_n$.

\subsection{Our contributions:}

Our results hold for a variety of models that we constructed by considering the extensions discussed in the previous subsection.  Therefore, we give a  relevant and non-trivial improvement of the results  of \cite{CGJ2}, \cite{renato} and \cite{CG}, since here we describe how to obtain non-polynomial fractional hydrodynamic equations  in multidimensional lattices, as the scaling limit of  the generalized exclusion process, with a much wider  choice of rates  of slow bonds. These  results   introduce a great deal of novelty in the literature of hydrodynamic limits of IPS with long-range dynamics; the price to pay  is to  have a (slightly) long and  quite technical article so that our results can be properly presented.

The methodology of our proof boils down to the adaptation of the  entropy method, introduced in \cite{GPV}, to accommodate our collection of models. We  follow the general strategy presented in Chapter 4 of \cite{kipnis1998scaling}. Nevertheless, since the results of that chapter are stated and proved for a particle system  evolving on the torus (a \textit{finite} domain), we made some effort to derive analogous arguments that can be applied to our setting since our model is evolving on an  \textit{infinite} volume. This careful approach was already present in Appendix D of \cite{ddimhydlim} and here we also included it so that the article is as self-contained as possible. 

Keeping in line with the goal of presenting the text in the most rigorous and transparent  possible way, in this work we do not shy away from stating some technical definitions (such as the ones regarding the Skorohod topology) and all the analysis and measure theory results that are needed in our arguments. We do so because (up to the best of our knowledge) these statements have not been provided explicitly yet, at least for a model evolving on an infinite domain with degenerate rates such as ours. We have then embraced the project of presenting and proving the corresponding lemmas in the clearest way and also obtaining the estimates for the several multidimensional integrals or sums that we have to deal with along the proofs. All these technical details appear in the  appendices.

A crucial argument  is the derivation of replacement lemmas, which allow the substitution of generic local functions of the dynamics by functions of the conserved quantity of the system, in this case functions of the density of particles. In this article they are obtained by assuming an entropy bound with respect to a measure associated to a non-constant profile, in opposition to the replacement lemmas stated in both \cite{renato} and \cite{CG}. This means that the derivation of our integral equations can be done by assuming a weaker requirement than in those two works, which is also quite novel.

\subsection{Remarks and open questions:}

Now,  we comment about some possible extensions of our work. A natural question is related to the derivation of the behavior of the  fluctuations of the system around the typical profile for the collection of models that we introduce. Another question would be to obtain this type of hydrodynamic limits  for other long range constrained dynamics as, for example, the zero-range process or the inclusion process. Since in those models the occupation variables are unbounded the proofs are more demanding. Another extension would be to apply our ideas to derive non-polynomial equations for a model with an asymmetric transition probability, such as the one studied in \cite{sethuraman}.

\subsection{Outline of the article:}

Finally, we describe the outline of this article: in Section \ref{model}, we describe the dynamics that produces all the features of the models as we have described above, moreover we introduce all the  definitions required to state our main theorem. In Section \ref{secheurlin}, we give heuristic arguments (which are applied rigorously in the following  sections) for the derivation of  the various hydrodynamic equations. In Section \ref{sectight}, we prove tightness of our sequence of probability measures, which then guarantees that it has  at least one limit point. In Sections \ref{estenerg} and \ref{secchar}, by following the entropy method introduced in \cite{GPV}, we prove that all such limit points are concentrated on weak solutions of the partial differential equations that are  stated in Subsection \ref{sechydeqsdif}.  The strategy of the proof consists in using the Dynkin's martingale, which is a discretized version of the notion of weak solution of the corresponding equation.  In Section \ref{replem} we show  the Replacement Lemma which, in our context, is  crucial in order to obtain a  nonlinear hydrodynamic equation. For completeness we conclude this article with four appendices: Appendix \ref{secfracoper} is devoted to obtaining an upper bound in $L^1(\mathbb{R}^d)$ of our fractional operators. In Appendix \ref{secdiscconv}, we prove the convergence of the discrete operators to the corresponding continuous (fractional) ones. In Appendix \ref{useest}, we provide estimates for various multidimensional sums, which were applied in Sections \ref{secheurlin} and \ref{sectight}. In Appendix \ref{antools}, we state and prove several results in  analysis and measure theory, since we did not find  in the literature their versions encompassing our case. 

\section{Microscopic models and main results}\label{model}

\subsection{Stochastic setup}

The purpose of this work is to derive the hydrodynamic limit for a class of long-range, kinetically constrained, symmetric generalized exclusion processes in arbitrary dimensions, accounting for the presence of a \textit{slow barrier} with constraints, thereby generalizing the results of \cite{jara2009hydrodynamic, CGJ2, renato, CG, gabriel}. 

Fix $d \geq 1$. We denote elements of $\mathbb{R}^d$ by $\hat{u}$ and $\hat{v}$, unless it is assumed that $d=1$. Moreover ${\rmd} \hat{u} $ resp. ${\rmd} \hat{v} $ denotes the (product) Lebesgue measure ${\rmd} u_1 \; {\rmd} u_2 \;{\cdots} \; {\rmd} u_d$ resp. ${\rmd} v_1 \; {\rmd} v_2 \;{\cdots} \; {\rmd} v_d$. Furthermore, {we} denote by $\{\hat{e}_1, \cdots, \hat{e}_d\}$ the canonical basis of $\mathbb{R}^d$ and $\hat{0}:= (0,0, \cdots, 0) \in \mathbb{R}^d$. {We will also resort to the following} alternative representation of $\hat{u} \in \mathbb{R}^d$, precisely, we write $\hat{u} = ( \hat{u}_{\star}, u_d ) 
	$ where $
	\hat{u}_{\star} = (u_1, u_2, \ldots, u_{d-1}) \in \mathbb{R}^{d-1}$.
Naturally, in the case $d=1$, $( \hat{u}_{\star}, u_d )$ is identified with $u_d \in \mathbb{R}$. 

Our process evolves on the \textit{lattice} $\mathbb{Z}^d$, whose elements are called \textit{sites} and will be recurrently denoted by the Latin letters $\hat{x}, \hat{y}, \hat{z}$. In what follows, given $\hat{x} \neq \hat{y} \in \mathbb{Z}^d$, $\{\hat{x}, \hat{y}\}$ denotes the \textit{unordered} pair and $(\hat{x}, \hat{y})$ denotes the \textit{ordered} pair. Each unordered pair will be called a \textit{bond}.

In order to define the \textit{generalized} exclusion process, we fix $ N_{\text{e}} \in \mathbb{N}_+:=\{1, 2, \ldots\}$ and enforce a rule which allows at most $ N_{\text{e}}$ particles per site. In the particular case  $N_{\text{e}}=1$ we recover the usual exclusion process. Thus, our {state space} is the set $\Omega = \{0,1, \ldots, N_{\text{e}} \}^{\mathbb{Z}^d}$ and we call its elements \textit{configurations}, which will be {recurrently} denoted by Greek letters such as $\eta$ {and $\xi$}.  Given a configuration $\eta \in \Omega$ and a site $\hat{x} \in \mathbb{Z}^d$, we say that the site $\hat{x}$ is empty resp. occupied if $\eta(\hat{x})=0$ resp. $\eta(\hat{x}) > 0$. Moreover, for any $\hat{x} \in \mathbb{Z}^d$,
\begin{equation} \label{bndeta}
	\quad 0 \leq \eta(\hat{x}) \leq N_{\text{e}}
	,
	\end{equation}
and for every configuration of particles $\eta \in \Omega$, we associate the configuration of \textit{complementary particles} $\tilde{\eta} \in \Omega$, defined by $\widetilde{\eta}(\hat{x}):=N_{\text{e}}- \eta(\hat{x})$. Contrarily  to the (diffusive) setting in \cite{ddimhydlim}, for {dimension} $d \geq 2$ {our dynamics allows diagonal jumps}, i.e., jumps between $\hat{x}$ and $\hat{y}$ where $\hat{x}$ and $\hat{y}$ differ in more than one coordinate.

Our model is composed by a linear combination of "fractional" processes already studied in the literature, and for that reason it is convenient to recall them and some of their properties. But first, let us introduce some relevant definitions. 
\begin{definition}\label{def:micro-op}
	Let $\{\hat{e}_1,\dots,\hat{e}_d\}$ be the canonical basis of $\mathbb{R}^d$, and $\hat{x},\hat{y}\in\mathbb{Z}^d$ be such that $\hat{x}\neq\hat{y}$.
	\begin{itemize}
		\item For each $1\leq j\leq d$ and $n\in\mathbb{Z}$ introduce the shift operator $\tau_j^n:\eta\mapsto\tau_j^n\eta$ defined on $\hat{x} \in \mathbb{Z}^d$ by $\tau_j^n\eta(\hat{x})=\eta(\hat{x}+n\hat{e}_j)$. Naturally, we identify $\tau_j^0:=\mathbf{1}$, with the latter being the identity in $\Omega$;
		\item The translation operator in higher dimensions is defined as $\tau^{\hat{x}}=\sum_{j=1}^{d}\tau_j^{x_j}$, for any $\hat{x}\in\mathbb{Z}^d$;
		\item For each \textit{bond} $\{\hat{x},\hat{y}\}$ the jump operator $\eta^{\hat{x},\hat{y}}$ is defined as 
		{\begin{equation} \label{jumppart}
				\eta^{\hat{x},\hat{y}}(\hat{z})=
				[\eta(\hat{x})-1]\mathbbm{1}_{\{\hat{z}= \hat{x}\}} 
				+[\eta(\hat{y})+1]\mathbbm{1}_{\{\hat{z}= \hat{y}\}}
				+\eta(\hat{z})\mathbbm{1}_{\{\hat{z} \neq \hat{x}, \hat{y}\}}.
			\end{equation}}
	\end{itemize}
	All of these operators are extended naturally to any function $f:\Omega\to\mathbb{R}$ through the identity $Af(\eta)=f(A\eta)$, with $A$ being any of the previous operators. Now we introduce their associated difference operators, acting on functions in $\Omega$.
	\begin{itemize}
		\item For each $1\leq j\leq d$ and integer $n\in\mathbb{Z}$ we define the difference operator $\nabla_j^{n}:=\tau_j^n-\mathbf{1}$. The second-order difference operator is defined as $\Delta_j^n:=-\nabla_j^{-n}\nabla_j^{n}$. Precisely, $\Delta_j^n:=\tau_j^{n}-2\mathbf{1}+\tau_j^{-n}$. For $n=1$ we write simply $\nabla_j\equiv \nabla_j^{1}$ and $\Delta_j \equiv \Delta_j^{1}$;
		\item The previous operators are extended to $\nabla_{\hat{x}}:=\tau^{\hat{x}}-\mathbf{1}$ and $\Delta_{\hat{x}}:=\tau^{\hat{x}}-2\mathbf{1}+\tau^{-\hat{x}}$, for $\hat{x} \in \mathbb{Z}^d$;
		
	\item For any $f: \Omega \mapsto \mathbb{R}$ and $\hat{x},\hat{y} \in \mathbb{Z}^d$, we denote $\nabla_{\hat{x},\hat{y}} f ( \eta ):=f(\eta^{\hat{x}, \hat{y} }) - f(\eta)$;  
		\item Moreover, for any $G:\mathbb{R}^d\to\mathbb{R}$ and $\hat{u},\hat{v} \in \mathbb{R}^d$ we will also write $\nabla_{\hat{v}}G(\hat{u}):= G(\hat{u} + \hat{v}) -  G(\hat{u})$ and
		$\Delta_{\hat{v}}G(\hat{u}) := G(\hat{u}+\hat{v}) -2 G(\hat{u}) + G(\hat{u}-\hat{v})$.
	\end{itemize}
\end{definition} 

\subsection{Infinitesimal generator}

Throughout this text, let $\abs{\;\cdot\;}$ be the Euclidean distance in $\mathbb{R}^d$. Hence for any $\hat{u}\in\mathbb{R}^d$ we write $\abs{\hat{u}}^2=(u_1)^2+\dots+(u_d)^2$. {In what follows, let $\hat{x},\hat{y}\in\mathbb{Z}^d$ be arbitrary.} In order to enforce the {(generalized)} \textit{exclusion} rule, {a possible choice} for a particle to move from $\hat{x}$ to $\hat{y}$ is that ${a}_{\hat{x},\hat{y}}(\eta)>0$, where \begin{equation} \label{excrule}
\forall \eta \in \Omega, \quad {a}_{\hat{x},\hat{y}}(\eta) :=\eta(\hat{x}) [N_{\text{e}} - \eta(\hat{y})]\leq N_{\text{e}}^2.
\end{equation}
Note that this choice corresponds, in the nearest-neighbor case, to a dynamics that is gradient in the sense that the instantaneous current of the system, i.e. the difference between the jump rates to the right and  to the left can be written as the gradient of some local function. We adopt the same choice since in this case there are no extra terms in the dynamics that would have to require a more refined analysis.

In \cite{jara2009hydrodynamic}, the  transition probability in $\mathbb{Z}^d$ which governs the movements of the particles is given, for any $\hat{z}\in\mathbb{Z}^d$, by
\begin{equation}\label{transition prob}
	p_\gamma(\hat{z}) = c_{\gamma} \frac{1}{\abs{\hat{z}}^{d+\gamma}}  \mathbbm{1}_{ \{ \hat{z} \neq \hat{0}  \} } 
	,\quad\gamma \in (0,2),
\end{equation}
where $c_{\gamma}$ is a normalizing  constant. The restriction $\gamma \in (0,2)$ is required to obtain fractional PDEs at the macroscopic level. Therefore, given an initial configuration $\eta \in \Omega$ and sites $\hat{x}$, $\hat{y}$ such that ${a}_{\hat{x},\hat{y}}(\eta)>0$, a particle jumps from $\hat{x}$ to $\hat{y}$ with probability $p_\gamma(\hat{y}-\hat{x})$. 

In \cite{jara2009hydrodynamic, renato, CG} it was obtained the hydrodynamic limit for the usual exclusion process (i.e., $N_{ \text{e} }=1$), where $p(\cdot)$ was given by \eqref{transition prob}. In all those works, the hydrodynamic equation was given by $\partial_t \rho=\Delta^{ \gamma / 2 } \rho^m$ for some $m \in \mathbb{N}_+$, where $\rho$ is the density of particles; and $ \Delta^{ \gamma / 2 }$ is the fractional Laplacian, given in the forthcoming Definition \ref{def:regi_lap}. 

On the other hand, in \cite{CGJ2}, the hydrodynamic equation was given by  $\partial_t \rho=  \mathbb{L}_{\kappa}^{\gamma} \rho$ for some $\kappa \geq 0$. Above, for any $\kappa \geq 0$, $\mathbb{L}_{\kappa}^{\gamma}$ is a "distorted" fractional Laplacian operator defined as
\begin{equation} \label{distfrac}
\mathbb{L}_{\kappa}^{\gamma}
		:= 
		\kappa  \Delta^{ \gamma / 2 }
		+(1-\kappa) \Delta_{\star}^{ \gamma / 2 }		,
\end{equation}
where $ \Delta_{\star}^{ \gamma / 2 }$ is the regional fractional Laplacian, as given in Definition \ref{def:dist_lap}. 

In this work, we present a particle system whose hydrodynamic equation is given by
\begin{align} \label{hydeq}
	\partial_t \rho = \mathbb{L}_{\kappa}^{\gamma} F(\rho)
\end{align}
for some $\kappa >0$, where $F$ is either a polynomial, or a power series. A motivation for the later is the interpolating model introduced in \cite{gabriel}. There, by applying the Generalized Binomial Theorem, 
it was observed, for the case $N_{ \text{e} }=1$, that 
\begin{align} \label{defrholam}
    \rho^{m}=[N_{ \text{e} }-(N_{ \text{e} }-\rho)]^{m}
    = \sum_{k\geq 0}(-1)^{k}\binom{m}{k} N_{ \text{e} }^{m-k}  
    (N_{ \text{e} }-\rho)^k, 
\end{align}
for any $m \in (0,1) \cup (1,2)$, and where, for each $k\geq0$, $\binom{m}{k}=\frac{m(m-1)\cdots(m-(k-1))}{k!}$ is the generalized binomial coefficient, with the convention $\prod_{\emptyset}:=1$. 

Motivated by the last display, we will consider $F$ in \eqref{hydeq} given by
\begin{align} \label{defF}
F(\rho):=b_0 + \sum_{k=1}^{\infty} b_k^+ \rho^k - \sum_{k=1}^{\infty} b_k^- (N_{ \text{e} }-\rho)^k
\end{align} 
for some $b_0 \in \mathbb{R}$ and sequences $(b_k^+)_{k \geq 1}, (b_k^-)_{k \geq 1} \subset \mathbb{R}$. We will require, at least, $F$ to be absolutely convergent for every $\rho\in [0,N_{\text{e}}]$. Thus, throughout the rest of the paper we will assume that
\begin{equation} \label{convabsF}
f_\infty:= \sum_{k=1}^{\infty} (|b_k^{+}| +|b_k^{-}|) N_{ \text{e} }^k< \infty.
\end{equation}
Keeping in mind \eqref{hydeq}, we observe that some recent works regarding hydrodynamic limit of the symmetric exclusion process with long jumps correspond to some particular choices of $F$, given in \eqref{defF}. For instance, $F(\rho)=\rho$ in \cite{jara2009hydrodynamic}; $F(\rho)=\rho^2$ in \cite{renato}; and $F(\rho)=\rho^m$ in \cite{CG}, for any $m \in \mathbb{N}_+$. In those works, the display in \eqref{defF} is actually a polynomial, thus \eqref{convabsF} holds trivially. Nevertheless, from \eqref{defrholam} together with the choices
	\begin{align*}
		b_0=N_{ \text{e} }^{m},
		\quad
		b_{k}^{+}=0 \quad \text{and} \quad b_{k}^{-}=-(-1)^{k} N_{ \text{e} }^{m-k} \binom{m}{k},
	\end{align*}
for all $k\geq 1$, we get $F(\rho)=\rho^{m}$ for any $m \in (0,1) \cup (1,2)$. We stress that Theorem \ref{hydlim}, our main result, requires extra conditions on $F$ besides \eqref{convabsF}, see Hypothesis \ref{hyp:coeff} below. It will be important in this article to distinguish whether $F$ is linear or not, as we will employ different techniques in each case. 
\begin{definition} \label{defFlin}
	We say that $F$ given by \eqref{defF} is \textbf{linear} if for every $k \geq 2$ it holds $b_k^+=b_k^-=0$; thus, $F(\rho)=b_0 + b_1^+ \rho - b_1^- ( N_{ \text{e} } - \rho)$. Otherwise, we say that $F$ is \textbf{nonlinear}.
\end{definition}
Now we will construct some rates that lead to a Markov process whose hydrodynamic equation is given by \eqref{hydeq}. For the sake of illustration, let us first assume that $d=1$. In order to derive the diffusive porous medium equation $\partial_t \rho = \Delta \rho^k$ for any $k\in\mathbb{N}_+$, in \cite{GLT09} it was imposed a kinetic constraint on the exchange rate of occupations between sites $x$ and $x+1$, that we write as ${r}^{(k)}_{x,x+1}$, being given by
\begin{align} \label{ratpordif}
\forall\eta\in\Omega, \quad {r}^{(k)}_{x,x+1}(\eta) := \sum_{j=1}^k \underset{i \notin \{0, 1 \} }{\prod_{i=j-k}^j} \eta(x + i)={r}^{(k)}_{x+1,x}(\eta).  
\end{align}
For each $\hat{x},\hat{y} \in \mathbb{Z}^d$ the multidimensional \textit{long-range} constraint $c^{(k)}_{\hat{x},\hat{y}}: \Omega \mapsto  [0, \; k N_{\text{e}}^{k-1}]$ is defined as the average constraint over all directions, 
\begin{align} \label{ratcub}
	c^{(k)}_{\hat{x},\hat{y}}:= \frac{1}{d} \sum_{j=1}^d c^{(k),j}_{\hat{x},\hat{y}}
	,\quad d\geq 1,
\end{align}
where for each direction $\hat{e}_j$ we define
\begin{align}\label{eq:rates_porous}
	c^{(k),j}_{\hat{x},\hat{y}}
	:=
	\frac{r_{\hat{x},\hat{y}}^{(k),j}+  r_{\hat{y},\hat{x}}^{(k),j} }{2}
	=c^{(k),j}_{\hat{y},\hat{x}}
	\quad\text{and}\quad
	r_{\hat{x},\hat{y}}^{(k),j}(\eta)
	:=
	\sum_{\ell=1}^k \underset{i \notin \{0, 1 \} }{\prod_{i=-k+\ell}^\ell} \tau^{i}_j\eta(\hat{x}+\mathbbm{1}_{\{i\geq 1\}}(\hat{y}-\hat{x}))
	.
\end{align}
For $N_{\text{e}}=1$ the corresponding process was introduced in \cite{renato}. We note that the expressions in \eqref{eq:rates_porous} are presented in a more compact form than in the aforementioned reference. For $k=1$, we adopt the convention that $c^{(1),j}_{\hat{x},\hat{y}}(\eta) = r_{\hat{x},\hat{y}}^{(1),j}(\eta) =1$, for every $\eta \in \Omega$, every $\hat{x}, \hat{y} \in \mathbb{Z}^d$ and every $\eta \in \Omega$. One constructs the multidimensional long-range constraint $c^{(k)}_{\hat{x},\hat{y}}$ above from the one-dimensional, nearest-neighbor rates in the following way. Let $d=1$. In order to define the (long-jumps) constraint $r^{(k),1}_{x,y}$, for $y> x+1$, one can rewrite the terms on the right-hand side of \eqref{ratpordif} of the form $\eta(x+i)$ as $\eta\big((x+1)+(i-1) \big)$, for any $i \in \{2, \ldots, k\}$, and substitute $x+1\equiv y$. This leads to ${r}^{(k),1}_{x,y}$. It turns out that this reasoning does not produce a symmetric constraint with respect to $x$ and $y$. The constraint $c^{(k),1}_{x,y}$ is the result of a symmetrization of $r^{(k),1}_{x,y}$. The model can then be lifted to higher dimensions $d\geq 1$ by first choosing some direction $\hat{e}_j \in \{\hat{e}_1, \ldots, \hat{e}_d \}$ according to the uniform distribution, and afterwards moving it with the unidimensional rate along that direction. Concretely, the constraint, with respect to a direction $\hat{e}_j$, corresponds to replacing, in $r_{\hat{x},\hat{y}}^{(k),1}$, the variables $x$ and $y$ by $\hat{x}$ and $\hat{y}$, and $\tau$ by $\tau_j$. This leads then to $r_{\hat{x},\hat{y}}^{(k),j}(\eta)$ for $j\geq 1$ which can then be symmetrized. Let us see a couple of examples.
\begin{exmp}
	For instance, for $k=3$ and $d=1$, it holds  	\begin{align*}
		{r}^{(3)}_{x,x+1}(\eta)&= \eta(x-2) \eta(x-1) +  \eta(x-1)\eta(x+2) +\eta(x+2)\eta(x+3),
		\\
		{r}^{(3)}_{x,y}(\eta) &= \eta(x-2) \eta(x-1) + \eta(x-1) \eta (y + 1 ) + \eta ( y + 1 ) \eta ( y + 2).
	\end{align*}
On the other hand, for $k=3$ and $d>1$, we get for any $j \in \{1, \ldots, d\}$:
\begin{align*}
		2c^{(3),j}_{\hat{x},\hat{y}}(\eta)=& 
		\left[
		\eta(\hat{x}-2\hat{e}_j)\eta(\hat{x}-\hat{e}_j) + \eta(\hat{x}-\hat{e}_j)\eta(\hat{y}+\hat{e}_j) + \eta(\hat{y}+\hat{e}_j)\eta(\hat{y}+2\hat{e}_j)
		\right]
		\\
		+& 
		\left[
		\eta(\hat{y}-2\hat{e}_j)\eta(\hat{y}-\hat{e}_j) + \eta(\hat{y}-\hat{e}_j)\eta(\hat{x}+\hat{e}_j) + \eta(\hat{x}+\hat{e}_j)\eta(\hat{x}+2 \hat{e}_j)
		\right].
	\end{align*}
\end{exmp}
Similarly as it was done in \cite{CGJ2}, we introduce a slow barrier in our system  (that hinders the flow of mass in the microscopic system), producing fractional boundary conditions at the macroscopic level. This leads to the distorted fractional Laplacian given in \eqref{distfrac}. Similarly as in \cite{CGJ2}, we introduce the following definition.  \begin{definition}[Slow-bonds]\label{def:slow-bonds}
    Denote by $\mcb{B}$ the set of \textit{bonds}, $\mcb{S}$ an hyperplane playing the role of a \textit{slow} barrier in the system, and $\mcb{F}$ their difference, \textit{i.e.,} the set of "fast" bonds. Precisely,
    \begin{align}
	\mathcal{B}:=\big\{ \{\hat{x}, \hat{y} \}: \hat{x} \neq \hat{y} \in \mathbb{Z}^d \big\}
	,\quad
	\mcb{S}:=\big\{ \{ \hat{x},\hat{y} \} \in \mcb B: x_d < 0, y_d \geq 0   \big\}
	,\quad\text{and}\quad
	\mcb{F}:= \mcb{B} \setminus \mcb{S}.
\end{align}
\end{definition}

We now describe the effect of $\mcb{S}$ in our system. The \textit{slow} factor is expressed through a convergent sequence regulating the current at the barrier. 
\begin{definition}[Slow factor]\label{def:alpha}
    Let $(\alpha_n)_{n \geq 1} \subset (0, \infty)$ be a convergent sequence, and denote its limit by $\alpha\geq0$. For any $n \geq 1$ and any bond $\{\hat{x}, \hat{y} \}\in\mcb B$, define
	\begin{equation}
			\alpha_{\hat{x},\hat{y}}^n
			=
			\mathbbm{1}_{\{\hat{x},\hat{y}\}\in\mcb F}
			+\alpha_n\mathbbm{1}_{\{\hat{x},\hat{y}\}\in\mcb S}.
	\end{equation}

\end{definition}

Motivated by the rates given in \cite{gabriel}, we now present a possible generator for a Markov process with hydrodynamic equation given by \eqref{hydeq}.

\begin{definition}[Generator of our model]  \label{def:gen_series}
Let $(\ell_n)_{n\geq 1}\subset \mathbb{N}_+$ be a sequence such that $\ell_n\xrightarrow{n \rightarrow \infty} \infty$, and $(b_k^\pm)_{k\geq 1}$ real-valued sequences. For each $\hat{x},\hat{y} \in \mathbb{Z}^d$, introduce the maps $c_{\hat{x},\hat{y}}^{n,j}$ and $c_{\hat{x},\hat{y}}^n$, defined for each $\eta\in\Omega$ as
	\begin{align}\label{cons-series}
		c_{\hat{x},\hat{y}}^{n,j}(\eta)
		:=
		\sum_{k=1}^{\ell_n} \big[ b_k^{+} c_{\hat{x}, \hat{y}}^{(k),j}(\eta )+  b_k^{-} c_{\hat{x}, \hat{y}}^{(k),j}(\widetilde{\eta} )\big]
		\quad\text{and}\quad
		c_{\hat{x},\hat{y}}^n
		:= \mathbbm{1}_{ \{ \{ \hat{x}, \hat{y}   \} \in \mcb F \} } \mathbbm{1}_{\{  |\hat{y}-\hat{x}|=1 \} } + \frac{1}{d}\sum_{j=1}^dc_{\hat{x},\hat{y}}^{n,j}
		.
	\end{align}
Assume that $(b_k^\pm)_{k\geq 1}$ are such that $c_{\hat{x},\hat{y}}^n(\eta) \geq 0$, for any $\hat{x},\hat{y} \in \mathbb{Z}^d$, $n \in \mathbb{N}_+$ and $\eta \in \Omega$. The generator of our process is denoted by $\mcb L_{n,\alpha}^\gamma$ and it acts on  local functions $f: \Omega \mapsto \mathbb{R}$, that is, functions depending on a finite number of coordinates, as 
\begin{align}  \label{defgenslow}
(\mcb L_{n,\alpha}^\gamma f)(\eta) 
		:= \frac{1}{2 N_{\text{e}}} \sum_{\hat{x}, \hat{y} } 
		p_\gamma( \hat{y} - \hat{x} ) 
		\alpha_{\hat{x},\hat{y}}^n
		c_{\hat{x}, \hat{y}}^{n,\gamma}(\eta)
		[a_{\hat{x},\hat{y}}(\eta)+a_{\hat{y},\hat{x}}(\eta)]
		(\nabla_{\hat{x},\hat{y}}f)(\eta).
\end{align}	
\end{definition}
The first term in \eqref{cons-series} will be referred to as the \textit{perturbation term}. As we shall see in more detail later on, this term guarantees a local mixing of the microscopic system, necessary for our arguments, but at a time-scale whose statistics are still invisible at the macroscopic scale. Above and in what follows, the discrete variables in any summation always range over $\mathbb{Z}^d$, unless it is stated otherwise. 
 
Next we  introduce measures defined on the state space $\Omega$ that will be useful in our arguments. 
\begin{definition}[{Product measure with a profile}] \label{defprodmeas} 
Fix a profile $h:\mathbb{R}^d\to(0,1)$ and $n \in \mathbb{N}$. We denote by $\nu_h^n$ the product probability measure on $\Omega$ with marginals given by
\begin{align} \label{defberhn}
\forall	\hat{z}\in\mathbb{Z}^d,\; \forall m \in \{ 0, 1, \ldots, N_{\text{e}}\}, \quad	\nu_h^n \big(\eta \in \Omega: \; \eta(\hat{z}) = m \big) 
	:=\binom{N_{\text{e}}}{m} 
	\big[h \big( \tfrac{\hat{z}}{n} \big) \big]^m 
	\big[1 - h \big( \tfrac{\hat{z}}{n} \big)\big]^{N_{\text{e}}-m}.
\end{align}
\end{definition}
In Section \ref{estenerg} below, it will be useful to consider the particular case where $h$ is constant, i.e. $h \equiv \beta$ for some $\beta \in (0,1)$.
\begin{definition}[Invariant measures] \label{definvmeas}
Fixed a parameter $\beta \in (0, \;1)$, we denote by $\nu_\beta$ the product probability measure on $\Omega$ characterized by marginals with \textit{Binomial} distribution, that is,
\begin{align}  \label{prod-bin}
\forall	\hat{z}\in\mathbb{Z}^d,\; \forall m \in \{ 0, 1, \ldots, N_{\text{e}}\}, \quad	\nu_\beta \big(\eta \in \Omega: \; \eta(\hat{z}) = m \big)
	:= 
	\binom{N_{\text{e}}}{m}
	\beta^m (1 - \beta)^{N_{\text{e}}-m}.	
\end{align}
In this way, under $\nu_\beta$, the random variables $(\eta(\hat{x}))_{\hat{x} \in \mathbb{Z}^d}$ are i.i.d. with Binomial distribution of parameters $N_{\text{e}}$ and $\beta$. Moreover, fixed $\beta\in (0,1)$, for every $k\in\mathbb{N}_+$ the measure $\nu_\beta$ is reversible with respect to $\mcb L^{(k),\gamma}$, and by linearity, also with respect to $\mcb L_{n,\alpha}^\gamma$. 
Therefore, it is also invariant for $\mcb L_{n,\alpha}^\gamma$, for every $n \in \mathbb{N}$.
\end{definition}

\subsection{Macroscopic operators}

\subsubsection{Spaces of test functions} \label{sectest}
We start by fixing the general notation related with the spaces that we are going to introduce. In what follows let $j\in\{1,\ldots,d\}$ and $p\in\mathbb{N}:=\{0,1,\ldots\}$.
\begin{itemize}
	\item We write $ \partial_j \equiv \partial_{\hat{e}_j}$, for the partial derivative in the direction $\hat{e}_j$, and shorten $\partial_{ij} \equiv \partial_i\partial_j$;
	\item $C^p( \mathbb{R}^d)$ corresponds to the set of real-valued functions in $\mathbb{R}^d$ that are $p$ times differentiable, with all derivatives of order up to $p$ being continuous (for $p=0$ it is identified with the set of continuous functions); 
	\item $C_c^p(\mathbb{R}^d)$ is the space of functions in $C^p(\mathbb{R}^d)$ with compact support.
	\item For $(X, \mu)$ a measure space we will write $L^p(X)\equiv L^p(X, \mu)$ whenever $\mu$ is the Lebesgue measure, and we denote the norm in $L^p(\mathbb{R}^d)$ by $\norm{\cdot}_p$. Similarly, for $G: X  \mapsto \mathbb{R}$ we write $\| G \|_{\infty} \equiv \| G \|_{\infty,X}:= \sup_{u \in X} |G(u)|$. In the particular case where $(X, \langle \cdot, \cdot \rangle_{X} )$ is a Hilbert space, we denote the norm in $X$ by $ \| \cdot \|_{X}$. Moreover, we denote the inner product in $L^2(\mathbb{R}^d)$ by $\langle \cdot,\cdot \rangle.$
	
\end{itemize}

Since we will study the temporal evolution of the conserved quantity in our Markov processes, we will deal with test functions that depend on both time and space. \begin{itemize}
	\item Let $C^{r_1,r_2}\big( [0, T] \times \mathbb{R}^d\big)$ be the space of functions in $C^{r_1}([0,T])$ in the first variable and $C^{r_2}(\mathbb{R}^d)$ on the second;
	
	\item Let $C_c^{r_1,r_2}([0,T] \times \mathbb{R}^d)$ be composed of  elements in $C^{r_1,r_2}([0,T]\times \mathbb{R}^d)$ with compact support;
	
	\item Fixed the time variable $s \in [0,T]$ we will write $G_s(\cdot):= G(s, \cdot)$, for $G:[0,T] \times \mathbb{R}^d \mapsto \mathbb{R}$;
	\item Given a metric space $N \subset L^2(\mathbb{R}^d)$, we denote by $L^2 (0,T; \; N)$ the set of all functions $G:[0,T]  \mapsto N$ such that $\int_0^T \| G_s(,\cdot) \|_{N}^2 \;{\rmd s} < \infty$.  
\end{itemize}
In this work, we describe the space/time evolution of the density of particles, which is the conserved quantity of the system, as a weak solution of a partial differential equation, that coins the name "hydrodynamic equation". This weak solution, $\rho$, is characterized through an integral equation  as $\mcb F_{t}^{\gamma,\kappa}(G,\rho,g)=0$, where
\begin{equation} \label{weakform}
	\begin{split}
		\mcb F_{t}^{\gamma,\kappa}(G, \rho, g)
		:= &
		\langle\rho_t,G_t \rangle- \langle g, G_0 \rangle
		- 
		\int_0^t \langle\rho_s,  \partial_s  G_s \rangle\; \;{\rmd s} 
		-  \int_0^t 
		\langle F ( \rho_s) ,\mathbb{L}_{\kappa}^{\gamma} G_s \rangle 
		\; \;{\rmd s}.
	\end{split}  
\end{equation}
In the last display, $F$ is given in \eqref{defF}; $\mathbb{L}_{\kappa}^{\gamma}$ is given in Definition \ref{def:dist_lap}; $g: \mathbb{R}^d \mapsto [0, \; N_{\text{e}}]$ and $\rho: [0,T] \times \mathbb{R}^d \mapsto [0, \; N_{\text{e}}]$ are measurable functions. We aim to define precisely a space of test functions such that the functional $\mcb F_{t}^{\gamma,\alpha}(\cdot,\rho, g)$ is well-defined. 
For the particular one-dimensional case, beyond a threshold value of $\gamma$, one can enlarge the space of test functions by including test functions \textit{without compact support}, being elements of $C^{1,2}$ which decay faster (regarding the space variable) than the inverse of any polynomial (as a motivation see Lemma \ref{lemF1F2}). When $\alpha=0$ we will deal also with functions that \textit{may be discontinuous} at the boundary $\mathbb{R}^{d-1} \times \{0\}$. This motivates the introduction of the spaces in Definition \ref{defschw2} below. 

\begin{definition}\label{defschw2}
	
	For each $d \geq 1$ define the following sets:
\begin{enumerate}
	\item $S^{2}(   \mathbb{R}^d)$, the subset of $C^{2}(  \mathbb{R}^d)$ whose elements $G$ satisfy
	\begin{align*}
		\forall k \in \mathbb{N}, \quad \| G \|_{k, \; S^{2}(   \mathbb{R}^d)} :=  \sup_{ \hat{u} \in  \mathbb{R}^d} |\hat{u}|^k \Big( |G(\hat{u})| + \sum_{j=1}^d | \partial_{\hat{e}_j} G (\hat{u}) | + \sum_{i,j=1}^d | \partial_{\hat{e}_i \hat{e}_j} G (\hat{u}) | \Big) < \infty;
	\end{align*}
	\item $S^{1, 2} \big( [0,T] \times  \mathbb{R}^d \big)$, the subset of $C^{1, 2}  \big( [0,T] \times  \mathbb{R}^d \big)$ whose elements $G$ satisfy \begin{align*}
	\forall k \in \mathbb{N},  \quad \sup_{s \in [0, T]} \big\{ \| G_s \|_{k, \; S^{2}(   \mathbb{R}^d)} + \| \partial_s G_s \|_{k, \; S^{2}(   \mathbb{R}^d)} \big\} < \infty
	;
\end{align*}

		\item $ S_{\gamma}^{d}$ given by
	\begin{equation} \label{defSdifgamma}
	S^{d}_{\gamma}:=
	\begin{cases}
		C_c^{1, 2} \big( [0,T] \times  \mathbb{R}^d \big)
		, \quad & \gamma \in (0, \; 1], \; d=1,
		\;\text{ or }\;\gamma \in (0, \; 2), \; d\geq 2,	
		\\
		S^{1, 2} \big( [0,T] \times  \mathbb{R}^d \big)
		, \quad & \gamma \in (1, \; 2), \; d=1;
	\end{cases}
\end{equation}

	\item  $S_{\gamma,0}^{d}$ as the set of functions $G: [0,T] \times \mathbb{R}^d \mapsto \mathbb{R}$ such that, for every $t\in [0,T]$, there exist $G^{-}, G^{+} \in S_{\gamma}^{d}$ satisfying
			\begin{align} \label{defGdisc}
		\forall (\hat{u}_{\star}, u_d) \in \mathbb{R}^d, \quad G_t(\hat{u}_{\star}, u_d) =\mathbbm{1}_{\{ u_d < 0 \}} G^{-}_t(\hat{u}_{\star}, u_d) + \mathbbm{1}_{ \{ u_d \geq 0\}} G^{+}_t(\hat{u}_{\star}, u_d),
	\end{align}
		and $S_{\gamma,\star}^{d}$ as the subset of $S_{\gamma,0}^{d}$ whose elements G satisfy
		\begin{align}\label{SROB0}
			\forall 
			\hat{u} \in \mathbb{B}:=\mathbb{R}^{d-1} \times \{0\}, \; \forall s \in[0,T], \quad
			G^{-}_s(\hat{u})=G^{+}_s(\hat{u});
		\end{align}
	\item 	For all $\gamma \in (0,2) $ and each $d \geq 1,$ introduce
		\begin{equation} \label{defSalgamd}
			S_{\gamma,\kappa}^{d}:=
			\begin{cases}
				S_{\gamma}^{d},	
				&\kappa>0,\\
				S_{\gamma,0}^{d},	&\kappa=0.
			\end{cases}
	\end{equation}
	\end{enumerate}
	\end{definition}
Note that, in particular, by choosing $G^{-}=G^{+}$ one can see that $S_{\gamma}^{d} \subsetneq S_{\gamma,0}^{d}$. 

In what follows, we let $\mcb {M}^{+}$ denote the space of non-negative Radon measures on $\mathbb{R}^d$ endowed with the \textit{vague topology} \cite[Chapter 30]{bauer}. More precisely, a sequence $(\nu_n)_n$ in $\mcb M^+$ converges vaguely to $\nu\in \mcb M^+$ if, for any $f\in C_c^0(\mathbb{R}^d)$, it holds $\nu_n(f)\xrightarrow{n \rightarrow \infty}\nu(f)$, where $\nu(f_n):=\int f_n\rmd\nu$.

According to item \textit{(1)} in Proposition \ref{propaproxtest} below, $\mcb F_{t}^{\gamma,\kappa}(G,\rho,g)=0$ is well-defined for every $\kappa \geq 0$, $\gamma \in (0, \;2)$, $d \geq 1$ and $t \in [0,T]$, whenever $G \in  S_{\gamma,\kappa}^{d}$ and $g \in L^{\infty}(\mathbb{R}^d)$. Furthermore, having in mind the fact that elements of $\mcb {M}^{+}$ act on $C_c^{0}(  \mathbb{R}^d)$, we observe that the terms of \eqref{weakform} can be approximated by functions in $C_c^{0,0}\big( [0,T] \times \mathbb{R}^d \big)$, in the sense of item \textit{(2)} in Proposition \ref{propaproxtest}. We postpone the proof of the next proposition to Appendix \ref{antoolstest}.
\begin{prop} \label{propaproxtest}
	Let $\kappa \geq 0$, $\gamma \in (0, \;2)$, $d \geq 1$ and fix $G \in  S_{\gamma,\kappa}^{d}$. 
	\begin{enumerate}
		\item It holds
			\begin{align}\label{L1test}
				\sup_{s\in[0,T]}\norm{G_s}_1
				+\int_0^T
				\norm{\partial_s G_{s}}_1 \;
				\rmd s 
				+ \int_0^T
				\norm{\mathbb{L}_{\kappa}^{\gamma} G_s}_1 \;
				\rmd s
				<\infty,
			\end{align}
where $\mathbb{L}_{\kappa}^{\gamma}$ is given by Definition \ref{def:dist_lap} below; 								
		\item For every $\varepsilon >0$ and every $\lambda_1, \lambda_2 \geq 0$, there exist $\widetilde{G}, \; H \in C_c^{0,0}\big( [0,T] \times \mathbb{R}^d \big)$ satisfying
	\begin{align} \label{aproxtest}
		&
		\sup_{s\in[0,T]}
		\norm{G_s- \widetilde{G}_s
		}_1
		+\int_0^T
		\norm{
		(
		\lambda_1  \partial_s G_{s} + \lambda_2 \mathbb{L}_{\kappa}^{\gamma} G_s
		)  
		- H_s
		}_{ 1 }
		\; \rmd s 
		< \varepsilon. 
	\end{align}	

\end{enumerate}
\end{prop}

\subsubsection{Fractional Sobolev spaces and fractional operators}\label{subfracoper} 

In the remainder of this subsubsection, $\mathcal{O}$ denotes an open subset of $\mathbb{R}^d$. We introduce the fractional Sobolev spaces and operators as in \cite{hitchhiker} modulo some adaptations due to the possible discontinuity of functions at the barrier.

\begin{definition}[fractional Sobolev space $\mcb{H}^{\gamma/2}$] \label{sobdisc}
	For any $\gamma \in (0,2)$, we define the space $\mcb{H}^{\gamma/2}(\mathcal{O})$ as the space of functions $G \in L^2(\mathcal{O})$ such that
	\begin{align*}
		{[G]_{\mcb{H}^{\gamma/2}(\mathcal{O})}:=}
		c_\gamma\iint_{\mathcal{O}^2} \frac{[G(\hat{u}) - G(\hat{v}) ]^2}{|\hat{u}-\hat{v}|^{{d}+\gamma }} {\rmd \hat{u}} \; \rmd \hat{v} < \infty,
	\end{align*}
	where $c_\gamma$ is the normalizing constant given in \eqref{transition prob}. Next, define the (squared) norm in $\mcb{H}^{\gamma/2}(\mathcal{O})$ as
	\begin{align} \label{normsob}
		\| G \|_{\mcb{H}^{\gamma/2}(\mathcal{O})}^2 
		:= \norm{G}_{L^2(\mathcal{O})}  
		+ [G]_{\mcb{H}^{\gamma/2}(\mathcal{O})}.
	\end{align}
\end{definition}
We account for the possibility of discontinuities at the barrier by introducing the next space. 
\begin{definition}[Fractional spaces $\mcb H^{\gamma/2}_\star$ and $\mcb H^{\gamma/2}_\kappa$]\label{def:frac-space}
	Let $\iota \in \{ +,-\}$ and denote $\mathbb{R}_{\iota}^{d\star}:=\mathbb{R}^{d-1} \times \mathbb{R}_\iota$, where $\mathbb{R}_-=(-\infty,0)$ and $\mathbb{R}_+=(0, \infty)$. We define the space 
	\begin{align*}
		\mcb H_\star^{\gamma/2}(\mathbb{R}^{d})
		:=\left\{
		G: \exists G^\pm\in \mcb H^{\gamma/2}(\mathbb{R}_{\pm}^{d\star}):
		\quad 
		G = G^{\pm} \;\text{on} \; \mathbb{R}_{\pm}^{d\star}
		\right\}.
	\end{align*}	
Furthermore, we define $\norm{G}^2_{\mcb H_\star^{\gamma/2}(\mathbb{R}^{d})}:=\norm{G^{-}}_{\mcb{H}^{\gamma/2}(\mathbb{R}^{d\star}_-)}^2
	+\norm{G^{+}}_{\mcb{H}^{\gamma/2}(\mathbb{R}^{d\star}_+)}^2$. 	
		Moreover, for each $\kappa\geq0$ define
	\begin{align} \label{defHkap}
		\mcb H_\kappa^{\gamma/2}(\mathbb{R}^{d})
		:=\begin{cases}
			\mcb H_\star^{\gamma/2}(\mathbb{R}^{d})
			,&\kappa=0,\\
			\mcb H^{\gamma/2}(\mathbb{R}^{d})
			,&\kappa>0.
		\end{cases}
	\end{align}
	
\end{definition}
We make several remarks that will be useful later on.
\begin{rem} \label{reminc1}
	First, note that $(\mathbb{R}_-^d\cup\mathbb{R}_+^d)^c =\mathbb{R}^{d-1}\times\{0\}$, which corresponds to the macroscopic barrier. Moreover, one can see that $L^2 \big( 0,T ; \; \mcb H^{\gamma/2}(\mathbb{R}^d) \big) \subsetneq L^2 \big( 0,T ; \;  \mcb H_\star^{\gamma/2}(\mathbb{R}^{d}) \big)$ by taking $G^{-}=G^{+}=G$ for every $G \in L^2 \big( 0,T ; \; \mcb H^{\gamma/2}(\mathbb{R}^d) \big)$.
\end{rem}
Now we define the fractional operators that will play a role in this work.
\begin{definition}[Regional fractional Laplacian]\label{def:regi_lap}
We follow closely the notation given in \cite{reflected}. Let $\gamma \in (0,2)$ and $G:\mathbb{R}^d \mapsto \mathbb{R}$ be such that
		\begin{align*}
			\int_{\mathcal{O}} \frac{|G(\hat{u})|}{(1+|\hat{u}|)^{d+\gamma}} d\hat{u} < \infty.
		\end{align*}
The last condition is assumed in order to ensure that the right-hand side of \eqref{fraclap-eps} below is finite. Next, let $\varepsilon >0$. We introduce the operator $ \Delta_{\eps, \mathcal{O}}^{ \gamma / 2 }$ as acting on functions $G:\mathbb{R}^d \mapsto \mathbb{R}$ as		
		 		 \begin{equation}\label{fraclap-eps}
		\forall \hat{u} \in \mathcal{O}
		, \quad  
		 \Delta_{\eps, \mathcal{O}}^{ \gamma / 2 } G(\hat{u})
		:= c_\gamma \mathbbm{1}_{\{ \mathcal{O} \}}(\hat{u}) \int_{\mathcal{O}} \mathbbm{1}_{ \{\hat{v}: |\hat{v}-\hat{u}| \geq \varepsilon\}}(\hat{u}) \frac{G(\hat{u})-G(\hat{v})}{|\hat{v}-\hat{u}|^{d+\gamma}} \rmd \hat{v}.
	\end{equation}	
	The regional fractional Laplacian in $\mathcal{O}$, that we write as $ \Delta_\mathcal{O}^{ \gamma / 2 }$, is defined as the next limit, whenever it exists:
	\begin{align} \label{deflapfracreg}
		 \Delta_\mathcal{O}^{ \gamma / 2 }
		:= \lim_{\varepsilon \rightarrow 0^{+}}  \Delta_{\eps, \mathcal{O}}^{ \gamma / 2 }.
	\end{align}
Moreover, the fractional Laplacian is the resulting operator for $\mathcal{O}=\mathbb{R}^d$, in which case we simply write $ \Delta^{ \gamma / 2 }:=  \Delta_{\mathbb{R}^d}^{ \gamma / 2 }$. 	
\end{definition}
Note that for $\mathcal{O} \subsetneq \mathbb{R}^d$, the function $\Delta_\mathcal{O}^{ \gamma / 2 } G$ is only defined on $\mathcal{O}$ and does not depend on the values of $G$ on $\mathbb{R}^d \setminus \mathcal{O}$. A subtle difference between the regional fractional Laplacian here introduced and others present in the literature (see e.g. \cite{reflected}) is that here $\Delta_\mathcal{O}^{ \gamma / 2 }$ is defined for functions in the whole domain $\mathbb{R}^d$ instead of just $\mathcal{O}$, which allow us to apply it directly to functions whose domain is $\mathbb{R}^d$ with no further justification. 

The precise operator characterizing the effect of the barrier is $\mathbb{L}_{\kappa}^{\gamma}$, given in Definition \ref{def:dist_lap} below.
\begin{definition}[Fractional operators $ \Delta_{\star}^{ \gamma / 2 }$ and $\mathbb{L}_\kappa^\gamma$]\label{def:dist_lap}
		Whenever the corresponding limits exist, we introduce the operator
	\begin{equation} \label{deflapfracabu}
		 \Delta_{\star}^{ \gamma / 2 }:=
		 \Delta_{\mathbb{R}^{d*}_{-}}^{ \gamma / 2 }
		+
		 \Delta_{\mathbb{R}^{d*}_{+}}^{ \gamma / 2 }.
	\end{equation} 	
	The distorted fractional Laplacian $\mathbb{L}_\kappa^\gamma$ is defined, for any $\kappa\geq0$, through
		\begin{equation} \label{defLalfgam}
		\mathbb{L}_{\kappa}^{\gamma}
		= 
		\kappa \Delta^{ \gamma / 2 }
		+(1-\kappa) \Delta_{\star}^{ \gamma / 2 }.
	\end{equation}		
\end{definition}
Depending on the values of $\gamma$ and ${\alpha}$, various types of boundary conditions will arise. In order to state them, we present the directional fractional derivative $\partial^{\gamma/2}_d$. \begin{definition}[Directional fractional derivative $\partial^{\gamma/2}_d$] 	
	For any $\gamma \in (0,2)$ we define the directional derivatives $\partial^{\gamma/2}_{d,+},\partial^{\gamma/2}_{d,-}$ as acting on $G:\mathbb{R}^d \mapsto \mathbb{R}$, whenever the following limit exists, as
\begin{align*} 
		\partial^{\gamma/2}_{d,\pm} G(\hat{u}_{\star},0)
		:=\lim_{r \rightarrow 0^\pm} 
		\abs{r}^{2 - \gamma} \partial_{\hat{e}_d} G(\hat{u}_{\star},r ).
	\end{align*}
\end{definition}

\subsection{Hydrodynamic equations} \label{sechydeqsdif} 
In this subsection, we fix $\gamma \in (0,2)$, $d \geq 1$ and a measurable function $g: \mathbb{R}^d \mapsto [0, \; N_{\text{e}}]$. We recall the functional $\mcb F_{t}^{\gamma,\kappa}(G,\rho, g)$, see \eqref{weakform}.  

In what follows, we recall the spaces in Definitions \ref{defschw2}; the fractional spaces and operators, given in Definition \ref{def:frac-space} and \ref{def:dist_lap}, respectively; the function $F$ given in \eqref{defF}; and the distorted fractional operator $\mathbb{L}_{\kappa}^{\gamma}$ in \eqref{distfrac}. 

Each notion of solution will be related with the limit $\alpha:=\lim_{n\to+\infty}\alpha_n\geq 0$, associated with the transport of mass through the microscopic barrier, as in Definition \ref{def:alpha}. In this way, in what follows, one can see $\kappa$ as taking the role of $\alpha$ (see Theorem \ref{hydlim}).

Next we state two definitions that will be associated with the setting $\alpha >0$. We recall that $F$ is well defined since we assume, at least, \eqref{convabsF}.

\begin{definition}\label{def:weak-open}
For any $\kappa  \in \mathbb{R}_+\backslash\{1\}$, we say that $\rho :[0,T] \times \mathbb{R}^d \mapsto [0, \; N_{\text{e}}]$ is a weak solution of 
	\begin{equation} \label{eqhydfracbetazero}
	\begin{cases}
		\partial_t \rho_t(\hat{u}) = (\mathbb{L}_{\kappa }^{\gamma} F)  \big( \rho_t(\hat{u}) \big),\quad &  (t,\hat{u}) \in [0,T] \times \mathbb{R}^{d}, \\
		\partial^{\gamma/2}_{d, +} F \big( \rho_t(\hat{u}_{\star} ,0^{+}) \big)  = \partial^{\gamma/2}_{d, -} F \big( \rho_t(\hat{u}_{\star} ,0^{-}) \big), \quad & (t,\hat{u}_{\star}) \in (0,T] \times \mathbb{R}^{d-1}, \\
		\rho_0(\hat{u}) = g(\hat{u}), \quad & \hat{u} \in \mathbb{R}^d,  
	\end{cases}
\end{equation}
	if the following conditions hold:
	\begin{enumerate}
		\item 
		there exists $ \theta_1 \in (0, \; N_{\text{e}}) $ such that $\rho-\theta_1 \in L^2 \big(0, T ; \; L^2( \mathbb{R}^d ) \big)$;
		\item
		there exists $ \theta_2 \in (0, \; N_{\text{e}})$ such that $F(\rho)- F(\theta_2)  \in L^2 \big(0, T ; \; \mcb{H}_{ \kappa }^{ \gamma/ 2 } ( \mathbb{R}^d ) \big)$;
		\item 
		$\mcb F_{t}^{\gamma, \kappa }(G,\rho, g)=0$ for every $G \in   S_{\gamma, \kappa}^{d}$ and $t \in [0,T]$.
	\end{enumerate}    
        
\end{definition}

\begin{definition} \label{defehsd}
We say that $\rho :[0,T] \times \mathbb{R}^d \mapsto [0, \; N_{\text{e}}]$ is a weak solution of 
	\begin{equation} \label{ehsd}
	\begin{cases}
		\partial_t \rho_t(\hat{u}) = \Delta^{\gamma/2} F  \big( \rho_t(\hat{u}) \big), \quad & (t,\hat{u}) \in [0,T] \times \mathbb{R}^d, \\
		\rho_0(\hat{u}) = g(\hat{u}), \quad & \hat{u} \in \mathbb{R}^d, 
	\end{cases}
\end{equation}
if the items (1), (2) and (3) of Definition \ref{def:weak-open} are satisfied, with $\kappa=1$. 
\end{definition}


Now we introduce two notions of solutions that will be associated with $\alpha=0$. The following equations will arise depending on the speed of convergence of $\alpha_n\xrightarrow{n   \rightarrow \infty  }0$. 
\begin{definition}
We say that $\rho :[0,T] \times \mathbb{R}^d \mapsto [0, \; N_{\text{e}}]$ is a weak solution of 
	\begin{equation} \label{eqhydfracdifneu}
	\begin{cases}
		\partial_t \rho_t(\hat{u}) = (\mathbb{L}_{ 0 }^{\gamma} F)  \big( \rho_t(\hat{u}) \big),\quad &  (t,\hat{u}) \in [0,T] \times \mathbb{R}^{d}-\mathbb{B}, \\
		\partial^{\gamma/2}_{d, +} F \big( \rho_t(\hat{u}_{\star} ,0^{+}) \big)  = \partial^{\gamma/2}_{d, -} F \big( \rho_t(\hat{u}_{\star} ,0^{-}) \big)=0, \quad & (t,\hat{u}_{\star}) \in (0,T] \times \mathbb{R}^{d-1}, \\
		\rho_0(\hat{u}) = g(\hat{u}), \quad & \hat{u} \in \mathbb{R}^d,  
	\end{cases}
\end{equation}
	if the items (1), (2) and (3) of Definition \ref{def:weak-open} are satisfied, with $\kappa=0$. \end{definition}

Our final notion of solution corresponds to Definition \ref{def:weak-open} with $\kappa=0$, under the restricted space of test functions $S_{\gamma,\star}^{d} \subsetneq S_{\gamma,0}^{d}$ (see \eqref{SROB0}).
\begin{definition}\label{def:weak-star}
We say that $\rho :[0,T] \times \mathbb{R}^d \mapsto [0, \; N_{\text{e}}]$ is a weak solution of 
	\begin{equation} \label{ehlips}
	\begin{cases}
		\partial_t \rho_t(\hat{u}) = (\mathbb{L}_{ 0 }^{\gamma} F)  \big( \rho_t(\hat{u}) \big),\quad &  (t,\hat{u}) \in [0,T] \times \mathbb{R}^{d}-\mathbb{B}, \\
		\partial^{\gamma/2}_{d, +} F \big( \rho_t(\hat{u}_{\star} ,0^{+}) \big)  = \partial^{\gamma/2}_{d, -} F \big( \rho_t(\hat{u}_{\star} ,0^{-}) \big), \quad & (t,\hat{u}_{\star}) \in (0,T] \times \mathbb{R}^{d-1}, \\
		\rho_0(\hat{u}) = g(\hat{u}), \quad & \hat{u} \in \mathbb{R}^d,  
	\end{cases}
\end{equation}
	if the items (1), (2) of Definition \ref{defehsd} are satisfied, with $\kappa=0$; and $\mcb F_{t}^{\gamma, 0}(G,\rho,g)=0$ for every $G \in   S_{\gamma,\star}^{d}$ and for every $t \in [0,T]$.
\end{definition}
Concerning the uniqueness of solutions of the formulations just introduced, in Appendix \ref{uniqweak} we prove the following lemma.
\begin{lem} \label{lemuniq}	
Assume that $d \in \mathbb{N}_+$ and $F$ is increasing. Then there exists at most one weak solution of
\begin{itemize}
\item
\eqref{eqhydfracbetazero}, for any $\kappa \in \mathbb{R}_+ \setminus \{1\}$ and $\gamma \in (0,2)$; 
\item
\eqref{ehsd}, for any $\gamma \in (0,2)$;
\item
 \eqref{eqhydfracdifneu}, for any $\gamma \in (0,2)$;
 \item
 \eqref{ehlips}, for any $\gamma \in (0,1]$.
\end{itemize}

\end{lem}

\begin{rem} \label{remlemuniq}
Uniqueness of weak solutions of \eqref{ehlips} when $\gamma \in (1, \; 2)$ is left as an open problem.
\end{rem}	

\subsection{Space of Radon measures and empirical measures} \label{secradon}
The link between the microscopic and macroscopic scales, or more precisely, the correspondence of $\Omega$ with $\mcb{M}^{+}$ defined below, is realized through the \textit{empirical measure}, that is given in Definition \ref{defgabemp} below. 

Let us fix once and for all an arbitrary $T>0$, which leads to a finite time horizon $[0,T]$.
We will study the $n^\gamma$ time-accelerated process, generated by $n^\gamma \mcb L_{n,\alpha}^\gamma$. Therefore, we shorten $\eta_t^n:=\eta_{n^\gamma t}$. 
\begin{definition}[Empirical measure] \label{defgabemp}
	For any $n\in\mathbb{N}_+$ and $\eta\in\Omega$, define the random measure $\pi^n$ as
	\begin{equation}\label{gab:emp}
		\pi^n(\eta, {\rmd \hat{u}}) := \frac{1}{n^d} \sum_{ \hat{x} } \eta(\hat{x}) \delta_{ \hat{x} / n}(\rmd \hat{u}),
	\end{equation}
	{where $\delta_{\hat{x}/n}$ is the Dirac measure on $\hat{x}/n \in \mathbb{R}^d$. Its time evolution is defined as $\pi_t^n(\eta,\rmd\hat{u}):=\pi^n(\eta_{t}^n,\rmd\hat{u})$. In general, for $G:\mathbb{R}^d\to\mathbb{R}$ we will shorten} $\langle \pi^n(\eta), G \rangle\equiv \int_{\mathbb{R}^d} G(\hat{u}) \pi^n(\eta,\rmd\hat{u})$.
\end{definition} 
We will need to apply the empirical measure to  functions that are in $L^1(\mathbb{R}^d)$. The lift from a test function in $C^0_c(\mathbb{R}^d)$ to $L^1(\mathbb{R}^d)$ can be made by density arguments.
 
\begin{definition} \label{assoc}
	We say that a sequence $(\mu_n)_{n \geq 1}$ of probability measures on $\Omega$ is \textit{associated} to a measurable profile $g: \mathbb{R}^d \mapsto [0, N_{\text{e}}]$ if for every $\delta >0$ and every $G \in C_c^{0}(\mathbb{R}^d)$
	\begin{equation} \label{eqassoc}
		\lim_{n \rightarrow \infty} \mu_n \big( \eta \in \Omega: | \langle \pi^n(\eta), G \rangle - \langle g , G \rangle | > \delta    \big)=0.
\end{equation}	
	This means that the sequence of random empirical measures $(\pi^n_0)_{ n \geq 1}$ converges weakly to the deterministic measure induced by $g$ in $\mcb{M}^{+}$, in probability with respect to $\mu_n$. This can be interpreted as a weak Law of Large Numbers.
\end{definition}
Making use of an approximation sequence if needed, we can extend \eqref{eqassoc} to $G \in L^{1}(\mathbb{R}^d)$. Now, we introduce the Skorohod space of trajectories.
\begin{definition}
Let $(N, \| \cdot \|_N)$ be a metric space and let $\mcb{D}_N([0,T])$ be the space of the c\`adl\`ag (right-continuous and with left limits) trajectories.
\end{definition}
We note that what is often known in the literature as the \textit{Skorohod topology} of $\mcb{D}_N([0,T])$, is the topology defined by a metric \cite[Chapter 3]{Bill}, induced by a particular distance in $\mcb{D}_N([0,T])$ (see, for instance, Appendix B.0.1 in \cite{dismestotav} and (14.12), (14.13) in \cite{Bill} for more details). 

Next, we observe that $\Omega$ is a metrizable space, thus the space $\mcb D_{\Omega}([0, T])$ is well-defined. In particular, $(\eta_t^n)_{t \in [0,T]} \in \mcb D_{\Omega}([0, T]) $. 

We introduce the subset
\begin{equation} \label{defMcb}
	\mcb {M}^{+}_{N_{\text{e}}} := \Bigg\{ \tilde{\pi} \in  \mcb {M}^{+}: \quad \forall G \in C_c^0(\mathbb{R}^d), \; \; \Bigg| \int G \; \rmd \tilde{\pi} \Bigg| \leq N_{\text{e}} \|G \|_{\infty}   \Leb (\overline{\supp \; G})<\infty  \Bigg\},
\end{equation}
where $\Leb(\cdot)$ corresponds to the Lebesgue measure in $\mathbb{R}^d$, and $\overline{\supp \cdot}$ the closure of the support. 

As it is done in Appendix D of \cite{ddimhydlim}, one can define a metric, $\norm{\cdot}_{\mcb {M}^{+}}$, on $\mcb {M}^{+}$. Since $\mcb {M}^{+}_{N_{\text{e}}}$ is a compact subspace of $ \mcb {M}^{+}$, it holds that $(\mcb {M}^{+}_{N_{\text{e}}}, \;  \| \cdot \|_{ \mcb {M}^{+} })$ is a complete separable metric space (see Appendix D of \cite{ddimhydlim} and Theorem 31.5 in \cite{bauer}).  In particular, the space $\mcb D_{\mcb {M}^{+}_{N_{\text{e}}}}([0, T])$ is also well-defined, and $(\pi_t^n)_{t \in [0,T]} \in \mcb D_{\mcb {M}^{+}_{N_{\text{e}}}}([0, T])$. 

For  a probability measure $\mu_n$ on $\Omega$, we denote by $\mathbb{P}_{\mu_n}$ the probability measure on $\mcb{D}_{\Omega}([0,T])$ induced by $(\eta_{t}^{n})_{t \in [0,T]}$ and $\mu_n$, and we denote the expectation with respect to $\mathbb{P}_{\mu_n}$ by $\mathbb{E}_{\mu_n}$. Furthermore, let $(\mathbb{Q}_n)_{n\geq1}$ be the sequence of probability measures on $\mcb D_{\mcb M^+_{N_{\text{e}}}}([0,T])$ induced by the Markov process $(\pi_t^n)_{t\geq0}$ and $\mathbb{P}_{\mu_n}$, namely $\mathbb{Q}_n=\mathbb{P}_{\mu_n}\circ(\pi^n_\cdot)^{-1}$. 

\subsection{The main theorem}

In order to prove Theorem \ref{hydlim} we will assume the following. 
\begin{hipos}\label{hyp:coeff} 
	Let the sequences $(b_k^{\pm})_{k \geq 1}$ introduced in \eqref{defF} be such that:
	\begin{enumerate}[label=(\roman*)]
		\item There exists a \textit{slow growth}:  
\begin{align}\label{h2}
	 f_{\infty}'  :=\sum_{k=1}^{\infty} k( |b_k^+| + |b_k^-|) N_{\text{e}}^{k-1} < \infty;
\end{align}	

		\item {\textit{Kinetic lower bound}: there exist constants $\underline{b}^-, \underline{b}^+\geq0$ which are not simultaneously zero, and $\underline{k}^-,\underline{k}^+\in\mathbb{N}$ such that for every $n \in \mathbb{N}_+$, $\eta\in \Omega $ and $  
			\hat{x},\hat{y}\in\mathbb{Z}^d,
			$}
		\begin{align}\label{h3}
			c_{\hat{x},\hat{y}}^{n}(\eta)\geq \underline{b}^+c_{\hat{x},\hat{y}}^{(\underline{k}^+)}(\eta)
			+
			\underline{b}^-c_{\hat{x},\hat{y}}^{(\underline{k}^-)}(\tilde{\eta}).
		\end{align}
		If $\underline{b}^+=0$ then $\underline{k}^+:=1$, and if $\underline{b}^-=0$, then $\underline{k}^-:=1$. 	\end{enumerate}
\end{hipos}
We observe that  the previous \textit{assumptions} corresponds to a generalization of the results in \cite{gabriel} for the slow diffusion scenario. For fast diffusion regimes, \eqref{h2} does not hold, and the forthcoming Hypothesis \ref{hipotight} will be necessary in that context (see also Example \ref{exemfast}).

\begin{rem}[About the assumptions]\label{rem:ass}

The condition in \eqref{h2} is of microscopic nature, originated from requiring $\limsup_n c_{\hat{x},\hat{y}}^{n}(\eta)<\infty$ for any $\eta\in\Omega$. This implies that $F$ is regular enough in order to prove the uniqueness of solutions of the weak formulation in a simple way (see condition $(2)$ in \eqref{eqhydfracbetazero}--\eqref{ehlips}, Lemma \ref{lemuniq}, and Proposition \ref{propcond2weak}). 

Under Hypothesis \ref{hyp:coeff} \textbf{(ii)}, it holds $c_{\hat{x},\hat{y}}^n(\eta)
	\geq 0 $, hence the operator $\mcb L_n^\gamma$ is a well-defined Markov generator. From this fact, it is straightforward to check that $F$ is \textit{increasing}, as it can be written as $F(\rho)=\int_0^\rho D(v)\rmd v$, with $D(\rho):=\sum_{k\geq1} [b_k^{+} k\rho^{k-1} + b_k^{-}k(N_{\text{e}}-\rho)^{k-1}]$. 
	
	Hypothesis \ref{hyp:coeff} \textbf{(ii)} also implies that $F$ is either uniformly or polynomially bounded from below, depending on the values of $k^\pm$ and $\underline{b}^\pm$, as $F(\rho)\geq \underline{b}^+\rho^{k^+} -\underline{b}^-(N_{\text{e}}-\rho)^{k^-}$.
\end{rem}

In order to state the main result of this article, we introduce a last quantity. Let $\mu, \nu$ be two probability measures on $\Omega$. The relative entropy of $\mu$ with respect to $\nu$ is defined by
\begin{align*}
H( \mu | \nu ):= \sup_f \Bigg  \{ \int_{\Omega} f d \mu - \log \Bigg( \int_{\Omega} e^{f} d \nu \Bigg) \Bigg\},
\end{align*}
where the supremum above is carried over all the bounded functions $f: \Omega \mapsto \mathbb{R}$.

Our main result lives in the context of the \textit{slow}-diffusion scenario in Hypothesis \ref{hyp:coeff}.  
\begin{thm}\label{hydlim}
 Fix $\gamma \in (0,2)$, $d \geq 1$ and a measurable function $g: \mathbb{R}^d \mapsto [0, \; N_{\text{e}}]$. Let $F$ be as in \eqref{defF} and satisfying Hypotheses \ref{hyp:coeff}. Let $(\mu_n)_{n \geq 1}$ be a sequence of probability measures in $\Omega$ associated to g in the sense of Definition \ref{assoc} and such that
	\begin{equation}  \label{2entbound}
				\exists \theta \in (0,1), \; C_{\theta} >0: \forall n \geq 1, \quad H ( \mu_n | \nu_\theta ) \leq C_{\theta} n^d, 
	\end{equation} 
where $\nu_\theta$ is given in Definition \ref{definvmeas}. Then for any $t \in [0,T]$, $G \in C_c^0(\mathbb{R}^d)$ and $\delta>0$,
	\begin{equation*} 
		\lim_{n \rightarrow \infty}\mathbb{P}_{\mu_n}\Bigg(  \eta_{\cdot}^n \in \mcb {D}_{\Omega}([0,T]): \Bigg| \frac{1}{n^d} \sum_{ \hat{x} } G(\tfrac{\hat{x}}{n}) \eta_t^n(\hat{x}) - \int_{\mathbb{R}^d} G(\hat{u})\rho_t(\hat{u})\,\rmd \hat{u}  \,  \Bigg| > \delta \Bigg)=0,
	\end{equation*}
	where $\rho$ is the unique weak solution of 	\begin{align*}
	\begin{cases}
       \text{\eqref{eqhydfracbetazero} with} \quad \kappa=\alpha,   \quad   \; & \text{if}  \quad \alpha \in \mathbb{R}_+ \setminus \{1\} \quad \text{and} \quad \gamma \in (0,2);\\    
		\eqref{ehsd},   \quad   \; & \text{if} \quad \alpha=1 \quad \text{and} \quad \gamma \in (0,2); \\
		\eqref{eqhydfracdifneu},   \quad   \; & \text{if} \quad \alpha=0, \; \gamma \in (0,2)  \quad \text{and} \quad \lim_{n \rightarrow \infty}  \alpha_n r_n^\gamma =0; \\
	\eqref{ehlips},   \quad   \; & \text{if} \quad \alpha=0,  \; \gamma =1 \quad \text{and}  \quad \lim_{n \rightarrow \infty} \alpha_n r_n^\gamma  \in (0, \infty] ,
	\end{cases}
\end{align*}
where for any $n \in \mathbb{N}_+$ and $\gamma \in (0, \;2)$, $r_n^\gamma$ is given by
\begin{equation} \label{rngamma}
            r_n^\gamma
		:=
		\mathbbm{1}_{\gamma \in (0,1)}
		+\log(n)\mathbbm{1}_{\gamma = 1}
		+n^{\gamma-1}\mathbbm{1}_{\gamma \in (1,2)}
            .
\end{equation} 
Moreover, $(\mathbb{Q}_n)_{n \geq 1}$ is tight and all limit points are concentrated on trajectories of the form $\pi_t(\rmd \hat{u})=\rho_t(\hat{u}) \; \rmd \hat{u}$, where $\rho$ is a weak solution of \eqref{ehlips}, if
	\begin{align*}
	\alpha=0,  \; \gamma \in (1,2), \quad \text{and} \quad  \lim_{n \rightarrow \infty} \alpha_n r_n^\gamma  \in (0, \infty].
\end{align*}
\end{thm}
When $\gamma \in (1,2)$, we are not able to ensure the uniqueness of weak solutions for \eqref{ehlips}, recall Remark \ref{remlemuniq}. While tightness ensures the existence of converging subsequences of $(\mathbb{Q}_n)_{n \geq 1}$, uniqueness of the limit is needed in order to guarantee that the sequence converges. In order to show uniqueness, in this paper we rely on the regularity of $F$ as stated in the second condition of the definitions of weak solutions. This regularity is not expected in a fast-diffusion regime. 

In what follows we will state a weaker version of \eqref{h2} which is satisfied by the fast diffusion setting in \cite{gabriel} and also allows the derivation of the integral equations in Theorem \ref{hydlim}, namely condition $(3)$ in \eqref{eqhydfracbetazero}--\eqref{ehlips}. In order to do so, define the sequence $(f_n')_{n\geq 1}$ through
   \begin{align}  \label{f-f'}
        f_n':=\sum_{k=1}^{\ell_n} ( |b_k^+| + |b_k^-|) kN_{\text{e}}^{k-1}, \quad n \in \mathbb{N}_+.
    \end{align}
We stress that $(f_n')_{n\geq 1}$ depends on $(\ell_n)_{n\geq 1}$, but we do not carry it on $(\ell_n)_{n\geq 1}$ in order to avoid a heavier notation. 

We observe that without the presence of a slow barrier, such as happens in the case of the models in \cite{renato, CG}, in order to obtain the integral equations it would be enough to replace \eqref{h2} by the following hypothesis.
\begin{hipo}\label{hipotight}
Let $\gamma \in (0,2)$. The sequence $(\ell_n)_{n\geq 1}$ is such that $f_n' r_n^{\gamma} n^{-1}\xrightarrow{n\to \infty}0$.
\end{hipo}
Hypothesis \ref{hipotight} is enough to ensure that $(\mathbb{Q}_{n})_{ n \geq 1 }$ is tight, see Proposition \ref{Qntight} below. Moreover, it is weaker than \eqref{h2}, due to \eqref{rngamma}. Now we relate it with the fast diffusion setting as in \cite{gabriel} through the following example.
\begin{exmp} \label{exemfast}
    For $F(\rho)=\rho^{m}$ with $m \in (0,1)$, from the Generalized Binomial Theorem,
    \begin{align*}
        F(\rho)
        =\sum_{k\geq0}(-1)^k
        \binom{m}{k}N_{\text{e}}^{m-k}(N_{\text{e}}-\rho)^k
        \Rightarrow       
            f_{n}'
            =N_{\text{e}}^{m}
            \sum_{k=1}^{\ell_n} \Big|\binom{m}{k} \Big|k.
    \end{align*}
    From \cite[Lemma A.1]{gabriel}, since $|\binom{m}{k}\big|\lesssim k^{-(m+1)}$ for $k\geq 1$, then \eqref{convabsF} holds and $f_{n}'\lesssim \ell_n^{1-m}/(1-m)$. In particular, it is enough to let $\ell_n=n^{\delta}$, with $0<\delta<\frac{\min\{2-\gamma,1\}}{1-m}$ to get Hypothesis \ref{hipotight}.
    Moreover, recalling \eqref{cons-series}, this particular choice of $F$ is associated with the constraint
    \begin{align}
        c_{\hat{x},\hat{y}}^n(\eta)
        =m\sum_{k=0}^{\ell_n}
        (-1)^{k}\binom{m-1}{k}N_{\text{e}}^{m-1-k}
        c_{\hat{x},\hat{y}}^{(k+1)}(\tilde{\eta}),
    \end{align}
    and Hypothesis \ref{hyp:coeff} \textbf{(ii)} is satisfied with $\underline{b}^+=0$, $\underline{k}^\pm=1$ and $\underline{b}^-=m$, due to the change of the sign of the generalized binomial coefficients: $(-1)^{k+1}\binom{m}{k}>0$ (see \cite{gabriel} for more details).
\end{exmp}

However, due to the presence of the slow barrier, the derivation of the integral equations will require a slightly more technical statement than the one of Hypothesis \ref{hipotight}.  
\begin{hipo}  \label{hipoerror}
Let $\gamma \in (0,2)$. The sequences $(\alpha_n)_{n \geq 1}$ and $(\ell_n)_{n \geq 1}$ are such that 
\begin{align*}
\begin{cases}
\quad\; f_n' r_n^\gamma    \big( n^{-1} + \alpha_n \big)
\xrightarrow{n\to \infty}0, 
& \text{if} \quad 
\alpha_n r_n^\gamma \xrightarrow{n\to \infty}0 = \alpha, 
\\
n^{-1} f_n' \big( r_n^\gamma   + \alpha_n  n^{\gamma/2} \big)\xrightarrow{n\to \infty}, 
& \text{otherwise}.
\end{cases}
\end{align*}   
\end{hipo}
Hypothesis \ref{hipoerror} is actually weaker than \eqref{h2}. Indeed, under \eqref{h2} we have that $f_n' \leq C$ for any $n \geq 1$, for some finite constant $C>0$. This leads to
\begin{align*}
\lim_{n \rightarrow \infty}   n^{-1} f_n' \big( r_n^\gamma   + \alpha_n  n^{\gamma/2} \big)  \leq C \lim_{n \rightarrow \infty}    n^{-1}  \big( r_n^\gamma   + \alpha_n  n^{\gamma/2} \big)  \leq (1+ 2\alpha)  C \lim_{n \rightarrow \infty} n^{-1} n^{\gamma/2} =0,
\end{align*}
for every $\gamma \in (0, 2)$. Moreover, if $\lim_{n \rightarrow \infty} \alpha_n r_n^\gamma =0$, then
\begin{align*}
\lim_{n \rightarrow \infty}  f_n' r_n^\gamma    \big( n^{-1} + \alpha_n \big) \leq C \lim_{n \rightarrow \infty}   r_n^\gamma     n^{-1} +C \lim_{n \rightarrow \infty}  \alpha_n  r_n^\gamma  = C \cdot 0 + C \cdot 0 =0,
\end{align*}
for every $\gamma \in (0, 2)$. Above we applied \eqref{rngamma} to get $0 \leq \lim_{n \rightarrow \infty}  r_n^\gamma     n^{-1} \leq \lim_{n \rightarrow \infty} n^{-1} n^{\gamma/2} =0$.

We now present a remark on the entropy bound \eqref{2entbound}.
\begin{rem}[Relative entropy bound]\label{rem:entropy}

The bound \eqref{2entbound} is of paramount importance to derive the conditions ($1$) and ($2$) in the definition of weak solutions of \eqref{eqhydfracbetazero}\gab{--}\eqref{ehlips}, which are known in the literature as \textit{energy estimates}. {In particular, the regularity of $F(\rho)$} in the aforementioned conditions is crucial to apply Oleinik's method to derive the uniqueness of weak solutions (for some equations), in the same way as it is done in \cite{CGJ2, renato, CG}. See more details in Appendix \ref{uniqweak}.
	
\end{rem}	
	
\subsubsection{Strategy of the proof}

Now we describe the strategy of the proof of Theorem \ref{hydlim}. We follow the steps of the entropy method introduced in \cite{GPV}. To that end, for every $n \geq 1$, we recall that $\mathbb{Q}_n$ is the probability measure on $\mcb D_{\mcb {M}^{+}_{N_{\text{e}}}}([0, T])$, the space introduced at the end of Subsection \ref{secradon}, induced by $\mathbb{P}_{\mu_n}$ and $(\pi_t^n)_{t \in [0,T]}$. 
 In Section \ref{sectight}, we prove
that the sequence of probability measures $(\mathbb{Q}_{n})_{ n \geq 1 }$ is tight with respect to the Skorohod topology of $\mcb D_{\mcb {M}^{+}_{N_{\text{e}}}}([0, T])$, see Proposition \ref{Qntight}, from where it follows that $(\mathbb{Q}_n)_{n \geq 1}$ has at least one limit point $\mathbb{Q}$. In Appendix \ref{secabscont}, we prove that any limit point $\mathbb{Q}$ is concentrated on trajectories of measures that are absolutely continuous with respect to the Lebesgue measure, see 
Proposition \ref{Qabscont}. We observe that the models in \cite{kipnis1998scaling} deal with a finite volume, in opposition to our setting. Since up to the best of our knowledge, Proposition \ref{Qabscont} was not proved in the literature for systems with an infinite volume, we do so in this work. 
  In Section \ref{estenerg}, we show that any such limit point $\mathbb{Q}$ is concentrated on trajectories of measures satisfying some energy estimates, i.e., the first two conditions in the definitions of weak solutions for the corresponding hydrodynamic equations.
 We complete the characterization of $\mathbb{Q}$ in Section \ref{secchar}, where we prove that $\mathbb{Q}$ is concentrated on trajectories of measures satisfying the remaining condition in the aforementioned definitions, such as the corresponding integral equations. 
Section \ref{replem} is dedicated to the proof of  the Replacement  Lemma, i.e. Lemma  \ref{globrep}, which is applied in Sections \ref{estenerg} and \ref{secchar}. Finally, we prove some auxiliary results in Appendices \ref{secfracoper}, \ref{secdiscconv}, \ref{useest} and \ref{antools}.

 \section{Heuristic derivation} \label{secheurlin}

In order to observe a macroscopic evolution in our system, we recall that the time scale of our Markov process will be speeded up by $n^{\gamma}$, in the same way as it was done in \cite{jara2009hydrodynamic, CGJ2, renato, CG}. Before we present the proof of Theorem \ref{hydlim} (see Section \ref{sectight}), we give some heuristic arguments which lead to the integral equations referred in that theorem.

We start with a result which is known in the literature as Dynkin's formula. We omit its proof, but we observe that it can be obtained exactly in the same way as Proposition 3.1 of \cite{ddimhydlim}. Recall the spaces of functions given in Definition \ref{defschw2} and observe that $S_{\gamma}^{d} \subsetneq S_{\gamma, \star}^{d} \subsetneq S_{\gamma, 0}^{d}$, for any $\gamma \in (0,2)$ and $d \in \mathbb{N}_{+}$. In order to apply our reasoning to the largest space of test functions possible, we state Proposition \ref{dynkform} below to $G \in S_{\gamma,0}^{d}$.
\begin{prop} \label{dynkform}
Let $n \geq 1$, $\gamma \in (0,2)$, $d \geq 1$ and $G \in S_{\gamma,0}^{d}$. Then the process $( \mcb M_{t}^{n}(G) )_{0 \leq t \leq T}$ defined by
\begin{equation} \label{defMnt}
\mcb M_{t}^{n}(G) := \langle \pi_{t}^{n},G_t \rangle - \langle \pi_{0}^{n},G_0  \rangle  -  \int_0^t   \langle  \pi_{s}^{n},\partial_s G_s  \rangle \;\rmd s - \int_0^t  n^{\gamma} \mcb L_{n,\alpha}^\gamma ( \langle  \pi_{s}^{n},G_s  \rangle ) \;\rmd s
\end{equation}
is a martingale with respect to the natural filtration $\{ \mcb F_t \}_{t \geq 0}$, where $\mcb F_t:=\sigma ( \{ \eta_s^n \}_{s \in [0,t] } )$ for every $t  \geq 0$. Moreover, for every $t \geq 0$ it holds
\begin{equation} \label{defNnt}
\mathbb{E} \Big[ \big( \mcb M_{t}^{n}(G) \big)^2 \big] =  \mathbb{E} \Big[  \int_0^t n^{\gamma} [ \mcb L_{n,\alpha}^\gamma (  \langle \pi_{t}^{n},G_t \rangle^2 ) - 2 \langle \pi_{t}^{n},G_t \rangle  \mcb L_{n,\alpha}^\gamma (  \langle \pi_{t}^{n},G_t \rangle )  ] \;\rmd s \Big].
\end{equation}
\end{prop}

The main arguments to proceed are the following. Since the quadratic variation of the martingale $\mcb M_t^n(G)$ will vanish as $n\to\infty$, the limiting process is deterministic. One then analyses each term of the martingale above, which is the stochastic analogue of the integral equation \eqref{weakform} . The first three terms of \eqref{defMnt} converge in $L^1(\mathbb{P}_{\mu_n})$, as $n \rightarrow \infty$, to $\langle\rho_t,G_t \rangle -\langle g,G_0 \rangle
-\int_0^t \langle\rho_s,\partial_sG_s \rangle \rmd s$. 
What distinguishes the various integral equations described in Theorem \ref{hydlim} is the rightmost term in \eqref{defMnt}, which is  known in the literature as the \textit{integral term}, that we now study.  

Due to the symmetry of $p_{\gamma}(\cdot)$, $\alpha^n_{\cdot,\cdot}$ and $c^{n,\gamma}_{\cdot,\cdot}$, a simple computation shows that 
\begin{equation}  \label{intterm0}
\int_0^t n^{\gamma} \mcb L_{n,\alpha}^\gamma  ( \langle \pi_s^n, G_s \rangle ) \; \rmd s = \int_0^t \frac{n^{\gamma}}{n^d} \sum_{\hat{x}, \hat{y}}  G_s (\tfrac{\hat{x}}{n}) p_{\gamma}(\hat{y}-\hat{x}) \alpha_{\hat{x},\hat{y}}^n
		c_{\hat{x}, \hat{y}}^{n,\gamma}(\eta_s^n) [\eta_s^n(\hat{y})-\eta_s^n(\hat{x})] \; \rmd s. 
\end{equation}
The last equality comes from the following equality:
\begin{equation} \label{excrule2}
\forall \hat{x}, \hat{y} \in \mathbb{Z}^d, \; \forall \eta \in \Omega, \quad a_{\hat{y},\hat{x}}(\eta) - a_{\hat{x},\hat{y}}(\eta)   = N_{\text{e}} [\eta ( \hat{y} ) - \eta(\hat{x})],
\end{equation}
which is a direct consequence of \eqref{excrule}. 

It is now important to identify the "long-range gradient property" of the model. In \cite[Chapter 4, Definition 2.5]{kipnis1998scaling}, it is said that a \textit{nearest neighbor} dynamics satisfy the gradient property if the current can be expressed as the discrete gradient of some local function. This yields, for example in dimension $d=1$, that there is a function $h:\Omega\to\mathbb{R}$ such that $c_{0,1}(\eta)[\eta(1)-\eta(0)]
	=\nabla_1 h$, where $\mathbf{c}_{0,1}\geq0$ is the jump constraint.

For long-range symmetric jumps, the notion of gradient property introduced in \cite[Part II, subsection 2.4]{spohn:book} generalizes the nearest-neighbor one by requiring the existence of some local function $h:\Omega\to\mathbb{R}$ such that $\mathbf{c}_{x,y}(\eta)(\eta(y)-\eta(x))
	=\mathbf{h}(\tau^y\eta)-\mathbf{h}(\tau^x\eta)$. We will see in the next and subsequent results that the model given in Definition \ref{def:gen_series} satisfies this in a \textit{weaker} sense, which is also sufficient to carry on the arguments that rely on the gradient property. 
\begin{prop}\label{prop:grad}
Recall the definition of $\nabla_j$ in the fourth bullet point of Definition \ref{def:micro-op}.	For any $1\leq j\leq d$ , $k\geq1$ , $\hat{x},\hat{y}\in\mathbb{Z}^d$ and $\eta\in\Omega$ fixed, it holds
	\begin{align}\label{grad:d-dim}
		c_{\hat{x},\hat{y}}^{(k),j}(\eta)[\eta(\hat{y})-\eta(\hat{x})]
		&=
		P^{(k),j}(\tau^{\hat{y}}\eta)
		-P^{(k),j}(\tau^{\hat{x}}\eta)
		-\nabla_j
		A^{(k),j}_{\hat{x},\hat{y}}(\eta)
	\end{align}	
	where 
	\begin{align}
		P^{(k),j}(\eta)
		&=\frac12\prod_{i=0}^{k-1}\eta(i\hat{e}_j)
		+\frac12\prod_{i=0}^{k-1}\eta(-i\hat{e}_j),
		\label{pure}\\
		A^{(k),j}_{\hat{x},\hat{y}}(\eta)
		&=
		M^{(k),j}_{\hat{x},\hat{y}}(\eta)
		-M^{(k),j}_{\hat{y},\hat{x}}(\eta),
		\label{antisym}\\
		M^{(k),j}_{\hat{x},\hat{y}}(\eta)
		&=\frac12\sum_{\ell=1}^{k-1}
		\prod_{i_0=1}^\ell\eta(\hat{x}-i_0\hat{e}_j)
		\prod_{i_1=0}^{k-1-\ell}\eta(\hat{y}+i_1\hat{e}_j)
		\label{mixed}
		.
	\end{align}
\end{prop}
It is important to note that the map $A^{(k),j}$ is \textit{antisymmetric}, $A^{(k),j}_{\hat{x},\hat{y}}=-A^{(k),j}_{\hat{y},\hat{x}}$. It turns out that the antisymmetry is consequence of the symmetrization of the rates (as in \eqref{eq:rates_porous}). Precisely, replacing $c_{\hat{x},\hat{y}}^{(k),j}(\eta)$ in \eqref{grad:d-dim} by the (asymmetric) constraint $r_{\hat{x},\hat{y}}^{(k),j}$ yields the first term in $F^{(k),j}$ and the term $M_{\hat{x},\hat{y}}^{(k),j}$ in $A_{\hat{x},\hat{y}}^{(k),j}$; while the remaining terms are associated with $r_{\hat{y},\hat{x}}^{(k),j}$.

	\begin{definition}\label{def:FA}
		For any $1\leq j\leq d$ fixed and $\hat{x},\hat{y}\in\mathbb{Z}^d$, let
		\begin{align*}
			F^{n,j}_\pm
			&:=
			\sum_{k=1}^{\ell_n}b_k^\pm P^{(k),j}
			\quad\text{and}\quad
			M^{n,j,\pm}_{\hat{x},\hat{y}}
			:=
			\sum_{k=1}^{\ell_n}b_k^\pm M_{\hat{x},\hat{y}}^{(k),j},
		\end{align*}
		and for any $\eta\in\Omega$, let $F^{n,j}(\eta):=F^{n,j}_+(\eta)-F^{n,j}_-(\widetilde{\eta})
		$ and $ M_{\hat{x},\hat{y}}^{n,j}(\eta):=M_{\hat{x},\hat{y}}^{n,j,+}(\eta)-M_{\hat{x},\hat{y}}^{n,j,-}(\widetilde{\eta})$. Moreover, we define $F^n:=\frac{1}{d}\sum_{j=1}^dF^{n,j}$.
	\end{definition}

	Recall the constraint $c_{\cdot,\cdot}^{n,j}(\eta)$ in \eqref{cons-series}, the barrier-effect $\alpha^{n}_{ \hat{x} , \hat{y} }$ in Definition \ref{def:alpha} and $p_\gamma(\cdot)$ in \eqref{transition prob}. From \eqref{grad:d-dim}, a summation by parts, the fact that $G$ either has compact support or decays fast enough, and the antisymmetry property of $A_{\hat{x},\hat{y}}^{j}$, one sees that 
	\begin{align}
		\sum_{\hat{x},\hat{y}}G(\tfrac{\hat{y}}{n})
		c^{n,j}_{\hat{x},\hat{y}}(\eta)
		 [ \eta(\hat{y})-\eta(\hat{x})]
		\alpha^{n}_{ \hat{x} , \hat{y} }
		=&
		-c_\gamma\sum_{\hat{x},\hat{y}}
		\frac{
			G \big(\tfrac{\hat{y}}{n}\big)-G \big(\tfrac{\hat{x}}{n}\big)
		}{
			\abs{\hat{y}-\hat{x}}^{d+\gamma}
		}
		\alpha^{n}_{ \hat{x} , \hat{y} }
		F^{j}(\tau^{\hat{y}}\eta)
		\\
		&-
		\frac{c_\gamma}{n}\sum_{\hat{x},\hat{y}}
		\frac{
			n\nabla_j G \big(\tfrac{\hat{y}}{n}\big)
			-n\nabla_j G\big(\tfrac{\hat{x}}{n}\big)
		}{
			\abs{\hat{y}-\hat{x}}^{d+\gamma}
		}
		M^{j}_{\hat{x},\hat{y}}(\tau_j\eta)
		\alpha^{n}_{ \hat{x} , \hat{y} }
		,
	\end{align}
with the abuse of notation $\nabla_jG(\tfrac{\hat{x}}{n}):=G(\tfrac{\hat{x}+\hat{e}_j}{n})-G(\tfrac{\hat{x}}{n})$ for functions $G:\mathbb{R}^d\to\mathbb{R}$, in analogy with $\nabla_j$ as in the fourth bulled point in Definition \ref{def:micro-op}. At this point it is evident the identification with a discretization of the regional fractional Laplacian as in Definition \ref{def:regi_lap}. In order to simplify the presentation, let us introduce this precisely.
	\begin{definition}\label{def:discrete-flap}
	Fix $\gamma\in(0,2)$ and $n\in\mathbb{N}_{+}$. For any subset of bonds $\mcb C\subset \mcb B$ and $H:\mathbb{R}^d\to\mathbb{R}$ define the linear operator $\Delta^{\gamma/2}_{n,\mcb C}$ at $H$ by 
			\begin{align}\label{defKn}
				\Delta^{\gamma/2}_{n,\mcb C} H \big(\tfrac{\hat{x}}{n}\big)
				:=
				\frac{1}{n^d}
				\sum_{\hat{y}: \{\hat{x}, \hat{y} \} \in \mcb C }
				K_H^\gamma\big(\tfrac{\hat{x}}{n},\tfrac{\hat{y}}{n}\big)
				\quad\text{where}\quad
				K_H^\gamma(\hat{u},\hat{v})
				:=c_\gamma\frac{H(\hat{u})-H(\hat{v})}{\mid\hat{u}-\hat{v}\mid^{d+\gamma}},
			\end{align}						
			whenever finite. Moreover, fixed $\kappa\geq0$, we define the operator $\mathbb{L}_{n,\kappa}^\gamma$ as
			\begin{align*}
				\mathbb{L}_{n,\kappa}^\gamma:=
				\Delta^{\gamma/2}_{n,\mcb F}
				+\kappa\Delta^{\gamma/2}_{n,\mcb S}
			\end{align*}			
			with $\mcb F$ and $\mcb{S}$ as in Definition \ref{def:slow-bonds}. In particular, we identify $\mathbb{L}_{n,0}^\gamma$ and $\mathbb{L}_{n,1}^\gamma$ with $\Delta^{\gamma/2}_{n,\mcb F}$ and $\Delta^{\gamma/2}_{n,\mcb B}$, respectively.
	\end{definition}
Recall the spaces $S_{\gamma,\alpha}^{d}$ and $S_{\gamma,\star}^{d}$, which are given in Definition \ref{defschw2}. According to Theorem \ref{hydlim}, one of the possible settings is
\begin{equation} \label{condalppos}
\alpha >0, \quad \gamma \in (0,2), \quad d \geq 1,  \quad G \in S_{\gamma,\alpha}^{d}.
\end{equation}
In this case, we get the integral equation given by condition $(3)$ of \eqref{eqhydfracbetazero} resp. condition $(3)$ of \eqref{ehsd}, when $\alpha \in \mathbb{R}_+\backslash\{1\}$ resp. when $\alpha=1$.

On the other hand, if $\alpha=0$, exactly one of two possibilities must happen. The first one is 
\begin{equation} \label{condneu}
 \alpha = 0, \quad  \gamma \in (0, \; 2), \quad  d \geq 1, \quad G \in S_{\gamma,\alpha}^{d}, \quad \lim_{n \rightarrow \infty}  \alpha_n r_n^{\gamma} =0,
\end{equation} 
where $r_n^{\gamma}$ is given in \eqref{rngamma}. According to Theorem \ref{hydlim}, the second possibility is
\begin{equation} \label{condlip}
  \alpha  = 0, \quad \gamma \in [1, \; 2), \quad  d \geq 1, \quad G \in S_{\gamma,\star}^{d}, \quad \lim_{n \rightarrow \infty} \alpha_n r_n^{\gamma} >0.
\end{equation} 
If \eqref{condneu} holds, we get the integral equation given by the condition $(3)$ of \eqref{eqhydfracdifneu}. On the other hand, \eqref{condlip} leads to the integral equation described in condition $(3)$ of \eqref{ehlips}. 	
	
	It is important to mention that in Appendix \ref{secdiscconv} we will show the following. 
\begin{prop}\label{prop:conv_frac-lap}
	 For all $\kappa \geq 0$ , $\gamma \in (0,2)$ , $d \geq 1$ and any $G \in  S_{\gamma,\kappa}^{d}$ it holds that
	\begin{align} \label{convKnalpha}
		\quad \lim_{n \rightarrow \infty}		
			\sup_{s \in [0,T]}
			\frac{1}{n^d} \sum_{\hat{x}}    		
		\big| 
		\mathbb{L}_{n,\kappa}^\gamma G_s \big(\tfrac{\hat{x}}{n}\big)   - \mathbb{L}_{\kappa}^{\gamma} G_s \big(\tfrac{\hat{x}}{n}\big) 
		\big| =0.
	\end{align}
\end{prop}
	In this way, focusing on the integrand function in the rightmost term of \eqref{defMnt}, we get from \eqref{intterm0} and the equality $\mathbb{L}_{n,\alpha_n}^\gamma=\mathbb{L}_{n,\alpha}^\gamma+(\alpha_n-\alpha)\Delta^{\gamma/2}_{n,\mcb S}$ that for any $s\in [0,T]$ 
      \begin{align}\label{intterm}
    	n^\gamma\mcb 
    	L_{n,\alpha_n}^\gamma 
    	\langle \pi_s^n, G_s \rangle 
    	=\varepsilon_{n, \alpha_n}^{\gamma}(G_s,\eta_s^n) 
    	+\mcb{A}_{n, \alpha_n}^{\gamma}(G_s,\eta_s^n)
    	+ \mcb{F}_{n, \alpha_n}^{\gamma}(G_s,\eta_s^n) + \mcb{F}_{n}^{\gamma}(G_s,\eta_s^n) ,
    \end{align}
    where  $(\alpha_n-\alpha)
	\frac{1}{n^d}  \sum_{\hat{x}} 
	\Delta^{\gamma/2}_{n,\mcb S} G_s \big(\tfrac{\hat{x}}{n}\big) 
	F^{n}(\eta_s^n)$
\begin{align}
	& \varepsilon_{n, \alpha_n}^{\gamma}(G_s,\eta_s^n) : = \frac{ {n^{\gamma} } }{n^d}  \sum_{ \hat{x}}    \Bigg[ \sum_{ \hat{y }} \big[ G_s \big(\tfrac{\hat{y}}{n}\big) - G_s \big(\tfrac{\hat{x}}{n}\big)\big] p_{\gamma}(\hat{y}-\hat{x}) \mathbbm{1}_{ \{ \{ \hat{x}, \hat{y}   \} \in \mcb F \} } \mathbbm{1}_{\{  |\hat{y}-\hat{x}|=1 \} }   \Bigg]  \eta_s^n( \hat{x} ), \label{intnlin1} \\
&	\mcb{A}_{n, \alpha_n}^{\gamma}(G_s,\eta_s^n)		:=		\frac{1}{d}		\sum_{j=1}^d		\frac{1}{n^d}	\sum_{\hat{x},\hat{y}}	\frac{1}{n} K_{n\nabla_j G_s}^\gamma \big(\tfrac{\hat{x}}{n},\tfrac{  \hat{y}}{n} \big)  M^{n,j}_{\hat{x},\hat{y}}(\tau_j\eta_s^n) \alpha^{n}_{ \hat{x} , \hat{y} },\label{intnlinmix}\\	
	&		\mcb{F}_{n, \alpha_n}^{\gamma}(G_s,\eta_s^n) :=  (\alpha_n-\alpha)
	\frac{1}{n^d}  \sum_{\hat{x}} 
	\Delta^{\gamma/2}_{n,\mcb S} G_s \big(\tfrac{\hat{x}}{n}\big) 
	F^{n}(\eta_s^n),   \label{intnlin5b} \\
&\mcb{F}_{n}^{\gamma}(G_s,\eta_s^n) := \frac{1}{n^d}  \sum_{\hat{x}} 
	\mathbb{L}_{n,\alpha}^\gamma G_s \big(\tfrac{\hat{x}}{n}\big) 
    F^{n}(\eta_s^n)
    .  \label{intnlin5}	
\end{align}	
Let us focus first on the perturbation term $\varepsilon_{n, \alpha_n}^{\gamma}$; it corresponds to the \textit{ SEP perturbations}  added to  $c^{n}_{\hat{x},\hat{y}}$ in \eqref{cons-series}. From \eqref{bndeta} and \eqref{transition prob} we have that $\sup_{s \in [0,T] }\big|
   \varepsilon_{n, \alpha_n}^{\gamma}(G_s,\eta_s^n)
   \big|$ is bounded from above by $N_{ \text{e} } (c_{\gamma})^{-1} Y^{n}_{\text{SEP}}(G)$, where for any $\gamma \in (0,2)$, $d \geq 1$ and $G \in S_{\gamma,0}^{d}$, $Y^{n}_{\text{SEP}}(G)$ is given by
\begin{equation} \label{defYnSEP}
Y^{n}_{\text{SEP}}(G):=\sup_{s \in [0,T] } \frac{ {n^{\gamma} } }{n^d}\sum_{\hat{x}} \Bigg| \sum_{ \hat{y }} \big[ G_s \big(\tfrac{\hat{y}}{n}\big) - G_s \big(\tfrac{\hat{x}}{n}\big)\big]  \mathbbm{1}_{ \{ \{ \hat{x}, \hat{y}   \} \in \mcb F \} } \mathbbm{1}_{\{  |\hat{y}-\hat{x}|=1 \} }   \Bigg|.
\end{equation}
In Appendix \ref{appertSEP}, we prove the following result.
\begin{prop} \label{boundYnsep}
Let $\gamma \in (0,2)$, $d \geq 1$ and $G \in S_{\gamma,0}^{d}$. Then $\lim_{n \rightarrow \infty} Y^{n}_{\text{SEP}}(G) =0$.
\end{prop}
In particular, $\int_0^t \varepsilon_{n, \alpha_n}^{\gamma}(G_s,\eta_s^n)\;\rmd s$ converges, as $n \rightarrow \infty$, to zero in $L^1(\mathbb{P}_{\mu_n})$.

We note that the term \eqref{intnlinmix} is trivially equal to zero when $F$ is linear; it corresponds to the \textit{mixed} products in \eqref{grad:d-dim}, and is analogous to the one described in \cite[equation (3.3)]{renato}. In the next subsection we obtain some upper bounds for both \eqref{intnlinmix} and \eqref{intnlin5b}. The goal is to obtain conditions which assure that $\int_0^t \mcb{A}_{n, \alpha_n}^{\gamma}(G_s,\eta_s^n)\;\rmd s$ and $\int_0^t \mcb{F}_{n, \alpha_n}^{\gamma}(G_s,\eta_s^n)\;\rmd s$ converge, as $n \rightarrow \infty$, to zero in $L^1(\mathbb{P}_{\mu_n})$. In this way, the integral equations will be completely determined by $\int_0^t \mcb{F}_{n}^{\gamma}(G_s,\eta_s^n)\;\rmd s$, which should converge, as $n \rightarrow \infty$ and in $L^1(\mathbb{P}_{\mu_n} )$, to
$\int_0^t \langle\mathbb{L}_{\alpha}^{\gamma} G_s,F(\rho_s) \rangle \; \rmd s$.

\subsubsection{Analysis of the error terms  \eqref{intnlinmix} and \eqref{intnlin5b}.} 

We now analyse the term in \eqref{intnlinmix} when $F$ is nonlinear. Due to the action of the barrier $\mathbb{B}$ given in \eqref{SROB0}, we split it into three parts: $\mcb{A}_{n, \alpha_n}^{\gamma}=\mcb{A}_{n, \alpha_n}^{\gamma,1}+\mcb{A}_{n, \alpha_n}^{\gamma,2}+\mcb{A}_{n, \alpha_n}^{\gamma,3}$, where $\mcb{A}_{n, \alpha_n}^{\gamma,1}$ corresponds to jumps (which do not cross the barrier) along the direction determined by $\hat{e}_j$, for $j \in \{1, \ldots, d-1\}$; $\mcb{A}_{n, \alpha_n}^{\gamma,2}$ corresponds to jumps along along the direction determined by $\hat{e}_d$ through fast bonds; and $\mcb{A}_{n, \alpha_n}^{\gamma,3}$ corresponds to jumps through slow bonds. More exactly,
    \begin{align}
        \mcb{A}_{n, \alpha_n}^{\gamma,1}(G_s,\eta^n_s)
        &=
        \frac{1}{d}\sum_{j=1}^{d-1}\frac{1}{n^d}
	\sum_{\hat{x},\hat{y}}
	\frac{1}{n}
        K_{n\nabla_j G_s}^\gamma\big(\tfrac{\hat{x}}{n},\tfrac{\hat{y}}{n}\big)
        M^{n,j}_{\hat{x},\hat{y}}(\tau_j\eta_s^n)
	\alpha^{n}_{ \hat{x} , \hat{y} }
\label{intnlin2}
,\\
        \mcb{A}_{n, \alpha_n}^{\gamma,2}(G_s,\eta^n_s)
        &=
        \frac{1}{d}
        \frac{1}{n^d}
	\sum_{\{\hat{x},\hat{y}\}\in\mcb F}
	\frac{1}{n}
        K_{n\nabla_j G_s}^\gamma\big(\tfrac{\hat{x}}{n},\tfrac{\hat{y}}{n}\big)
        M^{n,d}_{\hat{x},\hat{y}}(\tau_d\eta_s^n)
\label{intnlin3}
,\\
        \mcb{A}_{n, \alpha_n}^{\gamma,3}(G_s,\eta^n_s)
        &=
        \frac{1}{d}
        \frac{1}{n^d}
	\sum_{\{\hat{x},\hat{y}\}\in\mcb S}
	\frac{1}{n}
        K_{n\nabla_j G_s}^\gamma\big(\tfrac{\hat{x}}{n},\tfrac{\hat{y}}{n}\big)
        M^{n,d}_{\hat{x},\hat{y}}(\tau_d\eta_s^n)
	\alpha_n.
\label{intnlin6}
    \end{align}
Hereinafter, given $(f_n)_{n \geq 1}, (g_n)_{n \geq 1} \subset [0, \infty)$, the notation $f_n \lesssim g_n$ means that there exist a positive constant $C$ (independent of $n$) such that $0 \leq f_n \leq C g_n$, for any $n \in \mathbb{N}_+$.

To treat \eqref{intnlin2}, we apply Proposition \ref{boundnl} below, which is a generalization of Proposition 3.1 in \cite{renato}. We postpone its proof to Appendix \ref{useest2}. Recall the definition of $r_n^\gamma$ in \eqref{rngamma}.
\begin{prop} \label{boundnl}
Let $\gamma \in (0,2)$ and $d \geq 1$. Assume that either $j \in \{1, \ldots, d-1\}$ and $G \in S_{\gamma,0}^{d}$; or $j \in \{1, \ldots, d\}$ and $G \in S_{\gamma}^{d}$. Then
\begin{align*}
        \sup_{s \in [0,T] }
        \frac{1}{n^d}
	\sum_{\hat{x},\hat{y}}
	\frac{1}{n}
        \big|
        	K_{n\nabla_j G_s}^\gamma\big(\tfrac{\hat{x}}{n},\tfrac{\hat{y}}{n}\big)\big|
        \lesssim \frac{r_n^\gamma}{n}. 
\end{align*}
\end{prop}
In the last result, we do not have $j=d $ for $G \in S_{\gamma, 0}^{d} {\setminus} S_{\gamma}^{d}$ since $G_s$ might be discontinuous between $\hat{z}/n$ and $(\hat{z}-\hat{e}_d)/n$ for some $s \in [0,T]$, whenever $z_d=0$. Thus, recalling the expression for $M^{n,j}_{\hat{x},\hat{y}}$ in Definition \ref{def:FA}, we get from Proposition \ref{boundnl} that
\begin{equation} \label{boundnlin2} 
 \forall G  \in S_{\gamma,0}^{d}, \quad \sup_{s \in [0,T] } \big|
  \mcb{A}_{n, \alpha_n}^{\gamma,1}(G_s,\eta^n_s)
  \big| 
  \lesssim 
  \frac{r_n^\gamma}{n}
  \sum_{k=1}^{\ell_n} k (|b_k^{+}| + |b_k^{-}| ) N_{\text{e}}^k (1 + \alpha_n) \lesssim   f_n'  \frac{r_n^\gamma}{n} ,
\end{equation}
with $f_n'$ as in \eqref{f-f'}. In order to study $\mcb{A}_{n, \alpha_n}^{\gamma,2}$ and $\mcb{A}_{n, \alpha_n}^{\gamma,3}$, by performing algebraic manipulations we get that $\sup_{s \in [0, \; T]} \big| \mcb{A}_{n, \alpha_n}^{\gamma,2}(G_s,\eta^n_s)\big|$ resp. $\sup_{s \in [0, \; T]} \big| \mcb{A}_{n, \alpha_n}^{\gamma,3}(G_s,\eta^n_s)\big|$, is bounded from above by  $ Y^{n,\gamma}_{\mcb F}(G)$ resp. $Y^{n,\gamma}_{\mcb S}(G)$, where
\begin{align} \label{defYn1G}
 \begin{split}
 	Y^{n,\gamma}_{\mcb F}(G):=
 	\frac{f_n'}{d n^d} \Bigg\{& \sum_{x_d,y_d\geq 1}
 	E_{x_d, y_d}^{n,\gamma}(\nabla_d G) 
 	+\sum_{x_d,y_d\leq -1}
 	E_{x_d, y_d}^{n,\gamma}(\nabla_d G)  
 	+\sum_{x_d\geq 0}E_{x_d, 0}^{n,\gamma}(G)
 	+\sum_{x_d\leq-1} E_{x_d, -1}^{n,\gamma}(G) \Bigg\},
 \end{split}
\end{align}
\begin{align}   \label{defYn2G}
 	\begin{split}
 		Y^{n,\gamma}_{\mcb S}(G)
 		:=   
 		\frac{\alpha_n f_n'}{ dn^d} \Bigg\{&
 		\sum_{x_d\geq 1}\sum_{y_d\leq -1} 
 		E_{x_d, y_d}^{n,\gamma}(\nabla_d G)
 		+
 		\sum_{x_d\leq -1}\sum_{y_d\geq 1}  
 		E_{x_d, y_d}^{n,\gamma}(\nabla_d G) 
 		+
 		\sum_{x_d \leq-1}  
 		E_{x_d, 0}^{n,\gamma}(G)
 		+
 		\sum_{x_d\geq 0}  
 		E_{x_d, -1}^{n,\gamma}(G)
 		\Bigg\}.
 	\end{split}
\end{align}
Above, for every $n \geq 1$ and every $z, w \in \mathbb{Z}$, we defined
\begin{equation} \label{defExyGn}
E_{w, z}^{n,\gamma}(G) 
		: =  
		\sup_{s \in [0,T]} \frac{1}{n^d} \sum_{\hat{x}_{*}, \hat{y}_{*} \in \mathbb{Z}^{d-1} } 
	\big|	K_G^\gamma\big(\tfrac{\hat{x}}{n},\tfrac{\hat{y}}{n}\big) \big|
		\mathbbm{1}_{\{
			\hat{x}=(\hat{x}_{*},w),\;\hat{y}=(\hat{y}_{*},z)
			\}}
			.
\end{equation}
In Appendix \ref{useest3}, the following result is proved. 
\begin{prop} \label{boundYn}
Let $\gamma \in (0,2)$, $d \geq 1$ and $G \in S_{\gamma,0}^{d}$. Then $Y^{n, G}_{\mcb F} \lesssim f_n' r_n^\gamma/n$.   

Moreover, $Y^{n, G}_{\mcb S}  \lesssim \alpha_n f_n' n^{\gamma/2}/n$, if $G \in S_{\gamma,\star}^{d}$; and $Y^{n, G}_{\mcb S}  \lesssim \alpha_n f_n' r_n^\gamma$, if $G \in S_{\gamma,0}^{d} \setminus S_{\gamma,\star}^{d}$.
\end{prop}
In order to treat the term in \eqref{intnlin5b}, we state another result.
\begin{prop} \label{L1alphagen}
Let $\gamma \in (0,2)$, $d \geq 1$ and $G \in S_{\gamma,0}^{d}$. Then
\begin{align*}
\sup_{s \in [0,T]} | \mcb{F}_{n, \alpha_n}^{\gamma}(G_s,\eta_s^n) | \lesssim |\alpha_n - \alpha|
\sup_{s \in [0,T]}\frac{1}{n^d}
\sum_{\hat{x}}    
\big|  
	\Delta^{\gamma/2}_{n,\mcb S} G_s \big(\tfrac{\hat{x}}{n} \big) 
\big| \lesssim
\begin{cases}
|\alpha_n - \alpha|, \quad &  G \in S_{\gamma,\star}^{d}; \\
|\alpha_n - \alpha| r_n^\gamma, \quad & G \in S_{\gamma,0}^{d} \setminus S_{\gamma,\star}^{d}.
\end{cases}
\end{align*}
\end{prop}
\begin{proof}
Recalling the expression for $F^n$ in Definition \ref{def:FA} and \eqref{convabsF}, the first upper bound in the last display comes from the fact that $\sup_{s \in [0,T]} | F^{n}(\eta_s^n) |  < \infty$, due to \eqref{intnlin5b}. The second one is a direct consequence of \eqref{contrlip2} and \eqref{contrslow2}.
\end{proof}
Combining \eqref{boundnlin2} with Propositions \ref{boundYn} and \ref{L1alphagen}, we can state the following hypothesis, under which the terms in \eqref{intnlinmix} and \eqref{intnlin5b} do not contribute to the integral equations. 
\begin{rem} \label{remhiperr}
Assume Hypothesis \ref{hipoerror}. Moreover, assume that one of the conditions \eqref{condalppos}, \eqref{condneu}, \eqref{condlip} hold. Then by combining \eqref{boundnlin2} with Propositions \ref{boundYn} and \ref{L1alphagen}, we get 
 \begin{align*}
    \int_0^T \big[ \big| \mcb{A}_{n, \alpha_n}^{\gamma}(G_s,\eta^n_s) \big|+ \big| \mcb{F}_{n, \alpha_n}^{\gamma}(G_s,\eta_s^n) \big| \big] \;\rmd s\xrightarrow{n \rightarrow \infty}0 \quad\text{in }\;L^1(\mathbb{P}_{\mu_n}).
\end{align*}
\end{rem}

\subsubsection{Analysis of the term \eqref{intnlin5}.} \label{analysprincterm}
Now we treat the term in \eqref{intnlin5}. Recalling the expression for $F^n$ in Definition \ref{def:FA}, for every $k \in \{1, \ldots, \ell_n\}$, we can use Proposition \ref{prop:conv_frac-lap} to replace, in $\mcb{F}_{n, \alpha_n}^{\gamma}$ in the previous display, $\mathbb{L}_{n,\alpha}^\gamma$ by $\mathbb{L}_{\alpha}^\gamma$.
	Note that for the \textit{linear} terms, that is, for $k=1$, applying Proposition \ref{Qabscont}, the term in \eqref{intnlin5} converges, in $L^1(\mathbb{P}_{\mu_n})$ and as $n \rightarrow \infty$, to 	\begin{align*}
		&  (b_1^{+} + b_1^{-})  \int_0^t \int_{\mathbb{R}^d} \mathbb{L}_{\alpha}^{\gamma} G_s(\hat{u})  \rho_s(\hat{u})  \;  \rmd \hat{u} \; \rmd s. 
	\end{align*}
If $F$ is nonlinear, we observe that it is not possible to replace \eqref{intnlin5} directly by the empirical measure applied to $\mathbb{L}_{\alpha}^{\gamma} G_s$, as it is the case when $F$ is linear. Thus, we will require a technical result which is known in the literature as \textit{Replacement Lemma}, see Lemma 6.1 in \cite{renato} or Lemma 3.3 in \cite{CG} for instance.  Before we state it, we present a generalization of \eqref{2entbound} for product measures with \textit{non-constant} measures. More exactly, it is given by
	\begin{equation}  \label{1entbound}
		\exists h \in Ref, \; C_{h} >0: \forall n \geq 1, \quad H ( \mu_n | \nu_h^n ) \leq C_{h} n^d, 
	\end{equation} 
	with $\nu_h^n$ given in Definition \ref{defprodmeas} and the set $Ref$ as in Definition \ref{Ref} below, taken from \cite[Definition 2.3]{ddimhydlim}.
\begin{definition} \label{Ref}
		Given $h: \mathbb{R}^d \mapsto (0,1)$, we say that $h \in Ref$ if
		\begin{enumerate}
			\item
			there exist $0 < a_h < b_h < 1$ such that $a_h \leq h(\hat{u}) \leq b_h$, for any $\hat{u} \in \mathbb{R}^d$ and;
			\item
			there exists $L_h >0$ such that $|h(\hat{u}) - h(\hat{v})| \leq L_h |\hat{u} - \hat{v}|$, for any $\hat{u}, \hat{v} \in \mathbb{R}^d$ and;
			\item
			there exist $R_h>0$ and $A_h \in [a_h, \; b_h]$ such that $h(\hat{u})=A_h$ whenever $|\hat{u}| \geq R_h$.
		\end{enumerate}		
	\end{definition}
Now we can state the conditions under which our Replacement Lemma holds.
\begin{hipo} \label{hiporepl}
At least one of the following condition holds:
\begin{itemize}
\item [i)]
\eqref{2entbound} and Hypothesis \ref{hyp:coeff} \textbf{(ii)};
\item [ii)]
\eqref{1entbound} and Hypothesis \ref{hyp:coeff}.
\end{itemize}
\end{hipo} 
\begin{rem} \label{remhiprep}
Hypothesis \ref{hiporepl} is actually weaker than the assumptions in Theorem \ref{hydlim}. 
\end{rem}
In order to deal with the nonlinear terms of $F^{n}(\eta_s^n)$ in Definition \ref{def:FA}, corresponding to $k \geq 2$, we apply some approximations of the identity, which is a standard procedure in the literature. 
\begin{definition} \label{defmedemp}
	For $\varepsilon>0$ and $\hat{u} \in \mathbb{R}^d$ fixed, define the approximation of the identity $\overrightarrow{\iota_{\varepsilon}^{ \hat{u}}}$ \textit{from the right} by 
	\begin{align} \label{aproxidright}
		\overrightarrow{\iota_{\varepsilon}^{ \hat{u} }}(\hat{v}) := \frac{1}{\varepsilon^d} \mathbbm{1}_{ \prod_{j=1}^d (u_j, u_j+\varepsilon] }(\hat{v}), \quad  \hat{v} \in \mathbb{R}^d.
	\end{align}
	Moreover, given $\eta \in \Omega$, $\ell\in \mathbb{N}^+$ and $\hat{z} \in \mathbb{Z}^d$, let $\overrightarrow{\eta}^{\ell} (\hat{z})$ be the empirical average in a $d-$dimensional box of side length $\ell$ at the right of $\hat{z}$:
	\begin{align} \label{medempright}
		\overrightarrow{\xi}^{\ell} \big(\hat{z}):= 
		\frac{1}{\ell^d} 
		{ \sum_{\hat{\omega} \in \llbracket1, \; \ell \rrbracket^d}
		\xi ( \hat{z} + \hat{1}\odot \hat{\omega} ) },
	\end{align}
where $\llbracket1, \; L \rrbracket^d:=\{1,\ldots,L\}^d$ for any $L \in \mathbb{N}_+$, $\hat{1}=(1,\dots,1)$ and $\odot$ is the Hadamard product, $\hat{1}\odot \hat{\omega}=\sum_{i=1}^d\hat{e}_i\omega_i$.
\end{definition}
The last display is motivated by equation (5.3) in \cite{tertumariana}.  Here and in what follows, when $u \in \mathbb{R}^{*}_{+}$, we identify $\overrightarrow{\xi}^{u} (\hat{z})\equiv\overrightarrow{\xi}^{\lfloor u \rfloor} (\hat{z})$, where $\lfloor u \rfloor$ is the integer part of $u$. In particular and importantly, for all $\hat{z} \in \mathbb{Z}^d$, $j \in \{1, \ldots, d\}$, $n \geq 1$ and any $\varepsilon >0$, it holds that
\begin{align*}
\prod_{i=0}^{k-1} \overrightarrow{\eta}_t^{\varepsilon n}(\hat{x} + i \varepsilon n \hat{e}_j )  =\prod_{i=0}^{k-1} \Big\langle  \pi_s^n, \overrightarrow{\iota_{\varepsilon}^{ \hat{x}/n + i \varepsilon \hat{e}_j } } \Big\rangle.
\end{align*}
Recall that $\widetilde{\eta}_t^n:=N_{\text{e} } - \eta_t^n$, for any $n \geq 1$ and any $t \in [0,T]$.
Now we are ready to state the following technical lemma, which will be proved in Section \ref{replem}.
\begin{lem} \textbf{(Replacement Lemma)} \label{globrep}
Assume Hypothesis \ref{hiporepl}. Moreover, let $\widetilde{G}: [0,T] \times \mathbb{R}^d  \mapsto \mathbb{R}$ be such that there exists $H \in L^1(\mathbb{R}^d)$ satisfying
\begin{equation}  \label{boundrep}
	 \begin{cases}
		\forall \hat{u} \in \mathbb{R}^d, \quad \sup_{s \in [0,T]} \big|\widetilde{G}_s(\hat{u})\big| \leq H(\hat{u}); \\
		\exists \delta  \in (0, \gamma), \; \; K>0: \; \; \forall \hat{u}=(\hat{u}_\star,u_d) \in \mathbb{R}^{d-1} \times \mathbb{R}^{\star}, \quad H(\hat{u}) \leq K(1 + |u_d|^{ - \delta } ).
	\end{cases}
\end{equation}
Let $\xi=\eta$ or $\xi=\tilde{\eta}$. For any $k \geq 2$, $j \in \{1, \ldots, d\}$ and $t \in [0,T]$, it holds that
\begin{align}
	& \varlimsup_{\varepsilon \rightarrow 0^{+}}\varlimsup_{n \rightarrow \infty} \mathbb{E}_{\mu_n} \Bigg[  \Bigg| \int_{0}^{t}\frac{ 1  }{n^d}\sum_{\hat{x}}  \widetilde{G}_s \big( \tfrac{\hat{x}}{n} \big) \Bigg\{ \prod_{i=0}^{k-1}  \xi_s^n(\hat{x}+ i \hat{e}_j) - \prod_{i=0}^{k-1} \overrightarrow{\xi}_{s}^{\varepsilon n} \big(\hat{x}+  i \varepsilon n \hat{e}_j) \Bigg \} \;\rmd s \, \Bigg| \Bigg] =0.
	\label{globrlpos}
\end{align}
\end{lem}
In Appendix \ref{secfracoper} we will show the following result.
	\begin{prop} \label{propL1alpha}
		Let $\kappa \geq 0$, $\gamma \in (0,2)$, $d \geq 1$ and $G \in S_{\gamma,\kappa}^{d}$. Then there exists $H^G \in L^1(\mathbb{R}^d)$ such that $\sup_{s \in [0,T]} | \mathbb{L}_{\kappa}^{\gamma} G_s(\hat{u}) | 
			\leq H^G(\hat{u})$, for any $\hat{u} \in \mathbb{R}^d$. Furthermore, there exists a constant $C\equiv C(G)>0$ satisfying $H^G(\hat{u})
			\leq C\left(1 + |u_d|^{- \gamma/2}\right)$, whenever $u_d \neq 0$. 
\end{prop}
From Proposition \ref{propL1alpha}, we have that $\widetilde{G}=\mathbb{L}_{\alpha}^{\gamma}G$ and $\delta =  \gamma/2$ satisfy \eqref{boundrep}. For the sake of illustration, specializing to $b_k^\pm=0$, for any $k\neq 2$, and $b_2^-=0,b_2^+=1$, that is, $F(\rho)=\rho^{2}$ as in \cite{renato}. Now as $n \rightarrow \infty$ and then $\varepsilon \rightarrow 0^+$, from Lemma \ref{globrep} the term in \eqref{intnlin5} with $\mathbb{L}_{n,\alpha}^\gamma$ replaced by $\mathbb{L}_{\alpha}^\gamma$ approximates
\begin{align*} 
	\int_0^t   \frac1d\sum_{j=1}^d \frac{1}{n^d} \sum_{\hat{x}}
	\mathbb{L}_{\alpha}^{\gamma}  G_s \big(\tfrac{\hat{x}}{n}\big) \overrightarrow{\eta}_{s}^{\varepsilon n} \big(\hat{x})   \overrightarrow{\eta}_{s}^{\varepsilon n} \big(\hat{x}+   \varepsilon n \hat{e}_j)   \;\rmd s
	.
\end{align*}
From Proposition \ref{propaproxtest}, the fact, due to Proposition \ref{Qabscont}, that $\rho \in [0,N_{\text{e}}]$, and Lebesgue's Differentiation Theorem, the previous display  converges, as $n \rightarrow \infty$ and in $L^1(\mathbb{P}_{\mu_n})$, to 
\begin{align*} 
	\int_0^t   \frac{1}{d }\sum_{j=1}^d 
	\int_{\mathbb{R}^d}  \mathbb{L}_{\alpha}^{\gamma}  G_s ( \hat{u} ) [\rho_s(\hat{u})]^{2} \;\rmd \hat{u} \;\rmd s =  \int_0^t   \int_{\mathbb{R}^d}  \mathbb{L}_{\alpha}^{\gamma}  G_s ( \hat{u} ) [\rho_s(\hat{u})]^{2} \;\rmd \hat{u} \; \rmd s.
\end{align*}
This leads to the integral equation given by the third condition in the definitions of weak solutions given in Section \ref{sechydeqsdif}.

\section{Tightness} \label{sectight}

The main goal of this section is to prove the next result.

\begin{prop} \label{Qntight}
Assume Hypothesis \ref{hipotight}. Then the sequence of probability measures $(\mathbb{Q}_{n})_{ n \geq 1 }$ is tight with respect to the Skorohod topology of $\mcb {D}_{ \mcb{M}_{N_e}^{+} }([0,T] )$.
\end{prop}
We follow closely the arguments given in previous works of the authors, in particular \cite{ddimhydlim}. With this in mind, motivated by the time horizon $[0, T]$, the set of stopping times $\tau$ such that $\tau \in [0,T]$ is denoted by $\mathcal{T}_T$. Next, we state a result which is equivalent to Lemma 4.2 in \cite{ddimhydlim}.
\begin{lem} \label{lemtight}
	The sequence $(\mathbb{Q}_{n})_{ n \geq 1 }$ is tight if
	\begin{equation} \label{T1sdif}
		\forall \varepsilon >0, \; \forall G\in C_c^2(\mathbb{R}^d), \quad  \lim _{\lambda \rightarrow 0^+} \varlimsup_{n \rightarrow\infty} \sup_{\tau  \in \mathcal{T}_{T}, \; 0 \leq t \leq \lambda} {\mathbb{P}}_{\mu _{n}}\big( \eta_{\cdot}^{n}: \quad | \langle\pi^{n}_{ T \wedge ( \tau+ t)},G\rangle-\langle\pi^{n}_{\tau},G\rangle | > \varepsilon \big)  =0.
	\end{equation}	
\end{lem}
\begin{proof}
	We begin by observing that the space $\mcb {M}^{+}$ is a Polish space, due to Theorem 31.5 of \cite{bauer}. This means that we are done if we can prove that $(\mathbb{Q}_n)_{n \geq 1}$ satisfies both conditions of Theorem 1.3 in Chapter 4 of \cite{kipnis1998scaling}. From \eqref{defMcb}, $\mcb {M}^{+}_{N_{\text{e}}}$ is a compact subset of $\mcb {M}^{+}$ since it is a closed subspace of $\mcb {M}^{+}$ regarding the vague topology (see Definition 31.1 and Theorem 31.2 of \cite{bauer}). Thus, choosing $K(t,\varepsilon)=\mcb {M}^{+}_{N_{\text{e}}}$ in  Theorem 1.3 of \cite{kipnis1998scaling}, Chapter 4, we have 
	\begin{equation*} 
		\forall t \in [0,T], \; \forall \varepsilon >0, \quad \sup_{n \geq 1} \mathbb{Q}_n \big( \pi_t^n \notin  K(t,\varepsilon) \; \big)=  \mathbb{Q}_n \big( \pi_t^n \notin \mcb {M}^{+}_{N_{\text{e}}} \big)  = 0 < \varepsilon,
	\end{equation*}	
	which is equivalent to the first condition of Theorem 1.3 in Chapter 4 of \cite{kipnis1998scaling}. The second one can be obtained from \eqref{T1sdif} in exactly the same way as it was done in Appendix D of \cite{ddimhydlim}, therefore we omit the details. 
\end{proof}
The following result comes by combining \eqref{defMnt} with Markov's and Chebychev's inequalities and Lemma \ref{lemtight}.
\begin{prop} 
	Recall the definition of $\mcb M_{t}^{n}(G)$ in \eqref{defMnt}. If, for any $G \in C_{c}^2(\mathbb{R}^d)$, 
	\begin{align} 
		& \varlimsup_{\lambda \rightarrow 0^+} \varlimsup_{n \rightarrow \infty} \sup_{\tau \in \mathcal{T}_T, 0 \leq \; t \leq \lambda} \mathbb{E}_{\mu_n} \Bigg[ \Big| \int_{\tau}^{T \wedge (\tau+ t)} n^{\gamma}  \mcb L_{n,\alpha}^\gamma \langle \pi_{s}^{n},G\rangle \;\rmd s \Big| \Bigg] = 0, \label{condger1pr1} \\
		& \varlimsup_{\lambda \rightarrow 0^+} \varlimsup_{n \rightarrow \infty} \sup_{\tau \in \mathcal{T}_T, \; 0 \leq t \leq \lambda} \mathbb{E}_{\mu_n} \left[ \left( \mcb M_{\tau}^{n}(G) -  \mcb M_{T \wedge (\tau+ t)}^{n}(G) \right)^2 \right] = 0,  \label{condger1pr2}
	\end{align}
	then the sequence $(\mathbb{Q}_{n})_{ n \geq 1 }$ is tight.
\end{prop}
We observe that the test functions $G$ in the last proposition \textit{do not depend} on time. Indeed, following Remark 4.3 in \cite{ddimhydlim}, we have that $C_{c}^2(\mathbb{R}^d) \subsetneq S_{\gamma}^{d}$. This leads to
\begin{align} \label{reltest}
	\forall \gamma \in (0,2), \; \forall d \in \mathbb{N}_{+}, \quad C_{c}^2(\mathbb{R}^d) \subsetneq S_{\gamma}^{d} \subsetneq S_{\gamma, \star}^{d} \subsetneq S_{\gamma, 0}^{d}. 
\end{align}
Last display is a direct consequence of Definition \ref{defschw2}. Next, we claim that under Hypothesis \ref{hipotight}
\begin{align} \label{claim1tight}
	\forall G \in C_{c}^2(\mathbb{R}^d), \quad \sup_{s \in [0,T]} \big| n^{\gamma}  \mcb L_{n,\alpha}^\gamma \langle \pi_{s}^{n},G\rangle\big| \lesssim 1.
\end{align}
In particular, \eqref{condger1pr1}  follows trivially. In order to see that \eqref{claim1tight} holds, we look at \eqref{intterm} and analyse all the components of $\mcb L_{n,\alpha}^\gamma \langle \pi_{s}^{n},G\rangle$, by collecting some results from Section \ref{secheurlin}. More exactly, it is enough to combine Propositions \ref{boundYnsep}, \ref{boundYn} and \ref{L1alphagen}, with \eqref{boundnlin2}, \eqref{reltest}, Hypothesis \ref{hipotight} and the fact that
\begin{align*} 
\forall \kappa \geq 0, \; \forall \gamma \in (0,2), \; \forall d \in \mathbb{N}_+, \; \forall G \in C_{c}^2(\mathbb{R}^d), \quad  \frac{1}{n^d} \sum_{\hat{x}}   \big| \mathbb{L}_{n,\kappa}^\gamma G_s \big(\tfrac{\hat{x}}{n}\big) \big| \lesssim 1,
\end{align*}
to get \eqref{claim1tight}. The last display is a direct consequence of \eqref{reltest}, \eqref{convKnalpha} and Proposition \ref{propL1alpha}.

 In order to obtain \eqref{condger1pr2}, we observe that it is enough to take $G \in C_{c}^2(\mathbb{R}^d)$. Nevertheless, in Section \ref{secchar} we will deal with time-dependent functions that \textit{may} be discontinuous in space. This motivates us to state Proposition \ref{lemtight2} below, since \eqref{condger2pr3} is necessary in the proof of Proposition \ref{derintnonlin} below. In particular, we observe that \eqref{condger1pr2} is a direct consequence of \eqref{defNnt} and \eqref{condger2pr1}.
\begin{prop} \label{lemtight2}
	Let $\gamma \in (0,2)$, $d \in \mathbb{N}_+$ and $G \in  S_{\gamma,0}^{d}$. Then under Hypothesis \ref{hipotight} it holds
	\begin{equation} \label{condger2pr1}
		\lim_{n \rightarrow \infty}  \sup_{s \in [0,T]} \big| n^{\gamma} \big[ \mcb L_{n,\alpha}^\gamma (  \langle \pi_{s}^{n},G_s \rangle^2 ) - 2 \langle \pi_{s}^{n},G_s \rangle  \mcb L_{n,\alpha}^\gamma (  \langle \pi_{s}^{n},G_s \rangle )  \big] \big| =0.
	\end{equation}
In particular, \eqref{condger1pr2} is satisfied, and it holds
	\begin{align} 
		\forall \delta_1 > 0, \quad \lim_{n \rightarrow \infty} \mathbb{P}_{\mu_n} \Big( \sup_{s \in [0,T]} \big| \mcb M_{s}^{n}(G) \big| > \delta_1 \Big) =0. \label{condger2pr3}
	\end{align}
\end{prop}
Before showing Proposition \ref{lemtight2}, we state a technical result that will be proved in Appendix \ref{useest2}.
\begin{prop} \label{boundnltight}
Let $\gamma \in (0,2)$ and $d \in \mathbb{N}_+$. For any $H \in S_{\gamma}^d$, it holds
\begin{align*}
 \frac{n^{\gamma}}{n^{2d} }\sup_{s \in [0,T]} \sum _{\hat{x}, \hat{y} }   \big[ H_s\big( \tfrac{\hat{y}}{n}\big) - H_s\big( \tfrac{\hat{x}}{n}\big) \big]^2  p_{\gamma}(\hat{y}-\hat{x}) \lesssim \frac{n+n^{\gamma}}{n^{d+1}} \leq \frac{n+n^{\gamma}}{n^{2}}  \lesssim \frac{r_n^\gamma}{n},  
\end{align*}
where $r_n^\gamma$ is given in \eqref{rngamma}.
\end{prop}
\begin{proof}[Proof of Proposition \ref{lemtight2}]
From Proposition \ref{dynkform}, we conclude that \eqref{condger1pr2} is a direct consequence of \eqref{condger2pr1}. Moreover, since $\big(\mcb M_{t}^{n}(G) \big)_{t \in [0,T]}$ is a right-continuous martingale, an application of Doob's inequality leads to \eqref{condger2pr3}. Thus, the proof ends if we can obtain \eqref{condger2pr1}. Keeping this in mind, from Definition \ref{def:gen_series}, the symmetry of $p_\gamma(\cdot)$ given in \eqref{transition prob}  and $\alpha_{\cdot, \cdot}^n$ in Definition \ref{def:alpha}, and \eqref{intterm0},  it is not very difficult to prove that the display inside the absolute value in \eqref{condger2pr1} can be rewritten as
\begin{align} \label{dynktight}
\frac{1}{2 N_{\text{e}}} \frac{n^{\gamma}}{n^{2d} }  \sum _{\hat{x}, \hat{y} }   \big[ G_s\big( \tfrac{\hat{y}}{n}\big) - G_s\big( \tfrac{\hat{x}}{n}\big) \big]^2  p_{\gamma}(\hat{y}-\hat{x}) c_{\hat{x}, \hat{y}}^{n,\gamma}(\eta_s^n) \alpha_{\hat{x},\hat{y}}^n \big[ a_{\hat{x},\hat{y}}(\eta_s^n) + a_{\hat{y},\hat{x}}(\eta_s^n)\big].
\end{align}
Recalling \eqref{rngamma}, the only missing ingredient to finish the proof is the following estimate.
\begin{equation} \label{contrslowtight}
\forall \gamma \in (0, 2), \; \forall d \in \mathbb{N}_+, \; \forall G \in  S_{\gamma,0}^{d}, \quad  \frac{n^{\gamma}}{n^{2d} } \sup_{s \in [0,T]} \sum _{ \{ \hat{x}, \hat{y} \} \in \mcb S }   \big[ G_s\big( \tfrac{\hat{y}}{n}\big) - G_s\big( \tfrac{\hat{x}}{n}\big) \big]^2  p_{\gamma}(\hat{y}-\hat{x}) \lesssim \frac{r_n^\gamma}{n^d}.
\end{equation}
The last bound is proved in Appendix \ref{useest}. The proof ends by combining the previous estimate with Proposition \ref{boundnltight}, and recalling \eqref{bndeta} and Hypothesis \ref{hipotight}.
\end{proof}

\section{Energy estimates} \label{estenerg}

Here and in what follows, $\mathbb{Q}$ will always denote a limit point for the sequence $(\mathbb{Q}_n)_{n \geq 1}$, whose existence is ensured due to Proposition \ref{Qntight}. Keeping this in mind, we now state a result whose proof is postponed to Appendix \ref{secabscont}.
\begin{prop} \label{Qabscont}
It holds
	\begin{equation} \label{defdens}
		\mathbb{Q} \Big( \pi_{\cdot} \in \mcb D_{\mcb {M}^{+}_{N_{\text{e}}}}([0, T]):\; \forall s \in [0, \; T], \quad  \pi_s(\rmd \hat{u} )=\rho(s,\hat{u}) \; \rmd  \hat{u}, \; \; 0 \leq \rho \leq N_{\text{e}}   \Big)=1.
	\end{equation}
	\end{prop}
Recall the definition of $\mcb{H}_{ \alpha }^{ \gamma/ 2 } ( \mathbb{R}^d )$ in \eqref{defHkap}. The next step for concluding the proof of Theorem \ref{hydlim} is to show that
\begin{align} 
& \mathbb{Q} \big( \pi_{\cdot} :  \quad \rho- N_{ \text{e} } \theta \in L^2 \big(0, T ; \; L^2( \mathbb{R}^d ) \;   \big)=1, \label{cond1weak} \\
& \mathbb{Q} \big( \pi_{\cdot} :  \quad  F(\rho)- F( N_{ \text{e} } \theta)  \in L^2 \big(0, T ; \; \mcb{H}_{ \alpha }^{ \gamma/ 2 } ( \mathbb{R}^d )  \; \big)=1. \label{cond2weak}
\end{align}
where $\theta$ is given in \eqref{2entbound} and $\rho$ is the density as in Proposition \ref{Qabscont}. We obtain \eqref{cond1weak} and \eqref{cond2weak} in Sections \ref{seccond1weak} and \ref{seccond2weak}, respectively. Before we do so, we state the following preliminary lemma, which is applied in various differents contexts.
\begin{lem} \label{lemstatener}
Let $(Y, \langle \cdot, \cdot \rangle_{Y})$ be a Hilbert space with norm $\| \cdot \|_{Y}$ and $N \subset Y$ be a normed vector space such that $(N, \| \cdot \|_{Y})$ is dense in $(Y, \| \cdot \|_{Y})$. For every $\rho$, which is given by Proposition \ref{Qabscont}, let $\ell_{\rho}: N \mapsto \mathbb{R}$ be a (random) linear functional on $N$. Assume there exists finite constants $K_0, K_1>0$ such that
\begin{align} \label{funcbound}
\mathbb{E}_{\mathbb{Q}} \big[ \sup_{G \in N} \big\{ \ell_{\rho}(G) - K_0  \|G\|_{Y }^2 \big\} \big] \leq K_1.
\end{align}
Then
\begin{align}\label{funcbound2}
	\mathbb{Q} 
	\Big(
	\pi_{\cdot}:\;\exists g_{\rho} \in Y: \forall G \in Y,\;
	\langle G , g_{\rho} \rangle_{Y} = \ell_{\rho}(G) 
	\Big) = 1 \quad\text{and}\quad
	\mathbb{E}_{\mathbb{Q}} \big[ \| g_{\rho} \|_{Y }  \big] \leq K_1 + K_0
	.
\end{align}
\end{lem}
\begin{proof}
For every $G \in N$ such that $\|G\|_{Y } \leq 1$, we have
$
\ell_{\rho}(G) = \big( \ell_{\rho}(G) - K_0  \|G\|_{Y }^2 \big) + K_0.
$
Combining this with \eqref{funcbound}, we get $\mathbb{E}_{\mathbb{Q}} \big[ \sup_{G \in N, \| G \|_Y \leq 1} \{ \ell_{\rho}(G) \}  \big] \leq K_1 + K_0$. Since $(N, \| \cdot \|_{Y})$ is dense in $(Y, \| \cdot \|_{Y})$, the proof ends by applying Riesz's Representation Theorem.
\end{proof}
In this section, we will apply the following result, whose proof is postponed to Appendix \ref{sectopsko}.
\begin{prop} \label{weakconv1}
Fix $j \geq 1$, $G_1, \ldots, G_j \in C_c^{0,0}([0,T] \times \mathbb{R}^d )$ and $c_1, \ldots, c_j \in \mathbb{R}$. Then
\begin{align*}
\mathbb{E}_{\mathbb{Q}} \Bigg[ \max_{1 \leq i \leq j} \Bigg\{ c_i +   \int_0^{T}  \langle \rho_s, \;   G_i(s, \cdot) \rangle  \; \rmd s  \Bigg\}  \Bigg] 
\leq  
\varlimsup_{n \rightarrow \infty} \mathbb{E}_{\mu_n} \Bigg[ \max_{1 \leq i \leq j} \Bigg\{ c_i +   \int_0^{T} \langle \pi_s^n, \;   G_i(s, \cdot) \rangle \; \rmd s \Bigg\}  \Bigg].
\end{align*}
\end{prop}
 In the remainder of this work, we will make full use of the entropy bounds \eqref{2entbound} and \eqref{1entbound} (see also Remark \ref{rem:entropy}). 
Recall the definition of $C_{\theta}>0$ and $\theta \in (0,1)$ in \eqref{2entbound}, and of $C_{h}>0$ and $h \in Ref$ in \eqref{1entbound}. Next we adapt Lemma 6.5 of \cite{ddimhydlim} to our model, leading to a result which is applied repeatedly in the remainder of this article. 
\begin{lem} \label{entjen}
For every $n \geq 1$, let $(\widetilde{Z}_{n,j})_{j \geq 1}$ be a sequence of random applications $\widetilde{Z}_{n,j}:\mcb D_{ \Omega } ( [0, \; T]) \mapsto \mathbb{R}$. For each $j \geq 1$, define $Z_{n,j}: = \max_{1 \leq i \leq j} \{ \widetilde{Z}_{n,i} \}$. Then, for any $n,j \geq 1$ and any $ B >0$ it holds \begin{align*}
	\mathbb{E}_{\mu_n} [ Z_{n,j} ] 
	\leq \frac{ C_{\lambda}}{B} 
	+  \frac{j}{Bn^d} 
	+ \frac{1}{Bn^d} \log 
	\Bigg( 
	\max_{1 \leq i \leq j} \mathbb{E}_{\lambda_n} 
		\big[ 
			\exp  (Bn^d \widetilde{Z}_{n,i})  
		\big]  
	\Bigg).
\end{align*}
Above, $C_{\lambda}=C_{\theta}$ and $\lambda_n=\nu_{\theta}$ under \eqref{2entbound}; or $C_{\lambda}=C_{h}$ and $\lambda_n=\nu_{h}^n$ under \eqref{1entbound}.   
\end{lem}
Next we present the proof for \eqref{cond1weak}.

\subsection{Proof of \eqref{cond1weak}} \label{seccond1weak}

In this subsection, we will make use of the function $\phi \in C^{\infty}(\mathbb{R})$ defined by 
\begin{align} \label{defhlambda}
 \phi(\lambda): = \log \Bigg(  \int_{\Omega} e^{\lambda \eta(\hat{0})}   \rmd \nu_{\theta}  \Bigg) 
 = \log  
 \Bigg( 
 \sum_{r=0}^{N_{\text{e}}} e^{\lambda r }\binom{N_{\text{e}}}{r} \theta^r (1-\theta)^{N_{\text{e}}-r}   
 \Bigg), \quad \lambda \in \mathbb{R}.
\end{align}
In particular, from the definition of $\nu_{\theta}$ in Definition \ref{definvmeas},
\begin{equation} \label{hinv}
			\forall \hat{x} \in \mathbb{Z}^d, \; \forall t \in [0,T], \; \forall n \geq 1, \; \forall \lambda \in \mathbb{R}, \; \quad \exp \big( \phi(\lambda) \big)= E_{\nu_\theta} \big[ \exp \big ( \lambda \eta_t^n(\hat{x}) \big) \big].   
		\end{equation}
One can check that
\begin{equation} \label{esthstat}
\|\phi'\|_{\infty} \leq N_{\text{e}} \quad \text{and} \quad \exists K_{\phi} >0: \forall \lambda \in \mathbb{R}, \quad - K_{\phi} \lambda^2 \leq  N_{ \text{e} } \theta \lambda - \phi(\lambda).
\end{equation}
The goal of this subsection is to prove the following result. 
\begin{prop} \label{propestenergstat}
Assume \eqref{2entbound}. Then 
\begin{align}  \label{tesestatb}
\mathbb{E}_{\mathbb{Q}} \Bigg[ \int_0^T \int_{\mathbb{R}^d}  [\rho(s, \hat{u}) - N_{ \text{e} } \theta]^2    \rmd \hat{u} \; \rmd s  \Bigg] \leq T(C_{\theta} +1) + K_{\phi} < \infty . 
\end{align}
 In particular, \eqref{cond1weak} holds.
\end{prop}
\begin{proof}
	We will make use of Lemma \ref{lemstatener}, by choosing $N=C_c^{0,0} ( [0, T] \times \mathbb{R}^d)$ and $Y=L^2 ([0,T] \times \mathbb{R}^d)$. For each $G \in N$, let us define the quantity $J_G: =\int_0^T \int_{\mathbb{R}^d} \phi \big( G_t (\hat{u} ) \big) \; \rmd \hat{u} \; \rmd t $. From \eqref{esthstat}, for every $G\in N$ fixed, it holds
\begin{align}
		\label{cond1:ell-to-j}
		\mathbb{E}_{\mathbb{Q}}
			\Big[ 
		\sup_{G\in N}
		\Big\{
		\ell_{\rho} (G) - K_{\phi} \| G \|_{Y}^2
		\Big\}
			\Big] 
		\leq 
		\mathbb{E}_{\mathbb{Q}}        
			\Bigg[ 
		\sup_{G\in N}
		\Bigg\{
		- J_G 
		+
		\int_0^{T} \int_{\mathbb{R}^d} 
		G_s( \hat{u})  \rho_s(\hat{u}) \;
		 \rmd \hat{u} \; \rmd s 
		\Bigg\}
			\Bigg] 
		,
	\end{align}
where $K_{\phi}$ is given in \eqref{esthstat} and the linear functional $N \mapsto \mathbb{R}$ is defined, for any $G \in N$, as
	\begin{align}
		\ell_{ \rho}(G):= \int_0^T \int_{\mathbb{R}^d} [\rho_s(\hat{u}) - N_{ \text{e} } \theta]  G_s(\hat{u})  \; \rmd \hat{u} \; \rmd s 
		.
	\end{align}
Let us consider the space $X=L^1 ([0,T] \times \mathbb{R}^d)$. Since $N \subset X$ and $(X,\| \cdot \|_X)$ is a separable metric space, $(N,\| \cdot \|_X)$ is also a separable metric space, hence we can fix a dense sequence $(G_j)_{j \geq 1} \subset N$ in $(N,\| \cdot \|_X)$. Now let $\hat{G}$ be any element of $N$ and pick an $\varepsilon > 0$. Since $(G_j)_{j \geq 1}$ is dense in $(N, \| \cdot \|_X)$, there exists $j_0 \geq 1$ such that $\|G_{j_0} - \hat{G} \|_X < \varepsilon$. Then
	\begin{align}\label{cond1-eq0}
		- J_{\hat{G}} +   \int_0^{T} \int_{\mathbb{R}^d} \hat{G}(s, \hat{u})  \rho_s(\hat{u}) \;  \rmd \hat{u} \; \rmd s  
		\leq   
		\sup_{j \geq 1} \Bigg\{ - J_{G_j} +   \int_0^{T} \int_{\mathbb{R}^d} G_j(s, \hat{u})  \rho_s(\hat{u})  \;  \rmd \hat{u} \; \rmd s  \Bigg\} + 2 N_{\text{e}} \varepsilon.
	\end{align}
The last display is a consequence of Proposition \ref{Qabscont}, \eqref{esthstat} and the Mean Value Theorem. Since $\varepsilon >0$ is arbitrary, in order to get \eqref{tesestatb} from Lemma \ref{lemstatener}, it is enough to prove that
\begin{equation} \label{bndEQcountstat}
\mathbb{E}_{\mathbb{Q}} \Bigg[ \sup_{j \geq 1} \Bigg\{ - J_{G_j} +   \int_0^{T} \int_{\mathbb{R}^d} G_j(s, \hat{u})  \rho_s(\hat{u}) \; \rmd  \hat{u} \; \rmd s \Bigg\}  \Bigg] \leq T  (C_{\theta} + 1).
\end{equation}
In order to do so, for every $j,n \geq 1$ we define $A_{n,j}: \mcb D_{\Omega} ( [0,T] ) \mapsto \mathbb{R}$ as 
\begin{align} \label{defZnG}
A_{n,j}:= \max_{1 \leq i \leq j} \Bigg\{ \frac{1}{T} \int_0^T \frac{1}{n^d} \sum_{ \hat{x} }  \big[ G_i\big(s, \tfrac{\hat{x}}{n}  \big) \eta_{s}^n(\hat{x}) - \phi \big( G_i \big(s, \tfrac{\hat{x}}{n} \big) \big) \big] \; \rmd s \Bigg\}.
\end{align} 
Since $H(\mu_n | \nu_{\theta}) \leq C_{\theta} n^d$, applying Lemma \ref{entjen} for $B=1$, and following the arguments of the proof of Proposition 5.1 in \cite{casodif}, we get from \eqref{hinv} that
\begin{equation} \label{bndexpn2}
\forall j, n \geq 1, \quad \mathbb{E}_{\mu_n} [ A_{n,j} ] \leq C_{\theta} + \frac{j}{n^d}.
\end{equation}
Next, for every $j,n \geq 1$, we define $\widetilde{A}_{n,j}: \mcb D_{\Omega} ( [0,T] ) \mapsto \mathbb{R}$ as 
\begin{align*} 
\widetilde{A}_{n,j}:= \max_{1 \leq i \leq j} \Bigg\{ \frac{1}{T} \int_0^T \Bigg[ \frac{1}{n^d} \sum_{ \hat{x} }   G_i \big(s, \tfrac{\hat{x}}{n}  \big) \eta_{s}^n(\hat{x}) - \int_{\mathbb{R}^d} \phi \big( G_s ( \hat{u} ) \big) \; \rmd  \hat{u} \Bigg] \; \rmd s  \Bigg\}.
\end{align*} 
From \eqref{defZnG}, for every $j \geq 1$, there exists $n_0(j)$ depending only on $j$ such that
\begin{align} \label{exptilZn}
 \forall n \geq n_0(j), \quad \mathbb{E}_{\mu_n} \big[ \big| A_{n,j} - \widetilde{A}_{n,j} \big| \big] \leq 1.
\end{align}
Combining the display in the last line with an application of Proposition \ref{weakconv1} for the choice $c_i:=-J_{G_i}$ for every $i \geq 1$, we see that
\begin{align*}
\mathbb{E}_{\mathbb{Q}} \Bigg[ \max_{1 \leq i \leq j} \Bigg\{ -   J_{G_i}   +   \int_0^{T}    \langle \rho_s, \;   G_i(s, \cdot) \rangle  \; \rmd s \Bigg\}  \Bigg]  \leq \lim_{n \rightarrow \infty} \mathbb{E}_{\mu_n} \big[T \widetilde{A}_{n,j} \big] \leq T  (C_{\theta} + 1),
\end{align*}
where the last inequality comes from \eqref{bndexpn2} and \eqref{exptilZn}. Combining the last display with the Monotone Convergence Theorem, we obtain \eqref{bndEQcountstat} and the proof ends.
\end{proof}
The next result is a direct consequence of the previous proposition. 
\begin{cor} \label{corestenergstat}
Recall $f_\infty$ as in \eqref{convabsF} and assume \eqref{2entbound}. Then
\begin{align*} 
\mathbb{E}_{\mathbb{Q}} \Bigg[\int_0^T \int_{\mathbb{R}^d} \big[ F \big( \rho_s(\hat{u}) \big)- F(N_{\text{e}} \theta) \big]^2   \; \rmd  \hat{u} \; \rmd s   \Bigg] \leq \Bigg(  \frac{  f_{\infty}  }{\theta  (1 - \theta  )} \Bigg)^2 [T(C_{\theta} +1) + K_{\phi}] < \infty.
\end{align*}
\end{cor}
\begin{proof} 
The proof is a direct consequence of the identity $a^k -b^k=(a-b)\sum_{i=0}^{k-1} a^i b^{k-1-i}$ for any $a,b \in \mathbb{R}$, \eqref{convabsF} and Proposition \ref{propestenergstat}.	
\end{proof}
\begin{rem} \label{1condfastgab}
Instead of requiring the rather strong condition $f_{\infty}' < \infty$ given by \eqref{h2}, we only need that $f_\infty < \infty$. In particular, Corollary \ref{corestenergstat} holds when $F$ is given by the fast diffusion model in \cite{gabriel}.
\end{rem}
Next we focus on the proof of \eqref{cond2weak}.

\subsection{Proof of \eqref{cond2weak}}  \label{seccond2weak}

In order to obtain \eqref{cond2weak} we need to introduce some relevant definitions and state a few technical lemmas. For the remainder of this subsection, $O$ denotes an open subset of $\mathbb{R}^d$. 
\begin{definition}\label{def:OY-spaces}
Let $C_c^{0,0} ( [0,T] \times O^2)$ be the space of the functions $H: [0,T] \times O^2 \mapsto \mathbb{R}$ that are continuous and have compact support. For every $\delta >0$, let $O_\delta$ and $\mu_{O_\delta }$ be defined as 
	\begin{equation}\label{deflebvar}
		O_{\delta}:=\{(\hat{u}, \hat{v}) \in O^2: |\hat{u}-\hat{v}| \geq \delta \}
		\quad\text{and}\quad
		\rmd\mu_{O_\delta}(\hat{u},\hat{v})
		=
		\frac{1}{|\hat{u}-\hat{v}|^{d+\gamma}}\mathbbm{1}_{\{\hat{u},\hat{v}\in O_\delta\}}
		 \rmd\hat{u} \; \rmd\hat{v} 
		,
	\end{equation}
respectively. Moreover, let $Y_\delta(O)$ be defined as
	\begin{equation} \label{defOYeps}
		Y_\delta(O):=L^2 ( [0,T] \times O^2, \; \rmd s \otimes d \mu_{O_\delta } ).
	\end{equation}
Concretely, $H\in Y_\delta(O)$ if and only if $\| H \|^2_{Y_\delta(O)}<\infty$, where 
	\begin{align*}
		\| H \|^2_{Y_\delta(O)}
		:= \int_0^T \iint_{O_\delta} \frac{[H_s( \hat{u},\hat{v})]^2}{|\hat{u} - \hat{v}|^{d + \gamma}} \; \rmd \hat{u} \; \rmd \hat{v} \; \rmd s.
	\end{align*}
\end{definition}
Now we introduce the linear functional of interest. 
\begin{definition}
	For every $O \subset \mathbb{R}^d$, $\delta \in (0,1]$ , $k \geq 1$, and $\lambda:[0,T]\times\mathbb{R}^d\to \mathbb{R}$ we define $\ell^{O_{\delta}}_{\lambda}$ through its action on functions $H:C_c^{0,0} ( [0,T] \times O^2)\to\mathbb{R}$ by 
		\begin{align} \label{defelllamb}
		& \ell^{O_\delta}_{\lambda}(H) 
		: =  
		\int_0^T  \iint_{ O_\delta}  
		\frac{ \lambda_s( \hat{v}) - \lambda_s( \hat{u}) }{|\hat{u} -  \hat{v}|^{d + \gamma}}
		H_s( \hat{u} , \hat{v}) \;
		\rmd \hat{u} \; \rmd \hat{v} \; \rmd s, \quad H \in C_c^{0,0} ( [0,T] \times O^2)
		.
	\end{align}
\end{definition}
Next, we state a preliminary lemma that will be useful to obtain \eqref{cond2weak}. We postpone its proof to Appendix \ref{prolemestfrac}. 
\begin{lem} \label{new:lemestfrac}
	Let $O \subset \mathbb{R}^d$ and for every $\delta>0$, let $O_\delta$ and $Y_\delta(O)$ be as in Definition \ref{def:OY-spaces}, and let $k \geq 1$ be fixed. If:
	\begin{enumerate}
		\item There exists $\theta_1 \in \mathbb{R}$ such that
		\begin{align} \label{new:assumneu}
			\mathbb{E}_{\mathbb{Q}} \Bigg[  \int_0^T \int_{O } [ \rho(s, \hat{u}) -\theta_1 ]^2  \rmd \hat{u} \; \rmd s  \Bigg] < \infty;
		\end{align}
		\item There exist positive constants $\kappa_0, \kappa_1$ such that, for any fixed $j \geq 1 $, any family $ H_1, \ldots, H_j \in C_c^{0,0} ( [0,T] \times O^2)$ and and for any $\delta >0$, it holds that
		\begin{align} \label{new:assumsobpos1}
			\mathbb{E}_{\mathbb{Q}} \Big[ \max_{1 \leq i \leq j} \big\{ \ell^{O_\delta}_{F(\rho)}(H_i)  - \kappa_0 \|  H_i \|_{Y_\delta(O)}^2 \big\} \Big]    \leq \kappa_1,
		\end{align}
		where $\kappa_0$ and $\kappa_1$ are independent of $j$, $H_1, \ldots, H_j$ and $\delta$;
	\end{enumerate}
	Then  
	\begin{align} \label{new:fracsobdpos}
		\mathbb{E}_{\mathbb{Q}} \Bigg[  
		\int_0^T \iint_{ O^2 } 
		\frac{  \big[  F(\rho_s(\hat{v}))-F(\rho_s(\hat{u}) )  \big] ^2}{| \hat{u} - \hat{v}|^{d+\gamma}} \; \rmd \hat{u} \; 
		\rmd \hat{v} \; \rmd s  
		\Bigg]   < \infty.
	\end{align}
\end{lem}
Besides the previous lemma, another ingredient necessary to obtain \eqref{cond2weak} is the Feynman-Kac's formula. It was stated and proved in Lemma A.1 in \cite{baldasso}, for the case where the measure is not invariant with respect to a generator. In Lemma \ref{feyn} below, we state the aforementioned formula with respect to our model, for future reference. In what follows, for every probability measure $\nu$ on $\Omega$ and all functions $g_1, g_2: \Omega \to \mathbb{R}$, $\langle g_1, g_2 \rangle_{\nu}$ denotes the scalar product in $L^{2} (\Omega, \nu )$.
\begin{lem} \label{feyn}
	Let $\nu$ be a probability measure on $\Omega$ and $V:[0, \; \infty) \times \Omega$ be a bounded function. Then, for every $t \in [0,T]$, for any $B>0$ and for any $n \geq 1$, it holds
	\begin{align*}
		\frac{1}{B n^d} \log \Bigg\{  \mathbb{E}_{\nu} \Bigg[ \exp \Bigg( \int_0^t B n^d V(s, \eta_s^n) \;\rmd s \Bigg)    \Bigg]  \Bigg\} \leq   \int_0^t  \Big\{ \sup_f \big\{ \langle V(s, \cdot), f \rangle_{\nu} -\frac{1}{B n^d}  \langle -n^\gamma\mcb L_{n,\alpha}^\gamma \sqrt{f} , \sqrt{f} \rangle_{\nu}  \Big\} \;\rmd s.
	\end{align*}
	Above, the supremum is taken over all the densities $f$ on $\Omega$ with respect to $\nu$.\end{lem}
We recall from the beginning of Section \ref{secheurlin} that $\eta_t^n:=\eta_{t n^{\gamma}}$, with generator given by $n^{\gamma} \mcb L_{n,\alpha}^\gamma$. The term $\langle -n^\gamma\mcb L_{n,\alpha}^\gamma \sqrt{f} , \sqrt{f} \rangle_{\nu}$ in Lemma \ref{feyn} is known in the literature as the Dirichlet form. In order to estimate it, we introduce some quadratic forms.
\begin{definition}\label{def:dir_form}
        Let $ \nu$ be a probability measure on $ \Omega $ and $f: \Omega \mapsto \mathbb{R}$. We define the quadratic form $\mcb D_{n,\alpha}^\gamma ( f | \nu ):=  \mcb D_{n,\mcb F}^\gamma (f | \nu ) + \alpha_n  \mcb D_{n,\mcb S}^\gamma (f | \nu )$, where for $\mcb C \in \{ \mcb F, \mcb S\}$,  
	\begin{align} \label{defDnFS}
		\mcb D_{n,\mcb C}^\gamma (f | \nu ) 
            :=   
            \sum_{ \{ \hat{x} , \hat{y} \} \in \mcb C} 
            p_\gamma( \hat{y}  - \hat{x} ) 
           [ \mcb I^{n}_{\hat{x},\hat{y}}  (f | \nu ) + \mcb I^{n}_{\hat{y},\hat{x}}  (f | \nu ) ].
	\end{align} 
Above, $\mcb I^{n}_{\hat{x},\hat{y}}  (f | \nu )$ is given by  
\begin{equation} \label{defIxy}
 \mcb I^{n}_{\hat{x},\hat{y}}  (f | \nu ) 
	:= \frac{1}{2 N_{\text{e}}} \int_{\Omega} a_{\hat{x}, \hat{y}}(\eta) c_{\hat{x}, \hat{y}}^{n}(\eta)  [f (\eta^{\hat{x},\hat{y}}) - f (\eta)]^2 d \nu,
\end{equation}
    with $c_{\hat{x},\hat{y}}^{n}$ being given by \eqref{cons-series}. 
	\end{definition}	
From \eqref{cons-series} and \eqref{prod-bin}, for every $f: \Omega \rightarrow \mathbb{R}$ and any $\beta \in (0,1)$,
\begin{equation} \label{claimzw3}
	\forall \hat{z} \neq \hat{w} \in \mathbb{Z}^d, \; \forall n \in \mathbb{N}_+,  \quad \int_{\Omega} a_{ \hat{z}, \hat{w} }(\eta) c^{n}_{ \hat{z}, \hat{w} }(\eta) f(\eta)  d \nu_{\beta} = \int_{\Omega } a_{ \hat{w}, \hat{z} }(\eta) c^{n}_{ \hat{w}, \hat{z} }(\eta) f ( \eta^{ \hat{w}, \hat{z} } )   d \nu_{\beta}.
\end{equation}
Combining \eqref{claimzw3} with \eqref{defgenslow}, the symmetry of $p(\cdot)$ and of $r^n_{\cdot,\cdot}$ and Definition \ref{def:dir_form}, by performing some algebraic manipulations we conclude that
\begin{align}  \label{boundconst}
	\langle -\mcb L_{n,\alpha}^\gamma f , \; f \rangle_{\nu_\beta}
	= \frac{1}{2} \mcb D_{n,\alpha}^\gamma (f| \nu_\beta ),
\end{align}
for any $\beta \in (0,1)$ and any $f: \Omega \rightarrow \mathbb{R}$. The equality in \eqref{boundconst} motivates us to take $\nu_{\theta}$ as a reference measure, when estimating the Dirichlet form in Lemma \ref{feyn}, where $\theta$ is given by \eqref{2entbound}.

The following simple discretization result will be useful in what follows.
\begin{prop}\label{prop:phi-conv}
Fix $\delta \in (0, 1)$, let $O_\delta$ be as in Definition \ref{def:OY-spaces}, and fix $H \in C_c^{0,0} ( [0,T] \times O^2)$. Moreover,  define $\Phi_{\delta}^{H} \in C_c^{0,0}([0,T] \times \mathbb{R}^d)$ and $(\Phi_{\delta}^{H, n})_{n \geq 1}: [0,T] \times \mathbb{R}^d \mapsto \mathbb{R}$ by
\begin{align}
\Phi_{\delta}^{H}(s, \hat{u}):=& \int_{ \mathbb{R}^d} 
\frac{  H(s, \hat{v}, \hat{u}) - H(s, \hat{u}, \hat{v})   }{|\hat{v} - \hat{u}|^{d+\gamma}}    \mathbbm{1}_{ \{ (\hat{u}, \; \hat{v}) \in O_\delta  \} } \; \rmd  \hat{v}, \quad \quad (s, \hat{u}) \in [0,T] \times \mathbb{R}^d,
 \label{defPhiH} 
 \\
\Phi_{\delta}^{H, n}(s, \tfrac{\hat{x}}{n} ) 
 := &
 \frac{1}{n^d}
 \sum_{\hat{y}} 
 \frac{  H \big(s, \tfrac{\hat{y}}{n}, \tfrac{\hat{x}}{n}\big) - H\big(s, \tfrac{\hat{x}}{n}, \tfrac{\hat{y}}{n}\big)  }{\big|\tfrac{\hat{y}}{n}-\tfrac{\hat{x}}{n}\big|^{d + \gamma  }}   \mathbbm{1}_{ \{ (\hat{x}/n, \;\hat{y}/n) \in O_\delta  \} },
 \quad \hat{x} \in \mathbb{Z}^d, \; n \geq 1 \; ,s\in [0,T]. 
 \label{defPhiHn}
\end{align}

Then,
\begin{equation} \label{convPhiH}
\lim_{n \rightarrow \infty} \frac{1}{n^d}  \sum_{\hat{x}} \int_0^T  \big| \Phi_{\delta}^{H}\big(s, \tfrac{\hat{x}}{n}\big) -  \Phi_{\delta}^{H, n}\big(s, \tfrac{\hat{x}}{n} \big)\big| \; \rmd s =0.
\end{equation}
\end{prop}
\begin{proof}
Since $H \in C_c^{0,0} ( [0,T] \times O^2)$, there exists $b_H>0$ such that for every $n \in \mathbb{N}_+$, it holds
\begin{align*}
 \forall \hat{x} \in \mathbb{Z}^d: |\hat{x}| \geq n b_H, \quad \sup_{s \in [0,T]} \big| \Phi_{\delta}^{H}\big(s, \tfrac{\hat{x}}{n} \big) \big| = \sup_{s \in [0,T]} \big| \Phi_{\delta}^{H, n}\big(s, \tfrac{\hat{x}}{n} \big) \big| =0.
\end{align*}
The proof follows then from the Dominated Convergence Theorem.  
\end{proof}
Recall the approximation of the identity given in \eqref{aproxidright}. If $F$ is nonlinear, in order to make use of the fact that $\mathbb{Q}$ is a limit point, we will mollify $\overrightarrow{\iota_{\varepsilon}^{ \hat{u} }} \in L^1(\mathbb{R}^d)$, since $\overrightarrow{\iota_{\varepsilon}^{ \hat{u} }} \notin C_c^0(\mathbb{R}^d)$. 
\begin{definition}\label{def:mol-iota}
For every $\varepsilon \in (0, \; 1/2)$, let $\overrightarrow{\widetilde{\iota}_{\varepsilon}^{ \hat{0} }}: \mathbb{R}^d \mapsto [ 0, \; \varepsilon^{-d}]$ be defined as
	\begin{equation} \label{defitil0}
		\overrightarrow{\widetilde{\iota}_{\varepsilon}^{ \hat{0} }}
		=
		\begin{cases}
			\varepsilon^{-d} \quad \text{on} \quad [\varepsilon^2, \; \varepsilon - \varepsilon^2]^d, \\
			0 \quad \text{on} \quad \mathbb{R}^d - [0, \; \varepsilon]^d,
		\end{cases}
	\end{equation}
and $\overrightarrow{\widetilde{\iota}_{\varepsilon}^{ \hat{0} }}$ on $[0, \; \varepsilon]^d \setminus [\varepsilon^2, \; \varepsilon - \varepsilon^2]^d$ is any smooth interpolation, such that $\overrightarrow{\widetilde{\iota}_{\varepsilon}^{ \hat{0} }} \in C_c^{\infty}(\mathbb{R}^d)$.	
				
	Through translations, for any $\hat{u} \in \mathbb{R}^d$, we define $\overrightarrow{\widetilde{\iota}_{\varepsilon}^{ \hat{u} }} \in C_c^{\infty}(\mathbb{R}^d)$ as 
	\begin{equation} \label{defitilv}
		\overrightarrow{\widetilde{\iota}_{\varepsilon}^{ \hat{u} }} (\hat{v}):= \overrightarrow{\widetilde{\iota}_{\varepsilon}^{ \hat{0} }} ( \hat{v} - \hat{u} ).
	\end{equation} 
 \end{definition}
Keeping in mind \eqref{defitil0} and \eqref{defitilv}, we state the following auxiliary result which will be necessary when $F$ is nonlinear.
\begin{prop}\label{prop:aprox-rho}
Let $d \geq 1$, $j \in \{1, \ldots, d\}$, $k \geq 1$, $\hat{u} \in \mathbb{R}^d$, $s \in [0, \; T]$ and $\varepsilon \in (0, \; 1/2)$. Then, for every $n\geq 1$ it holds that
\begin{equation} \label{aproxtil}
\mathbb{P}_{\mu_n} \Bigg( \Bigg| \prod_{r=0}^{k-1} \Big\langle \pi_s^n, \; \overrightarrow{\iota_{\varepsilon}^{ \hat{u} + r \varepsilon \hat{e}_j } } \Big\rangle - \prod_{r=0}^{k-1} \Big\langle \pi_s^n, \; \overrightarrow{ \widetilde{\iota}_{\varepsilon}^{ \hat{u} + r \varepsilon \hat{e}_j } } \Big\rangle \Bigg| \leq k 2^{d+1} N_{\text{e}}^{k} \varepsilon  \Bigg) = 1. 
\end{equation}
Moreover, for every $G \in L^1([0,T] \times \mathbb{R}^d)$, it holds
\begin{align}
&\qquad\qquad\quad
\lim_{\varepsilon \rightarrow 0^+} 
\mathbb{E}_{\mathbb{Q}} 
\Bigg[ 
\int_0^T \int_{ \mathbb{R}^d } 
\Bigg| 
G_s( \hat{u}) 
\left(
[\rho_s( \hat{u})]^k - \prod_{r=0}^{k-1} \Big\langle \rho_s, \; \overrightarrow{ \widetilde{\iota}_{\varepsilon}^{ \hat{u} + r \varepsilon \hat{e}_j } }\Big\rangle 
\right)
\Bigg|  \rmd \hat{u} \; \rmd s 
\Bigg] =0, \label{aproxidtil} \\
&  \varlimsup_{\varepsilon \rightarrow 0^+} \varlimsup_{n \rightarrow \infty} \mathbb{E}_{\mu_n} \Bigg[ \int_0^T \Bigg| \int_{\mathbb{R}^d}  G_s( \hat{u}) \prod_{i=0}^{k-1} \Big\langle  \pi_s^n, \overrightarrow{\iota_{\varepsilon}^{ \hat{u} + i \varepsilon \hat{e}_j } } \Big\rangle  \; \rmd \hat{u} - \frac{1}{n^d} \sum_{\hat{x}}  G_s \big( \tfrac{\hat{x}}{n} \big) \prod_{i=0}^{k-1} \Big\langle  \pi_s^n, \overrightarrow{\iota_{\varepsilon}^{ \hat{x}/n + i \varepsilon \hat{e}_j } } \Big\rangle  \Bigg| \;\rmd s  \Bigg]=0. \label{aproxdiscnl}
\end{align}
\end{prop}
\begin{proof}
First we observe that \eqref{aproxtil} is a direct consequence of \eqref{aproxidright},  \eqref{defitil0}, \eqref{defitilv} and the generalized exclusion condition \eqref{bndeta}.
On the other hand, \eqref{aproxidtil} is a consequence of Proposition \ref{Qabscont}, Lebesgue's Differentiation Theorem and the Dominated Convergence Theorem.
In order to obtain \eqref{aproxdiscnl}, we observe that from \eqref{aproxidright} and again \eqref{bndeta}, for every $n \geq 1$ such that $\max\{|v_1|, \ldots, |v_d|\}/n < \varepsilon/3$ it holds \begin{equation}  \label{bnddifmedemp}
\forall s \in [0, \; T], \quad \mathbb{P}_{\mu_n} \Bigg( \Bigg| \prod_{r=0}^{k-1} \Big\langle  \pi_s^n, \overrightarrow{\iota_{\varepsilon}^{ \hat{u} + r \varepsilon \hat{e}_j } }\Big\rangle - \prod_{r=0}^{k-1} \Big\langle  \pi_s^n, \overrightarrow{\iota_{\varepsilon}^{ \hat{u} + r \varepsilon \hat{e}_j  + \hat{v}/n } } \Big\rangle \Bigg| \leq 2 d k N_{\text{e}}^{k} \frac{|\hat{v}|}{\varepsilon n}  \Bigg) = 1.
\end{equation}
Approximating $G$ in $L^1([0,T] \times \mathbb{R}^d)$ by any sequence $(G_i)_{i \geq 1} \subset C_c^{0,1}([0,T] \times \mathbb{R}^d)$, we get 
\begin{align} \label{limaproxL1a}
 \lim_{n \rightarrow \infty} \frac{1}{n^d}  \sum_{\hat{x}} \int_0^T \int_{[0, 1]^d}    \big|  G_s\big(  \tfrac{\hat{x}}{n} \big) - G_s \big(  \tfrac{\hat{x}+ \hat{v}}{n}  \big) \big| \; \rmd  \hat{v} \; \;\rmd s =0.   
\end{align}
As such, \eqref{aproxdiscnl} is a direct consequence of \eqref{bnddifmedemp},  \eqref{limaproxL1a} and \eqref{bndeta}.
\end{proof}
The following auxiliary result (which will be crucial if $F$ is nonlinear) is proved in Appendix \ref{sectopsko}.
\begin{prop} \label{weakconv2}
Fix $j \geq 1$, $G_1, \ldots, G_j \in C_c^{0,0}([0,T] \times \mathbb{R}^d )$ and $c_1, \ldots, c_j \in \mathbb{R}$. Moreover, fix $k \geq 2$ and $\varepsilon \in (0, \; 1/2)$. Then, it holds
\begin{align*}
& \mathbb{E}_{\mathbb{Q}} \Bigg[ \max_{1 \leq i \leq j} \Bigg\{ c_i +   \int_0^{T} \int_{\mathbb{R}^d} G_i(s, \hat{u})  \prod_{r=0}^{k-1} \Big\langle \rho_s, \; \overrightarrow{\widetilde{\iota}_{\varepsilon}^{ \hat{u} + r \varepsilon \hat{e}_d } } \Big\rangle  \rmd \hat{u} \; \rmd s \Bigg\}  \Bigg] \\
\leq & \varlimsup_{n \rightarrow \infty} \mathbb{E}_{\mu_n} \Bigg[ \max_{1 \leq i \leq j} \Bigg\{ c_i +   \int_0^{T} \int_{\mathbb{R}^d} G_i(s, \hat{u})  \prod_{r=0}^{k-1} \Big\langle \pi_s^n,  \;\overrightarrow{\widetilde{\iota}_{\varepsilon}^{ \hat{u} + r \varepsilon \hat{e}_d } } \Big\rangle   \rmd \hat{u} \; \rmd s \Bigg\}  \Bigg].
\end{align*}
\end{prop} 
The following remark will be useful in that follows.
\begin{rem} \label{remyoung}
From the equality $u^2-v^2=(u+v)(u-v)$ and Young's inequality, we get
\begin{align*}
|  f (\eta) - f(\eta^{ \hat{z}, \hat{w} }) | \leq \frac{f(\eta)+f(\eta^{ \hat{z}, \hat{w} })}{ A_{ \hat{z}, \hat{w} } } + \frac{ A_{ \hat{z}, \hat{w} }  [  \sqrt{f (\eta)}  - \sqrt{f (\eta^{ \hat{z}, \hat{w} })}   ]^2 }{2 },
\end{align*}
for every $\hat{z}, \hat{w} \in \mathbb{Z}^d$, every $A_{ \hat{z}, \hat{w} }>0$, every $\eta \in \Omega$ and every $f: \Omega \mapsto [0, \; \infty)$. 
\end{rem}
Thanks to Corollary \ref{corestenergstat}, we observe that \eqref{cond2weak} is a direct consequence of the following result.
\begin{prop} \label{propcond2weak}
Assume that $F$ satisfies \eqref{h2} and that \eqref{2entbound} holds. Furthermore, assume that either $\alpha >0$ and $O=\mathbb{R}^d$; or either $\alpha =0$ and $O \in \{\mathbb{R}_{-}^{d\star}, \mathbb{R}_{+}^{d\star}  \}$. Then \eqref{new:fracsobdpos} holds. 
\end{prop}
\begin{proof}
	Fix $\delta \in (0,1]$ and let $O_\delta$ and $Y_\delta(O)$ be given as in Definition \ref{def:OY-spaces}. The proof ends by applying Lemma \ref{new:lemestfrac}, which requires the conditions \eqref{new:assumneu} and \eqref{new:assumsobpos1}. The former holds due to Corollary \ref{corestenergstat}, so now we focus on the latter. Thus, let us now fix $j \geq 1$ and $H_1, \ldots, H_j \in C_c^{0,0} ( [0,T] \times O^2)$. In the following, we divide our reasoning in three steps, keeping in mind the arguments of Proposition 4.7 in \cite{CG}. We begin by replacing the series corresponding to $F(\rho)$ by some convenient polynomial.
	
\textbf{I) First step: truncation of the series corresponding to $F(\rho)$.} Observe that there exists $M >0$ such that
\begin{equation} \label{boundHiM}
\forall i \in \{1, \ldots, j\}, \; \forall n \geq 1, \quad \int_0^T   \sum_{(\hat{x}, \hat{y}) \in  O^n(\delta)} \frac{n^{\gamma}}{n^d}   \frac{\big|H_i\big(s,  \tfrac{\hat{x}}{n}, \tfrac{\hat{y}}{n}\big)\big|}{ |\hat{y}-\hat{x}|^{ \gamma + d} }  \rmd s
 \leq M,
\end{equation}
where $O_\delta^n$ is the discretization of the set $O_\delta$, defined by
	\begin{equation*} 
			O_\delta^n:= \big\{ (\hat{x}, \hat{y}) \in O^2 \cap (\mathbb{Z}^d \times \mathbb{Z}^d): \; \big|\tfrac{\hat{y}}{n} - \tfrac{\hat{x}}{n}\big| \geq \delta \big\}.
	\end{equation*}
Next, we fix $m \in \mathbb{N}$ such that $m \geq 2$ and 
\begin{equation} \label{Ftail}
 \sum_{k = m+1}^{\infty} ( |b_k^+| + |b_k^-|) N_{\text{e}}^{k-1} k \big[  2 \max_{1 \leq i \leq j} \big\{ \| \sqrt{|H_i| } \; \|^2_{Y_\delta(O)} \big\} + N_{\text{e} } M      \big]  \leq 1.
\end{equation}
In the last line we applied \eqref{h2}. Introducing
\begin{align*}
F_m(\rho):=\sum_{k=1}^{m}[ b_k^+ \rho^k -  b_k^- (N_{\text{e}}-\rho)^k ],
\end{align*}
	we get from \eqref{defelllamb}, \eqref{Ftail} and Proposition \ref{Qabscont} that for every $\kappa_0$, it holds
\begin{align*}
	&	\mathbb{E}_{\mathbb{Q}} \Big[ \max_{1 \leq i \leq j} \big\{ \ell^{O_\delta}_{F(\rho)}(H_i)  - \kappa_0 \|  H_i \|_{Y_\delta(O)}^2 \big\} \Big] \leq  \mathbb{E}_{\mathbb{Q}} \Big[ \max_{1 \leq i \leq j} \big\{ \ell^{O_\delta}_{F_m(\rho)}(H_i)  - \kappa_0 \|  H_i \|_{Y_\delta(O)}^2 \big\} \Big] +  1.
\end{align*}
Therefore, in order to obtain \eqref{new:assumsobpos1}, it is enough to find positive constants $\kappa_0, \kappa_1$ such that
\begin{align} \label{assumFtrun}
\mathbb{E}_{\mathbb{Q}} \Big[ \max_{1 \leq i \leq j} \big\{ \ell^{O_\delta}_{F_m(\rho)}(H_i)  - \kappa_0 \|  H_i \|_{Y_\delta(O)}^2 \big\} \Big] \leq \kappa_1.
\end{align}
Next we claim that for every $\kappa_0 >0$, it holds
	\begin{align} \label{boundnonlindisc}
			\mathbb{E}_{\mathbb{Q}} 
			[ \max_{1 \leq i \leq j} \big\{ - \kappa_0 \|H_i \|_{Y_\delta(O)} + \ell_{F_m(\rho)}^{O_\delta} (H_i) \big\} ] 
			\leq
			\varlimsup_{n \rightarrow \infty} \mathbb{E}_{\mu_{n}} \big[  \max_{1 \leq i \leq j} \big\{ Z^{\kappa_0}_{n,H_i} \big\}  \big],
		\end{align}
where for any $\kappa_0 >0$ and any $i \in \{1, \ldots, j\}$,  the map $Z^{\kappa_0}_{n,H_i} : \mcb D_{\Omega} ( [0,T] ) \mapsto \mathbb{R}$ is given by 
	\begin{align} \label{defZnHikap}
		Z^{\kappa_0}_{n,H_i}(\eta_\cdot)
		:= 
		\int_0^T \Bigg\{ \frac{n^{\gamma}}{n^d} \sum_{ (\hat{x}, \hat{y}) \in  O_\delta^n } \Big[  
			F^{m_\star}(\tau^{\hat{y}} \eta_s^n )-F^{m_\star}(\tau^{x} \eta_s^n )
		\Big] 
		\frac{H_i \big(s, \tfrac{\hat{x}}{n}, \tfrac{\hat{y}}{n} \big)}{|\hat{y}-\hat{x}|^{ \gamma + d}}   -   \frac{\kappa_0 \| H_i \|^2_{Y_\delta(O)}}{T}  \Bigg\} \; \rmd s.
	\end{align}
The arguments for obtaining \eqref{boundnonlindisc} are quite technical, thus we postpone its proof to the end of the proof of the current proposition. In \eqref{defZnHikap}, $F^{m_\star}: \Omega \to \mathbb{R}$ is given by
\begin{align*}
 F^{m_\star}(\eta)
		:=\frac{1}{d}\sum_{j=1}^d\sum_{k=1}^{m}
		\left[
		b_k^+ \prod_{i=0}^{k-1}\eta(i\hat{e}_j)
		-b_k^- \prod_{i=0}^{k-1} \widetilde{\eta} (i\hat{e}_j)
		\right]. 
\end{align*}
Assuming \eqref{boundnonlindisc}, from Lemma \ref{entjen} (with $B=1$), the proof ends if we can prove that there exists some $\kappa_0 >0$ such that
	\begin{equation}  \label{cladynnonlin}
	\max_{1 \leq i \leq j} \;	\varlimsup_{n \rightarrow \infty} \frac{1}{n^d} \log \Bigg(  \mathbb{E}_{\nu_{\theta}} \Big[ \exp  (n^d Z^{\kappa_0}_{n,H_i})  \Big]  \Bigg) \leq 1.
	\end{equation}
	
\textbf{II) Second step: proof of \eqref{cladynnonlin}, assuming \eqref{boundnonlindisc}.} By performing algebraic manipulations, applying  \eqref{excrule} and using the uniform continuity of $H_i$ in an analogous way as it was done in the proof of Proposition 4.7 in \cite{CG}, the term on the left-hand side of the last display is bounded from above by
\begin{equation} \label{est2}
\begin{split}
\max_{1 \leq i \leq j} \varlimsup_{n \rightarrow \infty} \int_0^T \sup_{f}   \Bigg\{ &  \sum_{(\hat{x}, \hat{y}) \in  O^n(\delta)} \frac{n^{\gamma-d}}{2}  \frac{H_i\big(s,  \tfrac{\hat{x}}{n}, \tfrac{\hat{y}}{n} \big)}{ |\hat{y}-\hat{x}|^{ \gamma + d} } \int_{\Omega}  [\eta ( \hat{y} ) - \eta(\hat{x})]    c_{\hat{x},\hat{y}}^{m_\star}(\eta) f(\eta)\, \rmd \nu_{\theta}  \\
+ &n^{\gamma-d}\langle \mcb L_{n,\alpha}^\gamma  \sqrt{f},\sqrt{f} \rangle_{\nu_{\theta}} - \frac{\kappa_0 \| H_i \|^2_{Y_\delta(O)}}{T}  \Bigg\} \; \rmd s.  
\end{split}
\end{equation}
Above and in the remainder of this proof, all the suprema over $f$ are carried over all the densities $f$ with respect to $\nu_{\theta}$. In the last display, $c_{\hat{x},\hat{y}}^{m_\star}(\eta)$ is given by
\begin{align*}
c_{\hat{x},\hat{y}}^{m_\star}(\eta)
		:= \frac{1}{d}\sum_{j=1}^d
		\sum_{k=1}^{m} \big[ b_k^{+} c_{\hat{x}, \hat{y}}^{(k),j}(\eta )+  b_k^{-} c_{\hat{x}, \hat{y}}^{(k),j}(\widetilde{\eta} )\big],
\end{align*}
where the term $c^{(k,j)}_{\hat{x},\hat{y}}$is defined in $\eqref{eq:rates_porous}$. Analogously as it was done in the proof of Proposition 4.6 in \cite{CG}, we will combine \eqref{boundconst} and Definition \ref{def:dir_form} in order to obtain \eqref{cladynnonlin}. But first, due to \eqref{defIxy}, we will replace $c_{\hat{x},\hat{y}}^{m_\star}(\eta)$ by $c_{\hat{x},\hat{y}}^{n}(\eta)$  in \eqref{est2}. Since  $\lim_{n \rightarrow \infty} \ell_n =\infty$, without loss of generality, we can assume that $\ell_n > m$. From \eqref{cons-series}, \eqref{bndeta}, \eqref{boundHiM} and \eqref{Ftail}, the $\varlimsup$ in \eqref{est2} is bounded from above by 
  \begin{equation} \label{est3}
\begin{split}
1 + \max_{1 \leq i \leq j} \varlimsup_{n \rightarrow \infty} \int_0^T \sup_{f}   \Bigg\{ &  \sum_{(\hat{x}, \hat{y}) \in  O^n(\delta)} \frac{n^{\gamma-d}}{2}  \frac{H_i\big(s,  \tfrac{\hat{x}}{n}, \tfrac{\hat{y}}{n}\big)}{ |\hat{y}-\hat{x}|^{ \gamma + d} } \int_{\Omega} \frac{a_{ \hat{y}, \hat{x} }(\eta) - a_{ \hat{x}, \hat{y} }(\eta)}{N_{\text{e}}}  c_{\hat{x},\hat{y}}^{n}(\eta) f(\eta) \; \rmd \nu_{\theta}  \\
+ &n^{\gamma-d}\langle \mcb L_{n,\alpha}^\gamma \sqrt{f},\sqrt{f} \rangle_{\nu_{\theta}} - \frac{\kappa_0 \| H_i \|^2_{Y_\delta(O)}}{T}  \Bigg\} \; \rmd s,  
\end{split}
\end{equation}
In the last display we applied \eqref{excrule2}. Next, due to Remark \ref{remyoung} and \eqref{defIxy}, we have that the absolute value of the integral over $\Omega$ in \eqref{est3} is bounded from above by
\begin{align*}
\frac{ 2N_{\text{e}} f_{\infty}'}{A_{\hat{x} , \hat{y}}} + \frac{A_{\hat{x} , \hat{y}} }{2  } [ \mcb I^{n,\gamma}_{\hat{x},\hat{y}}  ( \sqrt{f} | \nu_{\theta} ) + \mcb I^{n,\gamma}_{\hat{y},\hat{x}}  ( \sqrt{f} | \nu_{\theta} ) ],
\end{align*}
for every $A_{\hat{x} , \hat{y}} >0$. In the last line, we combined \eqref{claimzw3} with the upper bound $a_{ \hat{z}, \hat{w} }(\eta)c^{n}_{\hat{z},\hat{w}}( \eta) \leq N_{\text{e}}^2 f_{\infty}'$, for every $\hat{z}, \hat{w} \in \mathbb{Z}^d$, every $n \geq 1$ and every $\eta \in \Omega$, due to \eqref{eq:rates_porous} and \eqref{h2}; and the fact that $f$ is a density with respect to $\nu_\theta$. Thus, we conclude that the sum inside the supremum in \eqref{est3} is bounded from above by
\begin{equation} \label{bound2termglopos}
    \sum_{(\hat{x}, \hat{y}) \in  O^n(\delta)} \frac{n^{\gamma}}{n^d}  \frac{\big|H_i\big(s,  \tfrac{\hat{x}}{n}, \tfrac{\hat{y}}{n}\big)\big|}{ |\hat{y}-\hat{x}|^{ \gamma + d} } \Bigg\{ \frac{ N_{\text{e}} f_{\infty}'}{A_{\hat{x} , \hat{y}}} + \frac{A_{\hat{x} , \hat{y}} }{4  } \big[ \mcb I^{n,\gamma}_{\hat{x},\hat{y}}  ( \sqrt{f} | \nu_{\theta} ) + \mcb I^{n,\gamma}_{\hat{y},\hat{x}}  ( \sqrt{f} | \nu_{\theta} ) \big] \Bigg\}.
\end{equation}
The next term in \eqref{est3} is $\langle n^{\gamma-d} \mcb L_{n,\alpha}^\gamma \sqrt{f},\sqrt{f} \rangle_{\nu_{\theta}}$, which is equal to $\tfrac{-1}{2} n^{\gamma-d}\mcb D_{n,\alpha}^\gamma (\sqrt{f} | \nu_\theta )$, due to \eqref{boundconst}. Moreover, from \eqref{defDnFS}, this term is bounded from above by
\begin{align*} 
 & \frac{n^{\gamma-d}}{2} \sum_{ (\hat{x}, \hat{y}) \in  O^n(\delta) }  c_{\gamma} |\hat{y} - \hat{x}|^{-d-\gamma} \big( - \alpha^n_{\hat{x}, \hat{y}} \big) \big[ \mcb I^{n,\gamma}_{\hat{x},\hat{y}}  ( \sqrt{f} | \nu_{\theta} ) + \mcb I^{n,\gamma}_{\hat{y},\hat{x}}  ( \sqrt{f} | \nu_{\theta} ) \big] \\
 \leq & \frac{n^{\gamma-d}}{2} \sum_{ (\hat{x}, \hat{y}) \in  O^n(\delta) }  c_{\gamma} |\hat{y} - \hat{x}|^{-d-\gamma} ( - K_{\alpha} ) \big[ \mcb I^{n,\gamma}_{\hat{x},\hat{y}}  ( \sqrt{f} | \nu_{\theta} ) + \mcb I^{n,\gamma}_{\hat{y},\hat{x}}  ( \sqrt{f} | \nu_{\theta} ) \big] ,
\end{align*}
for any $n$ large enough. Combining the last display with \eqref{bound2termglopos}, we get that the $\varlimsup$ in \eqref{est3} is bounded from above by zero, when we choose $A_{\hat{x} , \hat{y}}=2 K_{\alpha} c_{\gamma} \big|H_i\big(s,  \tfrac{\hat{x}}{n}, \tfrac{\hat{y}}{n}\big)\big|^{-1}$ and $\kappa_0 =   N_{\text{e}} f_{\infty}' ( K_{\alpha} c_{\gamma} )^{-1} >0$. In particular, the inequality in \eqref{cladynnonlin} holds. 

\textbf{III) Final step: proof of \eqref{boundnonlindisc}.} In order to finish the proof, we need to show \eqref{boundnonlindisc}. The remaining arguments depend on whether $F$ is linear or nonlinear.

\begin{itemize}

\item First, we treat the case when $F$ is linear. From \eqref{bndeta} and Proposition \ref{prop:phi-conv}, we get
\begin{align} \label{convlinear}
 \lim_{n \rightarrow \infty} \mathbb{E}_{\mu_{n}} \Bigg[  \max_{1 \leq i \leq j} \Bigg\{ \frac{1}{n^d}  \sum_{\hat{x}} \int_0^T  \big| \Phi_{\delta}^{H_i}\big(s, \tfrac{\hat{x}}{n}\big) -  \Phi_{\delta}^{H_i, n}\big(s, \tfrac{\hat{x}}{n} \big) \big| \eta_s^n(\hat{x}) \; \rmd s \Bigg\} \Bigg]  =0.
\end{align}
Next, observe that $F_m(\rho)=F(\rho)=b_0 - b_1^- N_{\text{e} } + (b_1^+ + b_1^-) \rho$, since we imposed $m \geq 2$. Thus, from \eqref{defelllamb} and \eqref{defPhiH} we get that $ \ell_{F_m(\rho)}^{O_\delta} (H_i) = (b_1^+ + b_1^-) \langle \Phi_{\delta}^{H_i}, \rho \rangle$, for any $i \in \{1, \ldots, j\}$. Thus, choosing $c_i:= - \kappa_0  \|H_i \|_{Y_\delta(O)}$ for every $i \in \{1, \ldots, j\}$, from Proposition \ref{weakconv1} and \eqref{convlinear}, we get that the left-hand side of \eqref{boundnonlindisc} is bounded from above by
\begin{align*}
 \varlimsup_{n \rightarrow \infty} \mathbb{E}_{\mu_{n}} \Bigg[  \max_{1 \leq i \leq j} \Bigg\{  - \kappa_0  \|H_i \|_{Y_\delta(O)} + \int_0^T  \frac{b_1^+ + b_1^-}{n^d} \sum_{\hat{x}} \eta_s^n(\hat{x}) \Phi_{\delta}^{H_i, n} \big(s, \tfrac{\hat{x}}{n}  \big) \; \rmd s \Bigg\} \Bigg].
\end{align*}
From \eqref{defPhiHn}, the term in the last line is equal to the right-hand side of \eqref{boundnonlindisc} and the proof ends for the linear case.

\item

Finally, we treat the case when $F$ is nonlinear. In order to simplify the notation, in the remainder of the proof we will assume that $F(\rho)=\rho^k$ for some $k \in \mathbb{N}_+$ and $m \geq k$. Nevertheless, we observe that the proof for the general case is analogous. 

Let $\epsilon >0$. From \eqref{aproxidtil}, there exists $\varepsilon_1 \in (0, \; 1/2)$ satisfying
\begin{equation*} 
 \forall \varepsilon \in (0, \varepsilon_1], \quad \mathbb{E}_{\mathbb{Q}} \Bigg[ \max_{1 \leq i \leq j} \Bigg\{  \Bigg| \ell_{F_m(\rho)}^{O_\delta} (H_i) - \int_0^T \int_{ \mathbb{R}^d } \Phi_{\delta}^{H_i}(s, \hat{u}) \prod_{r=0}^{k-1} \Big\langle \rho_s, \; \overrightarrow{ \widetilde{\iota}_{\varepsilon}^{ \hat{u} + r \varepsilon \hat{e}_d } } \Big\rangle\; \rmd \hat{u} \; \rmd s \Bigg| \Bigg\} \Bigg] \leq \epsilon.
\end{equation*}
Thus, choosing $c_i:= - \kappa_0  \|H_i \|_{Y_\delta(O)}$ for every $i \in \{1, \ldots, j\}$, by applying Proposition \ref{weakconv2} and \eqref{aproxtil}, it is possible choose $\varepsilon_2 \in (0, \; 1/2)$ such that the left-hand side of \eqref{boundnonlindisc} is bounded from above by
\begin{align*} 
2 \epsilon + \varlimsup_{n \rightarrow \infty} \mathbb{E}_{\mu_n} \Bigg[ \max_{1 \leq i \leq j} \Bigg\{ c_i + \int_0^T \int_{ \mathbb{R}^d } \Phi_{\delta}^{H_i}(s, \hat{u}) \prod_{r=0}^{k-1} \Big\langle \pi_s^n,  \;\overrightarrow{\iota_{\varepsilon}^{ \hat{u} + r \varepsilon \hat{e}_d } } \Big\rangle \; \rmd \hat{u} \; \rmd s \Bigg\} \Bigg],
\end{align*}
for every $\varepsilon \in (0, \; \varepsilon_2]$. Thus, from \eqref{aproxdiscnl} and \eqref{medempright}, the $\varlimsup$ in the last display is bounded from above by
\begin{align} \label{bound1nonlin}
  \varlimsup_{\varepsilon \rightarrow 0^+} \varlimsup_{n \rightarrow \infty} \mathbb{E}_{\mu_n} \Bigg[ \max_{1 \leq i \leq j} \Bigg\{ c_i + \int_0^T  \frac{1}{n^d} \sum_{\hat{x}}  \Phi_{\delta}^{H_i} \big(s, \tfrac{\hat{x}}{n} \big) \prod_{r=0}^{k-1} \overrightarrow{\eta}_t^{\varepsilon n}(\hat{x} + r \varepsilon n \hat{e}_d ) \; \rmd s \Bigg\} \Bigg].
\end{align}
In the next step, we apply Lemma \ref{globrep}. More exactly, since $\Phi_{\delta}^{H_i}$ has compact support, there exists $b_i >0$ such that $\sup_{s \in [0,T]} \big| \Phi_{\delta}^{H_i}(s, \hat{u})\big|=0$ whenever $|\hat{u}| \geq b_i$, for any $i \in \{1, \ldots, j\}$. Keeping in mind \eqref{boundrep}, for every $i \in \{1, \ldots, j\}$, define $\widetilde{H}_i: \mathbb{R}^d \rightarrow \mathbb{R}$ by
\begin{align*}
\widetilde{H}_i(\hat{u}):= \frac{4 \| H_i \|_{\infty} \pi^{d-1}}{\gamma \delta^{\gamma}  } \mathbbm{1}_{ \{ |\hat{u}| \leq b_i   \}  }, \quad \hat{u} \in \mathbb{R}^d.
\end{align*}
In particular, from \eqref{defPhiH} and from \eqref{sphe} below, we have that
\begin{align*}
  \widetilde{H}_i \in L^1(\mathbb{R}^d) \cap L^{\infty}(\mathbb{R}^d) \quad \text{and} \quad \forall \hat{u} \in \mathbb{R}^d, \quad \sup_{s \in [0, T]} \big|\Phi_{\delta}^{H_i}(s, \hat{u}) \big| \leq  \widetilde{H}_i(\hat{u}),
\end{align*}
thus $\Phi_{\delta}^{H_i}$ satisfies \eqref{boundrep}; this holds for any $i \in \{1, \ldots, j\}$. Then, applying Lemma \ref{globrep} for $t=T$, \eqref{bndeta} and Proposition \ref{prop:phi-conv}, we conclude that the display in \eqref{bound1nonlin} is bounded from above by
\begin{align*}
   \varlimsup_{n \rightarrow \infty} \mathbb{E}_{\mu_n} \Bigg[ \max_{1 \leq i \leq j} \Bigg\{ c_i + \int_0^T  \frac{1}{n^d} \sum_{\hat{x}}  \Phi_{\delta}^{H_i, n} \big(s, \tfrac{\hat{x}}{n} \big) \prod_{r=0}^{k-1} \eta_t^{ n}(\hat{x} + r  \hat{e}_d ) \; \rmd s \Bigg\} \Bigg].
\end{align*}
Combining \eqref{defPhiHn} with the last display and the observation that $\epsilon>0$ is arbitrary, we obtain the inequality in \eqref{boundnonlindisc} and the proof ends for the nonlinear case.
\end{itemize}
  \end{proof}

\section{Characterization of limit points} \label{secchar}

In this section we complete the characterization of the limit point $\mathbb{Q}$, by showing that 
\begin{equation} \label{Qintform}
\mathbb{Q} \Big( \pi_{\cdot}:\; \forall t \in [0, \; T], \; \forall G \in  X, \quad \mcb F_{t}^{\gamma,\alpha}(G, \rho, g) = 0\Big)=1.
\end{equation}
In the last line, $\mcb F_{t}^{\gamma,\alpha}$ is given in \eqref{weakform} and $g$ is given by Theorem \ref{hydlim}. Moreover, above and in the remainder of this section, $X$ is space of test functions given by
\begin{align} \label{defX}
X:= 
\begin{cases}
S_{\gamma,\star}^{d}, \quad \text{if} \quad  \alpha=0,  \; \gamma \in [1,2), \quad \text{and} \quad  \lim_{n \rightarrow \infty} \alpha_n r_n^\gamma  \in (0, \infty]; \\
S_{\gamma,\alpha}^{d}, \quad \text{otherwise}.
\end{cases}
\end{align}
In the last display, the spaces $S_{\gamma,\alpha}^{d}$ and $S_{\gamma,\star}^{d}$ are given in Definition \ref{defschw2}, respectively; and $r_n^\gamma$ is given in \eqref{rngamma}. The first step for obtaining \eqref{Qintform} is to state the following approximation result, whose proof is postponed to Appendix \ref{antoolstest}.   
\begin{lem} \label{seplem}
	Let $\alpha \geq 0$, $\gamma \in (0, \;2)$ and $d \geq 1$. Then there exists $(H^i)_{i \geq 1} \subset X$ such that for every $\varepsilon >0$ and every $G \in X$, there exists $i_0 \geq 1$ depending only on $\varepsilon$ and $G$ such that 
\begin{equation*} 
\| H^{i_0} - G \|_X := \sup_{s\in[0,T]} \left\{	\| H_s^{i_0}- G_s \|_{1} \right\}+ \int_0^T \|	\partial_s H_s^{i_0} - \partial_s G_s	\|_{1}	\;\rmd s + \int_0^T \|  \mathbb{L}_{\alpha}^{\gamma}H_s^{i_0}-  \mathbb{L}_{\alpha}^{\gamma}G_s \|_{1} \; \rmd s < \varepsilon.
\end{equation*}	
\end{lem}
From Lemma \ref{seplem}, the elements of $X$ can be approximated by a countable set, which is crucial for obtaining the next result. 
\begin{cor} \label{corapro} 
Let $\alpha \geq 0$, $\gamma \in (0, \;2)$ and $d \geq 1$. Moreover, assume that for every
$G \in X$ and $\delta >0$, it holds $\mathbb{Q} \Big( \pi_{\cdot}:\;  \sup_{t \in [0,T]} \big|\mcb F_{t}^{\gamma,\alpha}(G, \rho, g) \big|> \delta\Big)=0.
$ Then \eqref{Qintform} holds.
\end{cor}
\begin{proof}
Let  $(H^i)_{i \geq 1} \subset X$ be given as in Lemma \ref{seplem} and define $i_0: X \times \mathbb{N}_{+} \mapsto \mathbb{N}_{+}$ by
\begin{align*}
i_0(G, j):= \min \big\{ i \geq 1:  \quad \| H^i - G \|_X < [j ( 2 N_{ \text{e} } + f_\infty )]^{-1} \big\}, \quad (G, j) \in X \times \mathbb{N}_{+},
\end{align*}
where $f_\infty$ is given by \eqref{convabsF}. From Lemma \ref{seplem}, $i_0$ is well-defined. Combining the last display with \eqref{weakform} and Proposition \ref{Qabscont}, we get
\begin{align*}
\forall j \in \mathbb{N}_+, \quad &\mathbb{Q} \Big( \pi_{\cdot}: \quad \sup_{G \in X} \; \sup_{t \in [0, \; T]} \big|\mcb F_{t}^{\gamma,\alpha}(G, \rho, g)\big| > 1/j \Big) \\
\leq & \mathbb{Q} \Big( \sup_{G \in X} \; \sup_{t \in [0, \; T]} \big|\mcb F_{t}^{\gamma,\alpha}(G, \rho, g) - \mcb F_{t}^{\gamma,\alpha}(H^{i_0(G, 2j)}, \rho, g) \big| > (2j)^{-1} \Big) \\
+ & \mathbb{Q} \Big( \pi_{\cdot}: \quad \sup_{G \in X} \sup_{t \in [0, \; T]} \big| \mcb F_{t}^{\gamma,\alpha}(H^{i_0(G, 2j)}, \rho, g) \big| > (2j)^{-1} \Big) = 0.
\end{align*}
In the last line we applied the assumptions of the current corollary to $H^i$, for any $i \geq 1$. 
\end{proof}
In what follows, we will apply the following lemma, whose proof is postponed to Appendix \ref{sectopsko}. It will be useful for applying Portmanteau's Theorem (Theorem 2.1 in \cite{Bill}).
\begin{lem} \label{openport2}
Let $d \geq 1$, $H \in C_c^0(\mathbb{R}^d)$, $ g \in L^{\infty}(\mathbb{R}^d)$, $G \in L^1(\mathbb{R}^d)$ and $\delta >0$. Then
\begin{align*}
\Bigg\{ 
\tilde{\pi}_{\cdot} \in \mcb{D}_{\mcb {M}^{+}_{N_{\text{e}}}}([0,T]):\;  
\Bigg|  \int_{\mathbb{R}^d} H(\hat{u}) \tilde{\pi}_0( \rmd \hat{u} )  
- 
\int_{\mathbb{R}^d} g( \hat{u} ) G( \hat{u} ) \; \rmd \hat{u}  
\Bigg| 
> \delta 
\Bigg\}
\end{align*}
is an open set, with respect to the Skorohod topology of $\mcb{D}_{ \mcb {M}^{+}_{N_{\text{e}}} }([0,T])$.
\end{lem}
In order to obtain \eqref{Qintform}, we will make use of the following technical result, whose proof is postponed to Appendix \ref{sectopsko}.  This lemma is useful for applying Portmanteau's Theorem. In what follows, recall the approximations of the identity $\tilde{\iota}_{\varepsilon_0}^{\hat{u}}$ given in Definition \ref{def:mol-iota}.
\begin{lem} \label{openport4}
Let $d \geq 1$, $\gamma \in (0,2)$ and $H, G \in C_c^{0,0}([0,T] \times \mathbb{R}^d)$. Fix $\varepsilon_0 \in (0, \;1)$, an integer $m \geq 2$ and $H^2, \ldots, H^{m}  \in C_c^{0,0}([0,T] \times \mathbb{R}^d)$.  
Then the function $\Psi: \mcb{D}_{ \mcb {M}^{+}_{N_{\text{e}}} }([0,T]) \mapsto \mathbb{R}$, given by \begin{equation} \label{defPs2}
\begin{split}
\Psi (\tilde{\pi}_{\cdot} )
:=& \sup_{t \in [0,T]}
 \Big|  
	\int_{\mathbb{R}^d}
	 H_t (\hat{u}) 
	 \tilde{\pi}_t( \rmd \hat{u} ) 
	 - \int_{\mathbb{R}^d} 
	 H_0( \hat{u}) 
	 \tilde{\pi}_0( \rmd \hat{u} ) 
	 - \int_0^t \int_{\mathbb{R}^d} 
	 G_s( \hat{u}) 
	 \tilde{\pi}_s( \rmd \hat{u} )\;\rmd s 
\\
-& \frac{1}{d} \sum_{k=2}^{ m } 
	\int_0^t  \int_{\mathbb{R}^d}  
	H^k_s(\hat{u})     \sum_{j=1}^d \prod_{i=0}^{k-1} 
	\int_{\mathbb{R}^d} \overrightarrow{\tilde{\iota}_{\varepsilon_0}^{ \hat{u} + i \varepsilon_0 \hat{e}_j } } ( \hat{v})  
	\tilde{\pi}_s( d \hat{v} ) 
	\; \rmd \hat{u} \; \rmd s  
\Big|
,
\end{split}
\end{equation} 
is lower semi-continuous, with respect to the Skorohod topology of $\mcb{D}_{ \mcb {M}^{+}_{N_{\text{e}}} }([0,T])$. 
\end{lem}
Now we are ready for stating the main result of this section.
\begin{prop} \label{derintnonlin}
Assume that Hypothesis \ref{hipoerror} is satisfied and that one of the conditions \eqref{condalppos}, \eqref{condneu}, \eqref{condlip} holds. Furthermore, assume that either $F$ is linear; or $F$ is nonlinear and Hypothesis \ref{hiporepl} holds. Then, \eqref{Qintform} holds.
 \end{prop}
\begin{proof}
In order to simplify the notation, above and in the remainder of this proof, we omit $\pi_{\cdot}$ from the sets where we are looking at.  According to Corollary \ref{corapro}, it is enough to verify that for any $\delta>0$ and any $G \in X$,
	it holds that $\mathbb{Q} \Big(  \sup_{t \in [0,T]} \big|\mcb F_t^{\gamma,\alpha}(G,\rho,g) \big|> \delta\Big)=0
	$. The aforementioned probability is bounded from above by 
	\begin{equation} \label{Qassoc}
		\mathbb{Q} \Big(  \sup_{t \in [0,T]} \big| F_t^{\gamma,\alpha}(G,\rho,\rho_0)\big| > \delta - \epsilon \Big)
		+\mathbb{Q} \Big( \big| \langle \pi_0, \; G_0\rangle - \langle g,G_0\rangle  \big|> \epsilon \Big)
	\end{equation}
	as a consequence of the inequality $\big|\mcb F_t^{\gamma,\alpha}(G,\rho,g)\big| \leq \big|F_t^{\gamma,\alpha}(G,\rho,\rho_0)\big|+\big| \langle \rho_0-g,G_0\rangle\big|$. Above and in the remainder of this proof, we take $\epsilon=\delta/16$. Since $G \in X \subset S_{\alpha}^{\gamma, d}$, we conclude that $G_0 \in L^1(\mathbb{R}^d)$, thus, there exists a sequence $(H_r)_{r \geq 1} \subset C_c^{0}( \mathbb{R}^d)$ which converges to $G_0$ in  $L^1(\mathbb{R}^d)$. Making use of this sequence jointly with Lemma \ref{openport2}, Portmanteau's theorem and the fact that $(\mu_n)_{n \geq 1}$ is associated with $g$ (see Definition \ref{assoc}) we have that the rightmost term in \eqref{Qassoc} is equal to zero.
			
Now we focus on the leftmost term in \eqref{Qassoc}. In the following, we divide our reasoning in three steps, keeping in mind the arguments of Proposition 4.2 in \cite{CG}. As it was done in the proof of Proposition \ref{propcond2weak}, we begin by replacing the series corresponding to $F(\rho)$ by some convenient polynomial.
	
\textbf{I) First step: truncation of the series corresponding to $F(\rho)$.} We are done if we can prove that for every  $G \in X$ and every $\delta >0$, it holds
\begin{align}  \label{charact2nonlin}
\mathbb{Q} \Bigg(  \sup_{t \in [0,T]} \Bigg| &  \langle \rho_t, G_t\rangle - \langle  \rho_0 , G_0\rangle  - \int_0^t \langle \rho_s, \partial_s G_s\rangle \;\rmd s -   \int_0^t \langle   F  ( \rho_s) ,  \mathbb{L}_{\alpha}^{\gamma} G_s \rangle  \; \rmd s \Bigg| > \delta - \epsilon  \Bigg)=0.
\end{align}
In order to do so, from now on fix $\delta >0$ and $G \in X \subset S_{\gamma,\alpha}^{d}$. Recalling that $F$ satisfies \eqref{convabsF}, we can choose $m=m(\delta, G)$ such that
\begin{align*}
	\sum_{k=m +1}^{\infty} (|b_{k}^{+}| +|b_{k}^{-})| )N_{\text{e}}^k \int_0^T \big\| \mathbb{L}_{\alpha}^{\gamma}G_s \big\|_1  \; \rmd s  \leq \frac{\delta}{16} = \epsilon.
\end{align*}
In the last line we applied Proposition \ref{propL1alpha}.
Due to \eqref{convKnalpha}, we can (and will) assume without loss of generality that $n$ is large enough such that
\begin{align} \label{truncKn}
	\sum_{k=m+1}^{\infty} (|b_{k}^{+}| +|b_{k}^{-}|) N_{\text{e}}^k \int_0^T  \frac{1}{n^d} \sum_{\hat{x}} \big| \mathbb{L}_{n,\alpha}^\gamma G_s \big(\tfrac{\hat{x}}{n}\big) \big|   \; \rmd s \leq 2\epsilon.  
\end{align}
Thus, invoking Proposition \ref{Qabscont}, in order to get \eqref{charact2nonlin}, we are reduced to show that 
\begin{equation} \label{claimQ0}
	\begin{split}
		\mathbb{Q} \Bigg(  \sup_{t \in [0,T]} \Bigg| &  \langle \rho_t, G_t\rangle - \langle  \rho_0, G_0\rangle  - \int_0^t \langle \rho_s, \partial_s G_s\rangle \;\rmd s \\
		- &  \int_0^t \sum_{k=1}^{ m } \langle     b_{k}^+ [ \rho_s  ]^{k} - b_{k}^- [N_{\text{e}} - \rho_s  ]^{k} \big),  \mathbb{L}_{\alpha}^{\gamma}G_s \rangle \;\rmd s \Bigg| > \delta - 3 \epsilon   \Bigg) =0.
	\end{split}
\end{equation}

\textbf{II) Second step: application of Portmanteau's theorem.} In order to simplify the notation, in the remainder of the proof we will assume that $b_{k}^-=0$ for every $k \geq 2$, but we observe that the general case is analogous. We treat the nonlinear terms in \eqref{claimQ0} by applying the approximations of the identity given in Definition \ref{def:mol-iota}. More exactly, combining \eqref{aproxidtil} with Proposition \ref{propL1alpha} and Markov's inequality, we can choose $\varepsilon_0$ small enough such that, for any $\varepsilon \in (0, \varepsilon_0]$,
\begin{align*} 
	& \mathbb{Q} \Bigg( \sum_{k=2}^{ m }  |b_{k}^+| \int_0^T \int_{\mathbb{R}^d} | \mathbb{L}_{\alpha}^{\gamma}G_s(\hat{u})|    \Bigg| [\rho_s(\hat{u})]^k - \frac{1}{d} \sum_{j=1}^d \prod_{i=0}^{k-1} \Big\langle  \pi_s, \overrightarrow{\widetilde{\iota}_{\varepsilon}^{ \hat{u} + i \varepsilon \hat{e}_j } } \Big\rangle \Bigg| \; \rmd \hat{u} \; \rmd s   \geq \epsilon \Bigg) =0
	.
\end{align*}
Thus, in order to get \eqref{claimQ0}, it is enough to prove that
\begin{equation} \label{claimQ0a}
	\begin{split}
		 \mathbb{Q} \Bigg(  \sup_{t \in [0,T]} \Bigg|   &\langle \rho_t, \; G_t\rangle - \langle  \rho_0, \; G_0\rangle  - \int_0^t \langle \rho_s, \; \partial_s G_s  + (b_1^{+} + b_1^{-} ) \mathbb{L}_{\alpha}^{\gamma} G_s \rangle \;\rmd s \\
		- & \frac{1}{d} \sum_{k=2}^{ m } \int_0^t  \int_{\mathbb{R}^d} \mathbb{L}_{\alpha}^{\gamma} G(s, \hat{u})   b_{k}^{+}   \sum_{j=1}^d \prod_{i=0}^{k-1} \Big\langle  \rho_s, \; \overrightarrow{\widetilde{\iota}_{\varepsilon_0}^{ \hat{u} + i \varepsilon_0 \hat{e}_j } } \Big\rangle \;  \rmd \hat{u} \; \rmd s  \Bigg| >  \delta - 4 \epsilon   \Bigg) =0.
	\end{split}
\end{equation}
Note from \eqref{aproxtest} that there exist $\widetilde{H}, H^1, H^2, \ldots H^m \in C_c^{0,0}([0,T] \times \mathbb{R}^d)$ such that
\begin{equation} \label{chotilGH}
	\begin{split}
		\int_0^T \Bigg\{ & N_{\text{e}} \big[\sup_{s \in [0, \; T]}  \| G_s  - \widetilde{H}_s \|_1 + \| \partial_s G_s +  (b_1^+ + b_1^-)  \mathbb{L}_{\alpha}^{\gamma} G_s - H^1_s \|_1 \; \rmd s 
		  \big] \\
		  + & \sum_{k=2}^{m} N_{\text{e}}^k  \| b_k^+   \mathbb{L}_{\alpha}^{\gamma} G_s - H^k_s \|_1 \Bigg\} \; \rmd s \leq \epsilon.
	\end{split}
\end{equation}
Thus, from Proposition \ref{Qabscont}, Lemma \ref{openport4} and Portmanteau's theorem, the last probability is bounded from above by
	\begin{align*}
		\varlimsup_{n \rightarrow \infty} \mathbb{P}_{\mu_n} \Bigg(  \sup_{t \in [0,T]} \Bigg| &  \langle \pi^n_t, \; \widetilde{H}_t\rangle - \langle  \pi^n_0, \; \widetilde{H}_0 \rangle   - \int_0^t  \langle \pi^n_s,  \; H^1_s \rangle \; ds \\
		- & \frac{1}{d} \sum_{k=2}^{m} \int_0^t  \int_{\mathbb{R}^d}  H^k_s( \hat{u})    \sum_{j=1}^d \prod_{i=0}^{k-1} \Big\langle  \pi^n_s, \; \overrightarrow{\tilde{\iota}_{\varepsilon_0}^{ \hat{u} + i \varepsilon_0 \hat{e}_j } } \Big\rangle    \;\rmd \hat{u} \; \rmd s \Bigg| > \delta - 5 \epsilon  \Bigg).
	\end{align*}
\begin{equation} \label{claimQ0b}
	\begin{split}
 \leq		 \varlimsup_{\varepsilon \rightarrow 0^+} \varlimsup_{n \rightarrow \infty} \mathbb{P}_{\mu_n} \Bigg(  \sup_{t \in [0,T]} \Bigg| &  \langle \pi_t^n, G_t\rangle - \langle  \pi_0^n, G_0\rangle  - \int_0^t \langle \pi_s^n, \partial_s G_s\rangle \; \rmd s  - (b_1^{+} + b_1^{-} ) \int_0^t \langle \pi_s^n,   \mathbb{L}_{\alpha}^{\gamma} G_s \rangle \; \rmd s \\
		- & \frac{1}{d} \sum_{k=2}^{ m } \int_0^t  \int_{\mathbb{R}^d} \mathbb{L}_{\alpha}^{\gamma} G(s, \hat{u})  b_{k}^{+}   \sum_{j=1}^d \prod_{i=0}^{k-1} \Big\langle  \pi_s^n, \overrightarrow{\iota_{\varepsilon}^{ \hat{u} + i \varepsilon \hat{e}_j } } \Big\rangle    \;\rmd \hat{u} \; \rmd s  \Bigg| > \delta - 7 \epsilon  \Bigg).
	\end{split}
\end{equation}
In the last display we applied \eqref{bndeta}, \eqref{chotilGH} and \eqref{aproxtil}.

\textbf{III) Final step: application of Dynkin's martingale (and Lemma \ref{globrep}, if $F$ is nonlinear).} Recalling the expression of the Dynkin's martingale in \eqref{defMnt} and of its integral term in \eqref{intterm}, we conclude that the display in \eqref{claimQ0b} is bounded from above by 
\begin{align*}
& \varlimsup_{n \rightarrow \infty} \mathbb{P}_{\mu_n} \Big( \sup_{s \in [0,T]} \big| \mcb M_{s}^{n}(G) \big| > \epsilon \Big) + \varlimsup_{n \rightarrow \infty} \mathbb{P}_{\mu_n} \Bigg( \int_0^T \big| \varepsilon_{n, \alpha_n}^{\gamma}(G_s,\eta_s^n)\big| \; \rmd s > \epsilon \Bigg)  \\
+ & \varlimsup_{n \rightarrow \infty} \mathbb{P}_{\mu_n} \Bigg( \int_0^T \big[  \big|  \mcb{A}_{n, \alpha_n}^{\gamma}(G_s,\eta^n_s)\big| + \big| \mcb{F}_{n, \alpha_n}^{\gamma}(G_s,\eta_s^n) \big| \big] \; \rmd s > \epsilon \Bigg) \\
+  & \varlimsup_{\varepsilon \rightarrow 0^+} \varlimsup_{n \rightarrow \infty} \mathbb{P}_{\mu_n} \Bigg(  \sup_{t \in [0,T]} \Bigg|  \int_0^t  \mcb{F}_{n, \alpha}^{\gamma}(G_s,\eta_s^n)  \; \rmd s  - (b_1^{+} + b_1^{-} ) \int_0^t \langle \pi_s^n,   \mathbb{L}_{\alpha}^{\gamma} G_s \rangle \; \rmd s \\
		- & \frac{1}{d} \sum_{k=2}^{ m } \int_0^t  \int_{\mathbb{R}^d} \mathbb{L}_{\alpha}^{\gamma} G(s, \hat{u})  b_{k}^{+}   \sum_{j=1}^d \prod_{i=0}^{k-1} \Big\langle  \pi_s^n, \overrightarrow{\iota_{\varepsilon}^{ \hat{u} + i \varepsilon \hat{e}_j } } \Big\rangle    \;\rmd \hat{u} \; \rmd s  \Bigg| > \delta - 10 \epsilon  \Bigg).
\end{align*}
From \eqref{condger2pr3}, the leftmost term in the first line of the last display is equal to zero. Moreover, the rightmost term in the first line is also equal to zero, by combining Markov's inequality with Proposition \ref{boundYnsep}.  
 
From the assumptions of the current proposition, Hypothesis \ref{hipoerror} is satisfied; and moreover, one of the conditions \eqref{condalppos}, \eqref{condneu}, \eqref{condlip} holds. Combining this with Markov's inequality and Remark \ref{remhiperr}, we conclude that the term in the second line of last display is also equal to zero.

It remains to treat the double $\varlimsup$ in the last display. By combining \eqref{intnlin5} and Definition \ref{def:FA}, we observe that this term is bounded from above by
\begin{align*}
& \varlimsup_{n \rightarrow \infty} \mathbb{P}_{\mu_n} \Bigg(   \int_0^T \Bigg| \sum_{k= m +1}^{\ell_n} \frac{b_k^{+}}{d n^d} \sum_{j=1}^d  \sum_{\hat{x}}\mathbb{L}_{n,\alpha}^\gamma G_s \big(\tfrac{\hat{x}}{n}\big)    P^{(k),j}(\tau^{\hat{x}}\eta_s^n)   \Bigg| \;\rmd s >  2 \epsilon  \Bigg) \\
+ & \varlimsup_{n \rightarrow \infty} \mathbb{P}_{\mu_n} \Bigg(  \int_0^T   \sum_{k=1}^{ m } \frac{|b_k^{+}|}{d n^d} \sum_{j=1}^d  \sum_{\hat{x}} \big|\mathbb{L}_{n,\alpha}^\gamma G_s \big(\tfrac{\hat{x}}{n}\big) - \mathbb{L}_{\alpha}^{\gamma} G_s \big( \tfrac{\hat{x}}{n} \big) \big| \;    P^{(k),j}(\tau^{\hat{x}}\eta_s^n)    \;\rmd s >  \epsilon  \Bigg) \\
+ & \varlimsup_{\varepsilon \rightarrow 0^+} \varlimsup_{n \rightarrow \infty} \mathbb{P}_{\mu_n} \Bigg( \sup_{t \in [0,T]} \Bigg| \int_0^t   \sum_{k=2}^{m} \sum_{j=1}^d \frac{b_k^{+}}{d n^d} \sum_{\hat{x}} P^{(k),j}(\tau^{\hat{x}}\eta_s^n)\mathbb{L}_{\alpha}^{\gamma} G_s \big( \tfrac{\hat{x}}{n} \big) \;\rmd s \\
- & \int_0^t   \sum_{k=2}^{m} \sum_{j=1}^d    \frac{b_k^{+}}{d} \int_{\mathbb{R}^d}  \prod_{i=0}^{k-1} \Big\langle  \pi_s^n, \overrightarrow{\iota_{\varepsilon}^{ \hat{u} + i \varepsilon \hat{e}_j } } \Big\rangle \mathbb{L}_{\alpha}^{\gamma} G(s, \hat{u}) \; \rmd \hat{u} \; \rmd s \Bigg| >  \delta - 13 \epsilon   \Bigg),
\end{align*}
with $P^{(k),j}$ as in \eqref{pure}. In the last display, since $\lim_{n \rightarrow \infty} \ell_n= \infty$, we are assuming without loss of generality that $\ell_n \geq m$.  

Due to \eqref{truncKn}, the $\varlimsup$ in the first line of the last display is equal to zero. The term in the second line is also equal to zero, by combining Markov's inequality with \eqref{pure}, \eqref{bndeta}, \eqref{convabsF} and  Proposition \ref{prop:conv_frac-lap}.

It remains to treat the double $\varlimsup$ in the last display. In the particular case where $F$ is linear, this term is trivially equal to zero and the proof ends. Therefore, in the remainder of the proof we will assume that $F$ is nonlinear. From the assumptions of the current proposition, Hypothesis \ref{hiporepl} holds and we are allowed to apply Lemma \ref{globrep}. The aforementioned term is bounded from above by
\begin{align*}
&\varlimsup_{\varepsilon \rightarrow 0^+} \varlimsup_{n \rightarrow \infty} \mathbb{P}_{\mu_n} \Bigg( \int_0^T \Bigg|  \sum_{k=2}^{m} \sum_{j=1}^d \frac{b_k^{+}}{d n^d} \sum_{\hat{x}}  \mathbb{L}_{\alpha}^{\gamma} G_s \big( \tfrac{\hat{x}}{n} \big) \prod_{i=0}^{k-1} \Big\langle  \pi_s^n, \overrightarrow{\iota_{\varepsilon}^{ \hat{x}/n + i \varepsilon \hat{e}_j } } \Big\rangle \\
- &\sum_{k=2}^{m} \sum_{j=1}^d \frac{b_k^{+}}{d}\int_{\mathbb{R}^d}  \mathbb{L}_{\alpha}^{\gamma} G(s, \hat{u}) \prod_{i=0}^{k-1} \Big\langle  \pi_s^n, \overrightarrow{\iota_{\varepsilon}^{ \hat{u} + i \varepsilon \hat{e}_j } } \Big\rangle  \Bigg| \; \rmd \hat{u} \; \rmd s > \epsilon   \Bigg) \\
+ &  \varlimsup_{\varepsilon \rightarrow 0^+} \varlimsup_{n \rightarrow \infty} \mathbb{P}_{\mu_n} \Bigg(   \int_0^T  \sum_{j=1}^d \sum_{k=2}^{m}   \frac{ b_k^{+} N^k_{\text{e} }}{2d n^d} \sum_{\hat{x}} \Big| \mathbb{L}_{\alpha}^{\gamma} G_s \big( \tfrac{\hat{x} + (k-1) \hat{e}_j }{n} \big) - \mathbb{L}_{\alpha}^{\gamma} G_s \big( \tfrac{\hat{x}}{n} \big) \Big| \;  \rmd s  > \epsilon \Bigg) \\
+ & \varlimsup_{\varepsilon \rightarrow 0^+} \varlimsup_{n \rightarrow \infty} \mathbb{P}_{\mu_n} \Bigg(\sup_{t \in [0,T]} \Bigg| \int_0^t   \sum_{k=2}^{m} \sum_{j=1}^d  \frac{ b_k^{+}}{d n^d} \sum_{\hat{x}}  \mathbb{L}_{\alpha}^{\gamma} G_s \big( \tfrac{\hat{x}}{n} \big)  \Bigg[ \prod_{i=0}^{k-1}  \eta_s^n(\hat{x}+ i \hat{e}_j) - \prod_{i=0}^{k-1} \overrightarrow{\eta}_{s}^{\varepsilon n} \big(\hat{x}+  i \varepsilon n \hat{e}_j) \Bigg] \;   \rmd s \Bigg| >  \epsilon  \Bigg).
\end{align*}
In the last display we applied the fact that $\epsilon=\delta/16$, besides \eqref{medempright}, \eqref{pure} and the change of variables $\hat{x} \mapsto \hat{z}+(k-1) \hat{e}_j$ in order to obtain the terms in the last two lines.

From Proposition \ref{propL1alpha} and \eqref{aproxdiscnl}, the first double $\varlimsup$ is equal to zero. The second one is also zero, due to Markov's inequality, Proposition \ref{propL1alpha} and an approximation of $\mathbb{L}_{\alpha}^{\gamma} G$ in $L^1([0,T] \times \mathbb{R}^d)$ by any sequence contained in $C_c^{0,1}([0,T] \times \mathbb{R}^d)$. Finally, by combining Proposition \ref{propL1alpha} and \eqref{convabsF} with Proposition 4.1 in \cite{gabriel} and Lemma \ref{globrep},we conclude that the final term is equal to zero and the proof ends.
\end{proof}

\section{Proof of the Replacement Lemma \ref{globrep}} \label{replem}    

The goal of this section is to prove Lemma \ref{globrep}. In order to do so, we \textit{replace} products of $\xi$'s by products of empirical averages over boxes of side length $\varepsilon n$, by performing two steps;
\begin{itemize}
\item
First, we replace products of $\xi$'s, by products of empirical averages over boxes of side length $\ell$. This step is known in the literature as the \textit{one-block} estimate and it will be obtained by performing only nearest-neighbor jumps;
\item
Afterwards, we increase the side length of our boxes from $\ell$ to $\varepsilon n$. This means that products of empirical averages over smaller boxes are replaced by products of the corresponding averages, over bigger ones. This step is known in the literature as the \textit{two-blocks} estimate.
\end{itemize}
In a more precise way, the first and second steps described above correspond in this work to \eqref{lemglobobe} and \eqref{lemglobtbe} below. In opposition to what was done in previous works (see \cite{renato, CG} for instance) we will \textit{not} choose $\ell$ as a function of $\varepsilon$ and $n$. Rather, it is enough to require that $\ell$ is a fixed natural number (which can be arbitrarily large), independent of $\varepsilon$ and $n$. This choice simplifies the arguments of the proof for \eqref{lemglobotbe}. 

Before stating the one-block and two-block estimates, we present a preliminary lemma from which we can avoid issues due to the presence of slow bonds (which were not present in \cite{renato, CG}). Thanks to Lemma \ref{globrlpos0} below, the contribution of a neighborhood of the boundary $\mathbb{B}=\mathbb{Z}^{d-1} \times \{0\}$ is negligible and we can focus only on fast bonds. Keeping the notation of Lemma \ref{globrep} in mind, in the remainder of this section we have either $\xi=\eta$; or $\xi=\tilde{\eta}$.

In what follows, for every $p\in\mathbb{N}_+$, let $\mathbb{A}_p^-$ and $\mathbb{A}_p^+$ be given by
\begin{align}
\label{defAkepsn}
\mathbb{A}^{+}_{p}
:= 
\{ (\hat{x}_{\star}, x_d) \in \mathbb{Z}^d: x_d \geq p\}
\quad\text{and}\quad
\mathbb{A}^{-}_{p}:= \{ (\hat{x}_{\star}, x_d) \in \mathbb{Z}^d: x_d \leq - p\}.
\end{align}
\begin{lem} \label{globrlpos0}
Let $\widetilde{G}: [0,T] \times \mathbb{R}^d  \mapsto \mathbb{R}$ be such that \eqref{boundrep} holds. Then, for any $k \geq 2$, $j \in \{1, \ldots, d\}$, and any $t \in [0,T]$, it holds
 \begin{equation*} 
 \varlimsup_{\varepsilon \rightarrow 0^{+}}\varlimsup_{n \rightarrow \infty} \mathbb{E}_{\mu_n} \Bigg[  \Bigg| \int_{0}^{t}  \sum_{\hat{x} \in  ( \mathbb{A}_{3k\varepsilon n}^- \cup \mathbb{A}_{3k\varepsilon n}^+ )^c }  \frac{ \widetilde{G}_s \big( \tfrac{\hat{x}}{n} \big)  }{n^d} \Bigg\{ \prod_{i=0}^{k-1}  \xi_s^n(\hat{x}+ i \hat{e}_j) - \prod_{i=0}^{k-1} \overrightarrow{\xi}_{s}^{\varepsilon n} \big(\hat{x}+  i \varepsilon n \hat{e}_j) \Bigg \} \;\rmd s \, \Bigg| \Bigg] =0.
 \end{equation*}
\end{lem}
\begin{proof}
The proof is immediate by combining \eqref{boundrep} and \eqref{bndeta} with the following observation: for every $H \in L^1(\mathbb{R}^d)$, it holds
\begin{equation} \label{L1near0}
\varlimsup_{\varepsilon \rightarrow 0^+} \varlimsup_{n \rightarrow \infty} \frac{1}{n^d} \sum_{\hat{x} \in  ( \mathbb{A}_{3k\varepsilon n}^- \cup \mathbb{A}_{3k\varepsilon n}^+ )^c } \big| H\big( \tfrac{\hat{x}}{n} \big) \big| 
\lesssim
\lim_{\varepsilon\to0^+} \int_{-\varepsilon}^{\varepsilon}
 \Bigg\{ \int_{\mathbb{R}^{d-1}}
\abs{H(\hat{u}_\star,u_d)} \; \rmd \hat{u}_\star \Bigg\}  \; \rmd u_d
=0.
\end{equation}
\end{proof}
From Lemma \ref{globrlpos0}, one can prove Lemma \ref{globrep} by showing, for $ \iota \in \{-, \; +\}$, that
\begin{align}\label{globrlpos1}
\varlimsup_{\varepsilon \rightarrow 0^{+}}\varlimsup_{n \rightarrow \infty} 
\mathbb{E}_{\mu_n} 
\Bigg[
    \Bigg| 
    \int_{0}^{t}  \sum_{\hat{x} \in \mathbb{A}_{3k\varepsilon n}^\iota}  \frac{ \widetilde{G}_s \big( \tfrac{\hat{x}}{n} \big)  }{n^d} \Bigg\{ \prod_{i=0}^{k-1}  \xi_s^n(\hat{x}+ i \hat{e}_j) 
    - \prod_{i=0}^{k-1} \overrightarrow{\xi}_{s}^{\varepsilon n} \big(\hat{x}+  i \varepsilon n \hat{e}_j) \Bigg \} \;
    \rmd s \, \Bigg| 
\Bigg] =0.
\end{align}
We present only the proof for $\iota=+$, but we observe that the proof for $\iota=-$ is analogous. First we observe that if $\widetilde{G}: [0,T] \times \mathbb{R}^d  \mapsto \mathbb{R}$ is such that \eqref{boundrep} holds, then $\widetilde{G}$ is bounded from above by some $H \in L^1(\mathbb{R}^d)$ such that, for any $n \geq 1$ and any $\varepsilon >0$, it holds
	\begin{align} \label{boundKM} 
		\frac{1}{n^d} \sum_{\hat{x} \in \mathbb{A}_{3k\varepsilon n}^{+}} 
		\big|H \big( \tfrac{\hat{x}}{n} \big)\big|
		\leq M; \quad
		\forall   \hat{x} \in \mathbb{A}_{3k\varepsilon n}^{+}, \;
		\big|H \big( \tfrac{\hat{x}}{n} \big) \big|  
		\leq K  + K \Bigg( \frac{|x_d|}{n} \Bigg)^{-\delta}
		\leq  K[1+(3k \varepsilon)^{-\delta}]
		,
	\end{align}
	for some finite constant $M,K>0$ and $\delta \in (0, \gamma)$. 
	
	Next, we discuss what we get when we require $F$ to satisfy \eqref{h3} in Lemma \ref{globrep}. We will make use of Hypothesis \ref{hyp:coeff} \textbf{(ii)} in order to be able to perform jumps between $\hat{x}$ and $\hat{x}+r \hat{e}_m$ for any $m \in \{1, \ldots, d\}$, when $|r| \geq  2$. With this in mind, we introduce the next definition.
\begin{definition}[Definition of $k^{\star}$ and $b^{\star}$] \label{defkstar}
From Hypothesis \ref{hyp:coeff} \textbf{(ii)}, exactly one of the two next conditions must hold:
\begin{enumerate}
\item
$\underline{b}^+ >0$ and $\underline{b}^- \geq 0$. In this case, we define $k^{\star}:=\underline{k}^+$ and $b^{\star}:=\underline{b}^+$. In particular, from  \eqref{h3} and \eqref{ratcub}, for every $n \in \mathbb{N}_+$, $\eta\in \Omega$ and $\hat{x}, \hat{y} \in \mathbb{Z}^d$, for every $j \in \{1, \ldots, d\}$ it holds
\begin{align*}
		 	c_{\hat{x},\hat{y}}^{n}(\eta)\geq b^{\star} c_{\hat{x},\hat{y}}^{(k^{\star})}(\eta) \geq \frac{b^{\star}}{2d} \Bigg\{ \mathbbm{1}_{ \{ k^{\star} > 1 \} } \prod_{i=1}^{k^\star -1} \eta( \hat{x} + i \hat{e}_j ) + \mathbbm{1}_{ \{ k^{\star} = 1 \} } \Bigg\}
			;
\end{align*}
\item
$\underline{b}^+ =0$ and $\underline{b}^- >0$. In this case, we define $k^{\star}:=\underline{k}^-$ and $b^{\star}:=\underline{b}^-$. In particular, from \eqref{h3} and \eqref{ratcub}, for every $n \in \mathbb{N}_+$, $\eta\in \Omega$ and $\hat{x}, \hat{y} \in \mathbb{Z}^d$, for every $j \in \{1, \ldots, d\}$ it holds
\begin{align*}
			c_{\hat{y},\hat{y}}^{n}(\eta)\geq 			b^{\star} c_{\hat{x},\hat{y}}^{(k^{\star})}(\widetilde{\eta}) \geq \frac{b^{\star}}{2d}  \Bigg\{ \mathbbm{1}_{ \{ k^{\star} > 1 \} } \prod_{i=1}^{k^\star -1} \widetilde{\eta}( \hat{y} + i \hat{e}_j ) + \mathbbm{1}_{ \{ k^{\star} = 1 \} } \Bigg\}.
\end{align*}
\end{enumerate}
\end{definition}
Next, we stress that \eqref{globrlpos1} is a direct consequence of the next result.
\begin{lem} \textbf{(One-block and two-blocks estimates)}  \label{lemglobotbe}
Let $\widetilde{G} : [0,T] \times \mathbb{R}^d  \mapsto \mathbb{R}$ be such that \eqref{boundrep} holds. Moreover, assume that $\ell \in \mathbb{N}_+$ is fixed. Then, under Hypothesis \ref{hiporepl}, for any $k \geq 2$, $j \in \{1, \ldots, d\}$ and any $t \in [0,T]$, we have
\begin{align}
\varlimsup_{\varepsilon \rightarrow 0^{+}}\varlimsup_{n \rightarrow \infty}\mathbb{E}_{\mu_n} \Bigg[ \Bigg| \int_{0}^{t}  \sum_{\hat{x} \in \mathbb{A}_{3k\varepsilon n}^{+}}  \frac{ \widetilde{G}_s \big( \tfrac{\hat{x}}{n} \big)  }{n^d} \Bigg\{ \prod_{i=0}^{k-1} \xi_s^{n} ( \hat{x} + i  \hat{e}_j) - \prod_{i=0}^{k-1} \overrightarrow{\xi}_s^{\ell} ( \hat{x} + i \ell \hat{e}_j) \Bigg\} \,  \; \rmd s \, \Bigg| \Bigg] =0,  \label{lemglobobe} \\
\varlimsup_{\varepsilon \rightarrow 0^{+}}\varlimsup_{n \rightarrow \infty}\mathbb{E}_{\mu_n} \Bigg[ \Bigg| \int_{0}^{t}  \sum_{\hat{x} \in \mathbb{A}_{3k\varepsilon n}^{+}} \frac{ \widetilde{G}_s \big( \tfrac{\hat{x}}{n} \big)  }{n^d}  \Bigg\{ \prod_{i=0}^{k-1} \overrightarrow{\xi}_s^{\ell} ( \hat{x} + i \ell \hat{e}_j) - \prod_{i=0}^{k-1} \overrightarrow{\xi}_s^{\varepsilon n} ( \hat{x} + i \varepsilon n \hat{e}_j) \Bigg\} \,  \; \rmd s \, \Bigg| \Bigg]  \leq \frac{f_{\star}(k^{\star})}{\ell^d},   \label{lemglobtbe}
\end{align}
where $f_{\star}: \mathbb{N} \mapsto [0, \infty)$ is given by
\begin{equation} \label{deffstar}
f_{\star}(m):= T k  \frac{8 M k^\star d N^k_{ \text{e} } }{a_h(1-b_h)} \mathbbm{1}_{ \{m > 1 \} }, \quad m \in \mathbb{N}.
\end{equation}
In the last display, $M$ and $k^{\star}$ are given by \eqref{boundKM} and Definition \ref{defkstar}, respectively. Moreover, $a_h$ and $b_h$ are given in Definition \ref{Ref}, for some $h \in Ref$ that arises from \eqref{1entbound}. 
\end{lem}
From Lemma \ref{lemglobotbe}, there exists a constant $C_0>0$ such that the double limit in \eqref{globrlpos1} is bounded from above by $C_0 \ell^{-d}$, for any $\ell \in \mathbb{N}_+$. Since $\ell$ is arbitrary, we conclude that \eqref{globrlpos1} holds, and the proof of Lemma \ref{globrep} ends.

In order to prove Lemma \ref{lemglobotbe}, we will make use of \eqref{cons-series} to perform nearest-neighbor jumps through fast bonds. More exactly, from \eqref{cons-series} and Definition \ref{def:dir_form}, we get, for every measure $\nu$ on $\Omega$ and $f: \Omega \mapsto \mathbb{R}$, that
\begin{align} \label{defDSEP}
\mcb D_{n,\alpha}^\gamma ( f | \nu ) \geq \mcb D_{ \text{SEP} }^\gamma ( f | \nu ) := p_\gamma( \hat{e}_1   ) \sum_{ \{ \hat{x} , \hat{y} \} \in \mcb F} I^{\text{SEP}}_{\hat{x},\hat{y}}  (f | \nu ),
\end{align}
where for every $\{ \hat{x} , \hat{y} \} \in \mcb F$, $I^{\text{SEP}}_{\hat{x},\hat{y}}   (f | \nu )$ is given by  
\begin{equation} \label{defISEP}
  I^{\text{SEP}}_{\hat{x},\hat{y}}  (f | \nu ) 
	:= \frac{1}{2 N_{\text{e}}} \int_{\Omega} \big\{ a_{\hat{x}, \hat{y}}(\eta)   \big[\nabla_{\hat{x}, \hat{y} } f(\eta) \big]^2 + a_{\hat{y}, \hat{x}}(\eta)   \big[\nabla_{\hat{y}, \hat{x} } f(\eta)\big]^2  \big\} d \nu.
\end{equation}
In the last line, $\nabla_{\cdot , \cdot } f(\eta)$ is given in Definition \ref{def:micro-op}. We present the proof for \eqref{lemglobobe} and \eqref{lemglobtbe} in Subsections \ref{secglobobe} and \ref{secglobtbe} below, respectively. 

We now make some comments on the different entropy bounds in Hypothesis \ref{hiporepl}. If only the weaker bound \eqref{1entbound} holds, we need to make use of Hypothesis \ref{hiporepl} \textbf{(ii)}, which requires $F$ to satisfy \eqref{h2}. This is necessary in order to obtain an upper bound for $f_{n}'$ that is independent of $n$, which can be combined with Proposition \ref{bound} for obtaining Lemma \ref{lemglobotbe}. We stress that this is the only step of the proof for \eqref{lemglobobe} and \eqref{lemglobtbe} where we need \eqref{h2}.

In the case where the stronger entropy bound \eqref{2entbound} holds, we are allowed to make use of Hypothesis \ref{hiporepl} \textbf{(i)}, which allows us to drop out the condition \eqref{h2}. This is possible because Proposition \ref{bound} can be replaced by \eqref{boundconst}, since product measures with respect to a constant profile are admissible.

In this section we will focus on the setting where \eqref{1entbound} holds, instead of \eqref{2entbound}. We do so because profiles $h \in Ref$ are more general than the constant ones (see Definition \ref{Ref}). This is why we require Hypothesis
\ref{hiporepl} \textbf{(ii)} in Subsections \ref{secglobobe} and \ref{secglobtbe}. We leave the case where Hypothesis \ref{hiporepl} \textbf{(i)} holds for the reader, but we observe that the arguments are totally analogous and, in fact, simpler.

Keeping Definition \ref{Ref} in mind,  we state a technical result that will be important later on. \begin{lem} \label{lem:anti-dir}
   Let $h \in Ref$ as given in Definition \ref{Ref}, $\hat{x}, \hat{y} \in \mathbb{Z}^d$ and $f$ a density with respect to $\nu_h^n$. Moreover, let $M_1 >0$ and $B: \Omega \mapsto [0, \; M_1]$ be such that $B(\eta_{\star})$ is independent of $\eta_{\star}( \hat{x} )$ and $\eta_{\star}( \hat{y} )$, for any $\eta_{\star} \in \Omega$. Then, for every $n \in \mathbb{N}_+$, 
\begin{align*}
&\frac{ 2  }{M_1} \Bigg| \int_{\Omega} B ( \eta ) [\eta( \hat{y} ) - \eta( \hat{x} ) ] f(\eta) \rmd \nu_h^n(\eta) \Bigg| \\
 \leq & \frac{1}{ N_{ \text{e} } } \int_{\Omega} \big[ a_{ \hat{x}, \hat{y} } (\eta) | \nabla_{\hat{x}, \hat{y} } f(\eta) | + a_{ \hat{y}, \hat{x} } (\eta) | \nabla_{\hat{y}, \hat{x} } f(\eta) | \big] \rmd \nu_h^n(\eta) +   \frac{ 2 N_{ \text{e} } L_h}{a_h ( 1 - b_h )} \frac{|\hat{x} - \hat{y} |}{n}.
 \end{align*} 
In particular, if $|\hat{y} - \hat{x}|=1$ and $\{\hat{x}, \hat{y} \} \in \mcb F$, for any $A >0$,
\begin{align*}
&\frac{ 2  }{M_1} \Bigg| \int_{\Omega} B ( \xi ) [\xi( \hat{y} ) - \xi( \hat{x} ) ]  f(\eta) \rmd \nu_h^n(\eta) \Bigg| \leq   \frac{I^{\text{SEP}}_{\hat{x},\hat{y}}  (\sqrt{f} | \nu_h^n )}{A} + \frac{  2 N_{ \text{e} } }{a_h ( 1 - b_h )} \Big( A + \frac{L_h}{n}   \Big).
\end{align*}  
\end{lem}
\begin{proof}
For every $\eta \in \Omega$, we have that $|\eta( \hat{y} ) - \eta( \hat{x} ) = | \widetilde{\eta} ( \hat{y} )- \widetilde{\eta} ( \hat{x} ) | = |\xi( \hat{y} ) - \xi( \hat{x} )|$, since $\widetilde{\eta}(\hat{z})= N_{ \text{e} } - \eta(\hat{z})$ for every $\hat{z} \in \mathbb{Z}^d$. Thus, in order to get the first upper bound, it is enough to combine the arguments in the proof of Lemma 4.3 in \cite{beatriz} with Definition \ref{Ref}. The second one holds by plugging the first one with Remark \ref{remyoung}.
\end{proof}

Keeping (15) in \cite{beatriz} in mind, we get from \eqref{defberhn} that for every $\eta \in \Omega$ and every $\hat{x}, \hat{y} \in \mathbb{Z}^d$ such that $a_{\hat{x}, \hat{y}}(\eta)>0$, it holds
\begin{align} \label{invexchzw}
\forall h \in Ref, \; \forall n \in \mathbb{N}_+, \quad	a_{\hat{x}, \hat{y}}(\eta) \nu_h^n(\eta)= q_{\hat{x}, \hat{y}}^n(h) a_{\hat{y}, \hat{x}}(\eta^{\hat{x}, \hat{y}}) \nu_h^n(\eta^{\hat{x}, \hat{y}}),  
\end{align}
where $q_{\hat{x}, \hat{y}}^n(h)$ is given by
\begin{align*}
	 q_{\hat{x}, \hat{y}}^n(h):= \frac{ h \big( \tfrac{ \hat{x} }{n}  \big) \big[ 1 - h\big( \tfrac{ \hat{y} }{n}  \big) \big]  }{ h\big( \tfrac{ \hat{y} }{n}  \big) \big[ 1 - h\big( \tfrac{ \hat{x} }{n}  \big) \big] }.
\end{align*}
Now we are ready to state a result which is analogous to \eqref{boundconst} for non-constant profiles. 
\begin{prop} \label{bound}
	Let $h \in Ref$, $n \in \mathbb{N}_+$ and $f$ a density with respect to $\nu_{h}^n$. Then there exists a constant $M_h >0$ depending only on $h$ and $\alpha$ such that 
	\begin{align}  \label{boundlip}
		\langle \mcb L_{n,\alpha}^\gamma \sqrt{f} , \sqrt{f} \rangle_{\nu_h^n}
		\leq 
		- \frac18\mcb D_{n,\alpha}^\gamma (\sqrt{f}| \nu_h^n ) + \frac{n^d}{n^\gamma}M_h  (1 + f_{n}'), 
	\end{align}
\end{prop}
where $f_{n}'$ is given by \eqref{f-f'}.
\begin{proof}
From \eqref{invexchzw} and \eqref{cons-series}, for every $\hat{z} \neq \hat{w} \in \mathbb{Z}^d$, it holds
\begin{align*}
	 & \int_{\Omega} c_{\hat{z}, \hat{w}}^{n}(\eta)   a_{\hat{z},\hat{w}}(\eta) \big[- {f}(\eta)+   \sqrt{f}(\eta) \sqrt{f}(\eta^{ \hat{z}, \hat{w} }) \big] \; \rmd \nu_{h}^n \\
	= & q^n_{\hat{z}, \hat{w}}(h) \int_{\Omega} c_{\hat{w}, \hat{z}}^{n} (\eta) a_{\hat{w},\hat{z}}(\eta) \big[- {f}  (\eta^{ \hat{w}, \hat{z} })+   \sqrt{f} (\eta^{ \hat{w}, \hat{z} }) \sqrt{f}(\eta) \big] \;  \rmd \nu_{h}^n.
\end{align*}
Combining the symmetry of $p_{\gamma}(\cdot)$, $r^n_{\cdot,\cdot}$ and $c^n_{\cdot,\cdot}$ with the last display, by applying arguments analogous to the ones described in the proof of Lemma 4.1 in \cite{beatriz}, the leftmost term in \eqref{boundlip} is bounded from above by
\begin{align} \label{posyounglip}
- & \frac18\mcb D_{n,\alpha}^\gamma (\sqrt{f}| \nu_h^n )  + \sum_{\hat{x}, \hat{y} }  \frac{ p_{\gamma} ( \hat{y} - \hat{x} )\alpha_{\hat{x},\hat{y}}^n}{16 N_{ \text{e} }} \int_{\Omega} c_{\hat{x}, \hat{y}}^{n}(\eta) \big[   G_{\hat{x},\hat{y}}^n (h,f, \eta )  +     G_{\hat{y},\hat{x}}^n (h, f, \eta ) \big]\; \rmd \nu_{h}^n. 
\end{align}
In the last line, for every $\hat{z} \neq \hat{w} \in \mathbb{Z}^d$, $G_{\hat{z},\hat{w}}^n (h, f, \eta )$ is given by
\begin{align*}
	G_{\hat{z},\hat{w}}^n (h, f, \eta ):= & a_{\hat{z},\hat{w}}(\eta) \big[\sqrt{f (\eta^{ \hat{z}, \hat{w} })}  + \sqrt{f(\eta)}  \big]^{2} [1 - q^n_{\hat{w}, \hat{z}}(h) ]^{2}.
\end{align*}
Combining \eqref{excrule}, Definition \ref{Ref} and the fact that $f$ is a density with respect to $\nu_{h}^n$, with \eqref{h2} and arguments analogous to the ones applied in the proof of Lemma C.1 in \cite{ddimhydlim}, the sum over $\hat{x}, \hat{y}$ in \eqref{posyounglip} is bounded from above by
\begin{align*}
 \frac{n^d}{n^{\gamma}} \widetilde{M}_h (1+ 2\alpha)(1 + f_n') \Bigg( \frac{1}{n^d} \sum_{|\hat{x}| < 2 R_h n  } \int_{| \hat{u} | \geq R_h} | \hat{u} |^{-\gamma - d} \rmd \hat{u} + \frac{1}{n^d} \sum_{|\hat{x}| < 2 R_h n  } \int_{| \hat{u} | \leq 5 R_h} | \hat{u} |^{-\gamma - d +2} \rmd \hat{u} \Bigg),
\end{align*}
where $R_h$ is given in Definition \ref{Ref} and $\widetilde{M}_h$ is a positive constant, depending only on $h$. In the last display we assumed without loss of generality that $\alpha_n \leq 2 \alpha$. The term $1 + f_n'$ is an upper bound for $c_{\hat{x}, \hat{y}}^{n}(\eta)$, due to \eqref{cons-series}. Applying \eqref{sphe} below, the proof ends.
\end{proof}

\subsection{One-block estimate} \label{secglobobe}

In this subsection, we prove \eqref{lemglobobe}. In order to do so, by following the reasoning in the beginning of Section 4.4 in \cite{CG}, we observe that for any $k \geq 2$, any $\ell \in \mathbb{N}_+$, any $j \in \{1, \ldots, d\}$ and any $\hat{x} \in \mathbb{Z}^d$, it holds
\begin{align*}
&\prod_{i=0}^{k-1} \xi(\hat{x} + i \hat{e}_j  ) - \prod_{i=0}^{k-1} \overrightarrow{\xi}^{\ell} (\hat{x} + i \ell \hat{e}_j)  
= \sum_{i=1}^{k} B^{k, \ell, \hat{x}}_{i,j}(\xi) \big[\xi \big( \hat{x} + (k-i) \hat{e}_j \big)- \overrightarrow{\xi}^{\ell}\big( \hat{x} + (k-i) \ell \hat{e}_j  \big) \big],
\end{align*}
where for any  $i \in \{1, \ldots, k\}$, $B^{k, \ell,\hat{x}}_{i,j}: \Omega \mapsto [0, \; N_{\text{e}}^{k-1}]$ is given by 
\begin{align} \label{defbxjeta}
B^{k, \ell, \hat{x}}_{i,j}(\eta_{\star}):= \prod_{m=0}^{k-i-1} \eta_{\star}(\hat{x} + m \hat{e}_j )  \prod_{p=k-i+1}^{k-1} \overrightarrow{\eta_{\star}}^{\ell}(\hat{x} + p \ell \hat{e}_j ), \quad \eta_{\star} \in \Omega.
\end{align}
Therefore, we conclude that \eqref{lemglobobe} is a direct consequence of the next result.
\begin{lem} \label{lemrep1}
Let $\widetilde{G} : [0,T] \times \mathbb{R}^d  \mapsto \mathbb{R}$ be such that \eqref{boundrep} holds. Then,  under Hypothesis \ref{hiporepl} \textbf{(ii)}, for any $k \geq 2$, $j \in \{1, \ldots, d\}$, any $t \in [0,T]$  and any $\ell \in \mathbb{N}$, 
\begin{equation*} 
\varlimsup_{\varepsilon \rightarrow 0^{+}}\varlimsup_{n \rightarrow \infty}\mathbb{E}_{\mu_n} \Bigg[ \Bigg| \int_{0}^{t}  \sum_{\hat{x} \in \mathbb{A}_{3k\varepsilon n}^{+}} \frac{ \widetilde{G}_s \big( \tfrac{\hat{x}}{n} \big)  }{n^d} B^{k, \ell, \hat{x}}_{i,j}(  \xi_s^n ) \big[\xi_s^n \big( \hat{x} + (k-i) \hat{e}_j \big)- \overrightarrow{\xi}_s^{\ell}\big( \hat{x} + (k-i) \ell \hat{e}_j  \big) \big] \, \rmd s \, \Bigg| \Bigg] =0.
\end{equation*}
\end{lem}
\begin{proof} 
Combining \eqref{1entbound} with Lemmas \ref{entjen} and \ref{feyn}, it is enough to prove that
\begin{equation} \label{lim2grep}
	\begin{split}
&	\varlimsup_{D \rightarrow \infty}	\varlimsup_{\varepsilon \rightarrow 0^{+}}\varlimsup_{n \rightarrow \infty} \int_0^t \sup_f \Bigg\{ \frac{n^\gamma}{D n^d} \langle \mcb L_{n,\alpha}^\gamma \sqrt{f} , \sqrt{f} \rangle_{\nu_h^n} \\
	+& \sum_{\hat{x} \in \mathbb{A}_{3k\varepsilon n}^{+} }  \frac{    \big|\widetilde{G}_s \big( \tfrac{\hat{x}}{n} \big) \big|  }{n^d} \Bigg| \int_{\Omega}  B^{k, \ell, \hat{x}}_{i,j}( \xi ) \big [\xi \big( \hat{x} + (k-i)  \hat{e}_j \big)- \overrightarrow{\xi}^{\ell}\big( \hat{x} + (k-i) \ell \hat{e}_j \big) \big]  f(\eta) \; \rmd \nu_{h}^n(\eta) \Bigg| \;\Bigg\} \rmd s=0.
	\end{split}
\end{equation}
Above, the supremum is carried over all densities $f$ with respect to $\nu_h^n$. From \eqref{boundrep} and telescopic arguments, the sum inside the supremum in \eqref{lim2grep} is bounded from above by
\begin{align}
&  \sum_{\hat{x} \in \mathbb{A}_{3k\varepsilon n}^{+} } \frac{ H \big( \tfrac{\hat{x}}{n} \big)  }{n^d} \sum_{y=k-i}^{(k-i) \ell - 1}   \Bigg| \int_{\Omega}  B^{k, \ell, \hat{x}}_{i,j}( \xi ) [ \eta( \hat{x} + y \hat{e}_j ) - \eta(\hat{x} +( y+1) \hat{e}_j  ) ]      f(\eta) \rmd \nu_{h}^n (\eta) \Bigg| \label{1tersupa01} \\
+& 	\sum_{\hat{x} \in \mathbb{A}_{3k\varepsilon n}^{+} }  \frac{  H \big( \tfrac{\hat{x}}{n} \big)   }{(\ell n)^d}  \sum_{\hat{\omega} \in \llbracket1, \; \ell \rrbracket^d} \sum_{r=1}^d \sum_{m=0}^{y_r-1}  \Bigg| \int_{\Omega}  B^{k, \ell, \hat{x}}_{i,j}( \xi )  [ \eta (\hat{z} ) - \eta (  \hat{z}+ \hat{e}_r ) ]  f(\eta) \rmd \nu_h^n (\eta) \Bigg|, \label{1tersupb1}
\end{align}
where $H \in L^{1}(\mathbb{R}^d)$ is given by \eqref{boundrep} and 
\begin{align} \label{defz3}
	\hat{z}=\hat{z}( \hat{x}, r, \hat{\omega},  m ):= \hat{x} + (k-i) \ell \hat{e}_j + m \hat{e}_r+ \sum_{y=1}^{r-1} \omega_y \hat{e}_y .
\end{align}
Combining \eqref{defbxjeta} with \eqref{boundKM} and Lemma \ref{lem:anti-dir}, the sum of \eqref{1tersupa01} and \eqref{1tersupb1} is bounded from above by a constant (depending only on $M$, $K$, $\delta$, $\gamma$, $N_{\text{e}}$, $k$ and $h$) times
 \begin{align}
 &  \ell  \Big( A + \frac{1}{n} \Big) + \frac{  \varepsilon^{-\delta}  }{ A  n^{d} } \sum_{y=k-i}^{(k-i) \ell - 1} \sum_{\hat{x} \in \mathbb{A}_{3k\varepsilon n}^{+}}   p (\hat{e}_j)  I_{\hat{x} + y \hat{e}_j, \hat{x} + (y+1) \hat{e}_j}^{SEP}(\sqrt{f} | \nu_h^n)   \label{eq1bl0b} \\
+ &  \frac{ \varepsilon^{-\delta}}{ A \ell^d n^{d} } \sum_{\hat{\omega} \in \llbracket1, \; \ell \rrbracket^d} \sum_{r=1}^d  \sum_{m=0}^{\ell-1} \sum_{\hat{x} \in \mathbb{A}_{3k\varepsilon n}^{+}} p(\hat{e}_r) I_{\hat{z} , \hat{z} +  \hat{e}_r}^{SEP}(\sqrt{f} | \nu_{h}^n) , \label{multsum}
\end{align}
for any $A>0$. In the last display, we assumed without loss of generality that $1 \leq (3k \varepsilon)^{-\delta}$. Before we proceed, we observe that a fixed bond $\{\hat{w}_1, \hat{w}_2\}$ appears \textit{at most} $k \ell$ times in the double sum in \eqref{eq1bl0b}.
Similarly, we observe from \eqref{defz3} that for any \textit{fixed} value for $(r, \hat{\omega}, m)$, any fixed bond of the form $\{\hat{w}_1, \hat{w}_2\}$ is counted \textit{at most once} in \eqref{multsum}.
Finally, we stress that all the bonds counted in the double sum in \eqref{eq1bl0b} and in the multiple sum in \eqref{multsum} are \text{fast}, due to \eqref{defAkepsn}. Thus, the sum of all the terms in \eqref{eq1bl0b} and \eqref{multsum} is bounded from above by
\begin{align} \label{1tersupa2}
	\ell  \Big( A + \frac{1}{n} \Big) + \frac{ \varepsilon^{-\delta} }{A n^{d} } (k+d) \ell \mcb D_{ \text{SEP} }^\gamma (\sqrt{f} | \nu_{h}^n ). 
\end{align}
In the last line we applied \eqref{defDSEP}. Finally, since \eqref{h2} holds, from \eqref{boundlip} and \eqref{defDnFS} we have that the second term inside the supremum in \eqref{lim2grep} is bounded from above by
\begin{align} \label{boundsup2}
\frac{M_h}{D}  (1 + f_{\infty}')  - \frac{n^{\gamma}}{8 D n^d} \mcb D_{n,\alpha}^\gamma  (\sqrt{f} | \nu_{h}^n ) \leq \frac{M_h}{D}  (1 + f_{\infty}') - \frac{n^{\gamma}}{8 D n^d} \mcb D_{ \text{SEP} }^\gamma (\sqrt{f} | \nu_{h}^n ),
\end{align}
where $M_h$ is a constant depending on $h$ and $\alpha$. Combining the last display with \eqref{1tersupa2}, the proof ends by choosing $A=8  (k+d) D K  \varepsilon^{-\delta} \ell  n^{-\gamma}$ and $D=\varepsilon^{-1}$. 
\end{proof}

\subsection{Global Moving Particle Lemma} \label{secglobmpl}

In the proof of \eqref{lemglobotbe}, given an initial state $\eta_i$ we will be interested in reaching some target configuration $\eta_f$. This will be done by moving particles through a series of bonds, constructing a path which starts at $\eta_i$ and ends at $\eta_f$. This motivates the following definition. 

\begin{definition} \label{path}
Let $L \in \mathbb{N}_+$ and $\eta_i, \eta_f \in \Omega$. We denote by $\mapsto_{m=0}^{L-1}\{\hat{x}_m^{0}, \; \hat{x}_m^{1}\}$ a collection of $L$ bonds $\{\{\hat{x}_m^0,\hat{x}_{m}^1\}\}_{ 0 \leq m < L }$ which induces a path $(\gamma_0(\eta_i),\ldots,\gamma_L(\eta_i)) \subset \Omega$ from $\eta_i$ to $\eta_f$. By this, we mean that $\gamma_0(\eta_i)=\eta_i$, $\gamma_L(\eta_i)=\eta_f$ and
\begin{equation} \label{defpath}
\forall m \in \{0, 1, \ldots, L-1\}, \quad a_{\hat{x}_r^0,\hat{x}_{r}^1}(\gamma_r(\eta_i)) \geq 1 \quad \text{and} \quad \gamma_{r+1}(\eta_i)=(\gamma_{r}(\eta_i))^{\hat{x}_r^0,\hat{x}_{r}^1}.
\end{equation}
\end{definition}
In Lemmas \ref{lem:moving} and \ref{lemrep2} below, one quantity of interest is $\nu(\eta_f)/ \nu(\eta_i)$, where $\eta_f$ can be obtained from $\eta_i$ through some path satisfying \eqref{defpath}. This ratio was always equal to one in \cite{renato, CG}, since there $\nu=\nu_{\theta}$ for some $\theta \in (0,1)$ and $N_{ \text{e} }=1$, thus $\nu$ was invariant under exchanges of particles. However this is not the case here. Nevertheless, we can still obtain an uniform upper bound, as it is stated in Lemma \ref{lemradnik} below.
\begin{lem} \label{lemradnik}
Let $L \in \mathbb{N}_+$, $\eta_i \in \Omega$ and a path $(\gamma_0(\eta_i),\ldots,\gamma_L(\eta_i))$ as in Definition \ref{path}. Moreover, let $h \in Ref$, $n \geq 1$ and $L \in \mathbb{N}_{+}$. Then, for every $m \in \{1, \ldots, L\}$, it holds
\begin{align*} 
\frac{\nu_h^n(\eta) }{ \nu_h^n \big( \gamma_{m}(\eta) \big) } = \prod_{r=0}^{m-1} \frac{ h \big( \tfrac{\hat{x}_r^0}{n}  \big) \big[ 1 - h \big( \tfrac{\hat{x}_r^1}{n} \big) \big] }{ h \big( \tfrac{\hat{x}_r^1}{n} \big) \big[1 - h \big( \tfrac{\hat{x}_r^0}{n} \big) \big]   }  \frac{ a_{\hat{x}_r^1,\hat{x}_{r}^0}\big(\gamma_{r+1}(\eta)\big) }{ a_{\hat{x}_r^0,\hat{x}_{r}^1}\big(\gamma_r(\eta)\big) }
\leq  \prod_{r=0}^{m-1} [ \widetilde{b}_h N^2_{ \text{e} } ] \leq [\widetilde{b}_h  N_{ \text{e} }^{2}]^L,
\end{align*}
where $\widetilde{b}_h:=[ a_h (1 - b_h ) ]^{-1}$. \end{lem}
\begin{proof}
The result is a direct consequence of Definition \ref{Ref}. 
\end{proof}
In the proof of \eqref{lemglobobe}, we only made use of nearest-neighbor jumps. However, in order to obtain \eqref{lemglobtbe}, we will be interested in going from a configuration $\eta$ to $\eta^{\hat{x}, \hat{x}+r \hat{e}_m}$ for any $\hat{x} \in \mathbb{Z}^d$, $m \in \{1, \ldots, d\}$ and any $r \in \mathbb{Z}$ such that $|r| \geq 2$. The procedure for doing so is described in this subsection, in which we always make use of an intermediate site $\hat{z}$ (this is crucial for obtaining Lemma \ref{lem:moving} below) and execute two \textit{main} operations; a leap from $\hat{x}$ to $\hat{z}$ and a leap from $\hat{z}$ to $\hat{y}=\hat{x}+r \hat{e}_m$, not necessarily in this order. 

Recall the definitions of $k^{\star}$ and $b^{\star}$ in Definition \ref{defkstar}. If $k^{\star}=1$, we have from Definition \ref{defkstar} that $c_{\hat{x}, \hat{y}}^{n}(\eta) \geq b^{\star} (2d)^{-1}$ whenever $\{ \hat{x}, \hat{y} \}$ is a \textit{fast} bond. This means that we can \textit{always} perform a jump from $\hat{x}$ to $\hat{y}$, as long as $a_{\hat{x}, \hat{y}}(\eta) >0$. The exact path depends on the value of $\eta ( \hat{z})$, as it is described below. Recall the notation in Definition \ref{path}.
\begin{itemize}
\item
If $\eta(\hat{z}) < N_{ \text{e} }$, we move through the path $\{ \hat{x}, \hat{z}\} \mapsto \{ \hat{z}, \hat{y}\}$, as it is illustrated in Figure \ref{easyxzzy};
\item
If $\eta(\hat{z}) = N_{ \text{e} }$, we move through the path $\{ \hat{z}, \hat{y}\} \mapsto \{ \hat{x}, \hat{z}\}$, as it is illustrated in Figure \ref{easyzyxz}.
\end{itemize}

\begin{figure}[htbp]
    \centering
    \begin{minipage}{0.45\textwidth}
        \centering
\begin{center}
		\begin{tikzpicture}[scale=0.33]
			------------------------------------------------------
			configuração inicial
			\draw [line width=1] (-7,10.5) -- (13,10.5) ; 
			\foreach \x in  {-7,-6,-5,-4,-3,-2,-1,0,1,2,3,4,5,6,7,8,9,10,11,12,13}
			\draw[shift={(\x,10.5)},color=black, opacity=1] (0pt,4pt) -- (0pt,-4pt) node[below] {};
			\draw[] (-2.8,10.5) node[] {};
			
			colocando os indices
				\draw[] (-7,10.5) node[above] {\footnotesize{$\eta_0=\eta$}};
				\draw[] (-5,10.3) node[below] {\footnotesize{$ \hat{x} $}};
			\draw[] (4,10.3)  node[below] {\footnotesize{$ \hat{z} $}};
			\draw[] (12,10.3) node[below] {\footnotesize{$ \hat{y} $}};
			
			partículas da conf inicial
			\shade[shading=ball, ball color=black!50!] (-5,10.65) circle (.3);
		\end{tikzpicture} 
	\end{center}	
	\begin{center}
		\begin{tikzpicture}[scale=0.33]
			------------------------------------------------------
			configuração inicial
			\draw [line width=1] (-7,10.5) -- (13,10.5) ; 
			\foreach \x in  {-7,-6,-5,-4,-3,-2,-1,0,1,2,3,4,5,6,7,8,9,10,11,12,13}
			\draw[shift={(\x,10.5)},color=black, opacity=1] (0pt,4pt) -- (0pt,-4pt) node[below] {};
			\draw[] (-2.8,10.5) node[] {};
			
			colocando os indices
				\draw[] (-7,10.5) node[above] {\footnotesize{$\eta_1=\eta_0^{\hat{x}, \hat{z} }$}};
				\draw[] (-5,10.3) node[below] {\footnotesize{$ \hat{x} $}};
			\draw[] (4,10.3)  node[below] {\footnotesize{$ \hat{z} $}};
			\draw[] (12,10.3) node[below] {\footnotesize{$ \hat{y} $}};
			
			partículas da conf inicial
			\shade[shading=ball,  ball color=black!50!] (4,10.65) circle (.3);
		\end{tikzpicture} 
	\end{center}
\begin{center}
		\begin{tikzpicture}[scale=0.33]
			------------------------------------------------------
			configuração inicial
			\draw [line width=1] (-7,10.5) -- (13,10.5) ; 
			\foreach \x in  {-7,-6,-5,-4,-3,-2,-1,0,1,2,3,4,5,6,7,8,9,10,11,12,13}
			\draw[shift={(\x,10.5)},color=black, opacity=1] (0pt,4pt) -- (0pt,-4pt) node[below] {};
			\draw[] (-2.8,10.5) node[] {};
			
			colocando os indices
				\draw[] (-7,10.5) node[above] {\footnotesize{$\eta_2=\eta_1^{ \hat{z}, \hat{y} }$}};
				\draw[] (-5,10.3) node[below] {\footnotesize{$ \hat{x} $}};
			\draw[] (4,10.3)  node[below] {\footnotesize{$ \hat{z} $}};
			\draw[] (12,10.3) node[below] {\footnotesize{$ \hat{y} $}};		
			partículas da conf inicial
			\shade[shading=ball, ball color=black!50!] (12,10.65) circle (.3);
		\end{tikzpicture} 
	\end{center}
	\caption{Path to send a particle from $\hat{x}$ to $\hat{y}$ through $\hat{z}$, when $k^{\star}=1$ , $\eta( \hat{z}) < N_{ \text{e} } $.}
	\label{easyxzzy}
        \label{fig:figure1}
    \end{minipage}
    \hfill
    \begin{minipage}{0.45\textwidth}
        \centering
        \begin{center}
		\begin{tikzpicture}[scale=0.33]
			------------------------------------------------------
			configuração inicial
			\draw [line width=1] (-7,10.5) -- (13,10.5) ; 
			\foreach \x in  {-7,-6,-5,-4,-3,-2,-1,0,1,2,3,4,5,6,7,8,9,10,11,12,13}
			\draw[shift={(\x,10.5)},color=black, opacity=1] (0pt,4pt) -- (0pt,-4pt) node[below] {};
			\draw[] (-2.8,10.5) node[] {};
			
			colocando os indices
				\draw[] (-7,10.5) node[above] {\footnotesize{$\eta_0=\eta$}};
				\draw[] (-5,10.3) node[below] {\footnotesize{$ \hat{x} $}};
			\draw[] (4,10.3)  node[below] {\footnotesize{$ \hat{z} $}};
			\draw[] (12,10.3) node[below] {\footnotesize{$ \hat{y} $}};
			
			partículas da conf inicial
			\shade[shading=ball, ball color=black!50!] (-5,10.65) circle (.3);
			\shade[shading=ball, ball color=black!50!] (4,10.65) circle (.3);
		\end{tikzpicture} 
	\end{center}		
	\begin{center}
		\begin{tikzpicture}[scale=0.33]
			------------------------------------------------------
			configuração inicial
			\draw [line width=1] (-7,10.5) -- (13,10.5) ; 
			\foreach \x in  {-7,-6,-5,-4,-3,-2,-1,0,1,2,3,4,5,6,7,8,9,10,11,12,13}
			\draw[shift={(\x,10.5)},color=black, opacity=1] (0pt,4pt) -- (0pt,-4pt) node[below] {};
			\draw[] (-2.8,10.5) node[] {};
			
			colocando os indices
				\draw[] (-7,10.5) node[above] {\footnotesize{$\eta_1=\eta_0^{\hat{z}, \hat{y} }$}};
				\draw[] (-5,10.3) node[below] {\footnotesize{$ \hat{x} $}};
			\draw[] (4,10.3)  node[below] {\footnotesize{$ \hat{z} $}};
			\draw[] (12,10.3) node[below] {\footnotesize{$ \hat{y} $}};
			
			partículas da conf inicial
			\shade[shading=ball, ball color=black!50!] (-5,10.65) circle (.3);
			\shade[shading=ball, ball color=black!50!] (12,10.65) circle (.3);
		\end{tikzpicture} 
	\end{center}
\begin{center}
		\begin{tikzpicture}[scale=0.33]
			------------------------------------------------------
			configuração inicial
			\draw [line width=1] (-7,10.5) -- (13,10.5) ; 
			\foreach \x in  {-7,-6,-5,-4,-3,-2,-1,0,1,2,3,4,5,6,7,8,9,10,11,12,13}
			\draw[shift={(\x,10.5)},color=black, opacity=1] (0pt,4pt) -- (0pt,-4pt) node[below] {};
			\draw[] (-2.8,10.5) node[] {};
			
			colocando os indices
				\draw[] (-7,10.5) node[above] {\footnotesize{$\eta_2=\eta_1^{\hat{x}, \hat{z} }$}};
				\draw[] (-5,10.3) node[below] {\footnotesize{$ \hat{x} $}};
			\draw[] (4,10.3)  node[below] {\footnotesize{$ \hat{z} $}};
			\draw[] (12,10.3) node[below] {\footnotesize{$ \hat{y} $}};		
			partículas da conf inicial
			\shade[shading=ball, ball color=black!50!] (4,10.65) circle (.3);
			\shade[shading=ball, ball color=black!50!] (12,10.65) circle (.3);
		\end{tikzpicture} 
	\end{center}
	\caption{Path to send a particle from $\hat{x}$ to $\hat{y}$ through $\hat{z}$, when $k^{\star}=1$ , $\eta( \hat{z}) = N_{ \text{e} } $.}
	\label{easyzyxz}
    \end{minipage}
\end{figure}
However the situation is trickier if $k^{\star} \geq 2$. In this case, we will make use of auxiliary particles, resp. auxiliary "holes", in the neighborhood of $\hat{x}$, resp. neighborhood of $\hat{y}$, if $k^{\star}=\underline{k}^+$, resp. $k^{\star}=\underline{k}^-$. This is necessary to send a particle from $\hat{x}$ to $\hat{y}$, resp. a "hole" from $\hat{y}$ to $\hat{x}$, when $k^{\star}=\underline{k}^+$, resp. $k^{\star}=\underline{k}^-$. In order to describe the details in a precise way, we introduce another definition. 
\begin{definition} \label{defomega12}
Let $\hat{w} \in \mathbb{Z}^d$ and $j \in \{1, \ldots, d\}$ be fixed. We define the set of configurations
\begin{equation*}
\Omega_{j}^{\star}( \hat{w})
:=
\begin{cases}
\quad \Big\{ 
\eta \in \Omega: \quad \prod_{i=1}^{2(k^\star -1)} \eta( \hat{w} + i \hat{e}_j ) \geq 1 \Big\}, \quad  & \text{if} \quad k^\star = \underline{k}^+, \\
\quad \Big\{ 
\eta \in \Omega: \quad \prod_{i=1}^{2(k^\star -1)} \widetilde{\eta}( \hat{w} + i \hat{e}_j ) \geq 1 \Big\}, \quad  & \text{if} \quad k^\star = \underline{k}^-,
\end{cases}
\end{equation*}
where $k^\star$ is given in Definition \ref{defkstar}. If $k^\star=1$, we adopt the convention $\Omega_{j}^{\star}( \hat{w}):=\Omega$. 

Furthermore, for $m \in \{1, \; 2\}$, for any $\eta \in \Omega$, we define $E_j^{\eta,m}(\hat{w}) \subsetneq \mathbb{N}$ as
\begin{equation*}
E_j^{\eta,m}(\hat{w}):=
\begin{cases}
\quad \big\{ i \in  \{1, \dots, m( k^\star -1) \}: \; \; \eta (\hat{w}  + i \hat{e}_j)    =0 \big\}, \quad  & \text{if} \quad k^\star = \underline{k}^+, \\
\quad \big\{ i \in  \{1, \dots, m( k^\star -1) \}:\; \; \widetilde{\eta} (\hat{w}  + i \hat{e}_j)    =0 \big\}, \quad  & \text{if} \quad k^\star = \underline{k}^-. 
\end{cases}
\end{equation*}
If $k^\star=1$, we adopt the convention $E_j^{\eta, m}(\hat{w}):=\varnothing$, the empty set. 
\end{definition}
In what follows, we describe the general procedure for constructing a path from $\eta$ to $\eta^{\hat{x},\hat{y}}$, for some $\hat{x},\hat{y} \in \mathbb{Z}^d$. When $k^{\star} \geq 2$, besides the condition $a_{\hat{x}, \hat{y}}(\eta) >0$, we will also require that there exists $j \in \{1, \ldots, d\}$ such that $\eta \in  \Omega_{j}^{\star}( \hat{x})$, resp. $\eta \in  \Omega_{j}^{\star}( \hat{y})$, if $k^\star = \underline{k}^+$, resp. $k^\star = \underline{k}^-$. This means that there exists a window of $2(k^\star -1)$ sites, oriented along the direction determined by $\hat{e}_j$, located immediately after $\hat{x}$, resp. after $\hat{y}$, without empty sites, resp. fully occupied sites, if $k^\star = \underline{k}^+$, resp. $k^\star = \underline{k}^-$. Given the intermediate site $\hat{z}$ between $\hat{x}$ and $\hat{y}$, we first check if the set $E_j^{\eta}(\hat{z})$ as in Definition \ref{defomega12} is empty or not.
\begin{itemize}
\item
If the aforementioned set is empty, it means that there exists some "cluster" at the neighborhood of $\hat{z}$. In particular, we can move using the paths illustrated in Figure \ref{easyxzzy}, resp. \ref{easyzyxz}, if $\eta(\hat{z}) < N_{ \text{e} }$, resp. $\eta(\hat{z}) = N_{ \text{e} }$.
\item
On the other hand, if $E_j^{\eta}(\hat{z})$ is not empty, we need to construct a "cluster" next to $\hat{z}$ in order to move from $\eta$ to $\eta^{\hat{x}, \hat{y}}$. This is illustrated in Figures \ref{figure-pathcase2} and \ref{figure-pathcase2hole}.
\end{itemize}

\begin{figure}[htbp]  \begin{center}
		\begin{tikzpicture}[scale=0.53]
			------------------------------------------------------
			configuração inicial
			\draw [line width=1] (-7,10.5) -- (13,10.5) ; 
			\foreach \x in  {-7,-6,-5,-4,-3,-2,-1,0,1,2,3,4,5,6,7,8,9,10,11,12,13}
			\draw[shift={(\x,10.5)},color=black, opacity=1] (0pt,4pt) -- (0pt,-4pt) node[below] {};
			\draw[] (-2.8,10.5) node[] {};
			
				\draw[] (-7,10.5) node[above] {\footnotesize{$\eta_0$}};
				\draw[] (-5,10.3) node[below] {\footnotesize{$ \hat{x} $}};
			\draw[] (4,10.3)  node[below] {\footnotesize{$ \hat{z} $}};
			\draw[] (12,10.3) node[below] {\footnotesize{$ \hat{y} $}};
			
			\shade[shading=ball, ball color=black!50!] (-5,10.65) circle (.3);
			\shade[shading=ball, ball color=blue!50!] (-4,10.65) circle (.3);
			\shade[shading=ball, ball color=blue!50!] (-3,10.65) circle (.3);
			\shade[shading=ball, ball color=red!50!] (-2,10.65) circle (.3);
			\shade[shading=ball, ball color=red!50!] (-1,10.65) circle (.3);
		\end{tikzpicture} 
	\end{center}	
	\begin{center}
		\begin{tikzpicture}[scale=0.53]
			------------------------------------------------------
			configuração inicial
			\draw [line width=1] (-7,10.5) -- (13,10.5) ; 
			\foreach \x in  {-7,-6,-5,-4,-3,-2,-1,0,1,2,3,4,5,6,7,8,9,10,11,12,13}
			\draw[shift={(\x,10.5)},color=black, opacity=1] (0pt,4pt) -- (0pt,-4pt) node[below] {};
			\draw[] (-2.8,10.5) node[] {};
			
				\draw[] (-7,10.5) node[above] {\footnotesize{$\eta_1$}};
				\draw[] (-5,10.3) node[below] {\footnotesize{$ \hat{x} $}};
			\draw[] (4,10.3)  node[below] {\footnotesize{$ \hat{z} $}};
			\draw[] (12,10.3) node[below] {\footnotesize{$ \hat{y} $}};
			
			\shade[shading=ball, ball color=black!50!] (4,10.65) circle (.3);
			\shade[shading=ball, ball color=blue!50!] (-4,10.65) circle (.3);
			\shade[shading=ball, ball color=blue!50!] (-3,10.65) circle (.3);
			\shade[shading=ball, ball color=red!50!] (-2,10.65) circle (.3);
			\shade[shading=ball, ball color=red!50!] (-1,10.65) circle (.3);
		\end{tikzpicture} 
	\end{center}
	\begin{center}
		\begin{tikzpicture}[scale=0.53]
			------------------------------------------------------
			configuração inicial
			\draw [line width=1] (-7,10.5) -- (13,10.5) ; 
			\foreach \x in  {-7,-6,-5,-4,-3,-2,-1,0,1,2,3,4,5,6,7,8,9,10,11,12,13}
			\draw[shift={(\x,10.5)},color=black, opacity=1] (0pt,4pt) -- (0pt,-4pt) node[below] {};
			\draw[] (-2.8,10.5) node[] {};
			
				\draw[] (-7,10.5) node[above] {\footnotesize{$\eta_2$}};
				\draw[] (-5,10.3) node[below] {\footnotesize{$ \hat{x} $}};
			\draw[] (4,10.3)  node[below] {\footnotesize{$ \hat{z} $}};
			\draw[] (12,10.3) node[below] {\footnotesize{$ \hat{y} $}};
			
			\shade[shading=ball, ball color=black!50!] (4,10.65) circle (.3);
			\shade[shading=ball, ball color=blue!50!] (5,10.65) circle (.3);
			\shade[shading=ball, ball color=blue!50!] (-3,10.65) circle (.3);
			\shade[shading=ball, ball color=red!50!] (-2,10.65) circle (.3);
			\shade[shading=ball, ball color=red!50!] (-1,10.65) circle (.3);
		\end{tikzpicture} 
	\end{center}
\begin{center}
		\begin{tikzpicture}[scale=0.53]
			------------------------------------------------------
			configuração inicial
			\draw [line width=1] (-7,10.5) -- (13,10.5) ; 
			\foreach \x in  {-7,-6,-5,-4,-3,-2,-1,0,1,2,3,4,5,6,7,8,9,10,11,12,13}
			\draw[shift={(\x,10.5)},color=black, opacity=1] (0pt,4pt) -- (0pt,-4pt) node[below] {};
			\draw[] (-2.8,10.5) node[] {};
			
				\draw[] (-7,10.5) node[above] {\footnotesize{$\eta_3$}};
				\draw[] (-5,10.3) node[below] {\footnotesize{$ \hat{x} $}};
			\draw[] (4,10.3)  node[below] {\footnotesize{$ \hat{z} $}};
			\draw[] (12,10.3) node[below] {\footnotesize{$ \hat{y} $}};
			
			\shade[shading=ball, ball color=black!50!] (4,10.65) circle (.3);
			\shade[shading=ball, ball color=blue!50!] (5,10.65) circle (.3);
			\shade[shading=ball, ball color=blue!50!] (6,10.65) circle (.3);
			\shade[shading=ball, ball color=red!50!] (-2,10.65) circle (.3);
			\shade[shading=ball, ball color=red!50!] (-1,10.65) circle (.3);
		\end{tikzpicture} 
	\end{center}
\begin{center}
		\begin{tikzpicture}[scale=0.53]
			------------------------------------------------------
			configuração inicial
			\draw [line width=1] (-7,10.5) -- (13,10.5) ; 
			\foreach \x in  {-7,-6,-5,-4,-3,-2,-1,0,1,2,3,4,5,6,7,8,9,10,11,12,13}
			\draw[shift={(\x,10.5)},color=black, opacity=1] (0pt,4pt) -- (0pt,-4pt) node[below] {};
			\draw[] (-2.8,10.5) node[] {};
			
				\draw[] (-7,10.5) node[above] {\footnotesize{$\eta_4$}};
				\draw[] (-5,10.3) node[below] {\footnotesize{$ \hat{x} $}};
			\draw[] (4,10.3)  node[below] {\footnotesize{$ \hat{z} $}};
			\draw[] (12,10.3) node[below] {\footnotesize{$ \hat{y} $}};
			
			\shade[shading=ball, ball color=blue!50!] (5,10.65) circle (.3);
			\shade[shading=ball, ball color=blue!50!] (6,10.65) circle (.3);
			\shade[shading=ball, ball color=red!50!] (-2,10.65) circle (.3);
			\shade[shading=ball, ball color=red!50!] (-1,10.65) circle (.3);
			\shade[shading=ball, ball color=black!50!] (12,10.65) circle (.3);
		\end{tikzpicture} 
	\end{center}

\begin{center}
		\begin{tikzpicture}[scale=0.53]
			------------------------------------------------------
			configuração inicial
			\draw [line width=1] (-7,10.5) -- (13,10.5) ; 
			\foreach \x in  {-7,-6,-5,-4,-3,-2,-1,0,1,2,3,4,5,6,7,8,9,10,11,12,13}
			\draw[shift={(\x,10.5)},color=black, opacity=1] (0pt,4pt) -- (0pt,-4pt) node[below] {};
			\draw[] (-2.8,10.5) node[] {};
			
				\draw[] (-7,10.5) node[above] {\footnotesize{$\eta_5$}};
				\draw[] (-5,10.3) node[below] {\footnotesize{$ \hat{x} $}};
			\draw[] (4,10.3)  node[below] {\footnotesize{$ \hat{z} $}};
			\draw[] (12,10.3) node[below] {\footnotesize{$ \hat{y} $}};
			
			\shade[shading=ball, ball color=blue!50!] (5,10.65) circle (.3);
			\shade[shading=ball, ball color=blue!50!] (-3,10.65) circle (.3);
			\shade[shading=ball, ball color=red!50!] (-2,10.65) circle (.3);
			\shade[shading=ball, ball color=red!50!] (-1,10.65) circle (.3);
			\shade[shading=ball, ball color=black!50!] (12,10.65) circle (.3);
		\end{tikzpicture} 
	\end{center}

\begin{center}
		\begin{tikzpicture}[scale=0.53]
			------------------------------------------------------
			configuração inicial
			\draw [line width=1] (-7,10.5) -- (13,10.5) ; 
			\foreach \x in  {-7,-6,-5,-4,-3,-2,-1,0,1,2,3,4,5,6,7,8,9,10,11,12,13}
			\draw[shift={(\x,10.5)},color=black, opacity=1] (0pt,4pt) -- (0pt,-4pt) node[below] {};
			\draw[] (-2.8,10.5) node[] {};
			
				\draw[] (-7,10.5) node[above] {\footnotesize{$\eta_6$}};
				\draw[] (-5,10.3) node[below] {\footnotesize{$ \hat{x} $}};
			\draw[] (4,10.3)  node[below] {\footnotesize{$ \hat{z} $}};
			\draw[] (12,10.3) node[below] {\footnotesize{$ \hat{y} $}};
			
			\shade[shading=ball, ball color=blue!50!] (-4,10.65) circle (.3);
			\shade[shading=ball, ball color=blue!50!] (-3,10.65) circle (.3);
			\shade[shading=ball, ball color=red!50!] (-2,10.65) circle (.3);
			\shade[shading=ball, ball color=red!50!] (-1,10.65) circle (.3);
			\shade[shading=ball, ball color=black!50!] (12,10.65) circle (.3);
		\end{tikzpicture} 
	\end{center}
	\caption{Path to send a particle from $\hat{x}$ to $\hat{y}$ through $\hat{z}$, when $k^{\star}=3=\underline{k}^+$, $E=\{1, 2\}$ and $\eta( \hat{z}) < N_{ \text{e} } $. 
	 }
	\label{figure-pathcase2}
\end{figure}
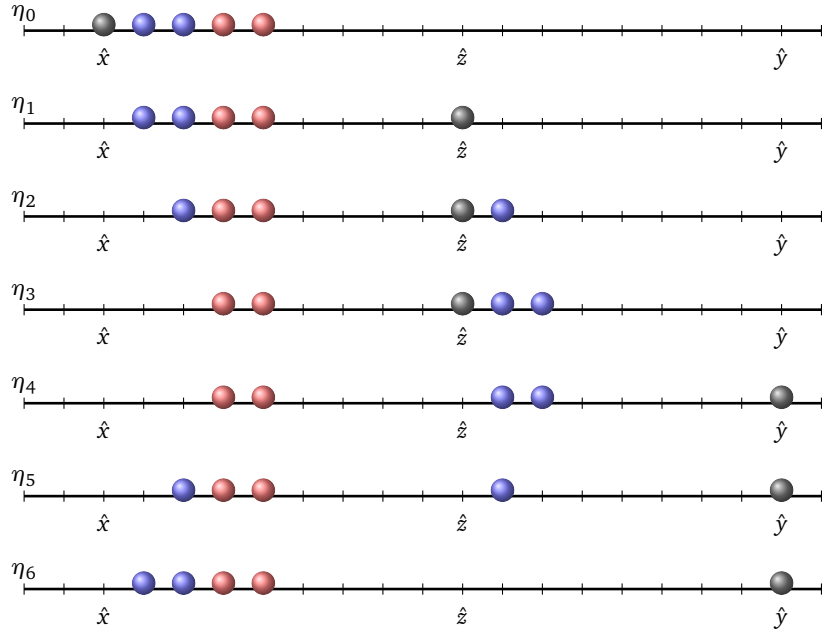

\begin{figure}[htbp] 
 \begin{center}
		\begin{tikzpicture}[scale=0.53]
			------------------------------------------------------
			configuração inicial
			\draw [line width=1] (-7,10.5) -- (13,10.5) ; 
			\foreach \x in  {-7,-6,-5,-4,-3,-2,-1,0,1,2,3,4,5,6,7,8,9,10,11,12,13}
			\draw[shift={(\x,10.5)},color=black, opacity=1] (0pt,4pt) -- (0pt,-4pt) node[below] {};
			\draw[] (-2.8,10.5) node[] {};
			
				\draw[] (-7,10.5) node[above] {\footnotesize{$\eta_0$}};
				\draw[] (-5,10.3) node[below] {\footnotesize{$ \hat{x} $}};
			\draw[] (0,10.3)  node[below] {\footnotesize{$ \hat{z} $}};
			\draw[] (8,10.3) node[below] {\footnotesize{$ \hat{y} $}};
			
			\shade[shading=ball, ball color=black!50!] (-5,10.65) circle (.3);
			\shade[shading=ball, ball color=black!50!] (1,10.65) circle (.3);
			\shade[shading=ball, ball color=black!50!] (2,10.65) circle (.3);
			\shade[shading=ball, ball color=blue!50!] (9,10.65) circle (.1);
			\shade[shading=ball, ball color=blue!50!] (10,10.65) circle (.1);
			\shade[shading=ball, ball color=red!50!] (11,10.65) circle (.1);
			\shade[shading=ball, ball color=red!50!] (12,10.65) circle (.1);
		\end{tikzpicture} 
	\end{center}	
	
 \begin{center}
		\begin{tikzpicture}[scale=0.53]
			------------------------------------------------------
			configuração inicial
			\draw [line width=1] (-7,10.5) -- (13,10.5) ; 
			\foreach \x in  {-7,-6,-5,-4,-3,-2,-1,0,1,2,3,4,5,6,7,8,9,10,11,12,13}
			\draw[shift={(\x,10.5)},color=black, opacity=1] (0pt,4pt) -- (0pt,-4pt) node[below] {};
			\draw[] (-2.8,10.5) node[] {};
			
				\draw[] (-7,10.5) node[above] {\footnotesize{$\eta_1$}};
				\draw[] (-5,10.3) node[below] {\footnotesize{$ \hat{x} $}};
			\draw[] (0,10.3)  node[below] {\footnotesize{$ \hat{z} $}};
			\draw[] (8,10.3) node[below] {\footnotesize{$ \hat{y} $}};
			
			\shade[shading=ball, ball color=black!50!] (-5,10.65) circle (.3);
			\shade[shading=ball, ball color=black!50!] (9,10.65) circle (.3);
			\shade[shading=ball, ball color=black!50!] (2,10.65) circle (.3);
			\shade[shading=ball, ball color=blue!50!] (1,10.65) circle (.1);
			\shade[shading=ball, ball color=blue!50!] (10,10.65) circle (.1);
			\shade[shading=ball, ball color=red!50!] (11,10.65) circle (.1);
			\shade[shading=ball, ball color=red!50!] (12,10.65) circle (.1);
		\end{tikzpicture} 
	\end{center}	
	
 \begin{center}
		\begin{tikzpicture}[scale=0.53]
			------------------------------------------------------
			configuração inicial
			\draw [line width=1] (-7,10.5) -- (13,10.5) ; 
			\foreach \x in  {-7,-6,-5,-4,-3,-2,-1,0,1,2,3,4,5,6,7,8,9,10,11,12,13}
			\draw[shift={(\x,10.5)},color=black, opacity=1] (0pt,4pt) -- (0pt,-4pt) node[below] {};
			\draw[] (-2.8,10.5) node[] {};
			
				\draw[] (-7,10.5) node[above] {\footnotesize{$\eta_2$}};
				\draw[] (-5,10.3) node[below] {\footnotesize{$ \hat{x} $}};
			\draw[] (0,10.3)  node[below] {\footnotesize{$ \hat{z} $}};
			\draw[] (8,10.3) node[below] {\footnotesize{$ \hat{y} $}};
			
			\shade[shading=ball, ball color=black!50!] (-5,10.65) circle (.3);
			\shade[shading=ball, ball color=black!50!] (9,10.65) circle (.3);
			\shade[shading=ball, ball color=black!50!] (10,10.65) circle (.3);
			\shade[shading=ball, ball color=blue!50!] (1,10.65) circle (.1);
			\shade[shading=ball, ball color=blue!50!] (2,10.65) circle (.1);
			\shade[shading=ball, ball color=red!50!] (11,10.65) circle (.1);
			\shade[shading=ball, ball color=red!50!] (12,10.65) circle (.1);
		\end{tikzpicture} 
	\end{center}	

 \begin{center}
		\begin{tikzpicture}[scale=0.53]
			------------------------------------------------------
			configuração inicial
			\draw [line width=1] (-7,10.5) -- (13,10.5) ; 
			\foreach \x in  {-7,-6,-5,-4,-3,-2,-1,0,1,2,3,4,5,6,7,8,9,10,11,12,13}
			\draw[shift={(\x,10.5)},color=black, opacity=1] (0pt,4pt) -- (0pt,-4pt) node[below] {};
			\draw[] (-2.8,10.5) node[] {};
			
				\draw[] (-7,10.5) node[above] {\footnotesize{$\eta_3$}};
				\draw[] (-5,10.3) node[below] {\footnotesize{$ \hat{x} $}};
			\draw[] (0,10.3)  node[below] {\footnotesize{$ \hat{z} $}};
			\draw[] (8,10.3) node[below] {\footnotesize{$ \hat{y} $}};
			
			\shade[shading=ball, ball color=black!50!] (0,10.65) circle (.3);
			\shade[shading=ball, ball color=black!50!] (9,10.65) circle (.3);
			\shade[shading=ball, ball color=black!50!] (10,10.65) circle (.3);
			\shade[shading=ball, ball color=blue!50!] (1,10.65) circle (.1);
			\shade[shading=ball, ball color=blue!50!] (2,10.65) circle (.1);
			\shade[shading=ball, ball color=red!50!] (11,10.65) circle (.1);
			\shade[shading=ball, ball color=red!50!] (12,10.65) circle (.1);
		\end{tikzpicture} 
	\end{center}	

 \begin{center}
		\begin{tikzpicture}[scale=0.53]
			------------------------------------------------------
			configuração inicial
			\draw [line width=1] (-7,10.5) -- (13,10.5) ; 
			\foreach \x in  {-7,-6,-5,-4,-3,-2,-1,0,1,2,3,4,5,6,7,8,9,10,11,12,13}
			\draw[shift={(\x,10.5)},color=black, opacity=1] (0pt,4pt) -- (0pt,-4pt) node[below] {};
			\draw[] (-2.8,10.5) node[] {};
			
				\draw[] (-7,10.5) node[above] {\footnotesize{$\eta_4$}};
				\draw[] (-5,10.3) node[below] {\footnotesize{$ \hat{x} $}};
			\draw[] (0,10.3)  node[below] {\footnotesize{$ \hat{z} $}};
			\draw[] (8,10.3) node[below] {\footnotesize{$ \hat{y} $}};
			
			\shade[shading=ball, ball color=black!50!] (8,10.65) circle (.3);
			\shade[shading=ball, ball color=black!50!] (9,10.65) circle (.3);
			\shade[shading=ball, ball color=black!50!] (10,10.65) circle (.3);
			\shade[shading=ball, ball color=blue!50!] (1,10.65) circle (.1);
			\shade[shading=ball, ball color=blue!50!] (2,10.65) circle (.1);
			\shade[shading=ball, ball color=red!50!] (11,10.65) circle (.1);
			\shade[shading=ball, ball color=red!50!] (12,10.65) circle (.1);
		\end{tikzpicture} 
	\end{center}	

 \begin{center}
		\begin{tikzpicture}[scale=0.53]
			------------------------------------------------------
			configuração inicial
			\draw [line width=1] (-7,10.5) -- (13,10.5) ; 
			\foreach \x in  {-7,-6,-5,-4,-3,-2,-1,0,1,2,3,4,5,6,7,8,9,10,11,12,13}
			\draw[shift={(\x,10.5)},color=black, opacity=1] (0pt,4pt) -- (0pt,-4pt) node[below] {};
			\draw[] (-2.8,10.5) node[] {};
			
				\draw[] (-7,10.5) node[above] {\footnotesize{$\eta_5$}};
				\draw[] (-5,10.3) node[below] {\footnotesize{$ \hat{x} $}};
			\draw[] (0,10.3)  node[below] {\footnotesize{$ \hat{z} $}};
			\draw[] (8,10.3) node[below] {\footnotesize{$ \hat{y} $}};
			
			\shade[shading=ball, ball color=black!50!] (8,10.65) circle (.3);
			\shade[shading=ball, ball color=black!50!] (9,10.65) circle (.3);
			\shade[shading=ball, ball color=black!50!] (2,10.65) circle (.3);
			\shade[shading=ball, ball color=blue!50!] (1,10.65) circle (.1);
			\shade[shading=ball, ball color=blue!50!] (10,10.65) circle (.1);
			\shade[shading=ball, ball color=red!50!] (11,10.65) circle (.1);
			\shade[shading=ball, ball color=red!50!] (12,10.65) circle (.1);
		\end{tikzpicture} 
	\end{center}	

 \begin{center}
		\begin{tikzpicture}[scale=0.53]
			------------------------------------------------------
			configuração inicial
			\draw [line width=1] (-7,10.5) -- (13,10.5) ; 
			\foreach \x in  {-7,-6,-5,-4,-3,-2,-1,0,1,2,3,4,5,6,7,8,9,10,11,12,13}
			\draw[shift={(\x,10.5)},color=black, opacity=1] (0pt,4pt) -- (0pt,-4pt) node[below] {};
			\draw[] (-2.8,10.5) node[] {};
			
				\draw[] (-7,10.5) node[above] {\footnotesize{$\eta_6$}};
				\draw[] (-5,10.3) node[below] {\footnotesize{$ \hat{x} $}};
			\draw[] (0,10.3)  node[below] {\footnotesize{$ \hat{z} $}};
			\draw[] (8,10.3) node[below] {\footnotesize{$ \hat{y} $}};
			
			\shade[shading=ball, ball color=black!50!] (8,10.65) circle (.3);
			\shade[shading=ball, ball color=black!50!] (1,10.65) circle (.3);
			\shade[shading=ball, ball color=black!50!] (2,10.65) circle (.3);
			\shade[shading=ball, ball color=blue!50!] (9,10.65) circle (.1);
			\shade[shading=ball, ball color=blue!50!] (10,10.65) circle (.1);
			\shade[shading=ball, ball color=red!50!] (11,10.65) circle (.1);
			\shade[shading=ball, ball color=red!50!] (12,10.65) circle (.1);
		\end{tikzpicture} 
	\end{center}	

	\caption{Path to send a particle from $\hat{x}$ to $\hat{y}$ through $\hat{z}$, when $k^{\star}=3=\underline{k}^-$, $E=\{1, 2\}$ and $\eta( \hat{z}) < N_{ \text{e} } $.  	
	 }
	\label{figure-pathcase2hole}
\end{figure}
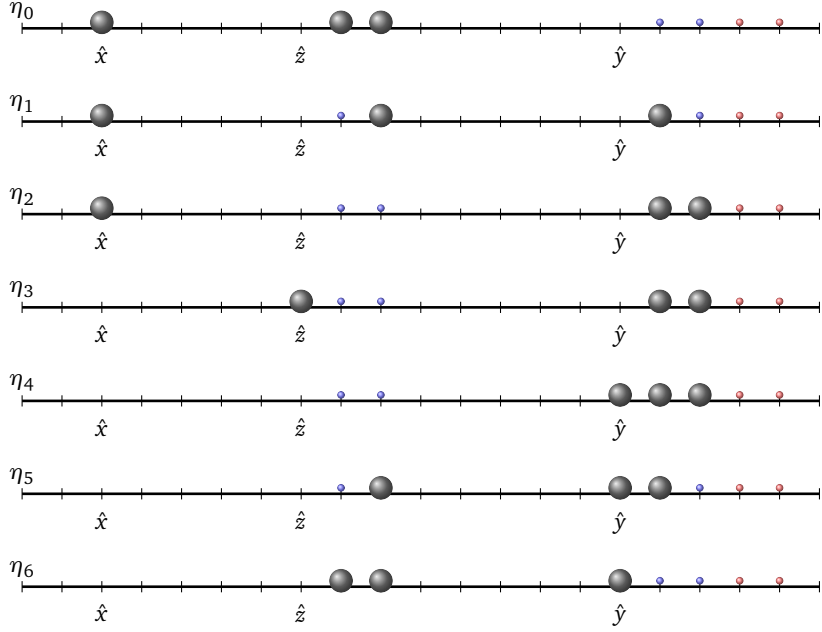

Observe that Figure \ref{figure-pathcase2} illustrates the path
\begin{align*}
\{ \hat{x}, \hat{z} \} \mapsto \{ \hat{x} + \hat{e}_j, \hat{z} +  \hat{e}_j \} \mapsto \{ \hat{x} +2 \hat{e}_j, \hat{z} + 2 \hat{e}_j \} \mapsto \{ \hat{z} , \hat{y}  \} \mapsto \{ \hat{z} + 2\hat{e}_j, \hat{x} + 2 \hat{e}_j \} \mapsto \{ \hat{z} + \hat{e}_j, \hat{x} +  \hat{e}_j \}
\end{align*}
and that Figure \ref{figure-pathcase2hole} illustrates the path
\begin{align*}
\{ \hat{z} + \hat{e}_j, \hat{y} +  \hat{e}_j \} \mapsto \{ \hat{z} +2 \hat{e}_j, \hat{y} + 2 \hat{e}_j \} \mapsto \{ \hat{x}, \hat{z} \} \mapsto \{ \hat{z} , \hat{y}  \} \mapsto \{ \hat{y} + 2\hat{e}_j, \hat{z} + 2 \hat{e}_j \} \mapsto \{ \hat{y} + \hat{e}_j, \hat{z} +  \hat{e}_j \}.
\end{align*}
Above we applied the notation in Definition \ref{path}. From Figures \ref{figure-pathcase2} and \ref{figure-pathcase2hole}, we observe the following:
\begin{itemize}
\item
If $k^\star = \underline{k}^+$, we have that $\eta \in  \Omega_{j}^{\star}( \hat{x})$. The particles at the sites $\{\hat{x}+i \hat{e}_j, 1 \leq k^\star - 1\}$ (represented by the big blue circles in Figure \ref{figure-pathcase2}) form a "mobile cluster" that is useful to construct the "cluster" next to $\hat{z}$. Moreover, the ones at $\{\hat{x}+i \hat{e}_j, k^\star \leq 2k^\star - 2\}$ (represented by the big red circles in Figure \ref{figure-pathcase2}) form a "fixed cluster" that is important to attract the "mobile cluster" back to its initial position. \item
On the other hand, when $k^\star = \underline{k}^-$ we get that $\eta \in  \Omega_{j}^{\star}( \hat{y})$. The "holes" at the sites $\{\hat{x}+i \hat{e}_j, 1 \leq k^\star - 1\}$ (represented by the small blue circles in Figure \ref{figure-pathcase2hole}) form a "mobile cluster" that is crucial to construct the "cluster" next to $\hat{z}$. Furthermore, the ones at $\{\hat{x}+i \hat{e}_j, k^\star \leq 2k^\star - 2\}$ (represented by the small red circles in Figure \ref{figure-pathcase2hole}) form a "fixed cluster" that is necessary to attract the "mobile cluster" back to its initial position.
\end{itemize}
Keeping in mind the procedure described above , we can state and prove the next result, which is a generalization of Lemma 5.3 in \cite{renato}. In what follows, we recall from \eqref{defAkepsn} the set $\mathbb{A}_p^+$.
\begin{lem} \label{lem:moving} \textbf{(Global moving particle lemma)}
Under Hypothesis \ref{hyp:coeff} \textbf{(ii)}, fix $i,j\in \{1, \ldots, d\}$ and $r \in \mathbb{Z}$ such that $|r|>k^\star$. Moreover, let $h \in Ref$, $n \geq 1$, $f: \Omega \mapsto \mathbb{R}$, and $Z_r\subset \mathbb{A}_{3|r|}^+$. Then, there exists a constant $C^{\star}$ depending only on $h$, $k^{\star}$, $b^{\star}$, $N_{\text{e} }$, $d$ and $\gamma$ such that
\begin{align}\label{gmpl2} 
        \sum_{\hat{x}\in Z_r}
        \int_{\Omega^\star} 
	a_{ \hat{x},\hat{x}+r\hat{e}_i}(\eta) \big[ f( \eta)  - f ( \eta^{\hat{x},\hat{x}+r\hat{e}_i} )  \big]^2 
	 \rmd \nu_h^n \leq C^{\star}
         |r|^{\gamma} \mcb D_{n,\mcb F}^\gamma (f | \nu_h^n ).  
    \end{align}
In the last display, $\Omega^\star:=\Omega_{j}^\star(\hat{x})$, if $k^{\star}=\underline{k}^+$; and $\Omega^\star:=\Omega_{j}^\star(\hat{x}+r\hat{e}_i)$, if $k^{\star}=\underline{k}^-$. Here, $\Omega_{j}^\star(\cdot)$ is given in Definition \ref{defomega12}.  
\end{lem}
\begin{rem}
The restriction $\hat{x}\in Z_r$ in the sum of last display is only to ensure that all the relevant bonds in the proof are \textit{fast}, which is crucial at the end of the proof.  
\end{rem}
\begin{proof} [Proof of Lemma \ref{lem:moving}]
For every $\hat{x} \in \mathbb{Z}^d$ and $\hat{\omega}\in  \llbracket1, \; \lfloor |r/2| \; \rfloor\rrbracket^d$, define $\hat{y}(\hat{x})$ and $\hat{z}(\hat{x}, \hat{\omega})$ by
\begin{align*}
\hat{y}(\hat{x}):=\hat{x}+r\hat{e}_i, \quad \hat{z}(\hat{x},\hat{\omega}):=
	\hat{x}+ \text{sign}(r) (\lfloor |r/2| \; \rfloor \hat{e}_i+\hat{1}\odot \hat{\omega}),
\end{align*}
where $\hat{1}\odot \hat{\omega}$ is given in Definition \ref{defmedemp}. Moreover, define $S_{\hat{x}} \subset \mathbb{Z}^d$ by
\begin{align*}
S_{\hat{x}}:= \big\{ \hat{z}(\hat{x},\hat{\omega}) -\hat{x}, \quad  \hat{\omega}\in \llbracket1, \; \lfloor |r/2| \; \rfloor\rrbracket^d \big\} \cup \big\{\hat{y}(\hat{x}) - \hat{z}(\hat{x},\hat{\omega}), \quad \hat{\omega}\in \llbracket1, \; \lfloor |r/2| \; \rfloor\rrbracket^d \big\}.
\end{align*}
for any $\hat{x} \in \mathbb{Z}^d$. Observe that $S_{\hat{x}}=S_{\hat{0}}$, for any $\hat{x} \in \mathbb{Z}^d$. Moreover, it holds
\begin{equation} \label{norep2}
 \max \{|v_i|, \; 1 \leq i \leq d, \; \hat{v}=(v_1, \ldots, v_d) \in S_{\hat{0}}\}  \leq |r| ; \quad \text{there are no repetitions in the set} \; S_{\hat{0}}.   
\end{equation}
Let us now fix $\hat{\omega} \in \llbracket1, \; \lfloor |r/2| \; \rfloor \rrbracket^d$, $\hat{x} \in Z_r$ and $\eta \in \Omega^\star$ such that $a_{ \hat{x}, \hat{y}( \hat{x} ) }(\eta)>0$. In the remainder of this work, for any finite set $S$, the number of its elements is denoted by $|S|$. Consider the set $E_j^{\eta,1} \big( \hat{z}(\hat{x},\hat{\omega}) \big)$, as in Definition \ref{defomega12}; in this proof we denote $\big|E_j^{\eta,1} \big( \hat{z}(\hat{x},\hat{\omega}) \big)\big|$ by $N^{\star}_{\hat{\omega}}$. Note that $ N^{\star}_{\hat{\omega}} \leq k^{\star}-1$, where $N^{\star}_{\hat{\omega}}$ is the (random) number of sites in the (random) window $\{ \hat{z} + \hat{e}_j, \hat{z} + 2\hat{e}_j, \ldots, \hat{z} + (k^{\star}-1) \hat{e}_j\}$ which needs to receive (and afterwards, send back) a particle, resp. a "hole" from the cluster in $\{ \hat{x} + \hat{e}_j, \ldots, \hat{x} + (k^{\star}-1) \hat{e}_j\}$, resp. in $\{ \hat{y} + \hat{e}_j, \ldots, \hat{y} + (k^{\star}-1) \hat{e}_j\}$, if $k^{\star}=\underline{k}^+$, resp. $k^{\star}=\underline{k}^-$. Recall the notation of Definition \ref{path}. We now construct a path $\mapsto_{m=0}^{1+ 2 N^{\star}_{\hat{\omega}}}\{\hat{r}_m^{0},\hat{r}_m^{1}\}$ from $\eta$ to $\eta^{ \hat{x}, \hat{y}( \hat{x} )}$, as it is illustrated in Figures \ref{figure-pathcase2} and \ref{figure-pathcase2hole}.
 \begin{itemize}
\item
If $\eta(\hat{z})<N_e$ and $k^{\star}=\underline{k}^+$, the path is $\{\hat{x},\hat{z} \} \mapsto \gamma_1 ( \hat{x}, \; \hat{z} ) \mapsto \{ \hat{z},\hat{y} \} \mapsto \gamma_2 ( \hat{z}, \; \hat{x} )$; 
\item
If $\eta(\hat{z})<N_e$ and $k^{\star}=\underline{k}^-$, the path is $\gamma_1 ( \hat{z}, \; \hat{y} ) \mapsto \{\hat{x},\hat{z} \} \mapsto   \{ \hat{z},\hat{y} \} \mapsto \gamma_2 ( \hat{y}, \; \hat{z} )$; 
\item
If $\eta(\hat{z})=N_e$ and $k^{\star}=\underline{k}^+$, the path is $\gamma_1 ( \hat{x}, \; \hat{z} ) \mapsto \{\hat{z},\hat{y} \} \mapsto   \{ \hat{x},\hat{z} \} \mapsto \gamma_2 ( \hat{z}, \; \hat{x} )$; 
\item
If $\eta(\hat{z})=N_e$ and $k^{\star}=\underline{k}^-$, the path is $\{\hat{z},\hat{y} \} \mapsto \gamma_1 ( \hat{z}, \; \hat{y} ) \mapsto \{ \hat{x},\hat{z} \} \mapsto \gamma_2 ( \hat{y}, \; \hat{z} )$.
\end{itemize}
Above, for $( \hat{r}_1, \; \hat{r}_2 ) \in \{ (\hat{x}, \; \hat{z}), \;  (\hat{z}, \; \hat{y})  \}$, $\gamma_1 ( \hat{r}_1, \; \hat{r}_2 )$ and $\gamma_2 ( \hat{r}_2, \; \hat{r}_1 )$ are the sub-paths given by
\begin{align*}
\gamma_1 ( \hat{r}_1, \; \hat{r}_2 ):= \mapsto_{p=1}^{N^{\star}_{\hat{\omega}}}
    \{\hat{r}_1+m_{p}\hat{e}_j,\hat{r}_2+m_{p}\hat{e}_j \}, \quad \gamma_2 ( \hat{r}_2, \; \hat{r}_1 ):= \mapsto_{p=1}^{N^{\star}_{\hat{\omega}}}
    \{\hat{r}_2+m_{N^{\star}_{\hat{\omega}} + 1 -p}\hat{e}_j,\hat{r}_1+m_{N^{\star}_{\hat{\omega}}+1 - p}\hat{e}_j \},
\end{align*} 
where $\{m_1, \ldots, m_{ N^{\star}_{\hat{\omega}} }\}$ is the \textit{increasing} ordering of the elements of $E_j^{\eta} \big( \hat{z}(\hat{x},\hat{\omega}) \big)$. Thus, for every $m \in \{0, \ldots 1 + 2 N^{\star}_{\hat{\omega}} \}$, it holds $a_{\hat{r}_m^0,\hat{r}_{m}^1}(\gamma_m(\eta)) \geq 1$, where $\gamma_{m+1}(\eta):=(\gamma_{m}(\eta))^{\hat{r}_m^0,\hat{r}_{m}^1}$ and $\gamma_{0}(\eta)=\eta$.

This reasoning is performed for each $\hat{\omega}\in\llbracket1, \; \lfloor |r/2| \; \rfloor\rrbracket^d$, each $\hat{x} \in Z_r$ and $\eta\in\Omega^\star$. From the Cauchy-Schwarz's inequality and the fact that $a_{ \hat{x},\hat{x}+r\hat{e}_j}(\eta) \leq N_{ \text{e} }^{2}$, the left-hand side of \eqref{gmpl2} is bounded from above by
\begin{align} \label{estpath0}
  \frac{ M^{\star} }{ \lfloor |r/2| \; \rfloor^d } \sum_{ \hat{\omega}\in \llbracket1, \; \lfloor |r/2| \; \rfloor \rrbracket^d } \sum_{\hat{x}\in Z_r}  \int_{ \Omega_0^\star } \sum_{m=0}^{ 1+ 2 N^{\star}_{\hat{\omega}} } \big[ f \big( (\gamma_{m}(\eta))^{\hat{r}_m^0,\hat{r}_{m}^1} \big) -  f  \big( \gamma_{m}(\eta) \big)  \big]^2  \; \rmd \nu_h^n \big( \gamma_{m}(\eta) \big)  
\end{align}
where $\Omega_0^\star:=\{ \eta \in \Omega^\star: \;  a_{ \hat{x},\hat{x}+r\hat{e}_i}(\eta) >0  \}$ $M^{\star}:= 2 k^{\star} N_{ \text{e} }^{2}[\widetilde{b}_h  N_{ \text{e} }^{2}]^{2 k^{\star}}$. In the last line we used the fact that $\nu_h^n(\eta) / \nu_h^n \big( \gamma_{s}(\eta) \big)$ is uniformly bounded from above by $[\widetilde{b}_h  N_{ \text{e} }^{2}]^{2 k^{\star}}$, due to Lemma \ref{lemradnik}, and to the observation that the whole path from $\eta$ to $\eta^{ \hat{x}, \hat{y}( \hat{x} )}$ is formed by exactly $2 + 2 N^{\star}_{\hat{\omega}} \leq 2 k^{\star}$ bonds.

 Next, we focus on the integral in \eqref{estpath0}. By construction, $c_{\hat{r}_m^0,\hat{r}_{m}^1}^{n}(\gamma_m(\eta))$ is uniformly bounded from below by $b^{\star} (2d)^{-1}$, see Definition \ref{defkstar}. Moreover, from \eqref{norep2}, we have that $p_{\gamma}^{-1} (\hat{r}_{m}^1 - \hat{r}_{m}^0)$ is uniformly bounded from above by $(|r| \sqrt{d})^{\gamma+d} ( c_{\gamma} )^{-1}$. Therefore, the integral over $\Omega_0^\star$ in \eqref{estpath0} is bounded from above by
\begin{align} \label{estpath1}
\tilde{C} \int_{ \Omega_{0}^\star } \sum_{m=0}^{ 1+ 2 N^{\star}_{\hat{\omega}} }  p_{\gamma} (\hat{r}_{m}^1 - \hat{r}_{m}^0) a_{\hat{r}_m^0,\hat{r}_{m}^1}\big(\gamma_m(\eta)\big) c_{\hat{r}_m^0,\hat{r}_{m}^1}^{n}\big(\gamma_m(\eta)\big) \big[ \nabla_{\hat{r}_m^0,\hat{r}_{m}^1}   f  \big( \gamma_{m}(\eta) \big)  \big]^2  \rmd \nu_h^n \big( \gamma_{m}(\eta) \big),
\end{align} 
where $\tilde{C}:= 2d (|r| \sqrt{d})^{\gamma+d} ( b^{\star} c_{\gamma} )^{-1}$. Above we used the fact that $a_{\hat{r}_m^0,\hat{r}_{m}^1}(\gamma_m(\eta)) \geq 1$, for every $m \in \{0, \ldots, 1+ 2 N^{\star}_{\hat{\omega}}\}$. The next step is to perform the change of variables $\eta_{*}:=\gamma_{m}(\eta)$ in \eqref{estpath1}. In order to do so, we partition $\Omega_{0}^\star$ in $2^{ k^{\star} }$ disjoint subsets, in such a way that the sequence of bonds $\mapsto_{m=0}^{1+ 2 N^{\star}_{\hat{\omega}}}\{\hat{r}_m^{0},\hat{r}_m^{1}\}$ is the same on each component. 

By construction, we need to look whether on $\eta( \hat{z}  ) < N_{ \text{e} } $ or $\eta( \hat{z}  ) = N_{ \text{e} } $; whether on $\eta( \hat{z} + \hat{e}_j )=0$ or $\eta( \hat{z} + \hat{e}_j )>0$; ...;  and  whether on $\eta \big( \hat{z} + (k^{\star} - 1) \hat{e}_j \big)=0$ or $\eta \big( \hat{z} + (k^{\star} - 1) \hat{e}_j \big)>0$. In this way, we write $\Omega_{0}^\star = \cup_{p =1 }^{ 2^{ k^{\star} } } \Omega_{p}^\star$, where $\mapsto_{m=0}^{1+ 2 N^{\star}_{\hat{\omega}}}\{\hat{r}_m^{0},\hat{r}_m^{1}\} = \mapsto_{m=0}^{1+ 2 \tilde{N}_{\hat{\omega},p}}\{\hat{r}_m^{0,p},\hat{r}_m^{1,p}\}$ is a deterministic sequence of bonds, for each $p \in \{1, \ldots, 2^{ k^{\star} }\}$. Thus, the display in \eqref{estpath1} can be rewritten as
\begin{align*}
& \tilde{C} \sum_{p=1}^{ 2^{ k^{\star} } } \sum_{m=0}^{ 1+ 2 \tilde{N}_{\hat{\omega},p} }   \int_{ \Omega_{p}^\star }  p_{\gamma} (\hat{r}_{m}^1 - \hat{r}_{m}^0) a_{\hat{r}_m^0,\hat{r}_{m}^1}\big(\gamma_m(\eta)\big) c_{\hat{r}_m^0,\hat{r}_{m}^1}^{n}\big(\gamma_m(\eta)\big) \big[ \nabla_{\hat{r}_m^0,\hat{r}_{m}^1}   f  \big( \gamma_{m}(\eta) \big)  \big]^2  \rmd \nu_h^n \big( \gamma_{m}(\eta) \big) \\
=& \tilde{C} \sum_{p=1}^{ 2^{ k^{\star} } } \sum_{m=0}^{ 1+ 2 \tilde{N}_{\hat{\omega},p} }   \int_{ \gamma_{m} ( \Omega_{p}^\star ) }  p_{\gamma} (\hat{r}_{m}^{1,p} - \hat{r}_{m}^{0,p}) a_{\hat{r}_m^{0,p},\hat{r}_{m}^{1,p}}( \eta_{*} ) c_{\hat{r}_m^{0,p},\hat{r}_{m}^{1,p} }^{n}( \eta_{*} ) \big[ \nabla_{\hat{r}_m^{0,p},\hat{r}_{m}^{1,p} }   f  ( \eta_{*} )  \big]^2  \rmd \nu_h^n ( \eta_{*} ) .
\end{align*}
Since $a_{\hat{r}_m^0,\hat{r}_{m}^1}(\gamma_m(\eta)) \geq 1$, for every $m \in \{0, \ldots, 1+ 2 N^{\star}_{\hat{\omega}}\}$, the map $\gamma_m :  \Omega_{p}^\star  \mapsto \Omega$ is one-to-one, for any $p \in \{1, \ldots, 2^{ k^{\star} } \}$ and any $m \in \{0, \ldots, 1+ 2 \tilde{N}_{\hat{\omega},p} \}$. Therefore, the last display is bounded from above by
\begin{align*}
\tilde{C} \sum_{p=1}^{ 2^{ k^{\star} } } \sum_{m=0}^{ 1+ 2 \tilde{N}_{\hat{\omega},p} } \mcb I^{n}_{ \hat{r}_{m}^{0,p} , \hat{r}_{m}^{1,p} }  ( f | \nu_h^n ),
\end{align*}
where $\mcb I^{n}_{ \cdot , \cdot}  ( f | \nu_h^n )$ is given in Definition \ref{def:dir_form}. Next, observe that for any $p \in \{1, \ldots, 2^{ k^{\star} } \}$ and any $m \in \{0, \ldots, 1+ 2 \tilde{N}_{\hat{\omega},p} \}$, $\{ (\hat{r}_{m}^{0,p}, \hat{r}_{m}^{1,p} ), \; 0 \leq m \leq 1+ 2 \tilde{N}_{\hat{\omega},p}  \}$ is always a subset of
\begin{align*}
\{ ( \hat{x}, \; \hat{z}) \} \cup \{ (\hat{z}, \; \hat{y}) \} \; \bigcup_{m=1}^{k^{\star} -1}  \; \{  ( \hat{r}_1^{\star} + m \hat{e}_j, \; \hat{r}_2^{\star} + m \hat{e}_j  ), \;  (  \hat{r}_2^{\star} + m \hat{e}_j, \; \hat{r}_1^{\star} + m \hat{e}_j  ) \}, 
\end{align*}
In the last display, $( \hat{r}_1^{\star}, \; \hat{r}_2^{\star} ) :=(\hat{x}, \; \hat{z})$, if $k^{\star}=\underline{k}^+$; and $( \hat{r}_1^{\star}, \; \hat{r}_2^{\star} ) :=(\hat{z}, \; \hat{y})$, if $k^{\star}=\underline{k}^-$. Thus, from the arguments described above, we conclude that the term in \eqref{estpath0} is bounded from above by a constant (depending only on $h$, $k^{\star}$, $b^{\star}$, $N_{\text{e} }$, $d$ and $\gamma$), times 
\begin{align*}
\frac{ |r|^{\gamma+d} }{ \lfloor |r/2| \; \rfloor^d }   \sum_{ \hat{\omega}\in \llbracket1, \; \lfloor |r/2| \; \rfloor\rrbracket^d } \sum_{\hat{x}\in \mathbb{A}_{2|r|}^+} \big[ p_{\gamma}( \hat{z} - \hat{x} )   I^{n}_{\hat{x},\hat{z}}  (f | \nu_h^n ) + p_{\gamma}( \hat{y} - \hat{z} )  I^{n}_{\hat{z},\hat{y}}  (f | \nu_h^n ) \big],
\end{align*}
where $I^{n}_{\hat{k}_1,\hat{k}_2}  (f | \nu_h^n ):=\mcb I^{n}_{\hat{k}_1,\hat{k}_2}  (f | \nu_h^n )+\mcb I^{n}_{\hat{k}_2,\hat{k}_1}  (f | \nu_h^n )$, for any $\hat{k}_1$, $\hat{k}_2 \in \mathbb{Z}^d$. In the last display we also used the assumption that $Z_r\subset \mathbb{A}_{3|r|}^+$. 

From \eqref{norep2}, we have that $y_d(\hat{x}) \geq 0$ and $z_d(\hat{x}, \hat{\omega} ) \geq 0$, for every $\hat{x}\in \mathbb{A}_{2|r|}^+$ and $\hat{\omega}\in \llbracket1, \; \lfloor |r/2| \; \rfloor\rrbracket^d$, thus we conclude that all the bonds in the last display are \textit{fast}. Moreover, also from \eqref{norep2}, we get that any fixed bond in $\mcb F$ is counted at most once in the last display. The proof ends by combining this with Definition \ref{def:dir_form}.
\end{proof}
As a final note, the introduction of the intermediate step $\hat{z}( \hat{x}, \hat{w} )$ was for obtaining the extra factor $\lfloor |r/2| \; \rfloor^{-d}$ that is going to be crucial for obtaining \eqref{lemglobtbe} in Subsection \ref{secglobtbe} below. 

\subsection{Two-blocks estimate} \label{secglobtbe}

In this subsection we show \eqref{lemglobotbe}. First, we note that for any $k \geq 2$, any $\ell \in \mathbb{N}$ and any $\hat{x} \in \mathbb{Z}^d$, it holds
\begin{align*}
&\prod_{i=0}^{k-1} \overrightarrow{\xi}^{\ell} ( \hat{x} + i \ell \hat{e}_j) - \prod_{i=0}^{k-1} \overrightarrow{\xi}^{\varepsilon n} ( \hat{x} + i \varepsilon n \hat{e}_j)  = \sum_{i=1}^{k} B^{\ell,\varepsilon n}_{i,j,k}(\tau^{\hat{x}}\xi) \big[  \overrightarrow{\xi}^{\ell}\big( \hat{x} + (k-i) \ell \hat{e}_j  \big) - \overrightarrow{\xi}^{\varepsilon n}\big( \hat{x} + (k-i) \varepsilon n \hat{e}_j  \big) \big],
\end{align*}
where for any  $i \in \{1, \ldots, k\}$, $B^{k, \ell,\hat{x}}_{i,j}: \Omega \mapsto [0, \; N_{\text{e}}^{k-1}]$ is given by 
\begin{align} \label{defbtilxjeta}
B^{\ell,\varepsilon n}_{i,j,k}(\eta_{\star}):= \prod_{m=0}^{k-i-1} \overrightarrow{\eta_{\star}}^{\ell}( m \ell \hat{e}_j )  \prod_{p=k-i+1}^{k-1} \overrightarrow{\eta_{\star}}^{\varepsilon n}( p \varepsilon n \hat{e}_j ), \quad \eta_{\star} \in \Omega.
\end{align}
Therefore, \eqref{lemglobtbe} is a direct consequence of the next result.
\begin{lem} \label{lemrep2}
Let $\widetilde{G} : [0,T] \times \mathbb{R}^d  \mapsto \mathbb{R}$ be such that \eqref{boundrep} holds and $\ell \in \mathbb{N}_+$. Then, under Hypothesis \ref{hiporepl} \textbf{(ii)}, for any $k \geq 2$, $j \in \{1, \ldots, d\}$, $i \in \{1, \ldots, k\}$ and any $t \in [0,T]$, we have that
\begin{equation*} 
\varlimsup_{\varepsilon \rightarrow 0^{+}}\varlimsup_{n \rightarrow \infty}\mathbb{E}_{\mu_n} \Bigg[ \Bigg| \int_{0}^{t}  \sum_{\hat{x} \in \mathbb{A}_{3k\varepsilon n}^{+}} \frac{ \widetilde{G}_s \big( \tfrac{\hat{x}}{n} \big)  }{n^d} B^{\ell,\varepsilon n}_{i,j,k}(\tau^{\hat{x}}\xi_s^n) \big[ \overrightarrow{\xi}_s^{\ell}\big( \hat{x} + (k-i) \ell \hat{e}_j  \big)  - \overrightarrow{\xi}_s^{\varepsilon n}\big( \hat{x} + (k-i) \varepsilon n \hat{e}_j  \big) \big] \;\rmd s \, \Bigg| \Bigg]
\end{equation*}
is bounded from above by $f_{\star}(k^{\star}) [ k \ell^d]^{-1}$, where $f_{\star}$ is given by \eqref{deffstar}.
\end{lem}
\begin{proof}
Analogously as it was done in the proof for Lemma \ref{lemrep1}, it is enough to prove that 
\begin{align} \label{claim2blo}
\varlimsup_{D \rightarrow \infty} \varlimsup_{\varepsilon \rightarrow 0^{+}}\varlimsup_{n \rightarrow \infty} \int_0^t \mathcal{G}( n, \varepsilon, D, s ) \; \rmd s \leq f_{\star}(k^{\star}),
\end{align}
where for any $n \geq 1$, $\varepsilon >0$, $D>0$ and $s \in [0, \;T]$, $\mathcal{G}( n, \varepsilon, D, s )$ is given by
\begin{equation} \label{lim3grep}
	\begin{split}
		 \sup_f \Bigg\{ &  \sum_{\hat{x} \in \mathbb{A}_{3k\varepsilon n}^{+} }  \frac{    \big|\widetilde{G}_s \big( \tfrac{\hat{x}}{n} \big) \big|  }{n^d} \Bigg| \int_{\Omega}   B^{\ell,\varepsilon n}_{i,j,k}(\tau^{\hat{x}}\xi) \big[ \overrightarrow{\xi}^{\ell}\big( \hat{x} + (k-i) \ell \hat{e}_j  \big)  - \overrightarrow{\xi}^{\varepsilon n}\big( \hat{x} + (k-i) \varepsilon n \hat{e}_j  \big) \big]  f(\eta) \rmd \nu_{h}^n \Bigg| \\
		+& \frac{n^\gamma}{D n^d} \langle \mcb L_{n,\alpha}^\gamma \sqrt{f} , \sqrt{f} \rangle_{\nu_h^n} \Bigg\}:=\mathcal{G}( n, \varepsilon, D, s ).
	\end{split}
\end{equation}
Above, the supremum is carried over all densities $f$ with respect to $\nu_{h}^n$. Now define $N:= \lfloor \varepsilon n/ \ell \rfloor$. From \eqref{medempright}, we get that $ \overrightarrow{\xi}^{\varepsilon n}\big( \hat{x} + (k-i) \varepsilon n \hat{e}_j  \big)$ can be rewritten as
\begin{align*}
&   \frac{1}{(N\ell)^d}
 \sum_{ \hat{r} \in \llbracket 0,N-1\rrbracket^d } \sum_{ \hat{y} \in \llbracket 1,\ell\rrbracket^d } \xi \Big( \hat{x} + (k-i)  \ell \hat{e}_j + \hat{y}\odot\hat{1} + \ell \Big[(k-i)(N-1) \hat{e}_j + \hat{r} \odot \hat{1} \Big] \Big).
\end{align*} 
Last equality holds, since for every $m \in \{1, \ldots, d\}$, any $u_{m} \in \{1, \ldots, N \ell \}$ can be decomposed as $u_{m} = r_{m} \ell + y_{m}$, for some $r_{m} \in \{0, \ldots, N-1\}$ and some $y_{m} \in \{1, \ldots, \ell\}$. In this way, we get that $\overrightarrow{\xi}^{\ell}\big( \hat{x} + (k-i) \ell \hat{e}_j  \big)  - \overrightarrow{\xi}^{\varepsilon n}\big( \hat{x} + (k-i) \varepsilon n \hat{e}_j  \big)$ equals
\begin{align*}
\frac{1}{(N\ell)^d}
 \sum_{ \hat{r} \in \llbracket 0,N-1\rrbracket^d } \sum_{ \hat{y} \in \llbracket 1,\ell\rrbracket^d }  \sum_{w=0}^d [  \xi( \hat{z}_w ) - \xi( \hat{z}_{w+1} ) ],
\end{align*}
where for any $w \in \{0, 1, \ldots, d+1\}$, $z_w=z_w(\hat{x}, \hat{y}, \hat{r})$ is defined through
\begin{align}\label{defzwx}
    \begin{cases}
    \hat{z}_0:= \hat{x} + \hat{y}\odot\hat{1} + (k-i) \ell \hat{e}_j,
    \\
    \hat{z}_1:=\hat{z}_0 + \ell (k-i)(N-1) \hat{e}_j
    ,&
    \\
     \hat{z}_{w+1}:=\hat{z}_{w} +  \ell r_w \hat{e}_w
    , & 1\leq w\leq d.
    \end{cases}
\end{align}
In particular, for every $w \in \{0, \ldots, d\}$, it holds
\begin{equation} \label{sizeedgezw}
    |\hat{z}_{w+1} - \hat{z}_{w}| \leq k \ell N = k \varepsilon n. 
\end{equation}
Our goal is to move particles through $\{\hat{z}_{0},  \hat{z}_{d+1}\}$, which will be done by applying Lemma \ref{lem:moving}. However, since this lemma is only stated for bonds of the form $\{ \hat{w}_0, \hat{w}_1 \}$ such that $\hat{w}_1 - \hat{w}_0$ is a multiple of $\hat{e}_{p}$ for some $p \in \{1, \ldots, d\}$, we will perform movements through $\{ \hat{z}_0, \hat{z}_1\}, \{ \hat{z}_1, \hat{z}_2\}, \ldots, \{ \hat{z}_d, \hat{z}_{d+1} \}$. Recalling that $\widetilde{G}$ satisfies \eqref{boundrep}, Lemma \ref{lem:anti-dir} with \eqref{defbtilxjeta}, we conclude that the sum in \eqref{lim3grep} is bounded from above by
\begin{align}
+& \frac{N_{\text{e}}^{k-2}}{2} \sum_{\hat{x} \in \mathbb{A}_{3k\varepsilon n}^{+} }  \frac{    H \big( \tfrac{\hat{x}}{n} \big) }{(nN\ell)^d} \sum_{ \hat{r} \in \llbracket 0,N-1\rrbracket^d } 
\sum_{ \hat{y} \in \llbracket 1,\ell\rrbracket^d } 
\sum_{w=0}^d 
 \int_{\Omega}   a_{ \hat{z}_{w+1}  , \hat{z}_w }(\eta) 
|\nabla_{\hat{z}_{w+1} , \hat{z}_w}f(\eta)|
\rmd \nu_{h}^n + C_0 M \varepsilon \label{bonundsup2bl2} \\ 
+& \frac{N_{\text{e}}^{k-2}}{2} \sum_{\hat{x} \in \mathbb{A}_{3k\varepsilon n}^{+} }  \frac{     H \big( \tfrac{\hat{x}}{n} \big) }{(nN\ell)^d}  \sum_{ \hat{r} \in \llbracket 0,N-1\rrbracket^d } 
\sum_{ \hat{y} \in \llbracket 1,\ell\rrbracket^d } 
\sum_{w=0}^d \int_{\Omega}   a_{ \hat{z}_w , \hat{z}_{w+1}   }(\eta)
|\nabla_{\hat{z}_w , \hat{z}_{w+1}}f(\eta) | 
\rmd \nu_{h}^n \label{bonundsup2bl1}.
\end{align}
The second term in \eqref{bonundsup2bl2} arises from \eqref{sizeedgezw} and \eqref{boundKM}; here $C_0$ is a positive constant depending only on $N_{\text{e}}, k, d$ and $h$. The first term in \eqref{bonundsup2bl2} can be treated exactly in the same way as the one in \eqref{bonundsup2bl1}, thus we will focus on the later one. To treat it, we will move particles between $\hat{z}_w$ and $\hat{z}_{w+1}$ by applying Lemma \ref{lem:moving}.

Recall the definition of $k^{\star}$ in Definition \ref{defkstar}. For the remainder of the proof, we analyze two cases separately: $k^{\star}=1$; and $k^{\star} >1$.

\textbf{I) The case} $k^{\star}=1$. 

From Lemma \ref{lemradnik}, $\nu_h^n(\eta) / \nu_h^n(\eta^{\hat{z}_w , \hat{z}_{w+1}})$ is uniformly bounded from above by $\widetilde{b}_h  N_{ \text{e} }^{2}$. Combining this with Remark \ref{remyoung} and the estimate $a_{ \hat{z}_w , \hat{z}_{w+1}   }(\eta) \leq  N_{ \text{e} }^{2}$, the term in \eqref{bonundsup2bl1} is bounded from above by
\begin{align}  
& \frac{N_{\text{e}}^{k}}{2} \sum_{\hat{x} \in \mathbb{A}_{3k\varepsilon n}^{+} }  \frac{ H \big( \tfrac{\hat{x}}{n} \big) }{(nN\ell)^d}  \sum_{ \hat{r} \in \llbracket 0,N-1\rrbracket^d } 
\sum_{ \hat{y} \in \llbracket 1,\ell\rrbracket^d } 
\sum_{w=0}^d A ( 1 + \widetilde{b}_h  N_{ \text{e} }^{2}  )  \label{kstara}\\
+ & \frac{N_{\text{e}}^{k-2}}{4A} \sum_{\hat{x} \in \mathbb{A}_{3k\varepsilon n}^{+} }  \frac{  H \big( \tfrac{\hat{x}}{n} \big) }{(nN\ell)^d}  \sum_{ \hat{r} \in \llbracket 0,N-1\rrbracket^d } 
\sum_{ \hat{y} \in \llbracket 1,\ell\rrbracket^d } 
\sum_{w=0}^d \int_{\Omega}   a_{ \hat{z}_w , \hat{z}_{w+1}   }(\eta)
[\nabla_{\hat{z}_w , \hat{z}_{w+1}} \sqrt{f(\eta)}  ]^2
\rmd \nu_{h}^n, \label{kstarb}
\end{align}
for any $A>0$. From \eqref{boundKM}, the last display is bounded from above by
\begin{align*}
A M   + \frac{ \varepsilon^{-\delta} }{A(nN\ell)^d}   \sum_{ \hat{r} \in \llbracket 0,N-1\rrbracket^d } 
\sum_{ \hat{y} \in \llbracket 1,\ell\rrbracket^d } 
\sum_{w=0}^d \sum_{\hat{x} \in \mathbb{A}_{3k\varepsilon n}^{+} }  \int_{\Omega}   a_{ \hat{z}_w , \hat{z}_{w+1}   }(\eta)
[\nabla_{\hat{z}_w , \hat{z}_{w+1}} \sqrt{f(\eta)}  ]^2
\rmd \nu_{h}^n,
\end{align*}
multiplied by a positive constant depending only on $N_{\text{e}}, k, d$ and $h$. From Lemma \ref{lem:moving} and \eqref{sizeedgezw}, the second term in the last display is bounded from above by
\begin{align*}
    \sum_{ \hat{r} \in \llbracket 0,N-1\rrbracket^d } 
\sum_{ \hat{y} \in \llbracket 1,\ell\rrbracket^d } 
\sum_{w=0}^d \frac{ \varepsilon^{-\delta} C^{\star}
         (k \varepsilon n)^{\gamma} }{A(nN\ell)^d}  \mcb D_{n,\mcb F}^\gamma (\sqrt{f} | \nu_h^n ) \leq  \frac{ \varepsilon^{\gamma-\delta} }{An^d}  C^{\star} k^{\gamma} 2d n^{\gamma} \mcb D_{n,\mcb F}^\gamma (\sqrt{f} | \nu_h^n ), 
\end{align*}
where $C^{\star}$ is the constant given in  Lemma \ref{lem:moving}. Applying exactly the same reasoning for treating the term in \eqref{bonundsup2bl2}, we conclude that the sum in \eqref{lim3grep} is bounded from above by
\begin{align} \label{kstarc}
C' \Big( A + \frac{\varepsilon^{\gamma-\delta}}{A} \frac{n^{\gamma}}{n^d} \mcb D_{n,\mcb F}^\gamma (\sqrt{f} | \nu_h^n )  \Big) + C_0 M \varepsilon,
\end{align}
for any $A >0$. In the last line, $C'>0$ is some positive constant depending only on $h$, $k$, $d$, $k^{\star}$, $b^{\star}$, $\gamma$, $M$ and $K$. Observe that the second term inside the supremum in \eqref{lim3grep} is bounded from above by the display in the left-hand side of \eqref{boundsup2}, due to
 \eqref{h2}, \eqref{boundlip} and \eqref{defDnFS}. Thus, combining the last display with \eqref{boundsup2}, the supremum in \eqref{lim3grep} is bounded from above by  
\begin{align*}
\frac{M_h}{D}  (1 + f_{\infty}')  + C_0 M \varepsilon + C' A + \Big( C' \frac{\varepsilon^{\gamma-\delta}}{A} - \frac{1}{8 D} \Big)  \frac{n^{\gamma}}{n^d} \mcb D_{n,\alpha}^\gamma  (\sqrt{f} | \nu_{h}^n ).
\end{align*}
where $M_h$ is a constant depending on $h$ and $\alpha$. Choosing $A=8D C' \varepsilon^{\gamma-\delta}$ and $D=\varepsilon^{( \delta - \gamma)/2}$, we conclude that the term on the left-hand side of \eqref{claim2blo} is equal to zero, since $\delta < \gamma$. This ends the proof for the case $k^{\star} = 1$; next, we treat the case $k^{\star} > 1$. 

\textbf{II) The case} $k^{\star} > 1$. 

In this case, for any $w \in \{0, \ldots, d\}$, let $\Omega_w^{\star}:=\Omega_1( \hat{z}_w )$, resp. $\Omega_w^{\star}:=\Omega_1( \hat{z}_{w+1})$ if $k^{\star}:=\underline{k}^+$, resp. $k^{\star}:=\underline{k}^-$, where $\Omega_1(\cdot)$ is given in Definition \ref{defomega12} (the choice of $m=1$ in $\Omega_m(\cdot)$ here is totally arbitrary). According to Lemma \ref{lem:moving}, a sufficient condition for transporting a particle between $\hat{z}_w$ and $\hat{z}_{w+1}$ is to construct some configuration $\eta_{\star} \in \Omega_{w}^{\star}$.

In what follows, we will make use of $\eqref{cons-series}$ to obtain $\eta_{\star} \in  \Omega_w^{\star}$ starting from $\eta \in \Omega_w$ by performing nearest-neighbor jumps, for some appropriate space $\Omega_w$, whose elements are required to have a minimum number of particles/"holes" in a $d$-dimensional box of side length $\ell$. More exactly, for every each $w \in \{0, 1, \ldots, d+1\}$, let $\hat{v}_w=\hat{v}_w(\hat{x}, \hat{r})$ be given by $\hat{v}_{w}:=\hat{z}_w-\hat{y}\odot\hat{1}$ and
\begin{equation}  \label{defomega12conf}
\Omega_w:=
\begin{cases}
\quad \Big\{ 
\eta \in \Omega: \quad \overrightarrow{\eta}^{\ell}(\hat{v}_w)\geq \frac{2 k^\star -2}{\ell^d}\Big\}
\quad & \text{if} \quad k^{\star}=\underline{k}^+; \\
\quad \Big\{ 
\eta \in \Omega: \quad \overrightarrow{ \widetilde{\eta} }^{\ell}(\hat{v}_{w+1})\geq \frac{2 k^\star -2}{\ell^d}\Big\}
\quad & \text{if} \quad k^{\star}=\underline{k}^+.  
\end{cases}
\end{equation}
In \eqref{defomega12conf}, $k^{\star}$ is given in Definition \ref{defkstar}. 
In particular, for every $\eta \in \Omega_w^c$, it holds
\begin{align*}
a_{ \hat{z}_{w}  , \hat{z}_{w+1} }(\eta)  = \eta ( \hat{z}_{w} ) \widetilde{\eta}( \hat{z}_{w+1} )   =  \eta ( \hat{v}_w + \hat{y}\odot\hat{1} ) \widetilde{\eta}( \hat{v}_{w+1} + \hat{y}\odot\hat{1} )  
\end{align*}
is equal to zero for at least $\ell^d - (2 k^\star -3) > \ell^d - 2 k^\star$ possibilities for $\hat{y} \in \llbracket 1,\ell\rrbracket^d$. This leads to 
\begin{align} \label{boundelld}
&  \frac{N_{\text{e}}^{k-2}}{2}\sum_{\hat{x} \in \mathbb{A}_{3k\varepsilon n}^{+} }  \frac{  H \big( \tfrac{\hat{x}}{n} \big)  }{(nN\ell)^d}  \sum_{ \hat{r} \in \llbracket 0,N-1\rrbracket^d } \sum_{ \hat{y} \in \llbracket 1,\ell\rrbracket^d } 
\sum_{w=0}^d \int_{\Omega_w^c}   a_{ \hat{z}_w , \hat{z}_{w+1}   }(\eta)
|\nabla_{\hat{z}_w , \hat{z}_{w+1}}f(\eta) | 
\rmd \nu_{h}^n \leq \frac{4 M k^\star d N^k_{ \text{e} } }{a_h(1-b_h) \ell^d}.
\end{align}
In the last display we combined \eqref{boundKM} with the fact that $f$ is a density with respect to $\nu_h^n$.  Thus, in order to treat \eqref{bonundsup2bl1}, we focus now on
\begin{align} \label{2sup1termb4}
\frac{N_{\text{e}}^{k-2}}{2} \sum_{\hat{x} \in \mathbb{A}_{3k\varepsilon n}^{+} }  \frac{   H \big( \tfrac{\hat{x}}{n} \big)  }{(nN\ell)^d} \sum_{ \hat{r} \in \llbracket 0,N-1\rrbracket^d } \sum_{ \hat{y} \in \llbracket 1,\ell\rrbracket^d } \sum_{w=0}^d \int_{\Omega_w}   a_{ \hat{z}_w , \hat{z}_{w+1}   }(\eta)| f( \eta ) - f ( \eta^{\hat{z}_w , \hat{z}_{w+1}  } )  | \rmd \nu_{h}^n.
\end{align}

In order to estimate the term in the last line, we perform nearest-neighbor jumps starting from some configuration $\eta \in \Omega_w$, which leads to some $\eta_{\star} \in \Omega_w^{\star}$; then we apply Lemma \ref{lem:moving} to move from $\eta_{ \star  } \in \Omega_w^{\star}$ to $\eta_{ \star }^{\hat{z}_w , \hat{z}_{w+1} }$; and finally we undo the original sequence of nearest-neighbor jumps, by performing them in reverse order and reverse direction. The final configuration $\eta^{\hat{z}_w , \hat{z}_{w+1} }$. This procedure is done for every $\hat{r} \in \llbracket 0,N-1\rrbracket^d$, $\hat{y} \in \llbracket 1,\ell\rrbracket^d$ and $w \in \{0, 1, \ldots, d\}$, in order to transport a particle from $\hat{z}_{w}$ to $\hat{z}_{w+1}$.  

If $k^{\star}=\underline{k}^+$, resp. $k^{\star}=\underline{k}^-$, we will move particles, resp. "holes", to the neighborhood of $\hat{z}_{w}$, resp. of $\hat{z}_{w+1}$, since in order to produce a configuration in $\eta_{\star} \in \Omega_w^{\star}$ we require a window of $2k^\star  -2$ nonempty sites, resp. not totally occupied sites, aligned along the direction of $\hat{e}_1$. Keeping this in mind, for each $w \in \{0,1, \ldots, d\}$, we define $N_{w}^{\star}:=|E_{2}^{\eta,2}(\hat{z}_w)|$, if $k^{\star}=\underline{k}^+$; and $N_{w}^{\star}:=|E_{2}^{\eta,2}(\hat{z}_{w+1})|$, if $k^{\star}=\underline{k}^-$. Here, $E_{1}^{\eta,2}(\cdot)$ is given in Definition \ref{defomega12}. In particular, $N_{w}^{\star} \leq 2k^\star  -2$.

If $N_{w}^{\star} = 0$, then $ \eta \in \Omega_w^{\star}$ and Lemma \ref{lem:moving} can be applied directly. Otherwise, $\eta \notin \Omega_w^{\star}$ and we require $N_{w}^{\star}$ extra particles/"holes". From \eqref{defomega12conf}, there are at least $2k^\star  -2$ such particles/"holes" in $B_w ( \ell)$, the $d$-dimensional box defined by
\begin{align*}
B_w ( \ell)
:=
\begin{cases}
\quad \Big\{ \hat{v}_w +\hat{u}\odot\hat{1}, \quad \hat{u}\in\llbracket 1,  \ell \rrbracket^d  \Big\}, \quad & \text{if} \quad k^{\star}=\underline{k}^+, \\
\quad \Big\{ 
\hat{v}_{w+1} +\hat{u}\odot\hat{1}, \quad \hat{u}\in\llbracket 1, \ell \rrbracket^d  \Big\}, \quad & \text{if} \quad k^{\star}=\underline{k}^-.
\end{cases}
\end{align*}
Next, we summarize the three steps we perform to treat the term in \eqref{2sup1termb4}.
\begin{enumerate}
    \item If $k^{\star}=\underline{k}^+$, resp. $k^{\star}=\underline{k}^-$, choose any $N_{w}^{\star}$ particles, resp. "holes", in $B_w ( \ell)$ and move them to $\{ \hat{z}_{w} + \hat{e}_1, \ldots, \hat{z}_{w} +2 (k^{\star} -1 ) \hat{e}_1 \}$, resp. $\{ \hat{z}_{w+1} + \hat{e}_1, \ldots, \hat{z}_{w+1} +2 (k^{\star} -1 ) \hat{e}_1 \}$, by using a total of $S_{\text{SEP}}= S_{\text{SEP}}(\hat{x}, \hat{r}, \hat{y}, w, \eta)$ nearest-neighbor jumps, in such a way that each particle, resp. each "hole" only passes through a fixed site at most once. In particular, $S_{\text{SEP}} \leq N_{w}^{\star} d \ell \leq 2 (k^{\star} -1) d \ell$. Thus, in the same way as it was done in the proof of Lemma \ref{lem:moving}, for every $\eta \in \Omega_w$ we have a path $\gamma\equiv\gamma(\hat{x},\hat{r},\hat{y},w,\eta)$ of the form $\gamma_{m+1}(\eta):=(\gamma_{m}(\eta))^{\hat{x}_m^{0},\hat{x}_m^{1}}$, for $0\leq m \leq S_{\text{SEP}}-1$, through its ordered sequence of bonds, that we represent as $\mapsto_{m=0}^{S_{\text{SEP}}-1} \{\hat{x}_m^{0},\hat{x}_m^{1}\}$. Letting $\gamma_0(\eta):=\eta$ we have that $\gamma_{S_{\text{SEP}}}(\eta) \in \Omega_{w}^{\star}.$. Moreover, by construction, it holds
\begin{equation} \label{embeddet}
\mapsto_{m=0}^{S_{\text{SEP}}-1} \{\hat{x}_m^{0},\hat{x}_m^{1}\} \subset \big\{ \{ \hat{u}_1, \hat{u}_2 \}: \quad |\hat{u}_2 - \hat{u}_1 | = 1, \; \hat{u}_1, \hat{u}_2 \in B_w ( \ell) \big\},
\end{equation}
In the remainder of this proof, we denote $(\gamma_{m}(\eta))$ by $\eta_m$, for every $m \in \{0,\ldots, S_{\text{SEP}} \}$.  Another condition that we impose on $\mapsto_{m=0}^{S_{\text{SEP}}-1} \{\hat{x}_m^{0},\hat{x}_m^{1}\}$ is that $a_{ \hat{z}_w , \hat{z}_{w+1}   }  (\eta_{S_{\text{SEP}}} )  \geq 1$, in order to apply Lemma \ref{lem:moving}.         
  
    \item After all the $S_{\text{SEP}}$ jumps are performed, the new configuration is $\eta_{S_{\text{SEP}}}\in \Omega_{w}^{\star}$. Now we execute the procedure described in the proof of Lemma \ref{lem:moving} to send a particle from $\hat{z}_{w}$ to $\hat{z}_{w+1}$; afterwards, the new configuration is  $\eta_{S_{\text{SEP}}}^{ \hat{z}_{w}, \hat{z}_{w+1} }$. Observe that this requires at most $2 k^{\star}$ jumps.

    \item The final step consists in bringing the $N_{w}^{\star}$ auxiliary particles back to their original sites in the starting configuration $\eta$. In order to do so, starting from $\eta_{S_{\text{SEP}}}^{ \hat{z}_{w}, \hat{z}_{w+1} }$, consider the reverse sequence $\mapsto_{m=0}^{S_{\text{SEP}}}\{\hat{x}_{S_{\text{SEP}}-m}^{1},\hat{x}_{S_{\text{SEP}}-m}^{0}\}$, thus sending back one particle/"hole" with each bond. This leads to the reverse path $\{\eta_{S_{\text{SEP}}-m}^{ \hat{z}_{w}, \hat{z}_{w+1} }, \; 0 \leq m \leq S_{\text{SEP}} \} $; the final configuration is $\eta^{ \hat{z}_{w}, \hat{z}_{w+1} }$, as desired. 
\end{enumerate}

With this procedure in mind, we get from the estimate $a_{ \hat{z}_w , \hat{z}_{w+1}   }(\eta) \leq N_{\text{e}}^{2}$ and the condition $a_{ \hat{z}_w , \hat{z}_{w+1}   }(\eta_{S_{\text{SEP}}}) \geq 1$ that the term in \eqref{2sup1termb4} is bounded from above by
\begin{align}
 & \frac{N_{\text{e}}^{k}}{2} \sum_{\hat{x} \in \mathbb{A}_{3k\varepsilon n}^{+} }  \frac{    H \big( \tfrac{\hat{x}}{n} \big)   }{(nN\ell)^d} \sum_{ \hat{r} \in \llbracket 0,N-1\rrbracket^d } \sum_{ \hat{y} \in \llbracket 1,\ell\rrbracket^d } \sum_{w=0}^d \int_{\Omega_w^{\star}}   a_{ \hat{z}_w , \hat{z}_{w+1}   }(\eta_{S_{\text{SEP}}}) \big| \nabla_{ \hat{z}_{w}, \hat{z}_{w+1} }
    f(\eta_{S_{\text{SEP}}})    \big|\rmd \nu_{h}^n ( \eta_{S_{\text{SEP}}} ) \label{intsup1termb4b} \\
+ & \frac{N_{\text{e}}^{k}}{2} \sum_{\hat{x} \in \mathbb{A}_{3k\varepsilon n}^{+} }  \frac{  H \big( \tfrac{\hat{x}}{n} \big)  }{(nN\ell)^d} \sum_{ \hat{r} \in \llbracket 0,N-1\rrbracket^d } \sum_{ \hat{y} \in \llbracket 1,\ell\rrbracket^d } \sum_{w=0}^d \int_{\Omega_w}   \big|
        f(\eta) - f(\eta_{S_{\text{SEP}}}) 
        \big| \rmd \nu_{h}^n (  \eta) \label{intsup1termb4a} \\
+ & \frac{N_{\text{e}}^{k}}{2} \sum_{\hat{x} \in \mathbb{A}_{3k\varepsilon n}^{+} }  \frac{   H \big( \tfrac{\hat{x}}{n} \big)  }{(nN\ell)^d} \sum_{ \hat{r} \in \llbracket 0,N-1\rrbracket^d } \sum_{ \hat{y} \in \llbracket 1,\ell\rrbracket^d } \sum_{w=0}^d \int_{\Omega_w}    \big| 
    f (\eta_{S_{\text{SEP}}}^{ \hat{z}_{w}, \hat{z}_{w+1} } ) -  f (\eta^{ \hat{z}_{w}, \hat{z}_{w+1} } ) 
    \big| \rmd \nu_{h}^n \big(  \eta_{S_{\text{SEP}}}^{ \hat{z}_{w}, \hat{z}_{w+1} } \big), \label{intsup1termb4c}
\end{align}
multiplied by $\widetilde{b}_h \; N_{ \text{e} }^{2 k^{\star} (2 d \ell +1 ) }$; this constant comes from Lemma \ref{lemradnik}. Applying exactly the same reasoning that was applied to treat the term in \eqref{bonundsup2bl1} in the case $k^{\star}=1$, the term in \eqref{intsup1termb4b} is bounded from above by 
\begin{equation} \label{kstard}
C' \Bigg[ A + \frac{\varepsilon^{\gamma-\delta}}{A} \frac{n^{\gamma}}{n^d} \mcb D_{n,\mcb F}^\gamma (\sqrt{f} | \nu_h^n )  \Bigg].
\end{equation}
In the last line, $C'$ is the same positive constant given by \eqref{kstarc}. Next we analyze the term in \eqref{intsup1termb4a}. From Remark \ref{remyoung} and the fact that $f$ is a density with respect to $\nu_h^n$, the integral in \eqref{intsup1termb4a} is bounded from above by
\begin{align*} 
 & \tilde{A} \Bigg( \int_{\Omega_w}    f(\eta) \rmd \nu_h^n ( \eta ) + [\widetilde{b}_h  N_{ \text{e} }^{2 k^{\star}  d \ell }]  \int_{\Omega_w} f(\eta_{S_{\text{SEP}}}) \rmd \nu_h^n (\eta_{S_{\text{SEP}}} ) \Bigg) \\
 + & [\widetilde{b}_h  N_{ \text{e} }^{2 k^{\star}  d \ell }] \int_{\Omega_w}  \frac{ k^{\star} d \ell }{ \tilde{A}  } \sum_{m=0}^{S_{\text{SEP}} - 1}   \big[\nabla _{ \hat{x}_m^0, \hat{x}_m^1 }   \sqrt{f ( \eta_{m} )}]^2 a_{ \hat{x}_m^0, \hat{x}_m^1 } ( \eta_m ) \rmd \nu_h^n ( \eta_m ) \\
\leq & C_1 ( \ell) \tilde{A}  +  \frac{ C_1 ( \ell)  \varepsilon^{-\delta}  }{\tilde{A} (nN\ell)^d} \sum_{ \hat{r} \in \llbracket 0,N-1\rrbracket^d } \sum_{ \hat{y} \in \llbracket 1,\ell\rrbracket^d } \sum_{w=0}^d \sum_{\hat{x} \in \mathbb{A}_{3k\varepsilon n}^{+} }  \underset{  | \hat{u}_1 - \hat{u}_2 | =1  }{ \sum_{ \hat{u}_1, \hat{u}_2 \in B_w^{ \text{ri} } (\ell )    } } I^{\text{SEP}}_{\hat{u}_1,\hat{u}_2}  (\sqrt{f} | \nu_h^n ) \\
\leq & C_1 ( \ell) \tilde{A}  + \frac{ C_1 ( \ell)  \varepsilon^{-\delta} (d+1) \ell^d }{\tilde{A} n^d p( \hat{e}_1 ) }  \mcb D_{ \text{SEP} }^\gamma ( \sqrt{f} | \nu_h^n ) \leq C_1 ( \ell) \tilde{A} + \frac{ C_1 ( \ell) \varepsilon^{-\delta} (d+1) \ell^d }{\tilde{A} n^d p( \hat{e}_1 ) } \mcb D_{n,\alpha}^\gamma ( \sqrt{f} | \nu_h^n ),
\end{align*}
for any $\tilde{A} >0$, where $I^{\text{SEP}}_{\cdot, \cdot}  (\sqrt{f} | \nu_h^n )$ and $\mcb D_{ \text{SEP} }^\gamma ( \sqrt{f} | \nu_h^n )$ are given in \eqref{defISEP} and \eqref{defDSEP}, respectively. In the last display we applied \eqref{boundKM} in the same way that we did for treating \eqref{kstara} and \eqref{kstarb}. Above and in the remainder of this proof, for $m \in \{1,2\}$, $C_m(\ell)$ is some positive constant depending only on $h$, $k$, $d$, $h$ $k^{\star}$, $b^{\star}$, $\gamma$, $M$, $K$ and $\ell$.

This very same reasoning can be applied to treat the term in \eqref{intsup1termb4c}, which is also bounded from above by the last display. Moreover, recall that the term in \eqref{2sup1termb4} is bounded from above by $\widetilde{b}_h \; N_{ \text{e} }^{2 k^{\star} (2 d \ell +1 ) }$, multiplied by the sum of the terms in \eqref{intsup1termb4b}, \eqref{intsup1termb4a} and \eqref{intsup1termb4c}. Therefore, combining all of this with \eqref{kstard} and \eqref{boundelld}, and applying exactly the same arguments for estimating the term in \eqref{bonundsup2bl2}, the sum in \eqref{lim3grep} is bounded from above by
\begin{align*}
C_2 ( \ell) [A +  \tilde{A}] +  C_2 ( \ell) \Bigg( \frac{\varepsilon^{\gamma-\delta}}{A} \frac{n^{\gamma}}{n^d} + \frac{  \varepsilon^{-\delta}  }{\tilde{A} n^d  } \Bigg) \mcb D_{n,\alpha}^\gamma ( \sqrt{f} | \nu_h^n ) + \frac{8 M k^\star d N^k_{ \text{e} } }{a_h(1-b_h) \ell^d} + C_0 M \varepsilon .
\end{align*}
In the same way as we did in the end of the proof for the case $k^{\star}=1$, we note that the second term inside the supremum in \eqref{lim3grep} is bounded from above by the left-hand side of \eqref{boundsup2}, due to
 \eqref{h2}, \eqref{boundlip} and \eqref{defDnFS}. Thus, combining the last display with \eqref{boundsup2}, we have that the supremum in \eqref{lim3grep} is bounded from above by the sum of $f_{\star}(k^{\star}) [T k \ell^d]^{-1}$, with  
\begin{align*}
\frac{M_h}{D}  (1 + f_{\infty}')  + C_0 M \varepsilon + C_2 ( \ell) [A +  \tilde{A}] +  \Bigg[ C_2 ( \ell) \varepsilon^{-\delta} \Bigg( \frac{\varepsilon^{\gamma}}{A}  + \frac{   n^{-\gamma} }{\tilde{A}   } \Bigg) - \frac{1}{16 D} \Bigg] \frac{n^{\gamma}}{n^{d}}  \mcb D_{n,\alpha}^\gamma ( \sqrt{f} | \nu_h^n ), 
\end{align*}
where $M_h$ is a constant depending on $h$ and $\alpha$. Choosing $A=32 C(\ell) D \varepsilon^{\gamma-\delta}$, $\tilde{A} =32 C(\ell) D \varepsilon^{-\delta} n^{-\gamma}$ and $D=\varepsilon^{( \delta - \gamma)/2}$, we conclude that the last display vanishes when first $n \rightarrow \infty$, and then $\varepsilon \rightarrow 0^+$, since $\delta < \gamma$. This leads to \eqref{claim2blo}, and the proof for the case $k^{\star} > 1$ ends.
\end{proof}


\appendix

\section{Fractional operators in $L^1(\mathbb{R}^d)$} \label{secfracoper}

In this section, our goal is to prove Proposition \ref{propL1alpha}. In order to do so, we extend the space of test functions $S_{\gamma,\kappa}^{d}$, given in \eqref{defSalgamd}. 
\begin{definition}  \label{deftesconttime}
	For each $d\in\mathbb{N}_+$ and $k \in \mathbb{N}$, introduce the following sets.
	\begin{enumerate} 
		\item Let $\mathcal{S}^k(\mathbb{R}^d)$ be the subset of $C^{0,k}([0, \infty) \times \mathbb{R}^d)$ whose elements $G$ satisfy, for every $r \in \mathbb{N}$ and every $j \in \{0, 1, \ldots, k\}$, 
		\begin{align} \label{defschw}
		\nnorm{G}_{r,j}:=
		|\hat{u}|^r \sup_{(s, \hat{u}) \in [0, T] \times \mathbb{R}^d } [G_s]_j(\hat{u}) < \infty,
		\end{align}
		 where $[G_s]_j(\hat{u})$ is a control for the absolute value of the spatial partial derivatives of $G_s$ up to order $j$, and is given by 
		 \begin{align} \label{semnorm}
		 [G_s]_j:=|G_s(\hat{u})| 
		 + \sum_{r=1}^j  \sum_{i_1=1}^d \ldots \sum_{i_r=1}^d | \partial_{i_1} \ldots \partial_{i_r }G_s |.
		 \end{align}
		\item Let $\mathcal{S}_c^k(\mathbb{R}^d)$ be the subset of $\mathcal{S}^k(\mathbb{R}^d)$ such that, for any $G\in \mathcal{S}^k(\mathbb{R}^d)$,
		\begin{align*}
			\exists K \subset \mathbb{R}^d \; \text{compact such that} \; |\hat{u}| \notin K \Rightarrow \; \sup_{s \in [0, \; T]} \big|G_s( \hat{u}) \big| = 0;
		\end{align*}
		\item Define the space $\mathcal S_{\gamma}^{d}$ by
		\begin{equation} \label{defSalgamd3}
		\mathcal S_{\gamma}^{d} := 
			\begin{cases}
				\mathcal{S}_c^2(\mathbb{R}^d), \quad & \gamma \in (0,1], \; d \geq 1 \quad \text{or} \quad \gamma \in (1,2), \; d \geq 2, \\
				\mathcal{S}^2(\mathbb{R}^d), \quad & \gamma \in (1,2), \; d=1.
			\end{cases}
		\end{equation} 
		\item Similarly to \eqref{defGdisc}, we introduce $\mathcal S_{\gamma,0}^{d}$ as the set of functions $G:[0, T] \times \mathbb{R}^d \mapsto \mathbb{R}$ such that
		\begin{align} \label{defGdisc2}
			\exists G^{\pm}\in \mathcal S_{\gamma}^{d}: G_s(\hat{u}) =\mathbbm{1}_{\{ u_d < 0 \}} G^{-}_s(\hat{u}) + \mathbbm{1}_{ \{ u_d \geq 0\}} G^{+}_s(\hat{u}),
		\end{align}
		for any $s \in [0, T]$, any $\gamma \in (0,2)$ and $d \geq 1$.
			\end{enumerate}
\end{definition}
For future reference, we note that from \eqref{defschw},
\begin{equation} \label{defnabladelta}
	\forall H \in \mathcal{S}^0(\mathbb{R}^d), \quad \nnorm{H}_{0,0} < \infty, 
	\quad\text{and}\quad
	\forall H \in \mathcal{S}^1(\mathbb{R}^d), \quad \nnorm{H}_{0,1}< \infty.
\end{equation}
\begin{rem} \label{remgentest}
	We note that, in particular, by choosing $G^{-}=G^{+}$, one sees that $ \mathcal S_{\gamma}^{d} \subsetneq \mathcal S_{\gamma,0}^{d}$. Moreover, combining \eqref{defSalgamd} with \eqref{defSalgamd3} and \eqref{defGdisc2}, we conclude that for any $\gamma \in (0, \; 2)$ and $d \geq 1$, it holds that $ S_{\gamma,\kappa}^{d} \subsetneq \mathcal S_{\gamma}^{d} \subsetneq \mathcal S_{\gamma,0}^{d}$ resp. $ S_{\gamma,\kappa}^{d} \subsetneq \mathcal S_{\gamma,0}^{d}$, if $\kappa >0$ resp. $\kappa=0$.
\end{rem}
Next, we state a couple of results (that will be proved later) in order to obtain Proposition \ref{propL1alpha}.
\begin{prop} \label{propLinfty}
	Let $\gamma \in (0, \; 2)$, $d \geq 1$ and $G \in  \mathcal S_{\gamma}^{d}$. Then there exists $H_1^G \in L^1(\mathbb{R}^d)$ and a constant $C_1\equiv C_1(G)>0$ such that, for all $\hat{u} \in \mathbb{R}^{d}$,
	\begin{equation} \label{boundH1G}
			\sup_{s \in [0, T]} \big|\Delta^{\gamma/2}G_s ( \hat{u} ) \big| 
			\leq H_1^G( \hat{u} )
			\leq C_1.
	\end{equation}
\end{prop}

\begin{prop} \label{propL1alpha0}
	Let $\gamma \in (0, \; 2)$, $d \geq 1$ and $G \in  \mathcal S_{\gamma,0}^{d}$. Then there exists $H_2^G \in L^1(\mathbb{R}^d)$ and a constant $C_2\equiv C_2(G)$ such that, for all $\hat{u} \in \mathbb{R}^{d}$, 
	\begin{equation} \label{boundH2G}
			\sup_{s \in [0,T]} \big| \Delta_{\star}^{\gamma/2}G_s( \hat{u} ) \big| 
			\leq H_2^G( \hat{u} )
			\leq C_2\left(1 + |u_d|^{- \gamma/2}\right).
	\end{equation}
\end{prop}
\begin{proof}[Proof of Proposition \ref{propL1alpha}]
	Recall the operator $\mathbb{L}_{\kappa}^{\gamma}$ given in Definition \ref{def:dist_lap}. For $\kappa=1$ resp. $\kappa=0$, we identify $\mathbb{L}_{1}^{\gamma}= \Delta^{\gamma/2}$ resp. $\mathbb{L}_{0}^{\gamma} =\Delta_{\star}^{\gamma/2}$, and Proposition \ref{propL1alpha} is a direct consequence of Remark \ref{remgentest} and Proposition \ref{propLinfty} resp. of Remark \ref{remgentest} and Proposition \ref{propL1alpha0}. 
	
	It remains to treat the regime $\kappa \in (0, 1) \cup (1, \infty)$. In this case, from Propositions \ref{propLinfty} and \ref{propL1alpha0}, there exist $H_1^G$ and $H_2^G$ satisfying \eqref{boundH1G} and \eqref{boundH2G}, respectively. Thus, $\sup_{s \in [0, \; T]} \big| \mathbb{L}_{\kappa}^{\gamma} G_s(\hat{u}) \big|$ is bounded from above by
	\begin{align*}
		\sup_{s \in [0, \; T]} \big| \kappa  \Delta^{\gamma/2}G_s(\hat{u}) + (1-\kappa)\Delta_{\mathbb{R}^{d\star}}^{\gamma/2}G_s(\hat{u}) \big| \leq  \kappa H_1^G(\hat{u}) + (1-\kappa) H_2^G(\hat{u}),
	\end{align*}
	for every $\hat{u} \in \mathbb{R}^d$. Above, we applied Remark \ref{remgentest}. In order to finish the proof, it is enough to define $H^G$ through $H^G ( \hat{u}):= \kappa H_1^G(\hat{u}) + (1-\kappa) H_2^G(\hat{u})$, for any $\hat{u} \in \mathbb{R}^d$.
\end{proof}
It still remains to prove Propositions \ref{propLinfty} and \ref{propL1alpha0}. In order to do so, we need to dominate $\big|\Delta^{\gamma/2}G_s\big|$ and $\big|\Delta_{ \star}^{\gamma/2}G_s\big|$ by some function in $L^1(\mathbb{R}^d)$, \textit{uniformly in time}. This is done by applying a series of estimates.

We recall the following equivalence of representations for the fractional Laplacian operator, from \cite[Lemma 3.2]{hitchhiker}.
\begin{lem}\label{lem:exp2}
	For any $(s, \hat{u}) \in [0,T] \times \mathbb{R}^d$ and any $G \in \mathcal{S}^2(\mathbb{R}^d)$, it holds
	\begin{align*} 
		 \Delta^{ \gamma/2 } G_s(\hat{u})
		=
		\frac{c_{\gamma}}{2}
		\int_{\mathbb{R}^d}
		\frac{\Delta_{\hat{w}}G_s(\hat{u})}{|\hat{w}|^{d+\gamma}} \;\rmd \hat{w},
	\end{align*}
where $\Delta_{(\cdot)}$ is given in the last item of Definition \ref{def:micro-op}.	
\end{lem}
For $G \in \mathcal{S}^2(\mathbb{R}^d)$, a second-order Taylor expansion of $G_s$ around $\hat{u}$ leads to 
\begin{equation}  \label{taylor2}
	|\Delta_{\hat{w}}G_s(\hat{u})| 
	\leq 
	\frac{| \hat{w} |^2}{2}  \sum_{i,j=1}^d \; \sup_{ \lambda \in [-1,1] } |\partial_{ij} G_s( \hat{u}+ \lambda \hat{w})|.
\end{equation}	
Applying polar coordinates and spherical coordinates for $d=2$ and $d \geq 3$, respectively, one obtains that for any $0 \leq R_1 \leq R_2\leq+\infty$ and $q \in \mathbb{R}$,
\begin{equation} \label{sphe}
	\begin{split}
		\int_{R_1 \leq |\hat{v}| \leq R_2} 
		|\hat{v}|^q 
		\rmd\hat{v} \leq 2 \pi^{d-1} \int_{R_1}^{R_2} r^{q+d-1} \;\rmd r
		.
	\end{split}
\end{equation}

We also introduce the following sets.
\begin{definition} \label{deftesdisctime}
	For every $k \in \mathbb{N}$, define $\mathcal{S}^k(\mathbb{R}^{d\star})$ and $\mathcal{S}_c^k(\mathbb{R}^{d\star})$ through
	\begin{align*} 
		\mathcal{S}^k(\mathbb{R}^{d\star}) 
		&:= \big\{ 
		G: [0, \infty) \times \mathbb{R}^d \mapsto \mathbb{R}, \exists G^{\pm}\in  \mathcal{S}^k(\mathbb{R}^d): G_s(\hat{u}) =\mathbbm{1}_{\{ u_d < 0 \}} G^{-}_s(\hat{u}) + \mathbbm{1}_{ \{ u_d \geq 0\}} G^{+}_s(\hat{u}) 
		\big\}
		,
		\\
		\mathcal{S}_c^k(\mathbb{R}^{d\star}) 
		&:= \big\{ 
		G \in \mathcal{S}^k(\mathbb{R}^{d\star}): \exists K \subset \mathbb{R}^d \; \text{compact set :} \; |\hat{u}| \notin K \Rightarrow \; \sup_{s \in [0, \; T]} |G_s( \hat{u})| = 0
		\big\}.
	\end{align*}
	
	For any $G \in \mathcal{S}_c^0(\mathbb{R}^{d\star})$, the set $\{ b >1: \; |\hat{u}| \geq b \Rightarrow \; \sup_{s \in [0, \; T]} |G_s( \hat{u})| = 0\}$ is nonempty and we define
	\begin{equation} \label{defbG}
		b_G:= \inf \big\{ b >1:\; |\hat{u}| \geq b \Rightarrow \; \sup_{s \in [0, \; T]} |G_s( \hat{u})| = 0 \big\} \in [1, \infty).
	\end{equation}
\end{definition}
From the Mean Value Theorem, for each $G \in \mathcal{S}^1(\mathbb{R}^d)$ it holds 
\begin{equation} \label{timelip}
\exists \hat{C}=\hat{C}(G)>0: \; \forall \hat{u}, \hat{w} \in \mathbb{R}^d, \quad	\sup_{s \in [0, \; T]} \big| G_s(\hat{u}+ \hat{w}) - G_s(\hat{u}) \big| \leq  \hat{C}  \nnorm{G}_{0,1} |\hat{w}|,
\end{equation}
where $\hat{C}$ is a constant depending only on $G$. In particular, if $G:[0, \infty) \times \mathbb{R}^d \mapsto \mathbb{R}$ satisfies \eqref{timelip}, and $\sup_{(s,\hat{u}) \in [0,T] \times \mathbb{R}^d } | G(s,\hat{u}) | < \infty$ (this later requirement is fulfilled for all test functions in this work), then for any $\delta \in [0,1]$, there exists a constant $\hat{C}(\delta,G)>0$ depending only on $\delta$ and $G$ such that
\begin{equation} \label{GSHoldgen}
	\begin{split}	
		&  \forall \hat{u}, \hat{w} \in \mathbb{R}^d,  \quad  \sup_{s \in [0, \; T]} \big|G_s( \hat{u} + \hat{w} ) - G_s(\hat{u}) \big| \leq   \hat{C}(\delta,G)  |\hat{w}|^{\delta}.
	\end{split}
\end{equation}
In particular, it holds
\begin{equation} \label{GSHold}
	\begin{split}	
		&  \forall G \in \mathcal{S}^1(\mathbb{R}^d), \; \forall \hat{u}, \hat{w} \in \mathbb{R}^d,  \quad  \sup_{s \in [0, \; T]} \big|G_s( \hat{u} + \hat{w} ) - G_s(\hat{u}) \big| \leq   2 \nnorm{G}_{0,1}  |\hat{w}|^{\delta}.
	\end{split}
\end{equation}
In the remainder of this section, for any $k \geq 1$, $C_k(\gamma, d)$ denotes a positive constant depending only on $\gamma \in (0,2)$ and $d \geq 1$. Analogously, $C_k(\gamma)$ depends only on $\gamma \in (0,2)$.

\subsection{Proof of Proposition \ref{propLinfty}}

For every $G \in \mathcal{S}^2(\mathbb{R}^d)$, define $I_1^G, \; I_2^G:[0,T] \times \mathbb{R}^d \mapsto [0, \; \infty)$ by
\begin{equation} \label{defI1I2G}
	I_1^G(s,\hat{u}):= \int_{|\hat{w}| \leq |\hat{u}|/2} \frac{|\Delta_{\hat{w}}G_s(\hat{u})|}{|\hat{w}|^{d+\gamma}} \;\rmd \hat{w}, \; I_2^G(s,\hat{u}):= \int_{|\hat{w}| > |\hat{u}|/2} \frac{| \Delta_{\hat{w}}G_s(\hat{u})|}{|\hat{w}|^{d+\gamma}} \;\rmd \hat{w}, \quad (s,\hat{u}) \in [0,T] \times \mathbb{R}^d,
\end{equation} 
and also $F_1^G, F_2^G:\mathbb{R}^d \mapsto [0, \; \infty)$ by
\begin{align}
	F_1^G(\hat{u}):= &\frac{\pi^{d-1}}{2^{2 - \gamma} (2 - \gamma) } (1+|\hat{u}|^2)  \sum_{i,j=1}^d \; \sup_{s \in [0,T], \; \hat{v}:  |\hat{v}-\hat{u}| \leq |\hat{u}|/2} |\partial_{ij} G_s( \hat{v})|, \quad \hat{u} \in  \mathbb{R}^d,  \label{defF1}
\end{align}
\begin{align} \label{defF2}
	F_2^G(\hat{u}):=
	\begin{cases}
		\nnorm{G}_{0,2} \int_{|\hat{w}| \leq 1} |\hat{w}|^{2-d - \gamma}\;\rmd\hat{w} + 4\nnorm{G}_{0,0} \int_{|\hat{w}| > 1} |\hat{w}|^{-d - \gamma}\;\rmd\hat{w}, \quad &   |\hat{u}| \leq 2
		, \\
		4 \nnorm{G}_{0,0} \int_{|\hat{w}| \geq |\hat{u}|/2 } |\hat{w}|^{-d - \gamma}\;\rmd\hat{w} \leq 4 \nnorm{G}_{0,0} \int_{|\hat{w}| \geq 1} |\hat{w}|^{-d - \gamma}\;\rmd\hat{w}, \quad &  |\hat{u}| > 2.
	\end{cases}
\end{align}
From \eqref{sphe}, \eqref{defnabladelta} and \eqref{taylor2}, for any $G \in \mathcal{S}^2(\mathbb{R}^d)$, it holds
\begin{align*}
 \forall \hat{u} \in \mathbb{R}^d, \quad \sup_{s \in [0,T]}  I_j^G(s,\hat{u}) \leq F_j^G(\hat{u}), \quad j \in \{1, \; 2\}.
\end{align*}
Thus, from Lemma \ref{lem:exp2} and the triangle inequality, for any $G \in \mathcal{S}^2(\mathbb{R}^d)$, for all $\hat{u} \in \mathbb{R}^d$ it holds
\begin{equation} \label{boundI1I2}
	\frac{2}{c_{\gamma}} \sup_{s \in [0,T]} \big| \Delta^{ \gamma/2 } G_s(\hat{u}) \big| \leq  \sup_{s \in [0,T]} \big\{ I_1^G(s,\hat{u})+I_2^G(s,\hat{u}) \big\} \leq    F_1^G( \hat{u}) + F_2^G( \hat{u}).
\end{equation}
From \eqref{defF1}, \eqref{semnorm}, \eqref{defF2} and \eqref{sphe}, it is not hard to prove that there exists $C_2(\gamma,d)$ such that 
\begin{equation} \label{defC0gd}
	\forall G \in \mathcal{S}^2(\mathbb{R}^d),\;\hat{u} \in \mathbb{R}^d,
	\quad
	F_1^G(\hat{u}) + F_2^G(\hat{u}) \leq \frac{16 \pi^{d-1}}{\gamma (2 - \gamma)} \big[ \nnorm{G}_{2,2} + \nnorm{G}_{0,2}   \big].
\end{equation}
From \eqref{defF1}, \eqref{defF2} and \eqref{sphe}, it is not hard to prove that there exists $C_1(\gamma,d)$ such that
\begin{equation} \label{defC1gd}
	\forall M\geq1,\;G \in \mathcal{S}^2(\mathbb{R}^d),\quad
	\int_{| \hat{u} | \leq M } \big[F_1^G(\hat{u}) + F_2^G(\hat{u})\big] \;\rmd \hat{u} \leq M^{d+2} C_{1}(\gamma,d) \nnorm{G}_{0,2} .
\end{equation}
From \eqref{defbG} and \eqref{sphe}, it follows that for any $G \in \mathcal{S}_c^0(\mathbb{R}^d)$, we have
\begin{equation} \label{boundacomp}
	\forall  \hat{u}: | \hat{u}| \geq 2 b_G ,  \quad  \sup_{s \in [0, \; T]} \Bigg\{  \int_{\mathbb{R}^d} \frac{| G_s(\hat{v}) - G_s(\hat{u})|}{| \hat{v} - \hat{u}|^{- \gamma - d} } d \hat{v} \Bigg\}  \leq \frac{2 \pi^{d-1} \nnorm{G}_{0,0}}{d| \hat{u} - b_G|^{ \gamma + d}} (b_G)^d,
\end{equation}
and from \eqref{boundacomp} and \eqref{sphe}, that for every $G \in \mathcal{S}_c^0(\mathbb{R}^d)$, it holds
\begin{equation} \label{boundbG}
 \int_{| \hat{u}| \geq 2 b_G } \sup_{s \in [0, \; T]} \Bigg\{  \int_{\mathbb{R}^d} \frac{| G_s(\hat{v}) - G_s(\hat{u})|}{| \hat{v} - \hat{u}|^{- \gamma - d} } d \hat{v} \Bigg\} \rmd \hat{u} \leq \frac{(2 \pi^{d-1})^2}{d \gamma} (b_G)^{d - \gamma} \nnorm{G}_{0,0}. 
\end{equation}
Keeping in mind the previous inequality, we state the next result.
\begin{prop} \label{propLqglobcomp}
	Let $\gamma \in (0, \; 2)$, $d \geq 1$ and $G \in \mathcal{S}_c^2(\mathbb{R}^d) $. Then there exists $H_1^G: \mathbb{R}^d \mapsto [0, \; \infty)$ such that $\sup_{s \in [0, T]} \big|\Delta^{\gamma/2}G_s( \hat{u} ) \big| \leq H_1^G( \hat{u} )$, for any $\hat{u} \in \mathbb{R}^d$, and
	\begin{align} \label{defC2gd}
		&	\exists C_2(\gamma,d):\;  \int_{\mathbb{R}^d} H_1^G( \hat{u} ) \; \rmd \hat{u} \leq C_2(\gamma, d) (b_G)^{d + 2}  \nnorm{G}_{0,2}   < \infty, \\
		& \forall \hat{u} \in \mathbb{R}^d, \quad 	H_1^G( \hat{u} )  \leq \frac{8 c_{\gamma} \pi^{d-1}}{\gamma (2 - \gamma)} \big[ \nnorm{G}_{2,2} + \nnorm{G}_{0,2}   \big] + c_{\gamma} \frac{2 \pi^{d-1} \nnorm{G}_{0,0}}{d }. \nonumber
	\end{align}
\end{prop}
\begin{proof}
	The result holds due to \eqref{deflapfracreg} (for $\mathcal{O}=\mathbb{R}^d)$, \eqref{boundI1I2}, \eqref{defC0gd}, \eqref{boundacomp}, \eqref{boundbG} and \eqref{defC1gd}.
	
\end{proof}
Keeping in mind \eqref{boundI1I2}, we state the following result.
\begin{lem} \label{lemF1F2}
	Let $\gamma \in (0,2)$ ,$d \geq 1$ and $G \in \mathcal{S}^2(\mathbb{R}^d)$. Then there exists $C_3(\gamma,d)$ such that
	\begin{equation} \label{defC3gd}
		\int_{ |\hat{u}| \geq 1  } \big|F_1^G(\hat{u}) \big| \; \rmd \hat{u}  \leq C_3(\gamma,d)  \nnorm{G}_{2+2d, \; 2}.
	\end{equation}
	Moreover, it holds
	\begin{align} \label{restrgd}
		\int_{ |\hat{u}| > 2  } \big|F_2^G(\hat{u})\big|  \;\rmd \hat{u}  \leq \frac{2^{\gamma+2}}{\gamma} (2 \pi^{d-1})^{2}  \nnorm{G}_{0,0} \int_2^{\infty} r^{d- 1 - \gamma} \; \rmd r.
	\end{align}
	In particular, if $\gamma > d$ (this means that $d=1$ and $1 < \gamma < 2$), there exists $C_4(\gamma)$ such that
	\begin{align} \label{defC4g}
		\int_{\mathbb{R}^d} \big[ \big|F_1^{G} (\hat{u}) \big| + \big|F_2^{G} (\hat{u}) \big| \big] \; \rmd \hat{u} \leq 	C_4(\gamma) \big[  \nnorm{G}_{0,2}  + \nnorm{G}_{4,2}  \big] < \infty.
	\end{align}
\end{lem}
\begin{proof}
	From \eqref{sphe}, \eqref{semnorm} and \eqref{defF1}, we get \eqref{defC3gd}.
	From \eqref{defF2} and \eqref{sphe}, it is not hard to obtain \eqref{restrgd}.
	Finally, in order to get \eqref{defC4g} when $\gamma > d$, it is enough to combine \eqref{defC2gd}, \eqref{sphe}, \eqref{defC3gd} and \eqref{restrgd}.
	This ends the proof.
\end{proof}
Choosing $H_1^G=c_{\gamma} F_1^G /2 + c_{\gamma} F_2^G /2$, the next result is a direct consequence of \eqref{boundI1I2}, \eqref{defC4g} and \eqref{defC0gd}, therefore we omit its proof.
\begin{prop} \label{propLqglobschw}
	Let $\gamma \in (1, \; 2)$, $d =1$ and $G \in \mathcal{S}^2(\mathbb{R}) $. Then there exists $H_1^G: \mathbb{R} \mapsto [0, \; \infty)$ such that $\sup_{s \in [0, T]} \big|\Delta^{\gamma/2}G_s( u ) \big| \leq H_1^G( u )$, for any $u \in \mathbb{R}$, and
	\begin{align}
		&\int_{\mathbb{R}} H_1^G( u ) \; \rmd u \leq \frac{c_{\gamma}}{2} C_4(\gamma) \big[   \nnorm{G}_{0,2} +  \nnorm{G}_{4,2} \big] < \infty, \label{C4sch} \\
		& \forall u \in \mathbb{R}, \quad 	H_1^G(u )  \leq \frac{8 c_{\gamma} \pi^{d-1}}{\gamma (2 - \gamma)} \big[ \nnorm{G}_{2,2} +  \nnorm{G}_{0,2}  \big] + c_{\gamma} \frac{2 \pi^{d-1} \nnorm{G}_{0,0}}{d }. \nonumber
	\end{align}	
\end{prop}
By combining \eqref{defSalgamd3} with Propositions \ref{propLqglobcomp} and \ref{propLqglobschw}, we finish the proof of Proposition \ref{propLinfty}. 

\subsection{Proof of Proposition \ref{propL1alpha0}}

In this subsection, we need to estimate $\big|\Delta_{\star}^{\gamma/2}G_s(\hat{u})\big|$. Keeping this in mind, for any $G \in \mathcal{S}^2(\mathbb{R}^d)$, we get from \eqref{deflapfracreg} that $\Delta_{\mathbb{R}^{d\star}}^{\gamma/2}G_s(\hat{u})$ can be rewritten as
\begin{equation} \label{lapfracalt}
	\frac{c_{\gamma}}{2} \int_{  | \hat{w}| < |u_d|}  \frac{ \Delta_{\hat{w}}G_s(\hat{u}) }{ |\hat{w}|^{\gamma+d}}\;\rmd\hat{w} +\begin{dcases}
		 c_{\gamma} \int_{| \hat{w} | \geq |u_d|, \; \; u_d+w_d <0}  \frac{\nabla_{\hat{w}}G_s(\hat{u})}{ |\hat{w}|^{\gamma+d}}\;\rmd\hat{w},  \quad & u_d < 0, \; s \in [0, T]
		 , \\
		c_{\gamma} \int_{| \hat{w} | \geq |u_d|, \; \; u_d+w_d >0}  \frac{\nabla_{\hat{w}}G_s(\hat{u})}{ |\hat{w}|^{\gamma+d}}\;\rmd\hat{w}, \quad & u_d > 0,  \; s \in [0, T],
	\end{dcases}
\end{equation}
where we recall $\Delta_{(\cdot)}$ and $\nabla_{(\cdot)}$ from the last item in Definition \ref{def:micro-op}. In particular, it holds
\begin{equation} \label{taylor0theta}
	\forall G \in \mathcal{S}^2(\mathbb{R}^d), \quad 
	\nnorm{\Delta_{(\cdot)} G(\hat{u})}_{0,0}\leq 4 \nnorm{G}_{0,0} 
	\quad \text{and} \quad 
	\nnorm{\nabla_{(\cdot)} G(\hat{u})}_{0,0} \leq 2 \nnorm{G}_{0,0}.
\end{equation}
Next, for every $G \in \mathcal{S}^1(\mathbb{R}^d)$, define $I_3^{G}: [0,T] \times \mathbb{R}^{d} \mapsto \mathbb{R}$ and $F_3^G: \mathbb{R}^{d} \mapsto \mathbb{R}$ by
\begin{align} \label{defI3G}
	I_3^{G} (s,\hat{u})
	&:= \int_{| \hat{w} | \geq |u_d|}  \frac{|\nabla_{\hat{w}}G_s(\hat{u})|}{ |\hat{w}|^{\gamma+d}}\;\rmd\hat{w}, \quad (s, \hat{u}) \in [0,T] \times \mathbb{R}^{d},
	\\
	\label{defF3}
	F_3^G(\hat{u})
	&:=
	\begin{cases} (\gamma \pi )^{-1} 4 \pi^{d} \nnorm{G}_{0,0} |u_d|^{-\gamma} \leq	(\gamma \pi )^{-1} 4 \pi^{d} \nnorm{G}_{0,0}, \quad & |u_d| \geq 1;  \\
	4 \pi^{d}	(\gamma \pi )^{-1} \big[ \nnorm{G}_{0,0} + 2 \nnorm{G}_{0,1}  |u_d|^{-\gamma/2} \mathbbm{1}_{ \{u_d \neq 0 \}} \big], \quad &  |u_d| \leq 1.
	\end{cases} 
\end{align}
From \eqref{sphe}, \eqref{GSHold} and the fact that $\gamma / 2 < 1$, it is not very hard to prove that
\begin{align} \label{claimboundI3}
	\forall \hat{u} \in \mathbb{R}^{d\star}, \quad \sup_{s \in [0,T]} \big|I_{3}^{G}(s, \hat{u})\big| \leq F_3^G(\hat{u}).
\end{align}
From \eqref{defF3} and \eqref{sphe}, we see that there exists $C_5(\gamma,d)$ such that
\begin{equation} \label{defC5gd}
	\forall M \geq 1, \; \forall G \in \mathcal{S}^1(\mathbb{R}^d), \quad \int_{| \hat{u} | \leq M } F_3^G(\hat{u}) \; \rmd \hat{u} \leq M^{d} C_{5}(\gamma,d) \nnorm{G}_{0,1}.
\end{equation}
Moreover, combining the triangle inequality with \eqref{boundI1I2}, \eqref{lapfracalt} and \eqref{claimboundI3}, we get
\begin{align} \label{bound1reg}
	\forall  \hat{u} \in  \mathbb{R}^{d\star}, \quad  \frac{2}{c_{\gamma}} \sup_{s \in [0,T]}\big|\Delta_{\star}^{\gamma/2}G_s(\hat{u})\big| \leq F_1^G( \hat{u}) + F_2^G( \hat{u}) + 2 F_3^G( \hat{u}).
\end{align}
for any $G \in \mathcal{S}^2(\mathbb{R}^d)$. In last line, $F_1^{G}$ and $F_2^{G}$ are given in \eqref{defF1} and \eqref{defF2}, respectively. 

With \eqref{boundbG} in mind, we state the next result.
\begin{prop} \label{propLqregcomp}
	Let $\gamma \in (0, \; 2)$, $d \geq 1$ and $G \in \mathcal{S}_c^2(\mathbb{R}^d) $. Then there exists $H_2^G: \mathbb{R}^d \mapsto [0, \; \infty)$ such that $\sup_{s \in [0, T]} \big| \Delta_{\star}^{\gamma/2}G_s( \hat{u} ) \big| \leq H_2^G( \hat{u} )$, for any $\hat{u} \in \mathbb{R}^d$, and
	\begin{align} \label{defC6gd}
		\exists C_6(\gamma,d):\; & \int_{\mathbb{R}^d} H_2^G( \hat{u} ) \; \rmd \hat{u} \leq C_6(\gamma, d) (b_G)^{d + 2}  \nnorm{G}_{0,2}  < \infty, \\
		\forall \hat{u} \in \mathbb{R}^{d\star}, \quad & H_2^G( \hat{u} )	\leq \frac{4 \pi^{d} c_{\gamma} }{\gamma \pi} \Bigg( \frac{4 \big[  \nnorm{G}_{2,2} + \nnorm{G}_{0,2}   \big] }{2 - \gamma}  +   2 \nnorm{G}_{0,1}   |u_d|^{-\gamma/2}  \Bigg)	. \nonumber
	\end{align}
\end{prop}
\begin{proof}
	Recalling Definition \ref{def:regi_lap}, the result holds due to the inequalities \eqref{bound1reg}, \eqref{defC0gd}, \eqref{boundacomp}, \eqref{boundbG}, \eqref{defC1gd} and \eqref{defC5gd}.
	
\end{proof}
Keeping in mind Proposition \ref{propLqglobschw}, we state the next result.
\begin{prop} \label{propLqregschw}
	Let $\gamma \in (1, \; 2)$, $d =1$ and $G \in \mathcal{S}^2(\mathbb{R}) $. Then there exists $H_2^G: \mathbb{R} \mapsto [0, \; \infty)$ such that $\sup_{s \in [0, T]} \big|\Delta_{\star}^{\gamma/2}G_s( u ) \big| \leq H_2^G( u )$, for any $u \in \mathbb{R}$, and
	\begin{align} \label{defC7g}
		\exists C_7(\gamma):\; & \int_{\mathbb{R}^d} H_2^G (u) \; \rmd u \leq  C_7(\gamma)\big[  \nnorm{G}_{0,2}  +  \nnorm{G}_{4,2}  \big] < \infty
		, \\
		\forall u \in \mathbb{R}^{*}, \quad & H_2^G(u )	
		\leq \frac{4  c_{\gamma} }{\gamma } 
		\Bigg( \frac{4 \big[  \nnorm{G}_{2,2} + \nnorm{G}_{0,2}   \big] }{2 - \gamma}  +   2 \nnorm{G}_{0,1}   |u|^{-\gamma/2}  \Bigg)
		. \nonumber
	\end{align}
\end{prop}
\begin{proof}
	The result is a direct consequence of \eqref{bound1reg}, \eqref{defC0gd}, \eqref{defF3}, \eqref{defC5gd} and \eqref{defC4g}.
\end{proof}
In order to obtain Proposition \ref{propL1alpha0}, we combine \eqref{defGdisc2} with  \eqref{deflapfracreg} and \eqref{deflapfracabu} to obtain
\begin{equation}  \label{lapGdisc}
	\Delta_{\star}^{\gamma/2}G_s(\hat{u}) 
	= \mathbbm{1}_{ \{ u_d < 0\}  } \Delta_{\star}^{\gamma/2}G^{-}_s(\hat{u}) 
	+ \mathbbm{1}_{ \{ u_d \geq 0\}}\Delta_{\star}^{\gamma/2}G^{+}_s(\hat{u}),
\end{equation}
for every $(s, \hat{u}) \in [0, T] \times \mathbb{R}^d$. Moreover, by combining \eqref{defSalgamd3} with Propositions \ref{propLqregcomp} and \ref{propLqregschw}, we conclude that for any $\gamma \in (0, \; 2)$, $d \geq 1$ and $\widetilde{G} \in \mathcal S_{\gamma}^{d}$, there exists $H_2^{\widetilde{G}} \in L^1(\mathbb{R}^{d})$ satisfying
\begin{align*}
	\begin{cases}
		\forall \hat{u} \in \mathbb{R}^{d}, \quad & \sup_{s \in [0,T]} \big|  \Delta_{\star}^{\gamma/2}\widetilde{G}_s( \hat{u} ) \big| \leq H_2^{\widetilde{G}}( \hat{u} ); \\
		\exists C_2(\widetilde{G})>0: \forall \hat{u} \in \mathbb{R}^{d\star}, \quad & \big|H_2^{\widetilde{G}}( \hat{u} ) \big| \leq C_2(\widetilde{G}) \big[1 + |u_d|^{- \gamma/2}\big].
	\end{cases}
\end{align*}
Plugging the last display with \eqref{lapGdisc}, we obtain Proposition \ref{propL1alpha0}.	

\section{Discrete convergences} \label{secdiscconv}

The goal of this section is to prove Proposition \ref{prop:conv_frac-lap}. We do so by stating a couple of results (that will be proved later), in the same way as it was done to obtain Proposition \ref{propL1alpha}. Recall Definitions \ref{def:discrete-flap} and \ref{deftesconttime}. 
\begin{prop} \label{convL1B}
	Let $\gamma \in (0, \; 2)$, $d \geq 1$ and $G \in \mathcal S_{\gamma}^{d}$. Then
	\begin{align*}
		\lim_{n \rightarrow \infty} \sup_{s \in [0,T]}\frac{1}{n^d}   \sum_{ \hat{x} }  \big |
		\Delta^{\gamma/2}_{n,\mcb B} G_s \big(\tfrac{ \hat{x}}{n} \big) - 
		 \Delta^{\gamma/2}   G_s  \big(\tfrac{ \hat{x}}{n} \big) \big|  =0.
	\end{align*}
\end{prop}

\begin{prop} \label{convL1F}
	Let $\gamma \in (0, \; 2)$, $d \geq 1$ and $G \in \mathcal S_{\gamma,0}^{d}$. Then
	\begin{align*}
		\lim_{n \rightarrow \infty} 
		\sup_{s \in [0,T]}\frac{1}{n^d}
		\sum_{ \hat{x} }  \big | \Delta^{\gamma/2}_{n,\mcb F} G_s \big(\tfrac{ \hat{x}}{n} \big)  - 
		 \Delta_{\star}^{ \gamma / 2 }   G_s  \big(\tfrac{ \hat{x}}{n} \big) \big| =0.
	\end{align*}
\end{prop}
Next we state a result which is a direct consequence of Propositions \ref{convL1B} and \ref{convL1F}, and  Definition \ref{def:discrete-flap}, therefore we omit its proof.
\begin{cor} \label{convL1S}
	Let $\gamma \in (0, \; 2)$, $d \geq 1$ and $G \in  \mathcal S_{\gamma}^{d}$. Then
	\begin{align*}
		\lim_{n \rightarrow \infty} \sup_{s \in [0,T]}\frac{1}{n^d}   \sum_{ \hat{x} }  \big | \Delta^{\gamma/2}_{n,\mcb S} G_s \big(\tfrac{ \hat{x}}{n} \big)  
		- \big[ \Delta^{ \gamma / 2 } - \Delta_{\star}^{ \gamma / 2 } \big]  G_s  \big(\tfrac{ \hat{x}}{n} \big)  \big|  =0.
	\end{align*}
\end{cor}
Now we are ready to present the proof of Proposition \ref{prop:conv_frac-lap}.
\begin{proof}[Proof of Proposition \ref{prop:conv_frac-lap}]
	From Definition \ref{def:discrete-flap} and \eqref{defLalfgam}, it holds $\mathbb{L}_{n,0}^\gamma=\Delta^{\gamma/2}_{n,\mcb F}$, $\mathbb{L}_0^\gamma=\Delta_{\star}^{\gamma/2}$, $\mathbb{L}_{n,1}^\gamma=\Delta^{\gamma/2}_{n,\mcb B}$ and $\mathbb{L}_1^\gamma=\Delta^{\gamma/2}$. Thus, we see that Proposition \ref{prop:conv_frac-lap} is a direct consequence of Remark \ref{remgentest} and Proposition \ref{convL1B}, resp. Proposition \ref{convL1F}, when $\kappa=1$, resp. $\kappa=0$.
	
	It remains to treat the regime $\kappa \in (0, 1) \cup (1, \infty)$. In this case, Proposition \ref{prop:conv_frac-lap} comes by combining the arguments above with Remark \ref{remgentest}, Proposition \ref{convL1F} and Corollary \ref{convL1S}.
\end{proof}

\subsection{Proof of Proposition \ref{convL1B}}

Keeping in mind Proposition A.1 of \cite{CGJ2}, in order to obtain Proposition \ref{convL1B}, we begin by stating a result which corresponds to the terms of $\hat{x} \in \mathbb{Z}^d$ such that $\hat{x}$ is far from the origin. 
\begin{prop} \label{convL1Bfar0}
	Let $\gamma \in (0, \; 2)$, $d \geq 1$ and $G \in  \mathcal S_{\gamma}^{d}$. Then
	\begin{align*}
		\varlimsup_{M \rightarrow \infty}  \varlimsup_{n \rightarrow \infty}   
		\sup_{s \in [0,T]} \frac{1}{n^d} \sum_{ |\hat{x}| \geq M n }   \big |
		\Delta^{\gamma/2}_{n,\mcb B} G_s \big(\tfrac{ \hat{x}}{n} \big) - 
		 \Delta^{\gamma/2}   G_s  \big(\tfrac{ \hat{x}}{n} \big) \big|  =0.
	\end{align*}
\end{prop}
In order to treat the case where $|\hat{x}| < M n$, we apply the following proposition. Recall $\Delta_{(\cdot)}$ from the last item in Definition \ref{def:micro-op}. In the remainder of this section, for any $k \geq 1$, $C_k(G)$ denotes a positive constant depending only on $G$. 
\begin{prop} \label{convL1Bnear0}
	Let $\gamma \in (0, \; 2)$, $d \geq 1$ and $G \in  \mathcal S_{\gamma}^{d}$. Then there exist $C_1(G), C_2(G)$ satisfying
	\begin{align}
		& 
		\frac{1}{n^d}
		\sum_{|\hat{z}| =  1}^{\varepsilon n - 1} 
		\frac{1}{|\hat{z}/n|^{\gamma+d}}
		\big| \Delta_{\hat{z}/n}G_s \big(\tfrac{\hat{x}}{n}\big) \big| 
		+
		\int_{| \hat{w} | < \varepsilon } 
		\frac{1}{|\hat{w}|^{d+\gamma}}
		\big| \Delta_{\hat{u}}G_s(\hat{w}) \big|
		\;\rmd\hat{w}  
		\leq C_1(G) \varepsilon^{2 - \gamma}, \label{1boundBnear0} \\
		&  
		\Bigg|
		\frac{1}{n^d}
		\sum_{|\hat{z}| =  \varepsilon n}^{\infty}
		\frac{1}{|\hat{z}/n|^{\gamma+d}}
		\Delta_{\hat{z}/n}G_s\big(\tfrac{\hat{x}}{n}\big) 
		-   
		\int_{|\hat{w}| \geq \varepsilon} 
		\frac{1}{ |\hat{w}|^{\gamma+d}}
		\Delta_{\hat{w}}G_s\big(\tfrac{\hat{x}}{n}\big)
		\;\rmd\hat{w}
		\Bigg|
		 \leq C_2(G) \frac{\varepsilon^{-\gamma - 1}}{n}, \label{2boundBnear0}
	\end{align}
	for every $\hat{x} \in \mathbb{Z}^d$, every $s \in [0, \; T]$, every $n \geq 1$ and every $\varepsilon \in (0, \;1)$.
\end{prop}
Now we obtain Proposition \ref{convL1B}, by assuming Propositions \ref{convL1Bfar0} and \ref{convL1Bnear0}.
\begin{proof}[Proof of Proposition \ref{convL1B}]
	From Definition \ref{defKn}, Lemma \ref{lem:exp2}, \eqref{1boundBnear0} and \eqref{2boundBnear0}, we obtain
\begin{align*}
		\varlimsup_{M \rightarrow \infty}  \varlimsup_{n \rightarrow \infty}   
		\sup_{s \in [0,T]} \frac{1}{n^d} \sum_{ |\hat{x}| < M n }   \big |
		\Delta^{\gamma/2}_{n,\mcb B} G_s \big(\tfrac{ \hat{x}}{n} \big) - 
		 \Delta^{\gamma/2}   G_s  \big(\tfrac{ \hat{x}}{n} \big) \big| =0.
\end{align*}	
Combining the last display with Proposition \ref{convL1Bfar0}, the proof of Proposition \ref{convL1B} ends. 		
\end{proof}
It remains to show Propositions \ref{convL1Bfar0} and \ref{convL1Bnear0}. We begin by obtaining the former result.
\begin{proof} [Proof of Proposition \ref{convL1Bfar0}]
	From Measure Theory, we get
	\begin{equation} \label{L1far0}
		\forall H \in L^1(\mathbb{R}^d), \quad \varlimsup_{M \rightarrow \infty}  \varlimsup_{n \rightarrow \infty}   \frac{1}{n^d}\sum_{ |\hat{x}| \geq M n }   \big|H  \big(\tfrac{ \hat{x}}{n} \big) \big| \lesssim \lim_{M \rightarrow \infty} \int_{ | \hat{u}| \geq M }  \big|H^G(\hat{u}) \big|  
		\;\rmd \hat{u} =0.
	\end{equation}
	Combining the last display with Proposition \ref{propLinfty}, we get
	\begin{equation*} 
		\varlimsup_{M \rightarrow \infty}  \varlimsup_{n \rightarrow \infty}   
		\sup_{s \in [0,T]} \frac{1}{n^d} \sum_{ |\hat{x}| \geq M n }  
		\big| \Delta^{\gamma/2}   G_s  \big(\tfrac{ \hat{x}}{n} \big) \big|  =0.
	\end{equation*}
	Therefore, it is enough to prove that
	\begin{equation} \label{2boundBfar0}
		\varlimsup_{M \rightarrow \infty}  \varlimsup_{n \rightarrow \infty}   
		\sup_{s \in [0,T]} \frac{1}{n^d} \sum_{ |\hat{x}| \geq M n }  \big|\Delta^{\gamma/2}_{n,\mcb B} G_s \big(\tfrac{ \hat{x}}{n} \big) \big| =0.
	\end{equation}
	For the remainder of the proof, we treat two cases differently: $G \in \mathcal{S}_c^2(\mathbb{R}^d)$ and $G \notin \mathcal{S}_c^2(\mathbb{R}^d)$. We begin with the former one.
	
	\textbf{I)} The case $G \in \mathcal{S}_c^2(\mathbb{R}^d)$. Without loss of generality, we assume $M \geq 2 b_G$, where $b_G$ is given in \eqref{defbG}. From \eqref{sphe}, it is not very hard to prove that
	\begin{align} \label{1supM}
		\forall G \in \mathcal{S}_c^2(\mathbb{R}^d), \quad \varlimsup_{M \rightarrow \infty}  \varlimsup_{n \rightarrow \infty}  \frac{2 \nnorm{G}_{0,0} c_{\gamma} }{n^d}\sum_{ |\hat{x}| \geq M n }  \frac{1}{n^d} \sum_{| \hat{y} | \leq b_G n } \Big| \frac{\hat{y}-\hat{x}}{n} \Big|^{- \gamma - d} =0.
	\end{align}
	From Definition \ref{defKn}, the double limit in \eqref{2boundBfar0} is bounded from above by
	\begin{align*}
		 & \varlimsup_{M \rightarrow \infty}  \varlimsup_{n \rightarrow \infty}  \frac{2 \nnorm{G}_{0,0} c_{\gamma} }{n^d}\sum_{ |\hat{x}| \geq M n }  \frac{1}{n^d} \sum_{| \hat{y} | \leq b_G n } \Bigg| \frac{\hat{y}-\hat{x}}{n} \Bigg|^{- \gamma - d}.
	\end{align*}
	Combining the last display with \eqref{1supM}, we get \eqref{2boundBfar0} and the proof ends for this case.
	
	\textbf{II)} The case $G \notin \mathcal{S}_c^2(\mathbb{R}^d)$. From \eqref{defSalgamd3}, we get that $d=1$, $\gamma \in (1, 2)$ and $G \in \mathcal{S}^2(\mathbb{R})$. Recall the definition of $F_1^G$ in \eqref{defF1}. From \eqref{taylor2}, for every $\hat{x} \in \mathbb{Z}^d$, $s \in [0,T]$ and $n \geq 1$ it holds
	\begin{align*}
		n^{\gamma} \sum_{|y | \leq | x |/2 } \big|\Delta_{\tfrac{ y}{n}}G_s \big(\tfrac{ x}{n})\big| p_{\gamma}(y)  \lesssim  & \sum_{| y | \leq |x |/2 } \frac{n^{\gamma}| y|^2}{n^2} p_{\gamma}(y) F_1^G  \big(\tfrac{ \hat{x}}{n} \big) \lesssim F_1^G  \big(\tfrac{ x}{n} \big) \int_0^{\frac{|x|}{2 n}} r^{1 - \gamma} \rmd r \lesssim F_1^G  \big(\tfrac{ x}{n} \big).
	\end{align*}
	Combining the last display with Lemma \ref{lemF1F2} and \eqref{L1far0}, we have
	\begin{equation} \label{2boundBfar0a}
		\varlimsup_{M \rightarrow \infty}  \varlimsup_{n \rightarrow \infty}   \sup_{s \in [0,T]} \frac{1}{n}\sum_{ |x| \geq M n }  n^{\gamma} \sum_{|y | \leq | x |/2 } \big|\Delta_{\tfrac{ y}{n}}G_s \big(\tfrac{ x}{n} \big) \big| p_{\gamma}(y)  \lesssim \varlimsup_{M \rightarrow \infty}  \varlimsup_{n \rightarrow \infty} \frac{1 }{n}\sum_{ |x| \geq M n } F_1^G  \big(\tfrac{ x}{n} \big) =0.
	\end{equation}
	Moreover, without loss of generality, we can assume that $M > 2$. Thus, from \eqref{taylor0theta} and the fact that $\gamma >1$, we get
	\begin{equation} \label{2boundBfar0b}
		\varlimsup_{M \rightarrow \infty}  \varlimsup_{n \rightarrow \infty}   
		\sup_{s \in [0,T]} \frac{1}{n} \sum_{ |x| \geq M n }  n^{\gamma} \sum_{|y | > | x |/2 } \big|\Delta_{\tfrac{ y}{n}}G_s \big(\tfrac{ x}{n} \big) \big| p_{\gamma}(y)   \lesssim \lim_{M \rightarrow \infty} M^{1 - \gamma} =0.
	\end{equation}
	From \eqref{2boundBfar0a}, \eqref{2boundBfar0b} and the fact that $\Delta^{\gamma/2}_{n,\mcb B} G_s \big(\tfrac{ x}{n} \big)$ can be rewritten as $\frac{c_{\gamma}}{2}$ times
	\begin{align*}
		 n^{\gamma} 
		 \sum_{| y | \leq | x|/2 } 
		 \Delta_{\tfrac{ y}{n}}G_s\big(\tfrac{ x}{n}\big) p_{\gamma}(y - x) 
		 + n^{\gamma} 
		 \sum_{| y | > | x |/2 } 
		 \Delta_{\tfrac{ y}{n}}G_s\big(\tfrac{ x}{n}\big) p_{\gamma}(y - x), 
	\end{align*}
	we get \eqref{2boundBfar0} and the proof ends.
\end{proof}
We end this subsection by presenting the proof of Proposition \ref{convL1Bnear0}. 
\begin{proof}[Proof of Proposition \ref{convL1Bnear0}]
	From \eqref{taylor2}, \eqref{semnorm} and \eqref{sphe}, the expression on the left-hand side of \eqref{1boundBnear0} is bounded from above by
	\begin{align*}
		& \frac{\nnorm{G}_{0,2}}{2} \Bigg[ \frac{1}{n^d} \sum_{|\hat{z}| =  1}^{\varepsilon n - 1} \Bigg( \frac{|\hat{z}|}{n} \Bigg)^{- \gamma - d} + \int_{| \hat{w} | < \varepsilon } |\hat{w}|^{2 - d - \gamma}\;\rmd\hat{w} \Bigg] 	\lesssim \int_{| \hat{w} | < \varepsilon } |\hat{w}|^{2 - d - \gamma}\;\rmd\hat{w} \lesssim \int_0^{\varepsilon} r^{1 - \gamma} \rmd r \lesssim {\varepsilon}^{2- \gamma},
	\end{align*}
	which leads to the upper bound in \eqref{1boundBnear0}. In order to obtain \eqref{2boundBnear0}, we need to estimate the difference between a sum ranging over $\hat{z} \in \mathbb{Z}^d$ such that $\hat{z}/n \in A_0^{\varepsilon} := \{ \hat{w} \in \mathbb{R}^d: \; | \hat{w}| \geq \varepsilon \}$ and an integral over $\hat{w} \in A_0^{\varepsilon}$. We do so by partitioning $A_0^{\varepsilon}$ in $2^d$ subsets, depending on whether the coordinates $w_1, \ldots, w_d$ of $\hat{w}$ are non-negative or negative. More exactly, define 
	\begin{align} \label{defA1eps}
		A_1^{\varepsilon} :=& \big\{ \hat{w} \in \mathbb{R}^d: \; | \hat{w}| \geq \varepsilon, \quad \min \{w_1, \ldots, w_d \} \geq 0 \big \},
	\end{align}
	and the remaining $2^{d}-1$ subsets in an analogous way. Therefore, if we can prove that there exists $C_2(G)$ satisfying 
	\begin{equation} \label{2boundBnear0pos}
		\sup_{s \in [0,T]} \Bigg|
		  \frac{1}{n^d}  
		  \sum_{\hat{z}/n \; \in \; A_1^{\varepsilon}} 
		  \frac{ \Delta_{\hat{z}/n}G_s(\hat{x}/n) }{|\hat{z}/n|^{\gamma+d}}  -   
		  \int_{\hat{w} \; \in \; A_1^{\varepsilon}} 
		  \frac{ \Delta_{\hat{w}/n}G_s(\hat{x}/n) }{ |\hat{w}|^{\gamma+d}}		  
		  \;\rmd\hat{w}  
		  \Bigg| 
		  \leq 
		  \frac{C_2(G)}{2^d} \frac{\varepsilon^{-\gamma-1}}{n},
	\end{equation} 
\eqref{2boundBnear0} follows. Next, we rewrite the term on the left-hand side of the last display as 
	\begin{align}\label{eqlapfrac3a}
		\begin{split}
			&\sup_{s \in [0,T]} 
			\Bigg| 
			\sum_{\hat{z}/n \; \in \; A_1^{\varepsilon}} 
			\int_{ \tfrac{z_1}{n} }^{\tfrac{ z_1 +1}{n}} 
			\ldots 
			\int_{ \tfrac{z_d}{n} }^{\tfrac{ z_d +1}{n}}
			\Bigg[  
			\frac{1}{| \hat{z}/n|^{\gamma+d}}  
			\Delta_{\hat{z}/n}G_s(\hat{x}/n)  
			-    
			\frac{1}{|\hat{w}|^{\gamma+d}} 
			\Delta_{\hat{w}}G_s(\hat{x}/n)
			\Bigg]\; \rmd\hat{w} \Bigg|
			\\&\qquad\qquad
			\leq   
			\sum_{\hat{z}/n \; \in \; A_1^{\varepsilon}} 
			\sup_{s \in [0,T]} 
			I_1 \Big( \hat{z}, n, \Delta_{(\cdot)}G_s\big(\tfrac{\hat{x}}{n}\big)  \Big) 
			+ \sum_{\hat{z}/n \; \in \; A_1^{\varepsilon}} 
			\sup_{s \in [0,T]}  
			\big| \Delta_{\hat{z}/n}G_s\big(\tfrac{\hat{x}}{n}\big) \big|  I_2( \hat{z}, n),
		\end{split}
	\end{align}
	where $\Delta_{(\cdot)}$ is given in the last item in Definition \ref{def:micro-op}; and for every $n \geq 1$, $\varepsilon \in (0,1)$, any $\hat{z}$ such that $\hat{z}/n \; \in \; A_1^{\varepsilon}$ and $f: \mathbb{R}^d \mapsto \mathbb{R}$, $I_1( \hat{z}, n, f)$ and $I_2( \hat{z}, n)$ are given by
	\begin{align}
		I_1( \hat{z}, n, f)&:= \int_{ \tfrac{z_1}{n} }^{\tfrac{ z_1 +1}{n}} \ldots \int_{ \tfrac{z_d}{n} }^{\tfrac{ z_d +1}{n}} 
		\frac{1}{ |\hat{w}|^{\gamma+d}}
		\big|  f \big( \tfrac{ \hat{z}}{n} \big) -     f( \hat{w} ) \big|\;\rmd \hat{w}  \geq 0, \label{defI1zn} \\
		I_2( \hat{z}, n)&:=\int_{ \tfrac{z_1}{n} }^{\tfrac{ z_1 +1}{n}} \ldots \int_{ \tfrac{z_d}{n} }^{\tfrac{ z_d +1}{n}}  
		\Bigg[  
		 \frac{1}{ |\hat{z}/n|^{\gamma+d}}
		-    
		\frac{1}{ |\hat{w}|^{\gamma+d}} \Bigg]\;\rmd\hat{w} \geq 0. \label{defI2zn}
	\end{align}
Combining \eqref{defnabladelta}, the Mean Value Theorem and \eqref{sphe}, it is not hard to prove that
	\begin{equation} \label{boundI1zna}
		\forall G \in \mathcal{S}^0(\mathbb{R}^d), \quad \sum_{\hat{z}/n \; \in \; A_1^{\varepsilon}} \sup_{s \in [0,T]} 
		I_1 \Big( \hat{z}, n, \Delta_{(\cdot)}G_s \big(\tfrac{\hat{x}}{n} \big)  \Big)
		\lesssim   \nnorm{G}_{1,0}  \frac{\varepsilon^{-\gamma-1} }{n},
	\end{equation}
for every $n \geq 1$ and every $\varepsilon \in (0, \; 1)$. Next, applying the Mean Value Theorem to the function $\widetilde{f}: \mathbb{R}^{d} \setminus \{ \hat{0} \}  \mapsto \mathbb{R}$, given by $\widetilde{f}(\hat{v}):= | \hat{v}|^{-\gamma-d}$ for any $\hat{v} \neq \hat{0}$, we get 
	\begin{equation} \label{boundI2zna}
				I_2( \hat{z}, n   ) 		\leq  
		n^{\gamma+d+1} | \hat{z} |^{-\gamma-d-1} \sum_{j=1}^d  \int_{ \tfrac{z_1}{n} }^{\tfrac{ z_1 +1}{n}} \ldots \int_{ \tfrac{z_d}{n} }^{\tfrac{ z_d +1}{n}} \Bigg( w_j - \frac{z_j}{n} \Bigg)\;\rmd\hat{w} =\frac{d n^{\gamma}}{2 | \hat{z} |^{\gamma+d+1}},
	\end{equation}
for every $n \geq 1$, $\varepsilon \in (0,1)$, any $\hat{z}$ such that $\hat{z}/n \; \in \; A_1^{\varepsilon}$. Combining \eqref{taylor0theta} and \eqref{boundI2zna}, we conclude that the second sum in \eqref{eqlapfrac3a} is bounded from above by a constant times
	\begin{align*}
		\frac{1}{n} \frac{\nnorm{G}_{0,0} }{n^d} 
		\sum_{| \hat{z} | \geq  \varepsilon n} 
		\big( \tfrac{ | \hat{z} | }{n} \big)^{- \gamma -d-1} 
		\lesssim \frac{\nnorm{G}_{0,0} }{n} \int_{| \hat{u}| \geq \varepsilon} | \hat{u}|^{- \gamma -d-1} \rmd \hat{u}   \lesssim  \frac{\varepsilon^{-\gamma-1} }{n} \nnorm{G}_{0,0} .
	\end{align*}
	In the last line we used \eqref{sphe}. Combining the last display with \eqref{eqlapfrac3a} and \eqref{boundI1zna}, we get \eqref{2boundBnear0pos} and the proof ends.
\end{proof}
This completes the proof of Proposition \ref{convL1B}. In next subsection, we obtain Proposition \ref{convL1F}.

\subsection{Proof of Proposition \ref{convL1F}}

We begin by stating a Lemma which is useful to estimate multidimensional discrete sums.
\begin{lem} \label{lemind}
	Let $j \in \mathbb{N}_+$, $z_j \neq 0 \in \mathbb{Z}$ and $q \in \mathbb{R}$ such that $q  + j - 1< 0$. Then, it holds
	\begin{align} \label{eqlemind}
		\sum_{z_1 \in \mathbb{Z}} \sum_{z_2 \in \mathbb{Z}} \ldots \sum_{z_{j-1} \in \mathbb{Z}} | (z_1, z_2, \ldots, z_{j-1}, z_j)|^q \leq C_{ind}(j,q) |z_j|^{q+j-1},
	\end{align}
	where $C_{ind}(j, \; q)$ is a positive constant depending only on $j$ and $q$.
\end{lem}
\begin{proof}
	If $j=1$, the result holds for $C_{ind}(j, \;q)=1$. For $j \geq 2$, the proof follows by induction.
\end{proof}
Keeping in mind Proposition A.6 of \cite{CGJ2}, in order to obtain Proposition \ref{convL1F}, we state a result which corresponds to the terms of $\hat{x} \in \mathbb{Z}^d$ such that $\hat{x}$ is far from the origin. 
\begin{prop} \label{convL1Ffar0}
	Let $\gamma \in (0, \; 2)$, $d \geq 1$ and $G \in  \mathcal S_{\gamma}^{d}$. Then
	\begin{align*}
		& \varlimsup_{M \rightarrow \infty}  \varlimsup_{n \rightarrow \infty}   
		\sup_{s \in [0, T]}\frac{1}{n^d}  
		\sum_{ |\hat{x}| \geq M n }     \big | \Delta^{\gamma/2}_{n,\mcb F} G_s \big(\tfrac{ \hat{x}}{n} \big)  - 
		 \Delta_{\star}^{ \gamma / 2 }   G_s  \big(\tfrac{ \hat{x}}{n} \big) \big|  =0.
	\end{align*}
\end{prop}
The next result corresponds to the terms of $\hat{x} \in \mathbb{Z}^d$ such that $x_d$ is close to the origin. 
\begin{prop} \label{convL1Fnear02}
	Let $\gamma \in (0, \; 2)$, $d \geq 1$ and $G \in \mathcal S_{\gamma}^{d}$. Then
	\begin{align*}
		&\varlimsup_{M \rightarrow \infty} \varlimsup_{\varepsilon \rightarrow 0^+}  \varlimsup_{n \rightarrow \infty}  \sup_{s \in [0, T]}\frac{1}{n^d} \sum_{\hat{x}: |x_d| < \varepsilon n, \; |\hat{x}| < M n}   \big | \Delta^{\gamma/2}_{n,\mcb F} G_s \big(\tfrac{ \hat{x}}{n} \big)  - 
		 \Delta_{\star}^{ \gamma / 2 }   G_s  \big(\tfrac{ \hat{x}}{n} \big) \big| =0.  
	\end{align*}
\end{prop}
Next, we prove Proposition \ref{convL1F}, by assuming Propositions \ref{convL1Ffar0} and \ref{convL1Fnear02}, that will be proved later on.
\begin{proof}[Proof of  Proposition \ref{convL1F}]
Combining \eqref{lapGdisc} with Propositions \ref{convL1Ffar0} and \ref{convL1Fnear02}, we are done if we can prove that for any $G	\in \mathcal S_{\gamma}^{d}$, it holds
	\begin{align} \label{triplimB2}
\varlimsup_{M \rightarrow \infty} \varlimsup_{\varepsilon \rightarrow 0^+} \varlimsup_{n \rightarrow \infty}  \sup_{s \in [0, T]}\frac{1}{n^d} \sum_{\hat{x}: |x_d| \geq \varepsilon n, \; |\hat{x}| < M n}      \big | \Delta^{\gamma/2}_{n,\mcb F} G_s \big(\tfrac{ \hat{x}}{n} \big)  - 
		 \Delta_{\star}^{ \gamma / 2 }   G_s  \big(\tfrac{ \hat{x}}{n} \big) \big|  =0, \end{align}
In order to treat the triple limit in last display, from Definition \ref{defKn} and recalling the last item in Definition \ref{def:micro-op}, we get 
	\begin{equation*} 
		\frac{n^d}{c_\gamma/2} \Delta^{\gamma/2}_{n,\mcb F} G_s \big(\tfrac{ \hat{x}}{n} \big) 
		=  
		\sum_{ | \hat{z} | = 1}^{\varepsilon n - 1} 
		\frac{\Delta_{\hat{z}/n}G_s \big(\tfrac{ \hat{x} }{n} \big)}{| \hat{z}/n|^{d+\gamma}} 
		+ 
		\sum_{ | \hat{z} |=\varepsilon n }^{ |x_d| - 1}  
		\frac{\Delta_{\hat{z}/n}G_s\big(\tfrac{ \hat{x} }{n}\big)}{| \hat{z}/n|^{d+\gamma}} 
		+
		\begin{dcases}
		2	\sum_{ | \hat{z} |\geq |x_d|, \;  z_d < - x_d } 
			\frac{\nabla_{\hat{z}/n}G_s\big(\tfrac{ \hat{x} }{n}\big)}{| \hat{z}/n|^{d+\gamma}},\; x_d < 0, 
			\\
		2	\sum_{ | \hat{z} |\geq |x_d|, \;  z_d \geq - x_d } 
			\frac{\nabla_{\hat{z}/n}G_s\big(\tfrac{ \hat{x} }{n}\big)}{| \hat{z}/n|^{d+\gamma}}, \; x_d \geq 0, 
		\end{dcases}
	\end{equation*}
	for any $\hat{x}$ such that $|x_d| \geq \varepsilon n$, where $\Delta_{(\cdot)}$ and $\nabla_{(\cdot)}$ are given in the last item in Definition \ref{def:micro-op}. Combining the last display with \eqref{lapfracalt}, in order to obtain \eqref{triplimB2}, it is enough to prove that the triple limits $\varlimsup_{M \rightarrow \infty} \varlimsup_{\varepsilon \rightarrow 0^+} \varlimsup_{n \rightarrow \infty}$ of each of the following suprema
	\begin{align}
		&    \sup_{s \in [0, T]}\frac{1}{n^d}  \sum_{\hat{x}: |x_d| \geq \varepsilon n, \; |\hat{x}| < M n} \Bigg( \frac{1}{n^d} \sum_{|\hat{z}| =  1}^{\varepsilon n - 1} \frac{ \abs{\Delta_{\hat{z}/n}G_s\big(\tfrac{\hat{x}}{n}\big)} }{|\hat{z}/n|^{\gamma+d}}
		  +   \int_{| \hat{w} | < \varepsilon } \frac{| \Delta_{\hat{w}}G_s(\hat{x}/n) |}{|\hat{w}|^{d+\gamma}}\;\rmd\hat{w} \Bigg),  \nonumber \\
		&\sup_{s \in [0, T]}\frac{1}{n^d}  \sum_{\hat{x}: |x_d| \geq \varepsilon n, \; |\hat{x}| < M n} \Bigg|  \frac{1}{n^d} \sum_{|\hat{z}| =  \varepsilon n }^{x_d- 1} \frac{ \Delta_{\hat{z}/n}G_s(\hat{x}/n) }{ |\hat{z}/n|^{-\gamma-d} } -    \int_{\varepsilon \leq | \hat{w} | < x_d/n } \frac{ \Delta_{\hat{w}}G_s(\hat{x}/n) } { |\hat{w}|^{\gamma+d}}\; \rmd\hat{w}  \Bigg|, \nonumber \\
		&\sup_{s \in [0, T]}\frac{1}{n^d}  \sum_{\hat{x}: x_d \leq - \varepsilon n, \; |\hat{x}| < M n} \Bigg|  \frac{1}{n^d}  \sum_{|\hat{z}| \geq x_d, \; z_d < - x_d } \frac{ \nabla_{\hat{z}/n}G_s(\hat{x}/n) }{|\hat{z}/n|^{\gamma+d}} -    \int_{x_d/n \leq | \hat{w} |, \; w_d < - x_d/n } \frac{ \nabla_{\hat{w}}G_s(\hat{x}/n) }{ |\hat{w}|^{\gamma+d}}\;\rmd\hat{w}  \Bigg|, \nonumber \\
		&\sup_{s \in [0, T]}\frac{1}{n^d}  \sum_{\hat{x}: x_d \geq \varepsilon n, \; |\hat{x|} < M n} \Bigg|  \frac{1}{n^d}  \sum_{|\hat{z}| \geq x_d, \; z_d > - x_d } \frac{ \nabla_{\hat{z}/n}G_s(\hat{x}/n) }{|\hat{z}/n|^{\gamma+d}} -    \int_{x_d/n \leq | \hat{w} |, \; w_d > - x_d/n } \frac{ \nabla_{\hat{w}}G_s(\hat{x}/n) }{ |\hat{w}|^{\gamma+d}}\;\rmd\hat{w}  \Bigg|, \nonumber \\
		&\sup_{s \in [0, T]}\frac{1}{n^d}  \sum_{\hat{x}: x_d \geq \varepsilon n, \; |\hat{x|} < M n} \Bigg|  \frac{1}{n^d}  \sum_{|\hat{z}| \geq x_d, \; z_d = - x_d } \frac{ \nabla_{\hat{z}/n}G_s(\hat{x}/n)}{|\hat{z}/n|^{\gamma+d}}   \Bigg|, \label{5trilim}
	\end{align}
	are all equal to zero. We stress that \eqref{5trilim} refers only to the supremum in the last line. From \eqref{1boundBnear0}, the triple limit of the first supremum in the last display is bounded from above by
	\begin{align*}
		\varlimsup_{M \rightarrow \infty} \varlimsup_{\varepsilon \rightarrow 0^+} \varlimsup_{n \rightarrow \infty}   \frac{1}{n^d} \sum_{  |\hat{x}| < M n } C_1(G) \varepsilon^{2 - \gamma} \leq C_1(G)  \varlimsup_{M \rightarrow \infty} \varlimsup_{\varepsilon \rightarrow 0^+}    (2M)^d \varepsilon^{2-\gamma} =0.
	\end{align*}
Next, from the triangle inequality and \eqref{taylor0theta}, with an application of Lemma \ref{lemind} for $q=- \gamma - d$ and $j=d$, we have that the triple limit in \eqref{5trilim} is also equal to zero. 
	
	In order to treat the three triple limits above \eqref{5trilim}, we claim that for every $G \in \mathcal{S}^2(\mathbb{R}^d)$, there exists $C_3(G)$ satisfying, for every $\hat{x} \in \mathbb{Z}^d$, every $s \in [0,T]$,
	\begin{align} 
		&\Bigg|  \frac{1}{n^d}  \sum_{|\hat{z}| =  \varepsilon n }^{x_d- 1} \frac{\nabla_{\hat{z}/n}G_s(\hat{x}/n) }{ |\hat{z}/n|^{-\gamma-d} } -    \int_{\varepsilon \leq | \hat{w} | < x_d } \frac{ \Delta_{\hat{w}}G_s(\hat{x}/n) } { |\hat{w}|^{\gamma+d}}\;\rmd\hat{w}  \Bigg|  \leq C_3(G) \frac{\varepsilon^{-\gamma - 1}}{n} ,\quad |x_d| \geq 2 \varepsilon n,\label{1boundFnear0} \\
		&\Bigg|  \frac{1}{n^d}  \sum_{|\hat{z}| \geq x_d, \; z_d > - x_d } \frac{ \nabla_{\hat{z}/n}G_s(\hat{x}/n)}{|\hat{z}/n|^{\gamma+d}} -    \int_{x_d \leq | \hat{w} |, \; w_d > - x_d } \frac{ \nabla_{\hat{w}}G_s(\hat{x}/n) }{ |\hat{w}|^{\gamma+d}}\; \rmd\hat{w}  \Bigg| \leq C_3(G) \frac{\varepsilon^{-\gamma - 1}}{n}, \quad x_d \geq 2 \varepsilon n,\label{2boundFnear0} \\
		&\Bigg|  \frac{1}{n^d}  \sum_{|\hat{z}| \geq x_d, \; z_d > - x_d } \frac{ \nabla_{\hat{z}/n}G_s(\hat{x}/n)}{|\hat{z}/n|^{\gamma+d}} -    \int_{x_d \leq | \hat{w} |, \; w_d < - x_d } \frac{ \nabla_{\hat{w}}G_s(\hat{x}/n) }{ |\hat{w}|^{\gamma+d}}\; \rmd\hat{w}  \Bigg| \leq C_3(G) \frac{\varepsilon^{-\gamma - 1}}{n},
  \quad x_d \leq -2 \varepsilon n,\label{3boundFnear0}
	\end{align}
	for every $n \geq 1$ and every $\varepsilon \in (0, \; 1)$. Indeed, if \eqref{1boundFnear0} , \eqref{2boundFnear0} and \eqref{3boundFnear0} hold, the sum of the triple limits of the three terms is bounded from above by
	\begin{align*}
		&3 C_3(G) \varlimsup_{M \rightarrow \infty} \varlimsup_{\varepsilon \rightarrow 0^+}  \varlimsup_{n \rightarrow \infty} \frac{1 }{n^d}\sum_{ |\hat{x}| < M n } \frac{\varepsilon^{-\gamma - 1}}{n} \leq 3 C_3(G) \varlimsup_{M \rightarrow \infty} \varlimsup_{\varepsilon \rightarrow 0^+}  \varlimsup_{n \rightarrow \infty} (2M)^d \frac{\varepsilon^{-\gamma - 1}}{n}  =0.
	\end{align*}
	We obtain \eqref{1boundFnear0}, \eqref{2boundFnear0} and \eqref{3boundFnear0} in the same way as it was done for proving \eqref{2boundBnear0}, since for \textit{each} of them, we have a sum ranging over $\hat{z} \in \mathbb{Z}^d$ and an integral over $\hat{w}$ such that $\hat{z}/n$ and $\hat{w}$ are in \textit{the same subset} of $\{ \hat{w} \in \mathbb{R}^d: \; \varepsilon \leq | \hat{w}| \}$. 
	
	For instance, order to prove \eqref{1boundFnear0}, we need to estimate the difference between a sum ranging over $\hat{z} \in \mathbb{Z}^d$ such that $\hat{z}/n \in B_0^{\varepsilon,n, x_d} := \{ \hat{w} \in \mathbb{R}^d: \; \varepsilon \leq | \hat{w}| < x_d/n \}$ and an integral over $\hat{w} \in B_0^{\varepsilon, n, x_d}$. We do so by partitioning $B_0^{\varepsilon, n, x_d}$ in $2^d$ subsets, depending on whether the coordinates $w_1, \ldots, w_d$ of $\hat{w}$ are non-negative or negative. More exactly, define 
	\begin{align*} 
		B_1^{\varepsilon,n, x_d} :=& \big\{ \hat{w} \in \mathbb{R}^d: \; \varepsilon \leq | \hat{w}| < x_d/n, \quad \min \{w_1, \ldots, w_d \} \geq 0 \big \},
	\end{align*}
	and the remaining $2^{d}-1$ subsets in an analogous way. Therefore, if we can prove that there exists $C_3(G)$ satisfying
	\begin{equation*} 
		\sup_{s \in [0,T]} \Bigg| \frac{1}{n^d}   \sum_{\hat{z}/n \; \in \; B_1^{\varepsilon, n, x_d}}  \frac{\Delta_{\hat{z}/n}G_s(\hat{x}/n)}{ |\hat{z}/n|^{-\gamma-d} }  -   \int_{\hat{w} \; \in \; B_1^{\varepsilon, n, x_d}} \frac{\Delta_{\hat{w}}G_s(\hat{x}/n)}{ |\hat{w}|^{\gamma+d}}\;\rmd\hat{w}  \Bigg| \leq \frac{C_3(G)}{2^d} \frac{\varepsilon^{-\gamma-1}}{n},
	\end{equation*} 
\eqref{1boundFnear0} follows. Next, we rewrite the term on the left-hand side of the last display as
	\begin{align*}
		& \sup_{s \in [0,T]} \Bigg|  \sum_{\hat{z}/n \; \in \; B_1^{\varepsilon, n, x_d}} \int_{ \tfrac{z_1}{n} }^{\tfrac{ z_1 +1}{n}} \ldots \int_{ \tfrac{z_d}{n} }^{\tfrac{ z_d +1}{n}} \big[  \big( \tfrac{| \hat{z}| }{n} \big)^{- \gamma -d}  \Delta_{\hat{z}/n}G_s(\hat{x}/n)  -    |\hat{w}|^{-\gamma-d} \Delta_{\hat{w}}G_s(\hat{x}/n) \big]\;\rmd\hat{w} \Bigg|  \\
		\leq &   \sum_{\hat{z}/n \; \in \; A_1^{\varepsilon}} \sup_{s \in [0,T]} 
			I_1 \Big( \hat{z}, n, \Delta_{(\cdot)}G_s\big(\tfrac{\hat{x}}{n}\big)  \Big) + \sum_{\hat{z}/n \; \in \; A_1^{\varepsilon}} \sup_{s \in [0,T]}  
			\big| \Delta_{\hat{z}/n}G_s\big(\tfrac{\hat{x}}{n}\big) \big|  I_2( \hat{z}, n). 
	\end{align*}
	The inequality in the last line holds, since $B_1^{\varepsilon,n, x_d} \subset A_1^{\varepsilon}$, given in \eqref{defA1eps}, and $I_1 \Big( \hat{z}, n, \Delta_{(\cdot)}G_s\big(\tfrac{\hat{x}}{n}\big)  \Big), I_2( \hat{z}, n)$, given in \eqref{defI1zn} and \eqref{defI2zn}, are non-negative. Therefore, keeping in mind \eqref{eqlapfrac3a}, we conclude from \eqref{boundI1zna}, \eqref{boundI2zna} and \eqref{taylor0theta} that \eqref{1boundFnear0} holds. By performing exactly the same reasoning, we get \eqref{2boundFnear0} and \eqref{3boundFnear0}. This ends the proof of Proposition \ref{convL1F}.	
\end{proof}
Now we prove Proposition \ref{convL1Ffar0}.
\begin{proof}[Proof of Proposition \ref{convL1Ffar0}]
	From \eqref{boundH2G} and \eqref{L1far0} we get
	\begin{align*} 
		& \varlimsup_{M \rightarrow \infty}  \varlimsup_{n \rightarrow \infty}  \sup_{s \in [0, T]}\frac{1}{n^d}  \sum_{   |\hat{x}| \geq M n }    \big|  \Delta_{\star}^{ \gamma / 2 }   G_s  \big(\tfrac{ \hat{x}}{n} \big) \big| =0.
	\end{align*}
	Therefore, we are done if we can prove that
	\begin{align} \label{2boundFfar0}
		& \varlimsup_{M \rightarrow \infty}  \varlimsup_{n \rightarrow \infty}  \sup_{s \in [0, T]}\frac{1}{n^d}  \sum_{   |\hat{x}| \geq M n }  \big | \Delta^{\gamma/2}_{n,\mcb F} G_s \big(\tfrac{ \hat{x}}{n} \big) \big|  =0.
	\end{align}
	From Definition \ref{defKn}, for every $G \in \mathcal{S}^2(\mathbb{R}^d)$, $s \in [0, T]$ and $n \geq 1$, it holds 
	\begin{equation} \label{KnFalt}
		\Delta^{\gamma/2}_{n,\mcb F} G_s \left(\tfrac{ \hat{x}}{n} \right):=
		\begin{dcases}
		n^{\gamma}	\sum_{z_d= - \infty}^{-1- x_d} \sum_{\hat{z}_{*} \in \mathbb{Z}^{d-1}} \big[ G_s \big(\tfrac{\hat{x}_{*} + \hat{z}_{*}}{n}, \tfrac{x_d + z_d}{n} \big) - G_s \big(\tfrac{ \hat{x}_{*}}{n}, \tfrac{x_d}{n} \big) \big]  p_{\gamma} ( \hat{z}_{*}, z_d ), \quad & x_d <0, \\
			n^{\gamma}\sum_{z_d= - x_d}^{\infty} \sum_{\hat{z}_{*} \in \mathbb{Z}^{d-1}} \big[ G_s \big(\tfrac{\hat{x}_{*} + \hat{z}_{*}}{n}, \tfrac{x_d + z_d}{n} \big) - G_s \big(\tfrac{ \hat{x}_{*}}{n}, \tfrac{x_d}{n} \big) \big]  p_{\gamma} ( \hat{z}_{*}, z_d ), \quad & x_d \geq 0,
		\end{dcases}
	\end{equation}
	If $G \in \mathcal{S}_c^2(\mathbb{R}^d)$, from \eqref{KnFalt}, \eqref{defbG} and \eqref{1supM}, we get \eqref{2boundFfar0} and the proof ends.
	
	On the other hand, if $G \notin \mathcal{S}_c^2(\mathbb{R}^d)$, we have from \eqref{defSalgamd3} that $d=1$, $\gamma \in (1, 2)$ and $G \in \mathcal{S}^2(\mathbb{R}^d)$. Thus, from \eqref{KnFalt}, \eqref{2boundBfar0a}, \eqref{2boundBfar0b} and  \eqref{taylor0theta}, we get \eqref{2boundFfar0} and the proof ends.
\end{proof}
Finally, we end this section by presenting the proof of Proposition \ref{convL1Fnear02}.
\begin{proof}[Proof of Proposition \ref{convL1Fnear02}]
	From Proposition \ref{propL1alpha0} and \eqref{L1near0} we get
	\begin{align*} 
		& \varlimsup_{M \rightarrow \infty} \varlimsup_{\varepsilon \rightarrow 0^+}  \varlimsup_{n \rightarrow \infty}   \sup_{s \in [0, T]}\frac{1}{n^d} \sum_{\hat{x}: |x_d| < \varepsilon n, \; |\hat{x}| < M n}    \big | \Delta_{\star}^{ \gamma / 2 }   G_s  \big(\tfrac{ \hat{x}}{n} \big) \big| =0.
	\end{align*}
	Therefore, we are done if we can prove that
	\begin{align} 
		&\varlimsup_{M \rightarrow \infty} \varlimsup_{\varepsilon \rightarrow 0^+}  \varlimsup_{n \rightarrow \infty}   \sup_{s \in [0, T]}\frac{1}{n^d} \sum_{\hat{x}: |x_d| < \varepsilon n, \; |\hat{x}| < M n}   \big | \Delta^{\gamma/2}_{n,\mcb F} G_s \big(\tfrac{ \hat{x}}{n} \big) - A_s^{n, G}(\hat{x})  \big|  =0,  \label{2boundFnear0a} \\
		&\varlimsup_{M \rightarrow \infty} \varlimsup_{\varepsilon \rightarrow 0^+}  \varlimsup_{n \rightarrow \infty}   \sup_{s \in [0, T]}\frac{1}{n^d} \sum_{\hat{x}: |x_d| < \varepsilon n, \; |\hat{x}| < M n}   \big | A_s^{n, G}(\hat{x})  \big|  =0,  \label{2boundFnear0b}
	\end{align}
	where $A_s^{n, G}$ is the term corresponding to $z_d=0$ in \eqref{KnFalt} and it is given by
	\begin{align*}
		A_s^{n, G}(\hat{x}):= n^{\gamma} \sum_{\hat{z}_{*} \in \mathbb{Z}^{d-1}} \big[ G_s \big(\tfrac{\hat{x}_{*} + \hat{z}_{*}}{n}, \tfrac{x_d }{n} \big) - G_s \big(\tfrac{ \hat{x}_{*}}{n}, \tfrac{x_d}{n} \big) \big]  p_{\gamma} ( \hat{z}_{*}, 0 ).
	\end{align*}
We distinguish the terms corresponding to $z_d \neq 0$ because they can be estimated by applying Lemma \ref{lemind} for $k=d$. Indeed, from \eqref{KnFalt}, the triple limit in \eqref{2boundFnear0a} is bounded from above by
	\begin{align*}
		&\varlimsup_{M \rightarrow \infty}  \varlimsup_{\varepsilon \rightarrow 0^+}  \varlimsup_{n \rightarrow \infty}   \frac{  2 \nnorm{G}_{0,1}  }{n^{d+\gamma/2}}  c_{\gamma} \sum_{\hat{x}: |x_d| < \varepsilon n, \; |\hat{x}| < M n} n^{\gamma} \sum_{|z_d|=1}^{\infty} \sum_{\hat{z}_{*} \in \mathbb{Z}^{d-1}} | (\hat{z}_{*},\; z_d) |^{-\gamma-d+\gamma/2} \\
		+ &\varlimsup_{M \rightarrow \infty}  \varlimsup_{\varepsilon \rightarrow 0^+}  \varlimsup_{n \rightarrow \infty}   \frac{ \nnorm{G}_{0,0} }{n^d} c_{\gamma} \sum_{\hat{x}: |x_d| < \varepsilon n, \; |\hat{x}| < M n} n^{\gamma} \sum_{|z_d|=n}^{\infty} \; \sum_{\hat{z}_{*} \in \mathbb{Z}^{d-1}} | (\hat{z}_{*},\; z_d) |^{-\gamma-d}.
	\end{align*}
	The term in the first line of the last display was obtained by applying \eqref{GSHold} for $\delta = \gamma/2 \in (0, \; 1)$ and can be further estimated by applying Lemma \ref{lemind} for $j=d$ and for $q=-d-\gamma/2$. In this way, we get \eqref{2boundFnear0a}.
	
It still remains to prove \eqref{2boundFnear0b}. If $d=1$, $A_s^{n, G}(\hat{x})=0$ and \eqref{2boundFnear0b} holds trivially. On the other hand, if $d \geq 2$, \eqref{2boundFnear0b} holds due to \eqref{taylor2}, \eqref{semnorm}, \eqref{taylor0theta} and \eqref{sphe}. This ends the proof.
\end{proof}

\section{Useful estimates} \label{useest}	

\subsection{Proof of Proposition \ref{boundYnsep}} \label{appertSEP}

In this subsection we present the proof for Proposition \ref{boundYnsep}. Recall the spaces of functions given in Definition \ref{defschw2}. 
\begin{proof}[Proof for Proposition \ref{boundYnsep}]

Let $\gamma \in (0,2)$ and $d \geq 1$. We remark that it is enough to treat the case $G \in S^{d}_{\gamma}$, due to the term $\mathbbm{1}_{ \{ \hat{x}, \hat{y} \in \mcb{F}  \} }$ in \eqref{defYnSEP}, Definition \ref{def:slow-bonds} and \eqref{defGdisc}. Now from \eqref{defYnSEP}, we get that $Y^n_{ \text{SEP} }(G)$ is bounded from above by
\begin{equation} \label{defYnSEPb}
\begin{split}
& \sup_{s \in [0,T] } \frac{ {n^{\gamma} } }{n^d} \Bigg\{ \sum_{\hat{x}} \sum_{j=1}^{d-1}  \big| \Delta_{\hat{e}_j/n} G_s \big(\tfrac{\hat{x}}{n}\big) \big|     + \sum_{\hat{x}_{\star} \in \mathbb{Z}^{d-1} }  \sum_{x_d \notin \{ -1,0 \} } \big| \Delta_{\hat{e}_d/n} G_s \big( \tfrac{\hat{x}_{\star}}{n},  \tfrac{x_d}{n} \big) \big| \Bigg\} \\
+& \sup_{s \in [0,T] } \frac{ {n^{\gamma} } }{n^d}  \sum_{\hat{x}_{\star} \in \mathbb{Z}^{d-1} } \big[ \big| \nabla_{\hat{e}_d/n} G_s \big( \tfrac{\hat{x}_{\star}}{n},  \tfrac{-2}{n} \big) \big| + \big| \nabla_{\hat{e}_d/n} G_s \big( \tfrac{\hat{x}_{\star}}{n},  \tfrac{-1}{n} \big) \big| \big], 
\end{split}
\end{equation}
where $\Delta_{\hat{e}_j/n}$ and $\nabla_{ \hat{e}_d/n}$ are given in the last bullet point of Definition \ref{def:micro-op}. Next, we treat two cases differently, according to \eqref{defSdifgamma}: $G \in S^{1, 2} \big( [0,T] \times  \mathbb{R}^d \big)$, if $\gamma \in (1,2)$ and $d=1$; and $G \in C_c^{1, 2} \big( [0,T] \times  \mathbb{R}^d \big)$, otherwise.

\textbf{I)} The case $\gamma \in (1,2)$, $d=1$ and $G \in S^{1, 2} \big( [0,T] \times  \mathbb{R} \big)$. Note that \eqref{defYnSEPb} can be rewritten as
\begin{align*}
& \sup_{s \in [0,T] } \frac{ {n^{\gamma} } }{n}   \big[ \big| \nabla_{1/n} G_s \big(   \tfrac{-2}{n} \big) \big| + \big| \nabla_{1/n} G_s \big(   \tfrac{-1}{n} \big) \big| \big] + \sup_{s \in [0,T] } \frac{ {n^{\gamma} } }{n} \big| \Delta_{1/n} G_s \big(   \tfrac{x_d}{n} \big) \big| +  \sup_{s \in [0,T] } \frac{ {n^{\gamma} } }{n}  \sum_{|x_d| \geq 2} \big| \Delta_{1/n} G_s \big(   \tfrac{x_d}{n} \big) \big|.
\end{align*}
Applying the Mean Value Theorem to $G_s$ in the terms with $\nabla_{1/n} G_s$ and a second-order Taylor expansion to $G_s$ in the terms with $\Delta_{1/n} G_s$, we have from \eqref{defschw} that the last display is bounded from above by
\begin{align*}
\frac{n^{\gamma} }{n^2} \Bigg[ \nnorm{G}_{0,1} + \nnorm{G}_{0,1} + \frac{\nnorm{G}_{0,2}}{n} + \frac{1}{n}  \sum_{|x_d| \geq 2} F_G \big(   \tfrac{x_d}{n} \big) \; \Bigg],
\end{align*}
where $F_G:\mathbb{R} \mapsto \mathbb{R}$ is given by $F_G(u):=\sup_{(s, v): s \in [0,T], |v - u| \leq |u|/2 } | \partial_{11}  G_s(v) |$, for any $u \in \mathbb{R}$. In particular, $F_G \in L^1(\mathbb{R})$, since $G \in S^{1, 2} \big( [0,T] \times  \mathbb{R} \big)$. Therefore, the last display is bounded from above by some positive constant times $n^{\gamma-2}$, which vanishes as $n \rightarrow \infty$, due to the fact that $\gamma < 2$.

\textbf{II)} The case $G \in C_c^{1, 2} \big( [0,T] \times  \mathbb{R}^d \big)$. From \eqref{defbG}, the display in \eqref{defYnSEPb} is bounded from above by
\begin{align*}
\sup_{s \in [0,T] } \frac{ {n^{\gamma} } }{n^d} \Bigg\{ \sum_{|\hat{x}| \leq b_G n} \sum_{j=1}^{d-1}  \big| \Delta_{\hat{e}_j/n} G_s \big(\tfrac{\hat{x}}{n}\big) \big| + \sum_{\hat{x}_{\star}: |\hat{x}_{\star}| \leq b_G n } \big[ \big| \nabla_{\hat{e}_d/n} G_s \big( \tfrac{\hat{x}_{\star}}{n},  \tfrac{-2}{n} \big) \big| + \big| \nabla_{\hat{e}_d/n} G_s \big( \tfrac{\hat{x}_{\star}}{n},  \tfrac{-1}{n} \big) \big| \big] \Bigg\}.
\end{align*} 
Applying a second-order Taylor expansion to $G_s$ in the terms with $\Delta_{\hat{e}_j/n} G_s$, and the Mean Value Theorem to $G_s$ in the terms with $\nabla_{\hat{e}_d/n} G_s$, we have from \eqref{defschw} that the last display is bounded from above by
\begin{align*}
\frac{n^{\gamma} }{n^2} \Bigg\{ \frac{(1 +2  b_G n)^d}{n^d} \nnorm{G}_{0,2} + \frac{(1 +2  b_G n)^{d-1 } }{n^{d-1 }} \nnorm{G}_{0,1} \Bigg\} \lesssim n^{\gamma-2},
\end{align*}
which vanishes as $n \rightarrow \infty$, due to the fact that $\gamma < 2$.

\end{proof} 

\subsection{Bounds applied in the proof of Proposition \ref{L1alphagen}} \label{proL1alphagen}

Recall the spaces of functions given in Definitions \ref{deftesconttime} and \ref{deftesdisctime}. In this subsection, we obtain the upper bounds applied in the proof of Proposition \ref{L1alphagen}. For $G \in S_{\gamma,\star}^{d} \subsetneq \mathcal S_{\gamma,\star}^{d}$, the following inequality will be useful:
\begin{equation} \label{contrlip2}
	\forall \gamma \in (0, \; 2), \; \forall d \geq 1, \; \forall G \in \mathcal S_{\gamma,\star}^{d},  \quad \frac{n^{\gamma}}{n^{d} } \sup_{s \in [0,T]} \sum _{ \{ \hat{x}, \hat{y} \} \in \mcb S }   \big| G_s\big( \tfrac{\hat{y}}{n}\big) - G_s\big( \tfrac{\hat{x}}{n}\big) \big| p_{\gamma}(\hat{y}-\hat{x}) \lesssim 1,
\end{equation}
where $\mathcal S_{\gamma,\star}^{d}$ is the subset of $\mathcal S_{\gamma,0}^{d}$ whose elements G satisfy \eqref{SROB0}. 

On the other hand, if $G \in S_{\gamma,0}^{d} \setminus S_{\gamma,\star}^{d} \subsetneq  \mathcal S_{\gamma,0}^{d} \setminus \mathcal S_{\gamma,\star}^{d}$, we will make use of the following estimates:
\begin{equation} \label{contrslow2}
	\begin{split}
		&	\forall \gamma \in (0, \; 2), \; \forall d \geq 1, \; \forall G \in \mathcal{S}_c^0(\mathbb{R}^{d\star}),  \quad \frac{n^{\gamma}}{n^{d} } \sup_{s \in [0,T]} \sum _{ \{ \hat{x}, \hat{y} \} \in \mcb S }   \big| G_s\big( \tfrac{\hat{y}}{n}\big) - G_s\big( \tfrac{\hat{x}}{n}\big) \big| p_{\gamma}(\hat{y}-\hat{x}) \lesssim r_n^\gamma, \\
		&	\forall \gamma \in (1, \; 2),  \forall G \in \mathcal{S}^0(\mathbb{R}^{*}), \quad \frac{n^{\gamma}}{n } \sup_{s \in [0,T]} \sum _{ \{ \hat{x}, \hat{y} \} \in \mcb S }   \big| G_s\big( \tfrac{\hat{y}}{n}\big) - G_s\big( \tfrac{\hat{x}}{n}\big) \big| p_{\gamma}(\hat{y}-\hat{x}) \lesssim n^{\gamma-1} = r_n^\gamma,
	\end{split}
\end{equation}
where $r_n^\gamma$ is given in \eqref{rngamma}. 

In order to obtain \eqref{contrlip2} and \eqref{contrslow2}, we state some useful bounds in what follows. From \eqref{defbG} and an application of Lemma \ref{lemind} for $j=d$ and $q=-\gamma-d$, it is not hard to prove that for every $k \in \mathbb{N}$, it holds
\begin{align} \label{claim1AGzw}
	\forall G \in \mathcal{S}_c^{0}(\mathbb{R}^{d\star}),    \; \forall z_d \neq w_d \in \mathbb{Z}, \quad A_{z_d, w_d}^{G,n,k} \lesssim n^{\gamma-1} |w_d - z_d|^{-\gamma-1}. 
\end{align}
In the last line, for every $G \in \mathcal{S}^{0}(\mathbb{R}^{d\star})$, for every $n \geq 1$, for every $k \in \mathbb{N}$ and every $z_d, w_d \in \mathbb{Z}$, $A_{z_d, w_d}^{G,n,k}$ is given by $A_{z_d, w_d}^{G,n,k}: = B_{z_d, w_d}^{G,n,k} + B_{w_d, z_d}^{G,n,k}$, where
\begin{equation} \label{defAGnkzw}
	B_{z_d, w_d}^{G,n,k}:=  \frac{n^{\gamma}}{n^d}  \sup_{s \in [0,T]} \sum_{\hat{z}_{*} \in \mathbb{Z}^{d-1} } \; \sum_{\hat{w}_{*} \in \mathbb{Z}^{d-1} } \big| G_s \big( \tfrac{\hat{w}_{*}}{n}, \tfrac{w_d}{n}   \big) - G_s \big( \tfrac{\hat{z}_{*}}{n}, \tfrac{z_d}{n}   \big) \big|^k p_{\gamma} (\hat{w}_{*} - \hat{z}_{*}, \; w_d - z_d  ).
\end{equation}
From \eqref{defAGnkzw}, \eqref{claim1AGzw} and \eqref{defbG}, it is not hard to prove that
\begin{equation} \label{claim2AGzw}
	\forall G \in \mathcal{S}_c^{0}(\mathbb{R}^{d\star}), \quad \frac{n^{\gamma}}{n^{d} } \sup_{s \in [0,T]} \sum _{ \{ \hat{x}, \hat{y} \} \in \mcb S }   \big| G_s\big( \tfrac{\hat{y}}{n}\big) - G_s\big( \tfrac{\hat{x}}{n}\big) \big|^k p_{\gamma}(\hat{y}-\hat{x}) \lesssim 1 + \sum_{x_d= 0}^{\infty} \sum_{y_d=-b_G n +1}^{- 1}  A_{x_d, y_d}^{G,n,k},
\end{equation}
for every $d \geq 1$ and every $k \in \mathbb{N}_+$. 
Similarly, if $\gamma \in (1, \; 2)$ from \eqref{defAGnkzw} and \eqref{claim1AGzw}, we get
\begin{equation} \label{claim2AGzwb}
	\forall G \in \mathcal{S}^{0}(\mathbb{R}^{*}), \quad \frac{n^{\gamma}}{n } \sup_{s \in [0,T]} \sum _{ \{ \hat{x}, \hat{y} \} \in \mcb S }   \big| G_s( \tfrac{\hat{y}}{n}) - G_s( \tfrac{\hat{x}}{n}) \big|^k p_{\gamma}(\hat{y}-\hat{x}) \lesssim 1 + \sum_{x_d= 0}^{n-1} \sum_{y_d=- n +1}^{- 1}  A_{x_d, y_d}^{G,n,k},
\end{equation}
for every $k \in \mathbb{N}_+$. Now from \eqref{defbG}, \eqref{GSHoldgen} and Lemma \ref{lemind}, we observe that if $G \in \mathcal{S}_c^0(\mathbb{R}^{d\star})$ satisfies \eqref{timelip}, then
\begin{align} \label{claim3AGzw}
	\forall n \geq 1, \; \forall z_d \neq w_d \in \mathbb{Z}, \; \forall \delta \in [0,1], \quad A_{z_d, w_d}^{G,n,1} \lesssim n^{\gamma-1-\delta} |w_d - z_d|^{-\gamma-1 + \delta}. 
\end{align}
Now we are ready to present the proof of \eqref{contrlip2}.
\begin{proof}[Proof of \eqref{contrlip2}]
	From the Mean Value Theorem and \eqref{defnabladelta}, we observe that every $G \in \mathcal S_{\gamma,\star}^{d}$ satisfies \eqref{timelip}.
	Next, we treat two cases differently: $G \in \mathcal S_{\gamma,\star}^{d} \cap \mathcal{S}_c^0(\mathbb{R}^{d\star})$ and $G \in \mathcal S_{\gamma,\star}^{d} \setminus \mathcal{S}_c^0(\mathbb{R}^{d\star})$. 
	
	\textbf{I)} The case $G \in \mathcal S_{\gamma,\star}^{d} \cap \mathcal{S}_c^0(\mathbb{R}^{d\star})$. Note that $G$ satisfies \eqref{timelip}. Applying \eqref{claim3AGzw} for $\delta = \gamma/2$, the rightmost term in \eqref{claim2AGzw} is bounded from above by a constant times
	\begin{align*}
		\frac{1}{n^2}\sum_{y_d= 1}^{b_G n} \sum_{x_d=0}^{\infty}  \Big(  \frac{x_d+y_d}{n} \Big)^{-\gamma/2 - 1}  \lesssim \int_0^{b_G}   \int_{0}^{\infty} (  u + v )^{-\gamma/2 - 1} \; \rmd v \; \rmd u  \lesssim \int_0^{b_G} u^ {-\gamma/2} \; \rmd u  \lesssim 1.
	\end{align*}
	Combining the last display with \eqref{claim2AGzw} for $k=1$, we get \eqref{contrlip2}. This ends the proof for this case.
	
	\textbf{II)} The case $G \notin \mathcal{S}_c^0(\mathbb{R}^{d\star})$. From \eqref{defSalgamd}, we get that $\gamma \in (1, \;2)$ and $d=1$. Next, observe that the rightmost term in \eqref{claim2AGzwb} can be rewritten as
	\begin{align*} 
		& \frac{n^{\gamma}}{n} \sum_{x=0}^{n-1} \sum_{y=-n+1}^{-1} \big[ \sup_{s \in [0,T]}   \big| G_s\big( \tfrac{y}{n}\big) - G_s\big( \tfrac{x}{n}\big) \big| p_{\gamma}(y - x) + \sup_{s \in [0,T]}   \big| G_s\big( \tfrac{x}{n}\big) - G_s\big( \tfrac{y}{n}\big) \big| p_{\gamma}(x - y) \big] .
	\end{align*}
	Since $G \in \mathcal S_{\gamma,\star}^{d}$, $G$ satisfies \eqref{timelip} and therefore, also satisfies \eqref{GSHoldgen}. Applying \eqref{GSHoldgen} for $\delta = \gamma/2$, we conclude that the last display is bounded from above by a constant depending on $G$, times
	\begin{align*}
		  \frac{1}{n^2} \sum_{x=0}^{n-1} \sum_{y=1}^{n-1}  \Bigg( \frac{x+y}{n} \Bigg)^{-\gamma/2-1}  \lesssim  \int_0^{1} \int_0^{1} ( u +v )^{-\gamma/2 -1} \; \rmd  u \; \rmd v   \lesssim 1.
	\end{align*}
	Combining the last display with \eqref{claim2AGzwb} for $k=1$, we get \eqref{contrlip2}. This ends the proof.
\end{proof}
We end this subsection by presenting the proofs of \eqref{contrslow2} and \eqref{contrslowtight}.
\begin{proof}[Proofs of \eqref{contrslow2} and \eqref{contrslowtight}]
	We treat two cases differently: $G \in \mathcal{S}_c^0(\mathbb{R}^{d\star})$ and $G \notin \mathcal{S}_c^0(\mathbb{R}^{d\star})$. 
	
	\textbf{I)} The case $G \in \mathcal{S}_c^0(\mathbb{R}^{d\star})$. From \eqref{claim1AGzw}, for every $k \in \mathbb{N}$, the rightmost term in \eqref{claim2AGzw} is bounded from above by a constant times
	\begin{align} \label{motrngamma}
		\frac{1}{n}\sum_{y_d= 1}^{b_G n} \frac{1}{n} \sum_{x_d=0}^{\infty} \Bigg(  \frac{x_d+y_d}{n} \Bigg)^{-\gamma - 1}  \lesssim  \frac{1}{n}\sum_{y_d= 1}^{b_G n}  \int_{0}^{\infty} \Bigg(  \frac{y_d}{n} + v \Bigg)^{-\gamma - 1} \rmd v \lesssim \frac{1}{n}\sum_{y_d= 1}^{b_G n} \Bigg(  \frac{y_d}{n}  \Bigg)^{-\gamma } \lesssim r_n^\gamma.
	\end{align}
	Combining \eqref{motrngamma} with \eqref{claim2AGzw} for $k=1$, we get \eqref{contrslow2}. Moreover, combining \eqref{motrngamma} with \eqref{claim2AGzw} for $k=2$, we obtain \eqref{contrslowtight}. This ends the proof for this case.
	
	\textbf{II)} The case $G \notin \mathcal{S}_c^0(\mathbb{R}^{d\star})$. From \eqref{defSalgamd}, we get that $\gamma \in (1, \;2)$ and $d=1$, thus 
	\begin{align*}
		\frac{n^{\gamma}}{n^{d} } \sup_{s \in [0,T]} \sum _{ \{ \hat{x}, \hat{y} \} \in \mcb S }   \big| G_s\big( \tfrac{\hat{y}}{n}\big) - G_s\big( \tfrac{\hat{x}}{n}\big) \big|^k p_{\gamma}(\hat{y}-\hat{x}) \leq \frac{n^{\gamma}}{n^{1} } (2 \nnorm{G}_{0,0} )^k \sum_{r \in \mathbb{Z}} |r| p_{\gamma}(r) \lesssim n^{\gamma - 1}.
	\end{align*}
	The sum over $r$ in the last line converges, since $\gamma >1$. For $k=1$, the last display is equivalent to \eqref{contrslow2}. Moreover, applying last display for $k=2$, we get \eqref{contrslowtight} and the proof ends. 
\end{proof}

\subsection{Proof of Propositions \ref{boundnltight} and \ref{boundnl} } \label{useest2}

First, we state an estimate which is analogous to \eqref{claim1AGzw}, but holds when $w_d$ is \textit{equal to} $z_d$. We observe from \eqref{defAGnkzw} and \eqref{transition prob} that $A_{z_d, z_d}^{\cdot, \cdot,\cdot} \equiv 0$ if $d=1$. More generally, due to \eqref{defAGnkzw} and \eqref{transition prob}, it holds  
\begin{align} \label{claim4AGzw}
\forall d \in \mathbb{N}_+, \;	\forall G \in \mathcal{S}_c^{0}(\mathbb{R}^{d}),    \; \forall z_d \in \mathbb{Z}, \quad A_{z_d, z_d}^{G,n,k} \lesssim n^{\gamma-2}, 
\end{align}
for every $k \in \mathbb{N}_+$. In order to obtain \eqref{claim4AGzw} for $d \geq 2$, we applied \eqref{GSHold} for $\delta=1/k$ and Lemma \ref{lemind} for $j=d-1$ and $q=\-\gamma-d+k \delta$.

Next, we present the proof of Proposition \ref{boundnltight}.
\begin{proof}[Proof of Proposition \ref{boundnltight}]
	From \eqref{defAGnkzw}, \eqref{defbG}, \eqref{claim1AGzw} and \eqref{claim4AGzw}, it is not hard to prove that for every for every $k \in \mathbb{N}_+$, we have
	\begin{equation} \label{claim5AGzw}
	\forall d \in \mathbb{N}_+, \; \forall G \in \mathcal{S}_c^{0}(\mathbb{R}^{d}),  \quad \frac{n^{\gamma}}{n^{d} } \sup_{s \in [0,T]} \sum _{ \hat{x}, \hat{y} }   \big| G_s\big( \tfrac{\hat{y}}{n}\big) - G_s\big( \tfrac{\hat{x}}{n}\big) \big|^k p_{\gamma}(\hat{y}-\hat{x}) \lesssim \max \Bigg\{ 1, \; \frac{n^{\gamma}}{n} \Bigg\}.
	\end{equation}
In order to obtain \eqref{claim5AGzw}, we applied \eqref{claim3AGzw} for $\delta=\gamma/2$. Now, combining \eqref{defschw} with the arguments used in the proof of Lemma \ref{lemF1F2}, it is not hard to show that for every $k \in \mathbb{N}_+$, 
	\begin{equation} \label{claim5AGzwb}
		\forall \gamma \in (1, \; 2), \forall G \in \mathcal{S}^{1}(\mathbb{R}),  \quad \frac{n^{\gamma}}{n } \sup_{s \in [0,T]} \sum _{ x, y }   \big| G_s\big( \tfrac{y}{n}\big) - G_s\big( \tfrac{x}{n}\big) \big|^k p_{\gamma}(y-x) \lesssim \max \Bigg\{ 1, \; \frac{n^{\gamma}}{n} \Bigg\}.
	\end{equation}
	
	In particular, Proposition \ref{boundnltight} is a direct consequence of \eqref{claim5AGzw} and \eqref{claim5AGzwb}.
\end{proof}
Next, we will be interested in applying Taylor expansions on $G \in \mathcal S_{\gamma,0}^{d}$. In order to do so, keeping in mind \eqref{defGdisc2}, we use the convention that
\begin{equation} \label{taylorGdisc}
	\begin{split}
		\partial_{i}  G_s (\hat{u} ):=& \mathbbm{1}_{\{ u_d < 0 \}} \partial_{i} G^{-}_s(\hat{u}) + \mathbbm{1}_{ \{ u_d \geq 0\}} \partial_{i} G^{+}_s(\hat{u}), \quad \hat{u} \in \mathbb{R}^d, \\
		\partial_{ij}   G_s (\hat{u} ):=& \mathbbm{1}_{\{ u_d < 0 \}} \partial_{ii}  G^{-}_s(\hat{u}) + \mathbbm{1}_{ \{ u_d \geq 0\}} \partial_{ii}  G^{+}_s(\hat{u}), \quad \hat{u} \in \mathbb{R}^d,
	\end{split}
\end{equation}
for every $s \in [0, \; T]$ and $i \in \{1, \ldots, d\}$. In last display, $G^{-}$ and $G^{+}$ are given in \eqref{defGdisc2}.

Now we present the proof of Proposition \ref{boundnl}. 
\begin{proof}[Proof of Proposition \ref{boundnl}]
	From the statement of Proposition \ref{boundnl}, there are two possibilities:
	\begin{itemize}
		\item
		$G \in S_{\gamma}^d \subsetneq \mathcal S_{\gamma}^d  \subsetneq \mathcal{S}^2(\mathbb{R}^d)$: then $G_s \in C^2(\mathbb{R}^d)$ for any $s \in [0,T]$. 
		\item
		$G \notin S_{\gamma}^d$: then $j \in \{1, \ldots, d-1\}$. This means that for every $\hat{u} \in \mathbb{R}^d$, the set $\{\hat{u} - r \hat{e}_j, r \in [0, \; 1]\}$ is either contained in $\mathbb{R}^{d-1} \times (- \infty, 0 )$ (where $G \equiv G^{-}$), or either contained in $\mathbb{R}^{d-1} \times [0, \infty)$ (where $G \equiv G^{+}$).  
	\end{itemize}
	In particular, using the convention \eqref{taylorGdisc}, we apply a second-order Taylor expansion on $G$ along the direction determined by $\hat{e}_j$, for every $\hat{u} \in \mathbb{R}^d$. Proceeding in this way, it holds  
	\begin{align} \label{errtaylor}
		 \forall n \in \mathbb{N}_+, \; \forall s \geq 0, \quad \exists \chi^n_{s, \hat{z}} \in [0, 1]: \quad G_s \big( \tfrac{\hat{z}+\hat{e}_j}{n} \big)- \nabla^{n}_j G_s \big( \tfrac{\hat{z}}{n} \big) = \frac{1}{n}  \partial_{j}  G_s \big( \tfrac{\hat{z}}{n} \big)  + \frac{1}{2n^2} \partial_{jj}   G_s \Big( \tfrac{\hat{z}+ \chi^n_{s, \hat{z}}}{n} \Big), 
	\end{align}
for any $\hat{z} \in \mathbb{Z}^d$. In particular, the sum over $\hat{x}$ and $\hat{y}$ in the statement of Proposition \ref{boundnl} is bounded from above by
	\begin{align*}
		& \frac{n^{\gamma}}{n^{d+2}} \sum_{\hat{x}} \sup_{s \in [0,T] } \Big| \partial_{jj}  G_s \Big( \tfrac{\hat{x}+ \chi^n_{s, \hat{x}}}{n} \Big) \Big|  +  \frac{n^{\gamma}}{n^{d+2}} \sum_{ \hat{y}} \sup_{s \in [0,T] } \Big| \partial_{jj}   G_s \Big( \tfrac{\hat{y}+ \chi^n_{s, \hat{y}}}{n} \Big) \Big|    \\
		+& \frac{n^{\gamma}}{n^{d+1}} \sum_{\hat{x}, \hat{y}} \sup_{s \in [0,T] } \big| \partial_{j}  G_s \big( \tfrac{\hat{y}}{n} \big) - \partial_{j}  G_s \big( \tfrac{\hat{x}}{n} \big) \big| p_{\gamma}( \hat{y} - \hat{x} ) . 
	\end{align*}
	In the first line of last display, we used the fact that $\sum_{\hat{z}} p_{\gamma}(\hat{z})=1$. Next, we observe that from \eqref{defF1}, \eqref{semnorm} and Lemma \ref{lemF1F2}, it is not very hard to show that
	\begin{align}
		\forall \gamma \in (0, \; 2), \; \forall d \geq 1, \; \forall G \in \mathcal S_{\gamma,0}^{d}, \quad \frac{n^{\gamma}}{n^{d+2}} \sum_{\hat{z}} \sup_{s \in [0,T] } \Big| \partial_{jj}  G_s \Big( \tfrac{\hat{z}+ \chi^n_{s, \hat{z}}}{n} \Big) \Big|   \lesssim n^{\gamma - 2} \lesssim \frac{r_n^\gamma}{\gamma}. \label{2ordnl}
	\end{align}
	On the other hand, from \eqref{claim5AGzw}, \eqref{claim2AGzwb} and \eqref{contrslow2}, it is not very hard to prove that
	\begin{align}
		\forall \gamma \in (0, \; 2), \; \forall d \geq 1, \; \forall G \in \mathcal S_{\gamma,0}^{d}, \quad \frac{n^{\gamma}}{n^{d+1}} \sum_{\hat{x}, \hat{y}} \sup_{s \in [0,T] } \big| \partial_{j}  G_s \big( \tfrac{\hat{y}}{n} \big) - \partial_{j}  G_s \big( \tfrac{\hat{x}}{n} \big) \big| p_{\gamma}( \hat{y} - \hat{x} ) \lesssim \frac{r_n^\gamma}{\gamma}. \label{1ordnl}
	\end{align}
	The proof ends by combining \eqref{2ordnl} and \eqref{1ordnl}.			
\end{proof}

\subsection{Proof of Proposition \ref{boundYn}}	  \label{useest3}
We end this section by presenting the proof of Proposition \ref{boundYn}.
\begin{proof}[Proof of Proposition \ref{boundYn}.]
Combining \eqref{defnabladelta} with the Mean Value Theorem, we get, for every $\gamma \in (0,2)$, every $d \geq 1$, every $H \in \mathcal S_{\gamma}^{d}$ and every $n \in \mathbb{N}_+$, that
\begin{align}
	\frac{1}{n^d}\sum_{x_d\geq0}E_{x_d, 0}^{n,\gamma}(\nabla_dH)  +\frac{1}{n^d}\sum_{x_d\geq0}E_{x_d, -1}^{n,\gamma}(\nabla_dH) 
	&\lesssim \frac{r_n^\gamma}{n}, \label{auxYn12a} 
	\\
	\frac{1}{n^d}\sum_{x_d\leq-1} E_{x_d, -1}^{n,\gamma}(\nabla_dH) 
	+ \frac{1}{n^d}\sum_{x_d\leq-1} E_{x_d, 0}^{n,\gamma}(\nabla_dH)  
	&\lesssim \frac{r_n^\gamma}{n}. \label{auxYn12b}
\end{align}
Next, combining the last display with \eqref{defGdisc} and Proposition \ref{boundnl}, we have for any $G \in S_{\gamma,0}^{d}$, that $Y^{n, G}_{\mcb F}$ is bounded from above by some constant, times $f_n' r_n^\gamma/n$, where $Y^{n, G}_{\mcb F}$ is given in \eqref{defYn1G}. 

It remains to treat $Y^{n, G}_{\mcb S}$: from its definition in \eqref{defYn2G}, this term is bounded from above by
	\begin{align}
		& 	2 \alpha_n f_n' \frac{n^{\gamma}}{n^d} 
		\sum_{x_d=1}^{\infty}
		\sum_{y_d=-\infty}^{-1} 
		\sum_{  \hat{x}_{\star}, \hat{y}_{\star} \in \mathbb{Z}^{d-1}  }  
		\big|  
			\nabla_d G^{-}_s \big(\tfrac{\hat{y}}{n}\big) - \nabla_d G^{+}_s\big(\tfrac{\hat{x}}{n}\big) 
		\big|  
		p_{\gamma}(\hat{y}-\hat{x}) \label{Yn2a}
		 \\
		+ & 
		\alpha_n f_n' \Bigg[ \sum_{x_d\leq-1}  A_{x_d, 0}^{G,n,1} 
		+   
		\sum_{x_d\geq0}   A_{x_d, -1}^{G,n,1}\Bigg], \label{Yn2b}
	\end{align}	
	where $G^{-}, G^{+}$ in \eqref{Yn2a} are elements of $S_{\gamma}^{d} \subsetneq \mathcal S_{\gamma}^{d}$, given by \eqref{defGdisc}. From \eqref{errtaylor}, \eqref{2ordnl} and \eqref{1ordnl}, the term in \eqref{Yn2a} is bounded from above by a constant, times $\alpha_n f_n' r_n^\gamma/n$. 

If $G \in S_{\gamma,\star}^{d} \subsetneq \mathcal S_{\gamma,\star}^{d} $, by choosing $\delta=\gamma/2$  in \eqref{claim3AGzw} and applying arguments analogous to the ones used to obtain \eqref{contrlip2}, we have that the term in \eqref{Yn2b} is bounded from above by a constant, times $\alpha_n f_n' n^{\gamma/2}/n$.

On the other hand, if $G \in  S_{\gamma,0}^{d} \setminus  S_{\gamma,\star}^{d}$, by applying \eqref{contrslow2}, we conclude that the term in \eqref{Yn2b} is bounded from above by a constant, times $\alpha_n f_n' r_n^\gamma$. This ends the proof. 		
\end{proof}

\section{Measure Theory and Analysis tools} \label{antools}		

\subsection{Uniqueness of weak solutions} \label{uniqweak}

In this subsection we prove Lemma \ref{lemuniq}. We begin by stating a result which is a consequence of \eqref{deflapfracreg} and the Dominated Convergence Theorem (DCT), following exactly the same arguments as in the proof of Proposition B.3 in \cite{CGJ2}. For the remainder of this subsection, unless it is specified otherwise, $\mathcal{O}$ is one of the three elements of $\{ \mathbb{R}^d,  \mathbb{R}_{-}^{d\star}, \mathbb{R}_{+}^{d\star}\}$.
\begin{prop} \label{fracintpart}
Let $\gamma \in (0, 2)$, $d \geq 1$ and $G \in  C_c^{\infty}(\mathcal{O}) \subset \mcb{H}^{ \gamma / 2}(\mathcal{O})$. Assume that $\rho \in \mcb{H}^{ \gamma / 2}( \mathcal{O} ) $ is bounded. Then,
\begin{align} \label{intpart}
 \int_{ \mathcal{O} } \rho( \hat{u} )  \Delta_{ \mathcal{O} }^{ \gamma / 2 } G  (\hat{u}) \; \rmd \hat{u} = - \frac{c_{\gamma}}{2}    \iint_{\mathcal{O}^2} \frac{[ G(\hat{u}) - G(\hat{v}) ] [ \rho(\hat{u}) - \rho(\hat{v}) ] }{|\hat{u}-\hat{v}|^{d + \gamma}} \rmd \hat{u} \; \rmd \hat{v}.  
\end{align}
\end{prop}
Next we present a result whose prove is a direct consequence of Theorem 1.4.2.4 in \cite{grisvard} and \eqref{normsob}, by following the strategy applied to obtain Proposition B.7 in \cite{CGJ2}.
\begin{prop} \label{densgamma01}
Let $\gamma \in (0, 1]$, $d \geq 1$ and $\mathcal{O} \in \{\mathbb{R}_{-}^{d\star},\mathbb{R}_{+}^{d\star}\}$. Then  $C_{c}^{\infty}(\mathcal{O})$ is dense in $\mcb {H}^{ \gamma / 2}(O)$.
\end{prop}
Next we define $\mathcal{P}_{\mathcal{O}}$ as the space of functions $G: [0,T] \times \mathcal{O} \mapsto \mathbb{R}$ such that there exist $k \in \mathbb{N}$ and $G_0, \ldots, G_k \in C_c^{\infty}(\mathcal{O})$ satisfying $G(s, \hat{u} )= \sum_{j=0}^{k} s^{j} G_j(t, \hat{u})$, for every $(s, \hat{u}) \in [0,T] \times \mathcal{O}$. From Theorem 7.38 in \cite{sobolevadams}, $\big( C_c^{\infty}(\mathbb{R}^d), \| \cdot \|_{\mcb{H}^{ \gamma / 2}(\mathbb{R}^d)} \big)$ is dense in $\big( \mcb{H}^{ \gamma / 2}(\mathbb{R}^d), \| \cdot \|_{\mcb{H}^{ \gamma / 2}(\mathbb{R}^d)} \big)$, for every $\gamma \in (0, 2)$. Combining this with Proposition \ref{densgamma01} and Proposition 23.2 (d) in \cite{MR1033497}, we can state a corollary of these results.
\begin{cor} \label{cordens1}
$\mathcal{P}_{\mathcal{O}}$ is dense in $L^2\big(0,T; \mcb{H}^{\gamma / 2}(\mathbb{R}^d) \big)$, for every $\gamma \in (0,2]$. Moreover, for every $\gamma \in (0,1]$, $\mathcal{P}_{\mathcal{O}}$ is dense in $L^2\big(0,T; \mcb{H}^{ \gamma / 2}( \mathcal{O} ) \big)$ for $\mathcal{O} \in \{ \mathbb{R}_{-}^{d\star}, \mathbb{R}_{+}^{d\star}\}$.
\end{cor}
For every $H: [0, T] \times \mathcal{O} \mapsto \mathbb{R}$, define $\mathcal{I}(H): [0, T] \times \mathcal{O} \mapsto \mathbb{R}$ by
\begin{equation} \label{intoper}
\mathcal{I}(H)(t,u):= \int_t^{T} H(s,u) \;\rmd s, \quad (t,u) \in [0, T] \times \mathcal{O}.
\end{equation}
Combining \eqref{intoper} with the Cauchy-Schwarz inequality and Fubini's Theorem in the same way as it was done to obtain Lemma B.9 in \cite{CGJ2}, we get the following result. 
\begin{lem} \label{lemuniqsdif}
Assume that $\gamma \in (0, \;2)$, $d \geq 1$ and $\mathcal{O}_1, \mathcal{O}_2 \in \{ \mathbb{R}^d,  \mathbb{R}_{-}^{d\star}, \mathbb{R}_{+}^{d\star}\}$ are such that $\mathcal{O}_1 \subset \mathcal{O}_2$. Let $\rho \in L^2 \left( 0,T; \mcb{H}^{ \gamma / 2}(\mathcal{O}_2) \right)$ be a bounded function and $(H_k)_{k \geq 1}$ be a sequence of functions in $\mathcal{P} \big([0,T], C_c^{\infty}(\mathcal{O}_2) \big)$ such that $\lim_{k \rightarrow \infty} \|H_k - \rho \|_{L^2\big(0,T; \mcb{H}^{ \gamma / 2}(\mathcal{O}_2) \big)}=0$. For every $k \geq 1$, define $G_k$ by $G_k:=\mathcal{I}(H_k)$. Finally, let $\widetilde{\rho} \in L^2 \left( 0,T; L^{ 2}(\mathcal{O}_2) \right)$. Then,
\begin{align*}
&\lim_{k \rightarrow \infty} \int_0^T \int_{\mathcal{O}_1} \widetilde{\rho}(s, \hat{u}) \partial_s G_k (s, \hat{u}) \; \rmd \hat{u} \;\rmd s = - \int_0^T \int_{\mathcal{O}_1}  \widetilde{\rho} (s, \hat{u}) \rho(s, \hat{u})  \rmd \hat{u} \;\rmd s < \infty, \\ 
&\lim_{k \rightarrow \infty}  \int_0^T \int_{\mathcal{O}_1} \rho(s, \hat{u})  \Delta_{\mathcal{O}_1}^{ \gamma / 2} G_k  (s, \hat{u}) \; \rmd \hat{u} \;\rmd s =  -\frac{c_{\gamma}}{4}  \int \int_{\mathcal{O}_1^2}  \frac{ \big\{  \int_0^T  [ \rho(s, \hat{u})  - \rho(s, \hat{v}) ]\;\rmd s \big\}^2  }{|\hat{u}-\hat{v}|^{d + \gamma}} \rmd \hat{u} \; \rmd \hat{v}. 
\end{align*}
\end{lem}
We end this subsection by presenting the proof of Lemma \ref{lemuniq}. 
\begin{proof}[Proof of Lemma \ref{lemuniq}]

Let us prove that \eqref{eqhydfracbetazero} has at most one weak solution whenever $\gamma \in (0,2)$ and $\kappa \in \mathbb{R}_+ \setminus \{1\}$. If $\rho_1,\rho_2$ are two weak solutions of \eqref{eqhydfracbetazero}, then $\rho_3 \in L^2 \big(0, T ; \; L^2( \mathbb{R}^d ) \big)$ and $\rho_4  \in L^2 \big(0, T ; \; \mcb{H}_{ \alpha }^{ \gamma/ 2 } ( \mathbb{R}^d ) \big)$, where $\rho_3$ and $\rho_4$ are defined by
\begin{align*}
\rho_3:= [\rho_1 - \theta_1] - [\rho_2 - \theta_1], \quad \rho_4 :=  [F (\rho_1) - \theta_2] - [F (\rho_2) - \theta_2],
\end{align*}
and $\theta_1, \theta_2$ are given in the items (1), (2) of Definition \ref{defehsd}. Now, combining arguments analogous to the ones in the proof of Proposition C.1 in \cite{renato} with \eqref{intoper}, Corollary \ref{cordens1} and Lemma \ref{lemuniqsdif}, we conclude that
\begin{align} \label{olein}
	\int_0^T \int_{\mathbb{R}^d } [ \rho_1(s, \hat{u}) - \rho_2(s, \hat{u}) ] \big[F \big( \rho_1(s, \hat{u}) \big) -  F \big( \rho_2(s, \hat{u}) \big) \big] \; \rmd \hat{u} \;\rmd s \leq 0.
\end{align}
Since $F$ is increasing, the inequality in \eqref{olein} is actually an equality, and $\rho_1 \equiv \rho_2$ almost everywhere on $[0,T] \times \mathbb{R}^d$. This ends the proof of the first item in Lemma \ref{lemuniq}. The proof for the remaining items is analogous. 
\end{proof}

\subsection{Proof of Lemma \ref{new:lemestfrac}} \label{prolemestfrac}

\begin{proof}[Proof of Lemma \ref{new:lemestfrac}]

Let $O$ be an open subset of $\mathbb{R}^d$ and $\delta >0$, as in the statement of Lemma \ref{new:lemestfrac}. We denote the Hilbert space $L^2([0,T] \times O \times O)$ by $Y(O)$. Thus, Theorem 4.13 of \cite{brezis2010functional} leads to the separability of the metric space $\big( Y(O), \; \| \cdot \|_{Y(O)} \big)$. From this, it is not very hard to prove that the metric space $(C_c^{0,0} ( [0,T] \times O^2), \; \| \cdot \|_{Y_{\delta}(O)} )$, given in Definition \ref{def:OY-spaces}, is also separable. Thus, we can fix a sequence $(G_i)_{i \geq 1} \subset C_c^{0,0} ( [0,T] \times O^2)$ such that $(H_i)_{i \geq 1}$ is dense in $(C_c^{0,0} ( [0,T] \times O^2), \| \cdot \|_{Y_{\delta}(O)} )$.

From Proposition \ref{Qabscont} and \eqref{convabsF}, $\mathbb{Q}-$ almost surely, we have that $|\ell^{O_\delta}_{F(\rho)}(G_1)| \leq f_\infty \| \sqrt{G_1} \|^2_{Y_{\delta}(O)}$. Thus, from the Monotone Convergence Theorem and \eqref{new:assumsobpos1}, we have that
	\begin{align} \label{boundsupseq}
		 \mathbb{E}_{\mathbb{Q}} \Big[ \sup_{i \geq 1} \big\{ \ell^{O_\delta}_{F(\rho)}(G_i)  - \kappa_0 \big\|  G_i \big\|_{Y_\delta(O)}^2 \big\} \Big] \leq \kappa_1, 
	\end{align}
where $\kappa_0$ and $\kappa_1$ are given in \eqref{new:assumsobpos1}. The next step is to replace the supremum in the last display by an analogous supremum, but over $H \in C_c^{0,0} ( [0,T] \times O^2)$. 

Combining H\"older's and Jensen's inequalities with \eqref{sphe} and Proposition \ref{Qabscont}, we get
	\begin{equation} \label{ineqleps}
		 \forall H \in N(O), \quad \mathbb{Q} \Bigg( \big| \ell^{O_\delta}_{F(\rho)}(H) \big| \leq \Big( \kappa_2   \delta^{-\gamma/2} \big\| F( \rho) - F ( N_{ \text{e} } \theta ) \big\|_{ Y(O) } \| H \|_{Y_\delta(O)} \Big)=1,
	\end{equation}
for some constant $\kappa_2$ independent of $\delta$. Moreover, from Jensen's inequality and \eqref{new:assumneu}, we get that there exists a positive constant $\kappa_3$ independent of $\delta$ such that $\mathbb{E}_{\mathbb{Q}} \big[ \big\| F( \rho) - F ( N_{ \text{e} } \theta ) \big\|_{Y(O)} \big]  \leq \kappa_3$. By combining this with the density of $(G_i)_{i \geq 1}$ in $(C_c^{0,0} ( [0,T] \times O^2), \| \cdot \|_{Y_{\delta}(O)} )$, \eqref{ineqleps} and \eqref{boundsupseq}, it is not hard to prove that
	\begin{equation} \label{boundsupF}
		\forall \delta >0, \quad \mathbb{E}_{\mathbb{Q}} \Big[ \sup_{H \in N(O)} \big\{\ell^{O_\delta}_{F(\rho)}(H) - \kappa_4(\delta) \| H \|^{2}_{Y_\delta(O)} \big\}  \Big] \leq \kappa_5(\delta),
	\end{equation}
	where $\kappa_4(\delta):=\kappa_0(1 + 2 \delta^{\gamma/2})$ and $\kappa_5(\delta):=\kappa_1 + \kappa_2 \kappa_3  + \kappa_0 \delta^{\gamma/2} (\delta^{\gamma/2} + 2)$. Furthermore, for every $\delta >0$ \textit{fixed}, we get from \eqref{boundsupF} that $\mathbb{Q}(E_{\delta})=1$, where $E_{\delta}$ is the event where the linear functional $\ell^{(k, \delta)}_{\rho, O}: C_c^{0,0} ( [0,T] \times O^2) \mapsto \mathbb{R}$ is continuous.

Thanks to Corollary 4.23 in \cite{brezis2010functional}, the metric space $(C_c^{0,0} ( [0,T] \times O^2), \; \| \cdot \|_{Y(O)})$ is dense in $( Y(O), \; \| \cdot \|_{Y(O)})$. From this, it is not hard to prove that for every $\delta >0$ \textit{fixed}, the metric space $(C_c^{0,0} ( [0,T] \times O^2), \; \| \cdot \|_{Y(O, \delta)})$ is dense in $\big( Y_{\delta}(O), \; \| \cdot \|_{Y_\delta(O)} \big)$. Combining this with the fact that $\mathbb{Q}(E_{\delta})=1$, \eqref{ineqleps} and \eqref{boundsupF}, it is not very hard to prove that
	\begin{equation} \label{boundsupY}
		\forall \delta >0, \quad \mathbb{E}_{\mathbb{Q}} \Big[ \sup_{H \in Y(O, \delta)} \big\{\ell^{O_\delta}_{F(\rho)}(H) - \kappa_6(\delta) \| H \|^{2}_{Y_\delta(O)} \big\}  \Big] \leq \kappa_7(\delta),
	\end{equation}
	where $\kappa_6(\delta):=\kappa_4(\delta)(1 + 2 \delta^{\gamma/2})$ and $\kappa_7(\delta):=\kappa_5(\delta) + \kappa_2 \kappa_3  + \kappa_4(\delta) \delta^{\gamma/2} (\delta^{\gamma/2} + 2)$. Furthermore, from \eqref{boundsupF}, Riesz's Representation Theorem and \eqref{boundsupY}, we get that
	\begin{align*}
		\forall \delta >0, \quad  \mathbb{E}_{\mathbb{Q}} \Bigg[  \int_0^T \iint_{ O_{\delta} } \frac{  \big[  F(\rho_s(\hat{v}))-F(\rho_s(\hat{u}) )  \big] ^2}{| \hat{u} - \hat{v}|^{d+\gamma}} \; \rmd \hat{u} \; 
		\rmd \hat{v} \; \rmd s    \Bigg]   \leq 4 \kappa_6(\delta) \kappa_7(\delta).
	\end{align*}
	Finally, taking $\delta \rightarrow 0^{+}$, we get from the Monotone Convergence Theorem that \eqref{new:fracsobdpos} holds, and the proof ends.
\end{proof}

\subsection{Results regarding the the topology of $\mcb D_{\mcb {M}^{+}_{N_{\text{e}}}}([0,T])$} \label{sectopsko}

In this subsection, we present the proof for Propositions \ref{weakconv1} and \ref{weakconv2}; and for Lemma \ref{openport4}. We begin with a classical result, which is proved in Lemma B.0.7 in \cite{dismestotav}.
\begin{lem} \label{lemotasko}
	Let $(N, \; \| \cdot \|_N)$ be a metric space, $ \big( \phi(s) \big)_{s \in [0, T]} \in \mcb{D}_N([0,T])$ and $\Big( \big( \phi_j(s) \big)_{s \in [0, T]} \Big)_{j \geq 1} \subset \mcb{D}_N([0,T])$ a sequence converging to $\big( \phi(s) \big)_{s \in [0, T]}$ with respect to the Skorohod topology. Then, 
	\begin{align}
		& \lim_{j \rightarrow \infty} \| \phi_j(0) - \phi(0) \|_{N}=0, \quad \lim_{j \rightarrow \infty} \| \phi_j(T) - \phi(T) \|_{N}=0, \label{convfint} \\
		& \lim_{j \rightarrow \infty} \| \phi_j(s) - \phi(s) \|_{N}=0, \quad \text{for almost every} \quad s \in [0, \; T]. \label{convmidt}
	\end{align}
\end{lem}
Next, we prove Propositions \ref{weakconv1} and \ref{weakconv2}.
\begin{proof}[Proof of Propositions \ref{weakconv1} and \ref{weakconv2}]
	Since $\mathbb{Q}$ is a limit point of $(\mathbb{Q}_n)_{n \geq 1}$, there exists a subsequence $(\mathbb{Q}_{n_k})_{k \geq 1}$ such that $(\mathbb{Q}_{n_k})_{k \geq 1}$ converges weakly to $\mathbb{Q}$. From Portmanteau's Theorem, it holds
	\begin{align} \label{portcont}
		\mathbb{E}_{\mathbb{Q}} [ \Psi (\pi_{\cdot} ) ] = \lim_{k \rightarrow \infty } \mathbb{E}_{\mathbb{Q}_{n_k}} [  \Psi (\pi_{\cdot} )  ] \leq \varlimsup_{ n \rightarrow \infty } \mathbb{E}_{\mathbb{Q}_{n}} [  \Psi (\pi_{\cdot} )  ], 
	\end{align}
	whenever $\Psi: \mcb D_{\mcb {M}^{+}_{N_{\text{e}}}}([0,T]) \mapsto \mathbb{R}$ is bounded and continuous. Next, fix $j \geq 1$, $G_1, \ldots, G_j \in C_c^{0,0}([0,T] \times \mathbb{R}^d )$ and $c_1, \ldots, c_j \in \mathbb{R}$. For every $i \in \{1, \ldots, j\}$, define $\Psi_i: \mcb D_{\mcb {M}^{+}_{N_{\text{e}}}}([0,T]) \mapsto \mathbb{R}$ by
	\begin{equation*}
		\Psi_i (\widetilde{\pi}_{\cdot} ):= c_i + \int_0^T \int_{\mathbb{R}^d}  G_i(s, \hat{u})   \; \widetilde{\pi}_t( \rmd \hat{u} ) \;\rmd s, \quad  \widetilde{\pi}_{\cdot} \in \mcb D_{\mcb {M}^{+}_{N_{\text{e}}}}([0,T]).
	\end{equation*} 
	It is not very hard to prove that $\Psi_i$ is  bounded, resp. continuous, by combining the triangle inequality and \eqref{defMcb}, resp. combining \eqref{convmidt} and the DCT. This holds for any $i \in \{1, \ldots, j\}$.
	Hence, $\max_{1 \leq i \leq j} \{ \Psi_i  \}: \mcb D_{\mcb {M}^{+}_{N_{\text{e}}}}([0,T]) \mapsto \mathbb{R}$ is also bounded and continuous. Thus, from Proposition \ref{Qabscont} and \eqref{portcont}, we obtain Proposition \ref{weakconv1}.

In order to obtain Proposition \ref{weakconv2}, fix $k \geq 2$ and $\varepsilon \in (0, \; 1/2)$. Next, for every $i \in \{1, \ldots, j\}$, define $\widetilde{\Psi}_i: \mcb D_{\mcb {M}^{+}_{N_{\text{e}}}}([0,T]) \mapsto \mathbb{R}$ by
	\begin{equation*}
		\widetilde{\Psi}_i (\widetilde{\pi}_{\cdot} ):= c_i + \int_0^T \int_{\mathbb{R}^d} G_i(s, \hat{u})  \prod_{r=0}^{k-1} \langle \widetilde{\pi}_s, \; \overrightarrow{ \widetilde{\iota}_{\varepsilon}^{ \hat{u} + r \varepsilon \hat{e}_d } }\rangle  \; \rmd \hat{u} \;\rmd s, \quad  \widetilde{\pi}_{\cdot} \in \mcb D_{\mcb {M}^{+}_{N_{\text{e}}}}([0,T]), 
	\end{equation*} 
where $\widetilde{\iota}_{\varepsilon}^{ \cdot }$ is given in Definition \ref{def:mol-iota}. It is not very hard to prove that $\widetilde{\Psi}_i$ is  bounded, resp. continuous, by combining Definition \ref{def:mol-iota}, the triangle inequality and \eqref{defMcb}, resp. combining \eqref{convmidt} and the DCT. This holds for any $i \in \{1, \ldots, j\}$.
	Hence, $\max_{1 \leq i \leq j} \{ \tilde{\Psi}_i  \}: \mcb D_{\mcb {M}^{+}_{N_{\text{e}}}}([0,T]) \mapsto \mathbb{R}$ is also bounded and continuous. Thus, from Proposition \ref{Qabscont} and \eqref{portcont}, we obtain Proposition \ref{weakconv2}.
\end{proof}
We end this subsection by presenting the proof for Lemma \ref{openport4}. 
\begin{proof}[Proof of Lemma \ref{openport4}]
Fix $\varepsilon_0 \in (0, \;1)$, an integer $m \geq 2$ and $G, H, H^2, \ldots, H^{m}  \in C_c^{0,0}([0,T] \times \mathbb{R}^d)$. 

Let $\widetilde{\pi}^*_{\cdot} \in \mcb D_{\mcb {M}^{+}_{N_{\text{e}}}}([0,T])$ and $(\widetilde{\pi}^k_{\cdot})_{k \geq 1} \subset \mcb D_{\mcb {M}^{+}_{N_{\text{e}}}}([0,T])$ be such that $(\widetilde{\pi}^k_{\cdot})_{k \geq 1}$ converges to $\widetilde{\pi}^*_{\cdot}$ with respect to the Skorohod topology. Furthermore, define $\widetilde{F}: [0, \; T] \mapsto \mathbb{R}$ by	
\begin{align*} 
\widetilde{F}(t)
:=&   
	\int_{\mathbb{R}^d}
	 H_t (\hat{u}) 
	 \tilde{\pi}_t( \rmd \hat{u} ) 
	 - \int_{\mathbb{R}^d} 
	 H_0( \hat{u}) 
	 \tilde{\pi}_0( \rmd \hat{u} ) 
	 - \int_0^t \int_{\mathbb{R}^d} 
	 G_s( \hat{u}) 
	 \tilde{\pi}_s( \rmd \hat{u} )\;\rmd s 
\\
-& \frac{1}{d} \sum_{k=2}^{ m } 
	\int_0^t  \int_{\mathbb{R}^d}  
	H^k_s(\hat{u})     \sum_{j=1}^d \prod_{i=0}^{k-1} 
	\int_{\mathbb{R}^d} \overrightarrow{\widetilde{\iota}_{\varepsilon_0}^{ \hat{u} + i \varepsilon_0 \hat{e}_j } } ( \hat{v})  
	\widetilde{\pi}_s( \rmd \hat{v} ) 
	\; \rmd \hat{u} \; \rmd s.
\end{align*} 
		Combining \eqref{defMcb} and Definition \ref{def:mol-iota} with the uniform continuity of $H$ and the right-continuity of $\widetilde{\pi}^*_{\cdot}$, it is not very hard to prove that
	\begin{equation} \label{rightF1F2}
	\forall s_0 \in [0, \; T), \quad \lim_{s \rightarrow s_0^+} \widetilde{F}(s) = \widetilde{F}(s_0).
	\end{equation}
	From the definition of $\Psi$ in Lemma \ref{openport4}, we have $\Psi( \widetilde{\pi}^*_{\cdot} ) = \sup_{s \in [0, \; T]} | \widetilde{F}(s) |$. Thus, for every $\varepsilon >0$, there exists $s_0 \in [0, \; T]$ such that $\big| \Psi( \widetilde{\pi}^k_{\cdot} ) - | \widetilde{F}(s_0)| \big| < \varepsilon$. If $s_0 = T$, we have from Definition \ref{def:mol-iota}, Lemma \ref{lemotasko} and the DCT that
		\begin{align*}
			\Psi( \widetilde{\pi}^{\star}_{\cdot} ) < |\widetilde{F}(s_0)| + \varepsilon = |\widetilde{F}(T)| + \varepsilon   \leq \varliminf_{k \rightarrow \infty} \Psi( \widetilde{\pi}^k_{\cdot} ) + \varepsilon.
		\end{align*}		
				On the other hand, if $s_0 \in [0, \; T)$, we get from  Lemma \ref{lemotasko}, \eqref{rightF1F2}, Definition \ref{def:mol-iota} and the DCT that
		\begin{align*}
			\Psi( \widetilde{\pi}^{\star}_{\cdot} ) < |\widetilde{F}(s_0)| + \varepsilon < \varliminf_{k \rightarrow \infty} \Psi( \widetilde{\pi}^k_{\cdot} )+ 2 \varepsilon.
		\end{align*}
	Since $\varepsilon >0$ is arbitrary, we conclude that $\Psi( \widetilde{\pi}^*_{\cdot} ) \leq \varliminf_{k \rightarrow \infty} \Psi( \widetilde{\pi}^k_{\cdot} )$. This ends the proof.
\end{proof}

In the next subsection we prove Proposition \ref{Qabscont}.

\subsection{Proof of Proposition \ref{Qabscont}} \label{secabscont}
\begin{proof}[Proof of Proposition \ref{Qabscont} ]	
	Let $\mathbb{Q}$ be a limit point of $( \mathbb{Q}_n)_{n \geq 1}$. From \eqref{bndeta}, Lemma \ref{lemotasko} and Portmanteau's Theorem, it is not very hard to prove that
	\begin{equation} \label{prop311ota}
		\forall G \in C_c^{0}(\mathbb{R}^d), \quad  \mathbb{Q} \Bigg(  \pi_{\cdot}:  \quad \sup_{s \in [0, \; T]} | \langle \pi_s, \; G \rangle | \leq  N_{\text{e}} \int_{\mathbb{R}^d} | G (\hat{u}) | \; \rmd \hat{u} \Bigg)   =1.
	\end{equation} 
Next, we fix $g_0 \in C_c^{\infty} ( \mathbb{R} )$ such that $0 \leq g_0 \leq 1$ on $\mathbb{R}$ and
	\begin{align} \label{propg0}
		g_0(u)=1 \quad \text{if} \quad  |u| \leq 1 \quad \text{and}  \quad g_0(u)=0 \quad \text{if} \quad |u| \geq 2.
	\end{align}
	Now, for every $r \in \mathbb{N}_+$, define $f_r \in C_c^0(\mathbb{R}^d)$ by
	\begin{equation}  \label{uryf}
		f_r(\hat{u})=1 \quad \text{if} \quad | \hat{u}| \leq r - 1 \quad \text{and}  \quad f_r(\hat{u})=g_0\big( | \hat{u}| - (r-1)  \big) \quad \text{if} \quad | \hat{u}|  > r-1.
	\end{equation}
	In particular, $f_r(\hat{u})=0$ when $|\hat{u}| \geq r+1$ and $f_r \in C_c^{\infty}(\mathbb{R}^d)$, for every $r \in \mathbb{N}_+$. Now, let $(\phi_k)_{k \geq 1}$ be a dense sequence in $\big(C_c^0(\mathbb{R}^d), \| \cdot \|_{\infty} \big)$; and for every $j, k \in \mathbb{N}_+$, define $\widetilde{f}_{j,k} \in C_c^0(\mathbb{R}^d)$ by $\widetilde{f}_{j,k}(\hat{u}):=f_j(\hat{u}) \phi_k(\hat{u})$, for any $\hat{u} \in \mathbb{R}^d$. Next, define $j_0: C_c^0(\mathbb{R}^d) \mapsto \mathbb{N}_+$ and $k_0: C_c^0(\mathbb{R}^d) \times  \mathbb{N}_+ \mapsto \mathbb{N}_+$ by
	\begin{align*}
		j_0(G):=& \min \big\{ j \geq 1:\; \forall \hat{u}  \in  \mathbb{R}^d: | \hat{u}| \geq j, \quad  G (\hat{u}) =0 \big\}, \quad G \in C_c^0(\mathbb{R}^d); \\
		k_0(G, \; i):= & \min \Bigg\{ j \geq 1:\;  \| \phi_{j} - G \|_{\infty} <  \frac{1}{2i  N_{\text{e}} [2 j_0(G)+2]^d}  \Bigg\}, \quad (G, \; i) \in C_c^0(\mathbb{R}^d) \times \mathbb{N}_+.
	\end{align*}
	Fix $i \geq 1$. From \eqref{bndeta}, we get that $\mathbb{Q}_n ( B_i ) =1$ for any $n \geq 1$, where 
	$B_i \subset \mcb D_{\mcb {M}^{+}_{N_{\text{e}}}}([0,T])$ is defined by
	\begin{align*}
		B_i:= \Bigg\{ \widetilde{\pi}_{\cdot}  :\; \forall G \in C_c^0(\mathbb{R}^d), \;   \sup_{s \in [0, \; T]} | \langle \widetilde{\pi}_s, \; G - \widetilde{f}_{j_0(G), k_0(G,i)} \rangle |  \leq \frac{1}{i} -  N_{\text{e}} \int_{\mathbb{R}^d} | G (\hat{u}) - \widetilde{f}_{j_0(G),k_0(G,i)}(\hat{u})| \; \rmd \hat{u} \Bigg\}.
	\end{align*}
	Combining Lemma \ref{lemotasko} with Portmanteau's Theorem, it is not hard to prove that $\mathbb{Q} (B_i)=1$, for any $i \geq 1$. Combining this with \eqref{prop311ota} and the triangle inequality, it is not hard to prove that
	\begin{equation} \label{lebota}
		\mathbb{Q} \Bigg(  \pi_{\cdot} :\;  \forall G \in C_c^0(\mathbb{R}^d), \quad \sup_{s \in [0, \; T]} | \langle \pi_s, \; G \rangle | \leq   N_{\text{e}}  \int_{\mathbb{R}^d} | G (\hat{u}) | \: \rmd \hat{u}  \Bigg)    =1.
	\end{equation}
Combining last display with	Urysohn's Lemma, the proof ends.
\end{proof}
In the next subsection we prove Proposition \ref{propaproxtest} and Lemma \ref{seplem}.
 
\subsection{ Proof of Proposition \ref{propaproxtest} and Lemma \ref{seplem}} \label{antoolstest}

In order to obtain Proposition \ref{propaproxtest}, we will apply the following result, whose proof is a direct consequence of \eqref{sphe}.
\begin{lem} \label{lemsupt}
Let $H \in \mathcal{S}^0(\mathbb{R}^d)$ and $\Phi(H): \mathbb{R}^d \mapsto \mathbb{R}$ be given by $\Phi(H)(\hat{u}):=\sup_{s \in [0, \; T]} | H_s(\hat{u})|$, for every $\hat{u} \in \mathbb{R}^d$. Then, $\Phi(H) \in L^1(\mathbb{R}^d)$.
\end{lem}
Next we state a result which is a generalization of the Stone-Weierstrass Theorem. 
\begin{lem} \label{lemstone}
Fix $K>0$ and $\widetilde{G} \in  C_K:= C^{1,2}([0,T] \times [-K, K]^d)$. Define $(H_k)_{k \geq 1} \subset C_K$ by
\begin{align*}
H_k(s, \hat{u}):=\sum_{\hat{\omega} \in \llbracket0, \; i \rrbracket^{d+1}} \widetilde{G} \big( \tfrac{s}{T}, \; \hat{v}(\hat{u})  \big) \binom{k}{ \omega_1 }   \frac{s^{\omega_1}(T-s)^{k-\omega_1}}{T^i}  \prod_{j=1}^{d}  \Bigg\{  \binom{k}{\omega_{j+1}} v_j( \hat{u} )^{ \omega_{j+1} }[1 - v_j( \hat{u} )]^{ k - \omega_{j+1} }  \Bigg\},
\end{align*}
for any $k \in \mathbb{N}_+$, where $v_j( \hat{u} )=(2K)^{-1}(u_j+K) \in [0, \;1]$ for any $j \in \{1, \ldots, d\}$. Then, $(H_k)_{k \geq 1}$, $(\partial_s H_k)_{k \geq 1}$, $(\partial_{i} H_k)_{k \geq 1}$ and $(\partial_{ij} H_k)_{k \geq 1}$  converge uniformly on $[0,T] \times [-K, K]^d$ to $\widetilde{G}$, $
\partial_s \widetilde{G}$, $\partial_{i} \widetilde{G}$ and $\partial_{ij}  \widetilde{G}$ respectively, as $k \rightarrow \infty$.  
\end{lem}
\begin{proof}
The proof is a direct consequence of the main result in \cite{berns}.
\end{proof}
Now we present the proof for Proposition \ref{propaproxtest}.
\begin{proof}[Proof of Proposition \ref{propaproxtest}]
Let $\kappa \geq 0$, $\gamma \in (0, \;2)$, $d \geq 1$. For any $H \in S_{\gamma}^{d}$, we have that $H$ and $\partial_s H$ belong to $\mathcal{S}^0(\mathbb{R}^d)$. In particular, from Lemma \ref{lemsupt}, $\max\{  \Phi(H), \Phi(\partial_s H)  \} ¸\in L^1(\mathbb{R}^d)$. Combining this with Proposition \ref{propL1alpha}, \eqref{L1test} holds for any $H \in S_{\gamma}^{d}$. Plugging this with \eqref{defGdisc}, we conclude that \eqref{L1test} is satisfied for any $G \in S_{\gamma,0}^{d}$.

Now, from \eqref{defSalgamd} and \eqref{defSdifgamma} we have that $S_{\gamma, \kappa}^{d} \subset S_{\gamma,0}^{d}$, thus \eqref{L1test} holds for any $G \in S_{\gamma, \kappa}^{d}$.

In order to obtain Proposition \ref{propaproxtest}, it is enough to combine Corollary 4.23 in \cite{brezis2010functional} with Lemma \ref{lemstone} and \eqref{sphe}.
\end{proof}

Finally we prove Lemma \ref{seplem}.

\begin{proof}[Proof of Lemma \ref{seplem}]
Let $\mcb P_0$ be the space of polynomials $P: [0, T] \times \mathbb{R}^d \mapsto  \mathbb{R}$ in the variables $s$, $u_1$, $\ldots$, $u_d$ with \textit{rational} coefficients. In particular, $\mcb P_0$ is a \textit{countable} set, thus in the remainder of this work we fix $(P_k)_{k \geq 1}$, an enumeration of the elements of $\mcb P_0$. Next, for every $k, r \in \mathbb{N}_+$, define $H_{k,r}: [0, T] \times \mathbb{R}^d \mapsto \mathbb{R}$ by $H_{k,r}(s, \hat{u}):= f_r(\hat{u}) P_k(s, \hat{u})$, for any $(s, \hat{u}) \in [0, T] \times \mathbb{R}^d$, where $(f_r)_{r \geq 1}$ is given by \eqref{uryf}. Thus, $\mcb P_1:=\{H_{k,r}, \; k, r \in \mathbb{N}_+ \} \subsetneq S_{\gamma}^{d}$ is a countable set. Now, define the space $\mcb P_2$ of functions $H: [0,T] \times \mathbb{R}^d \mapsto \mathbb{R}$ such that there exist $H^{-}, H^{+} \in \mcb P_1$ satisfying
			\begin{align*} 
		\forall (s, \hat{u}_{\star}, u_d) \in [0,T] \times \mathbb{R}^d, \quad H_s(\hat{u}_{\star}, u_d) =\mathbbm{1}_{\{ u_d < 0 \}} H^{-}_s(\hat{u}_{\star}, u_d) + \mathbbm{1}_{ \{ u_d \geq 0\}} H^{+}_s(\hat{u}_{\star}, u_d).
	\end{align*}
Then $\mcb P_2$ is countable. Now from the definition of $f_r$ in \eqref{uryf}, we get
\begin{equation} \label{defKg}
	\forall r \in \mathbb{N}_+, \quad \nnorm{f_r}_{0,2}  \leq  K_g:=1+(3d^{2}+2d) \nnorm{ g_0 }_{0,2},
\end{equation}	
where $g_0$ is given in \eqref{propg0} and $\nnorm{ \cdot }_{0,2}$ is given in Definition \ref{deftesconttime}. From the definition of $X$ in \eqref{defX}, $X \subset S_{\gamma,0}^{d}$. Then for every $G \in X$, let $G^{-}, G^{+} \in S_{\gamma}^{d}$ be given by \eqref{defGdisc}. Thus, the proof ends if we can show that there exist $P_{k_0^{-}}, P_{k_0^{+}} \in \mcb P_0$ such that 
\begin{equation} \label{condsep0}
\begin{cases}
P_{k_0^{-}}=P_{k_0^{+}}, \quad & \text{if} \quad X = S_{\gamma}^{d}, \\
\forall 
\hat{u} \in \mathbb{B}, \; \forall s \in[0,T], \quad P_{k_0^{-}}(s, \hat{u}) =P_{k_0^{+}}(s, \hat{u}), \quad & \text{if} \quad X = S_{\gamma, \star}^{d},
\end{cases}	
\end{equation}  
and $H^{-}$ and $H^{+}$ given by
\begin{equation} \label{defHpmsep}
 H^{-}:=H_{k_0^{-}, r_0} = \widetilde{f}_{r_0}  P_{k_0^{-}}, \quad H^{+}:=H_{k_0^{+}, r_0} = \widetilde{f}_{r_0}  P_{k_0^{+}},
\end{equation}
are elements of $\mcb P_1$ satisfying
\begin{align} \label{boundnorX}
	\max\{ \| H^{-} - G^{-} \|_X, \quad \| H^{+} - G^{+} \|_X \} < \varepsilon,
\end{align}
where $\| \cdot \|_X$ is given in the statement of Lemma \ref{seplem}. In \eqref{condsep0}, $\mathbb{B}$ is given by \eqref{SROB0}.

First, we treat the case where either $d \geq 2$, or $\gamma \in (0,1]$. From \eqref{defSdifgamma}, we have that $S^{d}_{\gamma} = C_c^{1, 2} \big( [0,T] \times  \mathbb{R}^d \big)$. In particular, $G^{-}$, $G^{+}$, $\partial_s G^{-}$ and $\partial_s G^{+}$ are elements of $\mathcal{S}_c^2(\mathbb{R}^d)$ , given in Definition \ref{deftesconttime}. Next, let $b^{\star}:=\max \{ b_{G^{-}}, b_{G^{+}}, b_{\partial_s G^{-}}, b_{\partial_s G^{+}}  \}>0$, where $b_{G^{-}}, b_{G^{+}}, b_{\partial_s G^{-}}, b_{\partial_s G^{+}}$ are given by \eqref{defbG}. Let $r_0:=\min \{r \in \mathbb{N}_+ : r \geq b^{\star} +1\}$. From Lemma \ref{lemstone}, there exist $P_{k_0^{-}}, P_{k_0^{+}} \in \mcb P_0$ such that
\begin{equation} \label{delta0bnd}
\sup_{s \in [0,T], \hat{u}: |\hat{u}| \leq r_0 +1 } \big\{ [\widetilde{G}^{-}_s]_2(\hat{u}) + [\partial_s \widetilde{G}^{-}_s]_0(\hat{u}) +  [\widetilde{G}^{+}_s]_2(\hat{u}) + [\partial_s \widetilde{G}^{+}_s]_0(\hat{u})  \big\} < \delta_0(\varepsilon),
\end{equation}
where $\widetilde{G}^{-} :=P_{k_0^{-}} -  G^{-}$, $\widetilde{G}^{+} :=P_{k_0^{+}} -  G^{+}$ and $\delta_0(\varepsilon)$ is given by
\begin{align} \label{defdelta0sep}
\delta_0(\varepsilon):= \big\{ (T+1) (1+\alpha) [C_2(\gamma,d) + C_6(\gamma,d) + 4 \pi^d ] (r_0+1)^{d+2} 8 d^{2} K_g \big\}^{-1} \varepsilon.
\end{align}
In the last line, $C_2(\gamma,d)$ and $C_6(\gamma,d)$ are given by \eqref{defC2gd} and \eqref{defC6gd}, respectively. Moreover, from Lemma \ref{lemstone} and Definition \ref{defschw2}, it is possible to choose $P_{k_0^{-}}, P_{k_0^{+}} \in \mcb P_0$ such that \eqref{condsep0} holds. Furthermore, from \eqref{uryf}, \eqref{delta0bnd} and \eqref{sphe}, we get
\begin{align} \label{boundnorXcompa}
\sup_{s\in[0,T]} \left\{	\|  \widetilde{H}^{-}_s \|_{1} + \|  \widetilde{H}^{+}_s \|_{1} \right\} + \int_0^T \big\{ \|	 \partial_s \widetilde{H}^{-}_s	\|_{1} + \|	 \partial_s \widetilde{H}^{+}_s	\|_{1} \big\} \rmd s < 4 \pi^d (T+1) (r_0+1)^d \delta_0(\varepsilon),
\end{align}
where $\widetilde{H}^{-}:=H^{-} - G^{-}$ and $\widetilde{H}^{+}:=H^{+} - G^{+}$. Now combining \eqref{defLalfgam} with Propositions \ref{propLqglobcomp} and \ref{propLqregcomp}, we have
\begin{align} \label{boundnorXcompb}
	 \int_0^T \Bigg\{  \frac{\|	 \mathbb{L}_{\alpha}^{\gamma} \widetilde{H}^{-}_s	\|_{1} + \|	 \mathbb{L}_{\alpha}^{\gamma} \widetilde{H}^{+}_s	\|_{1}}{T (1 +\alpha) (r_0+1)^{d+2}}  \Bigg\} \rmd s < [C_2(\gamma,d) + C_6(\gamma,d) ] [ \nnorm{\widetilde{H}^{-}}_{0,2}  + \nnorm{\widetilde{H}^{+}}_{0,2}]. 
\end{align}
In the last line, we applied the fact that $\max \{b_{\widetilde{H}^{-}}, b_{\widetilde{H}^{+}}  \} \leq b^{\star} +2 \leq r_0 +1$, which comes from \eqref{uryf}, \eqref{defbG} and the definition of $r_0$. Now from \eqref{defKg}, \eqref{defbG}, \eqref{defnabladelta} and Definition \eqref{deftesconttime}, we get that $\nnorm{\widetilde{H}^{-}}_{0,2}  + \nnorm{\widetilde{H}^{+}}_{0,2} \leq 8 d^{2} K_g \delta_0(\varepsilon)$. By combining this with \eqref{boundnorXcompb}, \eqref{boundnorXcompa} and \eqref{defdelta0sep}, we conclude that $H^{-}$ and $H^{+}$ given by \eqref{defHpmsep} indeed satisfy \eqref{boundnorX}. 

Next, we treat the case where $d =1$ and $\gamma \in (1,2)$. From \eqref{defSdifgamma}, we have that $S^{d}_{\gamma} = S^{1, 2} \big( [0,T] \times  \mathbb{R} \big)$. In particular, $G^{-}$, $G^{+}$, $\partial_s G^{-}$ and $\partial_s G^{+}$ are elements of $\mathcal{S}^2(\mathbb{R})$, given in Definition \ref{deftesconttime}. Thus, there exists $K_1 \geq 1$ such that
\begin{align} \label{delta1bnd}
	\sup_{s \in [0, T], \; u: |u| \geq K_1 } \big\{   (|u|^{4}+1) \big( [G^{-}_s]_2(\hat{u}) + [\partial_s G^{-}_s]_0(\hat{u}) +  [G^{+}_s]_2(\hat{u}) + [\partial_s G^{+}_s]_0(\hat{u}) \big) \big\} \leq \delta_1(\varepsilon),
\end{align}
where $\delta_1(\varepsilon)$ is given by
\begin{align} \label{defdelta1sep}
	\delta_1(\varepsilon):= \big\{ 96T (1 + \alpha ) [C_7(\gamma) + c_{\gamma} C_4(\gamma)]  K_g \}^{-1}\varepsilon.
\end{align}
In the last line, $C_4(\gamma)$ and $C_7(\gamma)$are given by \eqref{C4sch} and \eqref{defKg}, respectively. Moreover, from \eqref{L1test}, there exists $K_2 \geq K_1$ such that
	\begin{align} \label{delta2bnd}
	\int_{| u | \geq K_2} \Bigg\{ \int_0^T \big\{ | \partial_s G^{-}_s |  + | \partial_s G^{+}_s |  \big\} \; \rmd t +  \sup_{s \in [0, \; T]} \{ | G^{-}_s(u)  |  + | G^{+}_s(u)  | \} \Bigg\} \; \rmd u <  \delta_2(\varepsilon):=\frac{ \varepsilon }{24}.
\end{align}
Let $r_0:= \min \{r \in \mathbb{N}_+: r \geq K_2 \}$. From Lemma \ref{lemstone}, there exist $P_{k_0^{-}}, P_{k_0^{+}} \in \mcb P_0$ such that
\begin{equation} \label{delta3bnd}
	\sup_{s \in [0,T], \hat{u}: |u| \leq r_0 +1 } \big\{ [\widetilde{G}^{-}_s]_2(u) + [\partial_s \widetilde{G}^{-}_s]_0(u) +  [\widetilde{G}^{+}_s]_2(u) + [\partial_s \widetilde{G}^{+}_s]_0(u)  \big\} < \delta_3(\varepsilon),
\end{equation}
where $\widetilde{G}^{-} :=P_{k_0^{-}} -  G^{-}$, $\widetilde{G}^{+} :=P_{k_0^{+}} -  G^{+}$ and $\delta_3(\varepsilon)$ is given by
\begin{align} \label{defdelta3sep}
	\delta_3(\varepsilon):= \big\{ 96 (T+1) (1 + \alpha) ( r_0+1)^{4} [C_7(\gamma) + c_{\gamma} C_4(\gamma) +1]  K_g  \big\}^{-1} \varepsilon. 
\end{align}
Moreover, from Lemma \ref{lemstone} and Definition \ref{defschw2}, it is possible to choose $P_{k_0^{-}}, P_{k_0^{+}} \in \mcb P_0$ such that \eqref{condsep0} holds.

Next, we claim that $H^{-}$ and $H^{+}$ given by \eqref{defHpmsep} satisfy \eqref{boundnorX}. Indeed, from \eqref{uryf}, \eqref{delta2bnd} and \eqref{delta3bnd}, we get
\begin{align} \label{boundnorXnoncompa}
	\sup_{s\in[0,T]} \left\{	\|  \widetilde{H}^{-}_s \|_{1} + \|  \widetilde{H}^{+}_s \|_{1} \right\} + \int_0^T \big\{ \|	 \partial_s \widetilde{H}^{-}_s	\|_{1} + \|	 \partial_s \widetilde{H}^{+}_s	\|_{1} \big\} \rmd s < 8 [ \delta_2(\varepsilon) + (T+1)r_0 \delta_3(\varepsilon)],
\end{align}
where $\widetilde{H}^{-}:=H^{-} - G^{-}$ and $\widetilde{H}^{+}:=H^{+} - G^{+}$. Now combining \eqref{defLalfgam} with Propositions \ref{propLqglobschw} and \ref{propLqregschw}, we have
\begin{align} \label{boundnorXnoncompb}
	\int_0^T \Bigg\{  \frac{\|	 \mathbb{L}_{\alpha}^{\gamma} \widetilde{H}^{-}_s	\|_{1} + \|	 \mathbb{L}_{\alpha}^{\gamma} \widetilde{H}^{+}_s	\|_{1}}{T (1 +\alpha) [c_{\gamma} C_4(\gamma) + C_7(\gamma) ]}  \Bigg\} \rmd s < \nnorm{\widetilde{H}^{-}}_{0,2}  + \nnorm{\widetilde{H}^{-}}_{4,2} + \nnorm{\widetilde{H}^{+}}_{0,2} + \nnorm{\widetilde{H}^{+}}_{4,2}. 
\end{align}
Now from \eqref{defKg}, \eqref{defnabladelta} and Definition \eqref{deftesconttime}, we get that 
\begin{align*}
\nnorm{\widetilde{H}^{-}}_{0,2}  + \nnorm{\widetilde{H}^{-}}_{4,2} + \nnorm{\widetilde{H}^{+}}_{0,2} + \nnorm{\widetilde{H}^{+}}_{4,2} \leq 32 K_g [\delta_1(\varepsilon) + (r_0+1)^{4} \delta_3(\varepsilon) ].
\end{align*}
By combining the last display with \eqref{boundnorXcompb}, \eqref{boundnorXcompa}, \eqref{defdelta1sep} and \eqref{delta2bnd} and \eqref{defdelta1sep}, we conclude that $H^{-}$ and $H^{+}$ given by \eqref{defHpmsep} indeed satisfy \eqref{boundnorX}.   
\end{proof}

\subsection*{Acknowledgements}
P.C. was funded by the Deutsche Forschungsgemeinschaft (DFG, German Research Foundation) under Germany’s Excellence Strategy – EXC-2047/1 – 390685813. P.G. thanks Funda\c c\~ao para a Ci\^encia e Tecnologia FCT/Portugal for financial support through the projects UIDB/04459/2020 and UIDP/04459/2020. G.N. acknowledges financial support by the ANR grant MICMOV (ANR-19-CE40-0012) of the French National Research Agency (ANR).

\bibliographystyle{plain}
\bibliography{CGN24}
\vspace{1.cm}
\Addresses
\end{document}